\documentclass[11pt]{amsart}
\input{SHE-macros}
\usepackage{datetime}

\begin{document}

\title[Ergodicity and synchronization of KPZ]{Ergodicity and synchronization of \\ the Kardar-Parisi-Zhang equation}

\author[C.~Janjigian]{Christopher Janjigian}
\address{Christopher Janjigian\\ Purdue University\\  Department of Mathematics\\ 150 N University St. 
\\   West Lafayette, IN 47907\\ USA.}
\email{cjanjigi@purdue.edu}
\urladdr{http://www.math.purdue.edu/~cjanjigi}
\thanks{C.\ Janjigian was partially supported by the National Science Foundation grant DMS-2125961.}

\author[F.~Rassoul-Agha]{Firas Rassoul-Agha}
\address{Firas Rassoul-Agha\\ University of Utah\\  Mathematics Department\\ 155S 1400E\\   Salt Lake City, UT 84112\\ USA.}
\email{firas@math.utah.edu}
\urladdr{http://www.math.utah.edu/~firas}
\thanks{F.\ Rassoul-Agha was partially supported by National Science Foundation grants DMS-1811090 and DMS-2054630}
\thanks{Part of this material is based upon work supported by the National Science Foundation under Grant No.\ 1440140, while F.\ Rassoul-Agha was in residence at the Mathematical Sciences Research Institute in Berkeley, California, during the fall semester of 2021.}
\author[T.~Sepp\"al\"ainen]{Timo Sepp\"al\"ainen} \address{Timo Sepp\"al\"ainen\\ University of Wisconsin--Madison\\  Department of Mathematics\\ 480 Lincoln Drive \\  Madison, WI 53706\\ USA.}  \email{seppalai@math.wisc.edu}\urladdr{http://www.math.wisc.edu/~seppalai} \thanks{T.~Sepp\"al\"ainen was partially supported by National Science Foundation grants  DMS-1854619 and DMS-2152362 and by the Wisconsin Alumni Research Foundation.}  

\keywords{Busemann function, cocycle, continuum polymer, corrector, ergodic measure, hyperbolicity, infinite volume, Kardar-Parisi-Zhang equation, one force--one solution, random dynamical system, shape theorem, stationary distribution, stochastic homogenization, stochastic heat equation, synchronization}
\subjclass[2020]{60H15, 60K35, 60K37, 37H15, 37L55, 35R60}

\date{\usdate\today}

\begin{abstract}   
The Kardar-Parisi-Zhang (KPZ) equation on the real line is well-known to admit Brownian motion with a linear drift as a stationary distribution (modulo additive  constants). We show that these solutions are attractive, a result known as a one force--one solution (1F1S) principle or synchronization: the solution to the KPZ equation started in the distant past from an initial condition with a given slope will converge almost surely to a Brownian motion with that drift, which shows in particular that these invariant measures are totally ergodic. Our proof constructs the \emph{Busemann process} for the equation, which gives the natural jointly stationary coupling of all of these stationary solutions. Synchronization then holds simultaneously (on a single full probability event) for all but an at most  countable random set of asymptotic slopes. This set of exceptional slopes of instability for which synchronization fails is either almost surely empty or almost surely dense. Along the way, we prove a shape theorem which implies almost sure stochastic homogenization of the KPZ equation, for which the Busemann process gives the process of correctors. We also show that the forward and backward point-to-point and point-to-line continuum polymers converge to semi-infinite continuum polymers whose transitions are Doob transforms via Busemann functions of the transitions of the finite length polymers. 
\end{abstract}
\maketitle
\tableofcontents

\section{Introduction}

The Kardar-Parisi-Zhang (KPZ) equation  
	\begin{align}\label{KPZ}
	\partial_t \KPZ =\tfrac12\partial_{yy} \KPZ + \tfrac12(\partial_y \KPZ)^2 +  W,
	\end{align}
driven by space-time white noise $W$, first appeared in the physics literature in 1986 \cite{Kar-Par-Zha-86} as a prototypical model of the height interface  $\KPZ=\KPZ(t,y)$ of a growing surface in 1+1 dimensions.  It was  motivated by the expectation that a wide class of such models should exhibit universality, meaning that certain properly re-scaled statistics converge to model-independent limits. Over the intervening decades, analysis of the KPZ equation and related growth models has been an important source of new ideas in mathematics and physics.  Through many examples, the expectation of universality has been borne out, with the class now understood to encompass a wide variety of strongly interacting systems. See the recent surveys \cite{Cor-12,Qua-12,Qua-Spo-15,Hal-Tak-15}.

The analysis of how one should re-scale the solution of \eqref{KPZ} to see a non-trivial limit predates the introduction of the model, tracing back at least to Forster, Nelson, and Stephen in 1977 \cite{For-Nel-Ste-77}. As part of a broader study of randomly forced models in fluid dynamics, they undertook a dynamical renormalization group analysis of the closely related stochastic Burger's equation 
	\begin{align}\label{SBE}
	\partial_t u=\tfrac12\partial_{yy} u + \tfrac12\partial_y(u^2) +  \partial_y W. 
	\end{align}
 The connection between \eqref{KPZ} and \eqref{SBE} comes by ignoring the distributional nature of $W$ and differentiating formally  to see that if $h$ solves \eqref{KPZ}, then $u=\partial_y h$ should solve \eqref{SBE}.

 A basic observation in  \cite{For-Nel-Ste-77} is that one expects space-time white noise to be an invariant measure for \eqref{SBE}. This was proved  rigorously in \cite[Proposition B.2]{Ber-Gia-97}.  Correspondingly, Brownian motion (with drift) is invariant for the evolution of \eqref{KPZ} modulo an additive constant. Viewing \eqref{KPZ} as describing the free energy of a directed polymer model, this leads to the prediction of approximate local Brownianity.  As observed in \cite{Hus-Hen-Fis-85},  in 1+1 dimensions this suggests the KPZ scaling relation $2\chi = \xi$ for the free energy fluctuation exponent  $\chi$ and  the transverse path fluctuation exponent $\xi$. Combined with the KPZ scaling relation $\chi=2\xi-1$, this leads to the prediction $\chi = 1/3$ and $\xi = 2/3$  \cite{Kar-Par-Zha-86}.  Under this scaling  models in the class are expected to converge to universal limits described by the renormalization fixed point of the KPZ universality class. This fixed point was recently constructed in \cite{Mat-Qua-Rem-21,Dau-Ort-Vir-22-}. The convergence of one-point statistics of the KPZ equation started from the narrow wedge initial condition to the predicted Tracy-Widom limit was proven over a decade ago in \cite{Ami-Cor-Qua-11}, with process-level convergence of the KPZ equation recently shown in \cite{Vir-20-,Qua-Sar-20-}.

To reach these universal limits,   one typically centers and normalizes the height of the interface with  model-dependent (i.e., non-universal) terms, analogous  to the  mean and variance in the classical central limit theorem. 
The values of these non-universal quantities are predicted by the KPZ scaling theory  \cite{Kru-Mea-Hal-92, Spo-14} which describes the centering and scaling as functions of the   spatially-ergodic and temporally (increment-) stationary and ergodic measures of the model. These predictions underscore the importance of understanding the structure of stationary and ergodic distributions of growth models in the KPZ class in general and of the KPZ equation itself in particular. These topics are, in a sense to be described in more detail shortly, the main focus of the present paper. 

\subsection{Main contributions}
With the above context in mind, we briefly summarize our main contributions, before informally explaining their meaning in more detail and connecting our work to the rest of the literature.

\begin{enumerate}[label={\rm(\roman*)}, ref={\rm\roman*}]  
    \item We construct the \emph{Busemann process} of the KPZ equation, which provides the natural monotone and jointly stationary coupling of all the previously-known invariant measures modulo additive constants,  given by Brownian motions plus drift.  We prove that the finite dimensional marginals (in the drift parameter) of this coupling give the unique  couplings of Brownian motions with drift that are jointly stationary  and ergodic under the KPZ solution semi-group. 
    In particular, this resolves the conjecture that Brownian motion with drift is an ergodic (i.e.~extremal) stationary distribution for KPZ.
    \item  We show that for a fixed value of the conserved quantity $\lambda$ (asymptotic spatial slope at $\pm \infty$), synchronization by noise and a one force--one solution principle (1F1S) hold almost surely among all initial conditions with appropriate asymptotic slopes, with the pullback attractors provided by the Busemann process. We strengthen this to a \textit{quenched} 1F1S principle that considers all $\lambda$ 
    simultaneously in a typical realization of the driving white noise. In this setting we establish that synchronization and 1F1S 
    hold for all values $\lambda$ at which the Busemann process is continuous. We show that at the exceptional discontinuity values $\lambda$, there are at least two  pullback attractors. This random set of exceptional values is either empty or a countable dense subset of $\R$. As far as we are aware, this marks the first time that such a 1F1S principle has been proven for a stochastic Hamilton-Jacobi equation in a fully continuous and non-compact setting with rough forcing.
    
    \item We prove an almost sure locally uniform free energy density limit, known as a \textit{shape theorem}.  Almost sure stochastic homogenization of the KPZ equation follows as a corollary, with effective Hamiltonian $\overline{H}(p) = -1/24 + p^2/2.$ The (centered) Busemann process furnishes the associated stochastic process of correctors.
     \item We show that for any ergodic stationary distribution $\PMPinit$ for the KPZ equation modulo additive constants, there exists $\lambda>0$ so that $\PMPinit$ is the distribution of Brownian motion with drift $\lambda$ or else $\PMPinit$ is  supported on equivalence classes of continuous functions such that 
    \be\label{int100} \lim_{x\to-\infty}\frac{f(x)}{x}=-\lambda  \qquad\text{and}\qquad  \lim_{x\to\infty}\frac{f(x)}{x}=\lambda. \ee
    This implies a prediction implicit in the KPZ scaling theory \cite{Kru-Mea-Hal-92} (and more explicit in \cite{Spo-14}) that all spatially translation invariant and time ergodic stationary distributions of the KPZ equation are Brownian motions  with drift. It has also been conjectured that these are the  \textit{only} ergodic stationary distributions of the KPZ equation; see, for example, \cite[Remark 1.1]{Fun-Qua-15}. Settling this conjecture is now reduced to resolving the existence of an exceptional measure supported by functions of the  type in \eqref{int100}. This is Open Problem \ref{prob:Ergodic} in Section \ref{sec:OP}.
\end{enumerate}

\subsection{Ergodicity, one force--one solution, and pullback attraction}

Shortly after the early breakthroughs on the KPZ class in the physics literature, a group around Ya.~Sinai started a program on the ergodic theory of the forced Burgers equation and related stochastic Hamilton-Jacobi equations, beginning with the seminal work \cite{Sin-91}. That initial paper proved existence and uniqueness of stationary solutions to the (viscous) stochastic Burgers equation \eqref{SBE}, where the forcing  $\partial_x W$ is replaced by a term which is either a regular and periodic function of space and time or else periodic and regular in space and white in time.
It showed that solutions to the equations started from different initial conditions can be coupled to a process defined for all time which at each time level has marginal given by the stationary distribution in such a way that this process serves as a pullback attractor in the sense of Definition 9.3.1 in \cite{Arn-98}. 
The existence of a  unique globally defined  stochastic process which is measurable with respect to the history of the noise is commonly known as a \emph{one force--one solution} (1F1S) principle (see, e.g., the introduction of \cite{E-etal-00}).  Some authors (e.g., \cite{Bak-Kha-18}) include the pullback attractor property mentioned above as part of the definition.

There have been two main technical obstacles that require significant effort to overcome in the study of the ergodic theory, 1F1S principles, and pullback attraction in models of this type: working on non-compact spaces and working with rough noise. Both of these issues are present in our setting.  After \cite{Sin-91}, many subsequent works, e.g., \cite{Sin-93,Gom-etal-05,Dir-Sou-05,Bak-Kha-10,Bak-13,E-etal-00,Hoa-Kha-03,Kif-97,Dri-etal-22-}, proved similar results  for other viscous and inviscid Hamilton-Jacobi equations, including higher-dimensional ones, in compact or essentially compact settings and with noise more regular than what we consider. 

The first paper to prove  1F1S in a genuinely non-compact setting was \cite{Bak-Cat-Kha-14}, which studied the inviscid
 Burgers equation with a space-time Poissonian forcing. Several subsequent works \cite{Bak-16,Bak-Li-19,Dri-etal-22-} considered viscous and inviscid models in non-compact settings with what is known as kick forcing, i.e., forcing that has a product form and is Dirac at certain special (typically integer) times. Kick forcing   makes the model essentially semi-discrete, because the evolution between these pre-specified times is deterministic. In a similar sense, the Poissonian forcing in \cite{Bak-Cat-Kha-14} essentially pushes the model onto a (random) lattice.  In both of these cases, the induced discrete structure brings access to the many tools developed for lattice and semi-discrete growth models.
 
 Thanks in large part to recent advances in defining solutions to stochastic partial differential equations forced by rough noise, a handful of recent papers have made progress on models with rough forcing similar to the one we consider. Notably, \cite{Hai-Mat-18} proved ergodicity of the Brownian bridge measure and \cite{Ros-21-} proved a 1F1S principle for \eqref{KPZ} on the torus. Recently, \cite{Kni-Mat-22-} proved ergodicity of the open KPZ equation using the general results of \cite{Hai-Mat-18}. Ergodicity of a certain martingale problem formulation of \eqref{SBE} on the line was also recently shown in \cite{Gub-Per-20}, without considering the 1F1S or attractiveness questions.

 One particularly fruitful approach to proving 1F1S principles in some previous works, including \cite{Bak-Cat-Kha-14,Bak-16,Bak-Li-19}, follows the program pioneered by Newman and collaborators \cite{New-97,Lic-New-96,How-New-97,How-New-01} in the context of first-passage percolation, where straightness estimates for geodesics lead to directedness and coalescence, which essentially leads directly to the 1F1S principle.
This same method was also applied to other $1+1$ dimensional percolation models \cite{Cat-Pim-11,Cat-Pim-12,Cat-Pim-13,Fer-Pim-05,Wut-02}. 

\subsection{Busemann process}


We follow another 
method, also originally developed in the percolation and polymer literature  \cite{Dam-Han-14,Geo-etal-15,Geo-Ras-Sep-17-ptrf-1,Jan-Ras-20-aop,Jan-Ras-Sep-22-}. As far as we know, this approach is the first to prove a 1F1S principle in a fully continuous non-compact Hamilton-Jacobi equation like \eqref{KPZ} with rough forcing. Our method centers on (analogues of) quantities known as Busemann functions, which were originally introduced in the metric geometry literature by Busemann \cite{Bus-55}.

To describe the broad outlines of the argument, we begin with the observation that if $h$ satisfies  \eqref{KPZ}, then the asymptotic slopes at $\infty$ and $-\infty$,
\begin{align}
\lim_{x\to\infty}\frac{h(t,x)}{x}=\PMPslopep \qquad \text{ and } \qquad \lim_{x\to-\infty} \frac{h(t,x)}{x}=\PMPslopem, \label{eq:CCslope}
\end{align} are conserved by the dynamics of the equation. This is proven to hold for \eqref{KPZ} in the companion paper \cite{Alb-etal-22-spde-} and the existence of such asymptotic conserved quantities is typical for models of this type. Because of the existence of these conserved quantities, it is natural to expect that for each value of $\lambda=\PMPslopep=\PMPslopem$, there exists a unique stationary distribution with this asymptotic slope.  In the case of the KPZ equation \eqref{KPZ}, this corresponds to the natural prediction that the stationary distributions modulo additive  constants consist precisely of Brownian motions with linear drift; see also \cite[Remark 1.1]{Fun-Qua-15}.  The conjecture that this one-parameter family completely describes spatially \textit{translation invariant} stationary measures is one prediction implicit in the KPZ scaling theory, see \cite{Spo-14} and the arguments in \cite{Kru-Mea-Hal-92}.

An analogue of a Busemann function in this setting is given by the limit
\begin{align}\label{infBus}
   \Bus^\lambda(s,x,t,y) = \lim_{r\to -\infty} \bigl(h_r(t,y)-h_r(s,x)\bigr),
\end{align}
where $h_r$ solves \eqref{KPZ} on $[r,\infty)\times\R$  from the time-$r$ initial condition 
\begin{align}\label{infBusIC}
    h_r(r,z) = f(z) \qquad \text{ such that  }\qquad \lim_{\abs{z}\to\infty} \frac{f(z)}{z}=\lambda.
\end{align}
That is, the analogues of Busemann functions correspond precisely to the pullback attractors discussed above. We refer the reader to \cite{Jan-Ras-Sep-22-} for some discussion of the analogy to Busemann functions on the lattice in the setting of directed last-passage percolation, as well as some discussion of the connection to inviscid Hamilton-Jacobi equations. 

Our proof of   the limits \eqref{infBus} comes via a coupling constructed  from the known \cite{Ber-Gia-97,Fun-Qua-15} invariant measures given by Brownian motion with linear drift. The correct coupling  arises as a  weak   limit point of  Ces\`aro averages of the joint distribution of stationary solutions to \eqref{KPZ} coupled to the environment. This  produces a field of candidate Busemann functions on an extended space. These candidates are then shown to satisfy the limits \eqref{infBus} for a countable dense set of values of $\lambda$,  utilizing  the stochastic monotonicity of the associated semi-infinite polymer measures.

A consequence of our proof of  1F1S   and attractiveness   is a confirmation of the conjecture that the only ergodic and spatially translation-invariant stationary distributions (modulo constants) for the KPZ equation are Brownian motions with linear drift.  In particular we  verify the ergodicity of  these distributions under the KPZ evolution. Our results leave open the possibility of the existence of only one exceptional class of ergodic measures. These are measures supported on functions for which $\PMPslopep \in (0,\infty)$ and $\PMPslopem=-\PMPslopep$ in \eqref{eq:CCslope}. See Open Problem \ref{prob:Ergodic}.


Most  of the previously mentioned works which study the 1F1S principle prove the existence of analogues of the limit in \eqref{infBus} on an event of full probability which \textit{depends} on the value of $\lambda$. Because there are uncountably many such values, this leaves open the existence of  exceptional values of the conserved quantity for which the 1F1S principle fails. A central observation in \cite{Jan-Ras-20-aop,Jan-Ras-Sep-22-} (in the viscous and inviscid cases on the lattice, respectively) is that questions of this type are encoded into continuity properties of the \emph{Busemann process}. This  is the function-valued stochastic process obtained by extending the family $(\Bus^\lambda(s,x,t,y) : s,x,t,y \in \bbR)$, initially defined by \eqref{infBus} for a countable dense set of $\lambda$, to  left- or right-continuous  
processes of continuous functions $(\Bus^{\lambda\pm}(s,x,t,y) : s,x,t,y \in \bbR)$ indexed by $\lambda \in \bbR$. We construct this process and show that  1F1S  holds \textit{simultaneously} for all values of $\lambda$ for which $\Bus^{\lambda+}=\Bus^{\lambda-}$. We also show that in an appropriate sense, the 1F1S principle \emph{fails} off of this set if the complement is non-empty. 

Monotonicity of the Busemann process implies that the discontinuity set of $\lambda$ for  which $\Bus^{\lambda+}\neq\Bus^{\lambda-}$ is at most countable. Consequently,   the 1F1S principle holds simultaneously for all but at most countably many values of the conserved quantity $\lambda$. We show later that if this set of exceptional values of $\lambda$ for which the Busemann process is discontinuous is non-empty, then it is dense. We leave the question of whether or not this set is non-empty to future work. We do, however, note that a striking feature of \emph{all} previously studied models in the KPZ class for which exact computation is possible \cite{Fan-Sep-20,Bat-Fan-Sep-22-,Bus-Sep-Sor-22-,Sep-Sor-21-}, including both the lattice log-gamma polymer and the KPZ fixed point, is that the analogue of the Busemann process has in each case exhibited discontinuities. This implies that in those other settings, the 1F1S principle fails for a random, dense set of slopes.

\subsection{Solving the KPZ equation and our coupling}\label{sec:int:sol}
Direct well-posedness of an appropriately renormalized version of \eqref{KPZ} was shown only recently, 
first on the torus \cite{Hai-13,Hai-14,Gub-Imk-Per-15,Gub-Per-17} and then on the line \cite{Per-Ros-19}. 
For the case of \eqref{SBE} on the torus, see \cite{Gub-Per-17,Gub-Per-18-review}. For existence of solutions to \eqref{SBE} on the line see
\cite{Ber-Can-Jon-94}. Uniqueness of stationary energy solutions of \eqref{SBE}, as defined by \cite{Gon-Jar-10-},
was shown by \cite{Gub-Per-18}. 

While these methods define what it means to solve \eqref{KPZ} or \eqref{SBE}, the physically relevant notion of a solution to \eqref{KPZ} or \eqref{SBE} has been known for over thirty years and is given by what is called the Hopf-Cole solution. This solution to \eqref{KPZ} is defined by starting with the well-posed stochastic heat equation (SHE) with multiplicative white noise forcing,
 	\begin{align}\label{SHE}
	\partial_t\She=\tfrac12\partial_{yy} \She + \She  W,
	\end{align}
 and then defining $h(t,y) = \log \She(t,y)$ and $u(t,y) = \partial_y \log \She(t,y)$. These definitions agree with formal computations which ignore the distributional structure of $W$.  
 
 The Hopf-Cole solution arises as a limit of lattice and continuum models which lie in the KPZ class \cite{Ber-Gia-97,Alb-Kha-Qua-14-aop,Hai-Qua-18} and is the standard notion of solution to \eqref{KPZ} in both the physics and mathematics literature. See the surveys \cite{Cor-16,Cor-12,Hal-Tak-15,Qua-Spo-15,Qua-12}, the references therein, the discussion in the introduction of \cite{Hai-13}, and indeed \cite[equation (2)]{Kar-Par-Zha-86}. For agreement of some of the direct definitions of solutions to \eqref{KPZ} with the Hopf-Cole solution under certain hypotheses, see \cite[Theorem 3.19]{Per-Ros-19}, \cite[Theorem 1.1]{Hai-13} (on the torus), and \cite[Theorem 2.10]{Gub-Per-18} (up to a non-random additive term).

The Hopf-Cole transformation connecting \eqref{KPZ} and \eqref{SHE} is central to our work. While it may be possible to generalize some aspects of our method to settings which do not connect to a linear equation like \eqref{SHE}, the full strength of our results uses the linearity of \eqref{SHE} in an essential way. This comes through   a coupling based on the superposition principle.  We can simultaneously study all solutions to \eqref{SHE} started from all initial conditions and all initial times and to prove that the resulting process satisfies strong continuity properties. The details of this part are in the companion paper \cite{Alb-etal-22-spde-} that  constructs and analyses the Green's function of \eqref{SHE}. Some aspects of this construction previously appeared in \cite{Alb-Kha-Qua-14-aop,Alb-Kha-Qua-14-jsp}. This coupling lies behind the uniformity of our results, which essentially all hold on a single event of full probability, simultaneously for all choices of all parameters of the model.

The coupling of solutions to \eqref{SHE} in \cite{Alb-etal-22-spde-} also allows us to construct a coupling of the full field of continuum directed polymer measures. These are the finite length polymer measures for which \eqref{SHE} describes the evolution of the partition function and \eqref{KPZ} describes the evolution of the free energy. As will be seen in the sequel, there is essentially an equivalence between stationary distributions (modulo constants) to \eqref{KPZ}, Busemann functions, and semi-infinite length polymer measures.  Semi-infinite length polymer measures are probability measures on semi-infinite paths,  with  distinguished start or end times, that  are consistent with the finite-length polymer measures in the sense of Gibbs conditioning. 

The equivalence between stationary distributions and Busemann functions more properly holds for generalized Busemann functions, which need not arise as limits. Such objects are typically known as either correctors, by analogy with  the corresponding objects in stochastic homogenization, or else as covariant recovering cocycles. See the discussion in \cite{Jan-Ras-20-aop,Jan-Ras-20-jsp}. In the present paper, stationary distributions modulo constants for \eqref{KPZ} are understood through their equivalence to ratio stationary solutions to \eqref{SHE}, which are easier to work with in our analysis. See the discussion in Section  \ref{subsec:ergodic}. In various guises, these equivalences between these different objects lie behind many of the existing approaches to studying the 1F1S principle for models of this type, including the approach tracing back to Newman's method, where semi-infinite geodesics or characteristics play the role of infinite length polymer measures.

Returning back to \eqref{KPZ}, we make one note about the scaling properties of the KPZ equation. In principle, one could consider the more general class of models \[\partial_t h =\tfrac\nu 2\partial_{yy} h + \tfrac\lambda 2(\partial_y h)^2 +  \beta W,\]
where $\nu>0$ and $\lambda,\beta\neq 0$ are free parameters. Our choice to restrict attention to the case where $\lambda=\nu=\beta=1$ is justified by the scaling relations of the KPZ equation, which we record in the following remark. The following computations are purely formal, but are nevertheless correct.
\begin{remark}\label{rem:scaling}
Given a strictly positive $\nu$ and nonzero $\lambda$ and $\beta$,  let $h$ solve \eqref{KPZ} and call 
\[\wt h(t,y)=h_{\nu,\lambda,\beta}(t,y)=\nu\lambda^{-1} h(\nu^{-5}\lambda^4\beta^4t,\nu^{-3}\lambda^2\beta^2y).\]
Then $\wt h$ solves
 	\[\partial_t \wt h=\tfrac\nu2\partial_{yy} \wt h + \tfrac\lambda2(\partial_y \wt h)^2 + \beta \wt W,\]
where $\wt W(t,y)=\nu^{4}\lambda^{-3}\beta^{-3}W(\nu^{-5}\lambda^4\beta^4t,\nu^{-3}\lambda^2\beta^2y)$ is a new space-time white noise. 
\end{remark}
Because of this scaling property, it is without loss of generality to restrict attention to \eqref{KPZ}.

\subsection{Integrable inputs and the shape theorem}
Our general approach applies to non-solvable models. We do, however, use inputs from integrable probability at one  particular stage of our argument. This occurs in our proof of a locally uniform version of the free energy density limit (Lemma \ref{lm:Z-grid}), which is known as a \textit{shape theorem} in the the percolation and polymer literature. 

It was shown in \cite{Ami-Cor-Qua-11} that if $h$ is started from the ``narrow-wedge" initial condition at the time-space point $(0,0)$, then, in probability,
\begin{align}
    \lim_{t\to\infty} \frac{h(t,ty)}{t} = -\frac{1}{24} - \frac{y^2}{2}.\label{eq:probshp}
\end{align}
There have been significant refinements of this result since then, which we do not survey. Our methods require a locally uniform almost sure extension of this limit, so we show the following almost surely for all $C>0$:
\begin{align}\label{eq:shapeintro}
\lim_{n\to\infty}n^{-1}\sup_{\substack{(s,x,t,y)\tspb\in\tspb\varsets: \\[2pt] s,x,t,y\tspb\in\tspb[-Cn,Cn]}}\,\Bigl|\,h(t,y\viiva s,x)+\frac{t-s}{24}-\log \heat(t-s,y-x)\,\Bigr|=0,
\end{align}
where $\heat(t,y)$ is the heat kernel,
\begin{align}
    \heat(t,y) &= \frac{1}{\sqrt{2\pi t}}\tspb  e^{-\frac{y^2}{2t}}\tspb\ind_{(0,\infty)}(t),\label{eq:heat}
\end{align}
$h(t,y\viiva s,x)$ denotes the Hopf-Cole
solution to KPZ \eqref{KPZ} started from the ``narrow-wedge" initial condition at $(s,x)$ and $\varsets = \{(s,x,t,y) \in \bbR^4 : s < t\}$. The coupling of solutions which makes this statement possible is the one   developed in the companion paper \cite{Alb-etal-22-spde-}. The quantity on the right-hand side of the limit in \eqref{eq:probshp} often goes by the name of (\emph{quenched}) \emph{Lyapunov exponent} or (\emph{limiting}) \emph{free energy density}. Once the existence of the limit is proved,  applying shear transformations to the KPZ equation \eqref{KPZ}  implies that the form of the right-hand side is $a_0 - y^2/2$ for some $a_0 \in \R$. Knowing this, without the explicit value $a_0=-1/24$,  is in fact more than sufficient for the arguments in this paper where we apply \eqref{eq:shapeintro} and its consequences. See Remark \ref{rk:shear}.

We know of no method of computing the precise value  $-1/24$ on the right-hand side of \eqref{eq:probshp} without  integrable probability inputs.  To prove  \eqref{eq:shapeintro}, we use  
 strong tail estimates from  \cite{Cor-Gho-Ham-21,Cor-Gho-20-ejp,Cor-Gho-20-duke}. 
 In all lattice and semi-discrete models we are aware of, integrability is not needed for the shape theorem to hold (with a potentially unknown centering). It may be possible to develop a purely stochastic analytic method of proving \eqref{eq:shapeintro} with the sharp $-1/24$ term replaced by an unknown constant. As noted above, such arguments exist on the lattice under extremely mild assumptions. See, for example, \cite{Jan-Nur-Ras-22}. 

\subsection{Stochastic homogenization}
If $U$ is a bounded and uniformly continuous function and $\e>0$, call $u_\e(t,x)=-\e \KPZ\bigl(t/\e,x/\e\bigr)$, where $\KPZ$ solves \eqref{KPZ} with $h(0,x)=-\e^{-1}U(\e x)$. Then $u_\e$ solves the viscous Hamilton-Jacobi 
equation
\[\partial_t u_\e(t,x)-\tfrac{\e}{2}\partial_{xx}u_\e(t,x)+H(t/\e,x/\e,\partial_x u_\e)=0, \qquad u_\e(0,x)=U(x),\] 
with Hamiltonian $H(t,x,p)=\frac{p^2}2+W(t,x)$. 
As a consequence of \eqref{eq:shapeintro}, almost surely simultaneously for all $U$ and $u_\e$ as above, $u_\e$ converges locally uniformly as $\e\to0$ to the viscosity solution of the \textit{effective} Hamilton-Jacobi equation
\[\partial_t\overline u+{\overline H}(\partial_x \overline u)=0, \qquad \overline{u}(0,x) = U(x)\]
with effective Hamiltonian ${\overline H}(p)=\frac{p^2}2-\frac1{24}$. This is stochastic homogenization in the sense of \cite{Jin-Sou-Tra-17,Kos-Var-08}. 
The Busemann process furnishes the associated stochastic process of correctors, with
\[v(t,x)=-\Bus^{-p}(0,0,t,x)-px+\Bigl(\frac{p^2}2-\frac1{24}\Bigr)t\] 
solving the corresponding  corrector equation (with velocity $p$),
\[\partial_t v-\frac12\partial_{xx}v+\frac12(p+\partial_x v)^2+W=\frac{p^2}2-\frac1{24}.\]
See Remark \ref{rem:corrector} for an informal derivation of this equation following the treatment in \cite{Lio-Sou-03}.

\subsection{Notation and conventions} \label{sec:notation}
The integers are $\bbZ$, $\bbZ_+$ $=\{0,1,2,$ $\dots\}$, $\bbN=\{1,2,\dots\}$, the real numbers in $d$ dimensions are $\bbR^d$, the rational numbers are $\bbQ^d$. $\R_- =(-\infty,0]$ and $\R_+=[0,\infty)$. $\Udense$ denotes an arbitrary fixed countable dense subset of $\R$.  $\abs A$ is the cardinality of a finite set $A$. 
Directed time-space domains in $\R^4$ are denoted by $\varsets = \{(s,x,t,y)\in \bbR^4 : s<t\}$  {and} 
$\varset = \{(s,x,t,y) \in \bbR^4 : s \leq t\}$.  

When $X$ and $Y$ are metrizable spaces, $\cC(X,Y)$ denotes the space of continuous functions from $X$ to $Y$. We equip $\cC(X,Y)$ with the topology of uniform convergence on compact sets and with the corresponding Borel $\sigma$-algebra. 

$\sB(X)$ is the Borel $\sigma$-algebra of the metrizable space $X$. $\cM_1(X)$ is the space of Borel probability measures on $X$, equipped with the topology of weak convergence.  Given a signed Borel measure $\mu$ on $X$, the total variation measure $|\mu|$ is given by the sum of the positive and negative parts of the Jordan-Hahn decomposition of $\mu$ (see Definition 3.1.4 in \cite{Bog-07}). The total variation distance between signed Borel measures $\mu$ and $\nu$ on $X$ is $\norm{\mu-\nu}_{\rm TV} = |\mu-\nu|(X)$. In general, $(1/2)\norm{\mu}_{\rm TV} \leq \sup_A|\mu(A)| \leq \norm{\mu}_{\rm TV}$.  

$\cM_+(\R)$ is the space of non-negative measures on the real line, which we equip with the vague topology (i.e., the test functions are compactly supported and continuous).

Paths are denoted by  $x_{r:t}=\{x_s:r\le s\le t\}$, $x_{-\infty:t}=\{x_s:s\le t\}$, and similarly for $x_{m:\infty}$ and $x_{-\infty:\infty}$. For $t\le s\le s'\le t'$ in $[-\infty,\infty]$, let  $\Paths_{s:s'}$ be the $\sigma$-algebra on $\cC([t,t'],\R)$ generated by $X_{s:s'}=\{X_r:s\le r\le s'\}$, where $X$ is the coordinate random variable.

Given $s<t$, a function $F:\cC([s,t],\R)\to\R$ is nondecreasing if $F(X_{s:t})\le F(Y_{s:t})$ whenever $X_r\le Y_r$ for all $r\in[s,t]$.
Given two probability measures $Q_1$ and $Q_2$ on $\cC([s,t],\R)$, $Q_1$ is stochastically dominated by $Q_2$, abbreviated $Q_1\std Q_2$, if $E^{Q_1}[F]\le E^{Q_2}[F]$ for all bounded measurable nondecreasing functions $F:\cC([s,t],\R)\to\R$.


$\{B(x):x\in\R\}$ denotes two-sided standard Brownian motion with $B(0)=0$. 
If $\sA$ is an index set and $F$ and $G$ are $\sA$-indexed stochastic processes on a complete probability space $(\Omega,\sF,\bbP)$,  then   $F$ and $G$ are 
modifications
of one another if $\bbP\{F(\alpha) = G(\alpha)\}=1$ for each $\alpha \in \sA$, and 
indistinguishable if
$\bbP\{\forall \alpha \in \sA:  F(\alpha) = G(\alpha)\}=1$. 
 
If $\sF$ and $\sG$ are $\sigma$-algebras, $\sF\vee\sG=\sigma(\sF,\sG)$ is the smallest $\sigma$ algebra containing $\sF$ and $\sG$.

\textit{Physical solutions} to the KPZ equation \eqref{KPZ} and the SHE \eqref{SHE} are defined via superposition through \eqref{NEind:m9} and \eqref{eq:genKPZ}.

\section{Solutions of SHE and KPZ}
\label{sec:setting}
We turn to the discussion of our setting  and coupling
and then summarize some results   from the companion paper \cite{Alb-etal-22-spde-}. 

\subsection{Probability space}\label{sec:prob-space}
We work on a complete Polish probability space $(\Omega,\sF,\bbP)$ that supports a space-time white noise $W$ on $L^2(\bbR^2)$ and a  group of continuous measure-preserving automorphisms  described  momentarily. A white noise $W$ is a mean zero Gaussian process indexed by $f \in L^2(\bbR^2)$ that  satisfies
\[ \bbP\bigl\{W(a f+b g) = a W(f) + b W(g)\bigr\}=1 \quad\text{and}\quad  \mathbb{E}[W(f)W(g)] = \int_{\bbR^2}f(t,x)\tspa g(t,x)\tspa dt\tspa dx\]
for $a,b \in \bbR$ and $f,g \in L^2(\bbR^2)$. 

Denote by $\nulls$ the $\sigma$-algebra generated by the $\bbP$-null events in $\sF$. For $-\infty \leq a < b \leq \infty$, let $\timefun{a}{b}$ denote the $\sB([a,b]\times\R)$ measurable functions in $L^2(\R^2)$. 
Let $\filt{W,0}_{s:t} = \sigma(W(f) : f \in \timefun{s}{t})\vee \nulls$ be the $\sigma$-algebra generated by the white noise restricted to time interval $[s,t]$ 
and $\nulls$. For   $s \leq t$,   $\fil_{s:t} = \filt{W,0}_{s-:t+} = \bigcap_{a < s \leq t < b} \filt{W,0}_{a:b}$ is the augmented filtration of the white noise.  Abbreviate $\fil=\fil_{-\infty:\infty}= \sigma(W(f) : f \in L^2(\R^2))\vee \nulls$.  

On the plane, the shift maps $\shiftf{s}{y}$, the shear  $\shearf{s}{a}$ by $a$ relative to time $s$,   time and space reflection maps  $\reff_1$ and $\reff_2$,   rescaled dilation maps $\dif{\alpha}{\lambda}$, and the negation map $\nef$ act on  $f\in L^2(\bbR^2)$ as follows: 
\be\label{S7}\begin{aligned}
\shiftf{s}{y}  f (t,x) &= f(t+s,x+y) \quad \text{for} \quad  s,y \in \bbR; \\
\shearf{s}{a} f (t,x) &= f(t, x+ a(t-s)) \quad \text{for} \quad  s,a \in \bbR; \\
\reff_1 f (t,x) &= f(-t,x) \qquad\text{and}\qquad  \reff_2 f  (t,x) = f(t,-x); \\
\dif{\alpha}{\lambda} f(t,x) &= \sqrt{\alpha \lambda\tspa} \tspb f(\alpha t, \lambda x) \quad \text{for} \quad \alpha,\lambda > 0;\\
\nef f (t,x) &= - f(t,x). 
\end{aligned}\ee
Their    inverses are $\shiftf{t}{x}^{-1} = \shiftf{-t}{-x}, \reff_1^{-1} = \reff_1, \reff_2^{-1}=\reff_2,$ $\shearf{s}{a}^{-1} = \shearf{s}{-a}$, $\dif{\alpha}{\lambda}^{-1} = \dif{\alpha^{-1}}{\lambda^{-1}}$, and $\nef^{-1} = \nef$. 

We assume that $(\Omega,\sF,\bbP)$ is equipped with a group (under composition) of continuous measure-preserving automorphisms generated by  
translations $\{\shiftd{s}{y} : s,y \in \bbR\}$,
shears $\{\sheard{s}{a} : s,a \in \bbR\}$, reflections $\refd_1$ and $\refd_2$,  dilations $\{\did{\alpha}{\lambda} : \alpha,\lambda >0\}$, and negation $\ned$,
that act on $W$ by 
\be\label{S9}
\begin{aligned}
W \circ \shiftd{s}{y}(f) &= W(\shiftf{-s}{-y} f ),\\
W \circ \refd_1(f) &= W(\reff_1 f),  \\
W \circ \did{\alpha}{\lambda}(f) &= W(\dif{\alpha^{-1}}{\lambda^{-1}} f),
\end{aligned}
\qquad \ \ 
\begin{aligned} 
W \circ \sheard{s}{a}(f) &= W(\shearf{s}{-a} f), \\
W \circ \refd_2(f) &= W(\reff_2 f),\\
W \circ \ned(f) &= W(\nef f). \\
\end{aligned}
\ee
%
%
We require that this group of automorphisms also be measure-preserving for $(\Omega,\fil,\P)$ (in particular, the automorphisms preserve $\fil$) and that (i) if $(s,y) \neq (0,0)$ then $\shiftd{s}{y}$ is strongly mixing on $(\Omega,\fil,\P)$, and (ii)  if $a \neq 0$, then $\sheard{s}{a}$ is strongly mixing on $(\Omega,\fil,\P)$. A concrete example of a space that satisfies all of these requirements is described  in Appendix \ref{sec:WNsetting}. In that setting, $\fil=\sF$.

The setting described above places us into the framework of the companion paper \cite{Alb-etal-22-spde-}, where the construction of our coupling of solutions to the SHE \eqref{SHE} appears. We next discuss this coupling and summarize some of the key results from that work and their connection to the rest of the literature. 
We refer the reader to \cite{Alb-etal-22-spde-} for the proofs and a more in-depth discussion.

\subsection{Solving SHE}
\label{solve-SHE}
Let $\cM_+(\R)$ denote the space of locally finite positive Borel measures on $\R$ endowed with the vague topology. 
It is shown in Theorem 2.6 of \cite{Alb-etal-22-spde-}  that the  state space on which \eqref{SHE} admits non-explosive, non-negative, and not identically zero solutions is the subspace
\begin{align}
\ICM &= \Bigl\{f\in\cM_+(\R):   0<f(\R)\le\infty \text{ and }  \forall a >0:  \int_{\bbR} e^{-a x^2}\,f(dx) <\infty \Bigr\}. \label{eq:ICM}
\end{align}
The zero measure is excluded to avoid accounting for trivialities in our results and, in any case, the physical solution from a zero initial condition is identically zero. 
\begin{remark}
$\ICM$ contains each point mass $\delta_x$ for $x\in\R$. Moreover, $\ICM$  contains all measures of the form  $f(x)\tspa dx$ for Borel functions $f:\R\to (0,\infty)$ for which there exist $C\ge0$ and $a<2$ such that  $f(x)\le Ce^{C\abs{x}^a}$ for all $x\in\R$. In particular, by the law of the iterated logarithm, almost every sample path of geometric Brownian motion lies in $\ICM$.
\end{remark}
For a fixed initial time $s\in\R$ and a fixed (i.e.\ non-random) initial condition $f\in\ICM$, \cite{Che-Dal-14,Che-Dal-15} showed the existence and uniqueness of a continuous and $\{\fil_{s:t}\}$-adapted 
solution to the Duhamel formulation of \eqref{SHE},
\begin{align}
\She(t,y) = \int_{\bbR} \heat(t,y-z)\, f(dz) + \int_s^t \int_{\R} \heat(t-r,y-z) \She(r,z)\,W(dz\,dr), 
\quad y\in\R,\ t\in(s,\infty). \label{eq:genmild}
\end{align}
As in \eqref{eq:heat}, $\heat$ denotes the heat kernel. A fixed point of the mapping in \eqref{eq:genmild} satisfying appropriate measurability and moment conditions is  known as a \emph{mild solution} of the SHE. Previously, \cite{Ber-Can-95, Ber-Gia-97} proved the existence and uniqueness of a continuous adapted mild solution for certain random initial conditions taking values in a subspace of $\ICM$. 
Most importantly, the Hopf-Cole solution of KPZ \eqref{KPZ}, which arises by taking logs of the mild formulation of SHE \eqref{SHE}, is the physically relevant solution for fixed initial conditions and fixed initial times. 


As discussed in the introduction, our results require a coupling that (i) couples the solutions of the SHE from all initial times and all initial conditions simultaneously, (ii) agrees with the physically relevant mild solution for each given initial time and  deterministic or random initial condition, and (iii) satisfies pathwise continuity properties as a function of its arguments, including the measure-valued initial condition. The mild formulation is not obviously well-suited to such a coupling, because the fixed point problem in \eqref{eq:genmild} is usually formulated to prove almost sure existence and uniqueness  for a fixed initial time and initial condition. Because there are uncountably many initial times and initial conditions, one needs to prove that it is possible to glue these together consistently. The main idea we use to build the coupling is to construct the Green's function $\She(t,y\viiva s,x)$ of the equation, which solves
\be\begin{aligned}
    \partial_t \She(t,y\viiva s,x) &= \frac{1}{2} \partial_{yy} \She(t,y\viiva s,x) + \She(t,y\viiva s,x) W(t,y), \\
    \She(s,y\viiva s,x) &= \delta_x(y)
\end{aligned}\label{eq:FSol}\ee
for all $(s,x,t,y)\in\varsets$ and then define solutions to \eqref{SHE} from general initial conditions $f\in\sM_+(\R)$ via the superposition principle: 
\begin{align}\label{NEind:m9}
\begin{split}
&\She(t,y\viiva s,f) = \int_{-\infty}^\infty  \She(t,y\viiva s,x)\,f(dx)
\quad\text{and}\quad 
\She(s,dy\viiva s,f) = f(dy).
\end{split}
\end{align}
When the measure in \eqref{NEind:m9} is given by $f(x)\tspa dx$ we write
\begin{align}\label{NEind}
\begin{split}
&\She(t,y\viiva s,f) = \int_{-\infty}^\infty  \She(t,y\viiva s,x)f(x)\,dx\quad\text{and}\quad \She(s,y\viiva s,f) = f(y).
\end{split}
\end{align}
It will always be clear from context whether $f$ is a measure or a function.

The initial value problem  \eqref{eq:FSol} is first treated  through its mild formulation for dyadic rational   $s,x \in \bbR$.  The resulting solutions are then  glued together to define the field $\She(t,y\viiva s,x)$. A similar construction of the Green's function, without most of the estimates we need for the present work, appeared previously in \cite{Alb-Kha-Qua-14-aop,Alb-Kha-Qua-14-jsp}. 
 Our starting point is stated in the next theorem. It is a corollary of Theorem 2.2 and Lemma A.5 of \cite{Alb-etal-22-spde-}.
\begin{theorem}\label{thm:She1} 
There exists 
a process $\She=\{\She(t,y\viiva s, x): (s, x, t, y)\in\varsets\}$ taking values in 
$\sC(\varsets,\bbR)$ with $\She(t,\aabullet\viiva s,\aabullet)$ $\fil_{s:t}$-measurable which satisfies the following: for each $f\in\ICM$ and $s\in\R$, 
the function $\She(\aabullet, \aabullet)= \She(\aabullet,\aabullet \viiva s, f)$ defined through  \eqref{NEind:m9} is indistinguishable from the solution to \eqref{eq:genmild} constructed in {\rm\cite{Che-Dal-14,Che-Dal-15}}. 
\end{theorem}

We continue to discuss the properties of the field $\She(t,y\viiva s,f)$ informally, with references to precise statements in \cite{Alb-etal-22-spde-}. 

The automorphisms in \eqref{S9} act on the Green's function via the following identities, which are contained in Proposition 2.3 of \cite{Alb-etal-22-spde-} and which hold for fixed $a,r,z \in \bbR$ and all $(s,y,t,x) \in \varsets$ simultaneously on an event of full probability that depends only on $a,r,$ and $z$:
\begin{align} 
&\She(t+r,y+z\viiva s+r,x+z) \circ \shiftd{-r}{-z} = \She(t,y\viiva s,x)\label{eq:cov.shift} \\
&\She(-s,x\viiva-t,y)\circ \refd_1 = \She(t,y\viiva s,x)
\text{ and }\She(t,-y\viiva s,-x)\circ \refd_2 = \She(t,y\viiva s,x)\label{eq:cov.ref}\\
&e^{a(y-x)+\frac{a^2}2(t-s)}\She(t,y+a(t-r)\viiva s, x+a(s-r))\circ \sheard{r}{-a} = \She(t,y\viiva s, x)\label{eq:cov.shear}
\end{align}
Because $\bbP$ is invariant under the automorphisms in \eqref{S9}, these identities induce distributional symmetries. In particular, at the level of $\rnShe$, the invariance under shear transformations in \eqref{eq:cov.shear} implies that $y \mapsto \rnShe(t,y\viiva s,x)$ and $x\mapsto \rnShe(t,y\viiva s,x)$ are stationary processes. See Corollary 2.4 in \cite{Alb-etal-22-spde-} or Proposition 1.4 in \cite{Ami-Cor-Qua-11}.

The field $\She$ satisfies a self-consistency which is essentially the Chapman-Kolmogorov identity for the continuum directed polymer, which will be introduced in Section \ref{sec:poly}: for all $s<r<t$, and all $x,y \in \R$,
	\begin{align}\label{eq:CK}
	\She(t,y\viiva r,x)=\int_{-\infty}^\infty\She(t,y\viiva s,z)\She(s,z\viiva r,x)\,dz.
	\end{align}
See Lemma 3.12 in \cite{Alb-etal-22-spde-} or Theorem 3.1(vii) in \cite{Alb-Kha-Qua-14-jsp}. Tonelli's theorem then implies that for all $s<r<t$, all $y$, and all $f\in\sM_+(\R)$,
\begin{align}\label{CK}
\She(t,y\viiva r,f) = \int_{-\infty}^\infty  \She(t,y\viiva s,z)\tspb \She(s,z\viiva r,f)\, dz.
\end{align}
This identity is recorded as Theorem 2.6(v) in \cite{Alb-etal-22-spde-}. \eqref{CK} is the \emph{cocycle property} for the solution semi-group to \eqref{SHE} defined by superposition: for all $s<r<t$, all $y \in \R$, and all $f \in \sM_+(\R)$,
\be\label{CK8}  \She(t,y\viiva r,f)=\She\bigl(t,y\hspace{0.8pt}\big\vert\hspace{0.8pt}s,\She(s,\aabullet\,\viiva r,f)\bigr).
\ee
By way of analogy with the ordinary heat equation, where the heat semi-group defines physical solutions, and for reasons we elaborate more on in the context of the KPZ equation below, we call \eqref{NEind:m9} the \textit{physical solution} of the stochastic heat equation.

A key observation in \cite{Alb-etal-22-spde-} is contained the following remark, which motivates most of the properties of this solution semi-group which we discuss in the sequel.
\begin{remark}
The singularity in the solution to \eqref{eq:FSol} as $t\searrow s$ is fully captured by the heat kernel in the following precise sense. Define the process 
    \be\label{women} 
    \rnShe(t,y\viiva s,x)=\begin{cases}
    \frac{\She(t,y\viiva s,x)}{\heat(t-s,y-x)}&\text{if $t>s$ and}\\[2pt]  
    1&\text{if $t=s$.}
    \end{cases}\ee
Then  with $\P$-probability one, $\rnShe$ is continuous and everywhere strictly positive on $\varset=\{(s,x,t,y)\in\R^4 : s \leq t\}$. Moreover, it is locally H\"older $1/2-$ in $(x,y)$ and $1/4-$ in $(s,t)$ on $\varsets=\{(s,x,t,y)\in\R^4 : s < t\}$. Moreover, for each $T>0$, there exists $C=C(\omega,T)>0$ such that
\[
C^{-1}(1+|x|^4+|y|^4)^{-1} \leq \rnShe(t,y\viiva s,x) \leq C(1+|x|^4+|y|^4).\]
 These claims follow from Theorem 2.2 and Corollary 3.10 in \cite{Alb-etal-22-spde-}. The H\"older estimates in that paper become suboptimal near the boundary where $t=s$, which is why we do not discuss precise regularity on $\varset$.
\end{remark}
In words, the above remark says that $\She(t,y\viiva s,x) = \rnShe(t,y\viiva s,x)\heat(t-s,x-y)$ is, uniformly on compact sets in time and uniformly in all of space, just the heat kernel up to a small (sub-polynomial growth and decay), rough, multiplicative perturbation. It therefore inherits many regularity properties as a solution semi-group from the heat semi-group.

It is immediate from everywhere strict positivity of $\She(t,y\viiva s,x)$ and \eqref{NEind:m9} that for all non-zero $f \in \sM_+(\R)$, all $s<t$, and all $y \in\R$,
\begin{align}\label{eq:allpos}
    \She(t,y\viiva s,f)\in(0,\infty].
\end{align}
In Theorem 2.6 of \cite{Alb-etal-22-spde-}, it is shown that the space $\ICM$ in \eqref{eq:ICM} is a natural domain for \eqref{SHE} in the sense solutions which start in $\ICM$ are well-behaved and remain in $\ICM$ for all time: for all $f\in \ICM$, $s<t$, and all $y$ in $\R$, 
\begin{align}\label{pr:IC:1}
\She(t,y\viiva s,f)\in(0,\infty) 
\quad\text{and}\quad  \She(t,x\viiva s,f)\tspa dx\in\ICM.
\end{align}
Theorem 2.6 in \cite{Alb-etal-22-spde-} shows that if $f \in \ICM$, then the function  $(s,y,t) \mapsto \She(t,y\viiva s,f)$  is 
continuous on $\{(s,y,t) \in \bbR^3 : s<t\}$. It is shown in Appendix D of \cite{Alb-etal-22-spde-} that the space of strictly positive continuous density functions of  measures in $\ICM$,
\be\label{CICM}   \CICM=\Bigl\{ f\in \sC(\R,(0,\infty)):       \forall a >0,  \int_{\bbR} e^{-a x^2} f(x)\tspb dx <\infty \Bigr\},  \ee
admits a Polish topology, where the convergence is characterized by uniform convergence of $f$ 
on compact sets, as well as convergence of the integrals in the definition.  Theorem 2.9 in \cite{Alb-etal-22-spde-} shows that with this topology on $\CICM$,
the map $(f,s,t) \mapsto \She(t,\abullet \viiva s,f)\in\CICM$ is continuous on $\{(s,t) : s \leq t\}\times\CICM$. In particular, this solution semi-group induces a Feller process on $\CICM$. See Remark 2.12 in \cite{Alb-etal-22-spde-}.

Conversely, Theorem 2.6 of \cite{Alb-etal-22-spde-} shows that $\ICM$ is sharp as a domain for \eqref{SHE} in the sense that all other non-zero initial conditions are explosive: if $f \in \sM_+(\R)\backslash \ICM$, then
    \begin{align}
\frac1{2(t-s)}<\sup\Bigl\{a>0:\int e^{-ax^2}\,f(dx)=\infty\Bigr\}\quad\text{implies }\She(t,y\viiva s,f) = \infty\ \  \forall y\in\R.\label{pr:IC:ICMsh}
    \end{align}
To summarize these observations: $\ICM$ is preserved by the dynamics of \eqref{SHE}, all initial conditions in $\ICM$ become strictly positive and jointly continuous instantaneously and remain so for all time. All other non-zero initial conditions explode in finite time.

\subsection{Physical solutions of KPZ} 
\label{sec:KPZ}
At the level of the KPZ equation \eqref{KPZ} started from an initial condition $f$ at initial time $s$, the Hopf-Cole transformation, combined with the previous observations,  provides a coupling of all solutions which are started from a measurable function via the identification for $t \geq s$, $x,y\in\R$, and $f$ Borel measurable,
\begin{align}
\KPZ(t,y \viiva s, f) = \log \She (t,y \viiva s, e^f) = \log \int_{\R} \She(t,y \viiva s,x)e^{f(x)}dx.\label{eq:genKPZ}
\end{align}
This coupling shows that the non-explosive initial conditions $f$ for KPZ are precisely those satisfying
\[
\int_{\R} e^{f(x) - a x^2}dx <\infty \text{ for all }a>0. 
\]
By Theorem 2.6 in \cite{Alb-etal-22-spde-}, any such initial condition becomes instantaneously continuous. 

Calling
\[
\CKPZ=\bigg\{f \in \sC(\R,\R) : \int_{\R} e^{f(x) - a x^2}dx <\infty \text{ for all }a>0 \bigg\},
\]
one can check that $\CKPZ$ is Polish in the topology in which convergence is characterized by uniform convergence on compact sets combined with convergence of all integrals of the form $ \int_{\R} e^{f(x) - a x^2}dx$, for $a>0$. Indeed, this is just the topology on $\CKPZ$ induced by insisting that the bijection $f \mapsto \log f$ define a homeomorphism between $\CICM$ and $\CKPZ$.

It is shown in Corollary 2.10 of \cite{Alb-etal-22-spde-} that the map $(f,s,t) \mapsto \KPZ(t,\abullet \viiva s, f) \in \CKPZ$ is jointly continuous on $\CKPZ\times \{(s,t) : s \leq t\}$ and satisfies $h(t,\abullet\viiva t,f)=f(\abullet)$ for all $t \in \R$.  As a consequence, the solution map defined in \eqref{eq:genKPZ} induces a Feller process on the Polish space $\CKPZ$. See Remark 2.12 in \cite{Alb-etal-22-spde-}.

As mentioned above,    the expression in \eqref{NEind:m9} is indistinguishable from the physically relevant mild solution to \eqref{SHE} under the most general assumptions for which uniqueness of the mild solutions to \eqref{SHE} have been proven in the literature. This includes non-random initial conditions in $\CICM$. Because of this observation, the fact that the Hopf-Cole transformation of a mild solution defines the physically relevant solution to \eqref{KPZ} for a fixed initial condition, and the just-observed continuity, we will call the expression in \eqref{eq:genKPZ} the \emph{physical solution} to \eqref{KPZ}. To justify this, note that by separability of $\CKPZ$, this field is the unique (up to indistinguishability) jointly continuous extension of the classical Hopf-Cole solutions, as can be seen by starting the field at rational times from non-random functions coming from a countable dense subset of $\CKPZ$. 

The above continuity also shows that $\eqref{eq:genKPZ}$ defines a Feller process on the Polish space
\[
\CKPZt = \CKPZ/\!\sim,
\]
where $\sim$ is the equivalence relation on $\CKPZ$ defined by  $f \sim g$ if $f(x) = g(x) +c$ for some $c \in \R$ and all $x \in \R$. This process is the one which admits stationary distributions given by Brownian motion with drift and it is on this space that we study questions of ergodicity. We discuss Feller continuity and ergodicity in detail at the level of solutions to \eqref{SHE} in Section \ref{sec:erg}.

As is often the case in models of this type in non-compact settings, questions concerning ergodicity of stationary distributions are complicated by the presence of a conservation law in the dynamics of \eqref{KPZ}. It is shown in Proposition 2.13 of \cite{Alb-etal-22-spde-} that if  $\PMPslopem,\PMPslopep\in\R$ and $f$ is a locally bounded Borel-measurable  function  such that 
	\begin{align}\label{eq:growth-rate}
	\lim_{x\to \infty} \frac{f(x)}{x} = \PMPslopep \qquad \text{ and }\qquad \lim_{x\to -\infty} \frac{f(x)}{x} = \PMPslopem,
	\end{align}
then for all $s<t$, 
	\begin{align}\label{eq:cons}\lim_{x\to \infty} \frac{\KPZ(t,x\viiva s,f)}{x} = \PMPslopep \qquad \text{ and }\qquad \lim_{x\to -\infty} \frac{\KPZ(t,x\viiva s,f)}{x} = \PMPslopem.\end{align}
As noted previously, it has been known since \cite{Ber-Gia-97} that there is a one-parameter family of increment-stationary distributions modulo constants, i.e., on $\CKPZt$ (see Section \ref{subsec:ergodic}) for \eqref{KPZ}, understood through the Hopf-Cole transform of the mild formulation of the SHE in \eqref{eq:genmild}, given by Brownian motion with drift $\lambda$. It is natural to suspect that these form attractors for \eqref{KPZ} started from initial conditions satisfying \eqref{eq:growth-rate} if, for example, $\PMPslopem=\PMPslopep=\lambda$.

\subsection{Continuum directed polymer}\label{sec:CDRP} 
The Feynman-Kac interpretation of solution of the SHE \eqref{SHE} is as the partition function of the continuum directed polymer measure with a boundary condition given by the initial condition $f$ and a Dirac mass as a terminal condition. Because the SHE dynamics proceeds forward in time, this Feynman-Kac interpretation is as a backward Markov chain. Concretely, the continuum directed polymer measures are defined as follows:
For  $r<t$ in $\R$ and $f\in\ICM$, 
the backward continuum point-to-line polymer that starts at $(t,y)\in\R^2$ and terminates at time $r$ with terminal condition $f$ is the time-inhomogeneous 
Markov process with initial position $(t,y)$ and the 
following  transition probability density from time $s'$ to time 
$s$:  
	\be\label{pi-rf} \begin{aligned}    \pi_{r,f}(s,w\viiva s',w')
	&=\frac{\She(s',w'\viiva s,w)\She(s,w\viiva r,f)}{\She(s',w'\viiva r,f)} \\
	&=\frac{\She(s',w'\viiva s,w)\int_\R\She(s,w\viiva r,z)\,f(dz)}{\int_\R\She(s',w'\viiva r,z)\,f(dz)},\quad \text{for } s'>s>r   \text{ and }  w,w'\in\R.
	\end{aligned}\ee 
We denote the path distribution of this process on the space of continuous functions $\sC([s,t],\R)$, equipped with its Borel $\sigma$-algebra, by $\Poly_{(t,y),(r,f)}$. Many of our arguments involve analysis of these measures and their infinite-volume counterparts, which we introduce and study in Section \ref{sec:poly}.

These measures were originally introduced in \cite{Alb-Kha-Qua-14-jsp}.  The construction in that paper builds the measures for fixed initial and terminal conditions on an event  of full probability that depends on those conditions. Theorem 2.14 of \cite{Alb-etal-22-spde-} uses the coupling of solutions to \eqref{SHE} discussed in Section \ref{solve-SHE} to couple all these measures together on a single event of full probability. \cite{Alb-etal-22-spde-} also proves many basic properties of the measures, including that sample paths under these measures are all almost surely H\"older 1/2$-$ continuous.

The {\it  point-to-point quenched polymer} is the special case $\Poly_{(t,y),(r,x)}=\Poly_{(t,y),(r,\delta_x)}$.   This is the distribution of a path between the time-space points  $(t,y)$ and $(r,x)$ whose Markovian transition probability density from    
$(s',w')$ to $(s,w)$  is  
	\begin{align}\label{pi-rx}
	\pi_{r,x}(s,w\viiva s',w')=\frac{\She(s',w'\viiva s,w)\She(s,w\viiva r,x)}{\She(s',w'\viiva r,x)},\quad  s'>s>r\text{ and }w,w'\in\R. 
	\end{align}
	
\section{Main results}\label{sec:main}

We turn to the statements of the main results, with the proofs coming in the subsequent sections.
\subsection{Busemann process}\label{sub:Bus}
The main tool of this paper for studying the solutions of the SHE \eqref{SHE} and the KPZ equation \eqref{KPZ} is the \emph{Busemann process}. This  is  the jointly stationary monotone coupling of the spatially homogeneous stationary distributions of the KPZ equation that arises naturally from the dynamics itself. 
This process is real-valued and  parameterized by two  time-space pairs, slopes $\lambda\in\R$, and signs $\pm$. 
The first theorem gives  the existence and basic properties of the Busemann  process. 

\begin{theorem}\label{main:Bus}
There exists a stochastic process $\bigl\{\Bus^{\lambda\sig}(s,x,t,y):s,x,t,y,\lambda\in\R,\,\sigg\in\{-,+\}\bigr\}$ defined on the probability space $(\Omega,\sF,\P)$ and 
satisfying the following properties.
\begin{enumerate}[label={\rm(\alph*)}, ref={\rm\alph*}] \itemsep=3pt
\item\label{Z-meas} The process $\bigl\{\Bus^{\lambda\sig}(s,x,t,y):s,x,t,y,\lambda\in\R,\,\sigg\in\{-,+\}\bigr\}$ is a measurable function of the Green's function 
 $\She(\aabullet,\aabullet\tsp\viiva\aabullet,\aabullet)$.
\item\label{Bproc.indep} For any $T\in\R$, 
$\bigl\{\Bus^{\lambda\sig}(s,x,t,y):\sigg\in\{-,+\},x,y,\lambda\in\R,s,t\le T\bigr\}$ is $\fil_{-\infty:T}$-measurable and hence independent of 
$\fil_{T:\infty}$.
\item\label{Bproc.-=+} For each $\lambda\in\R$, $\P\{\Bus^{\lambda-}(s,x,t,y)=\Bus^{\lambda+}(s,x,t,y)\ \forall s,x,t,y\}=1$. 
\item\label{Bproc.BM} For each $t,\lambda\in\R$ and $\sigg\in\{-,+\}$,  the process  $\{\Bus^{\lambda\sig}(t,x,t,y):x,y\in\R\}$
has the same distribution under $\P$ as $B(y)-B(x)+\lambda(y-x)$, where $B$ is a 
two-sided standard Brownian motion.
\end{enumerate}
There exists an event $\Omega_0$ such that $\P(\Omega_0)=1$
and the following hold for all $\w\in\Omega_0$.
\begin{enumerate}[resume,label={\rm(\alph*)}, ref={\rm\alph*}] \itemsep=3pt
\item\label{Bproc.cont} For each  $\lambda\in\R$ and $\sigg\in\{-,+\}$, $\Bus^{\lambda\sig}\in\cC(\R^4,\R)$.
\item\label{Bproc.a} For all $x<y$,  $t$, and $\mu<\lambda$, 
	\begin{align}\label{Busmono}
	\begin{split}
	&\Bus^{\mu-}(t,x,t,y)\le\Bus^{\mu+}(t,x,t,y)\le\Bus^{\lambda-}(t,x,t,y)\le\Bus^{\lambda+}(t,x,t,y)\quad\text{and}\\
	&\Bus^{\mu-}(t,y,t,x)\ge\Bus^{\mu+}(t,y,t,x)\ge\Bus^{\lambda-}(t,y,t,x)\ge\Bus^{\lambda+}(t,y,t,x).
	\end{split}
	\end{align} 
\item\label{Bproc.b} For all $r,x, s, y, t, z, \lambda$ and all $\sigg\in\{-,+\}$
	\begin{align}\label{cocycle}
	\Bus^{\lambda\sig}(r,x,s,y)+\Bus^{\lambda\sig}(s,y,t,z)=\Bus^{\lambda\sig}(r,x,t,z).
	\end{align}
 \item\label{Bproc.c} For all $s,x,t,y,\lambda$  and all $\sigg\in\{-,+\}$
	\begin{align}\label{Buscont}
	\Bus^{\lambda-}(s,x,t,y)
	=\lim_{\mu\nearrow\lambda}\Bus^{\mu\sig}(s,x,t,y)
	\quad\text{and}\quad 
	\Bus^{\lambda+}(s,x,t,y)=\lim_{\mu\searrow\lambda}\Bus^{\mu\sig}(s,x,t,y).
	\end{align}
\item\label{Bproc.d} For all  $t>r$, all $s,x,y, \lambda$, and all $\sigg\in\{-,+\}$
	\begin{align}\label{NEBus2}
	e^{\Bus^{\lambda\sig}(s,x,t,y)}
	&= \int_{-\infty}^\infty  \She(t,y\viiva r,z)\tspa e^{\Bus^{\lambda\sig}(s,x,r,z)}\, dz.
	\end{align}
\end{enumerate}
\end{theorem}

Part \eqref{Bproc.indep} says that the Busemann process is adapted to the filtration of the white noise. Part \eqref{Bproc.-=+} says that when $\lambda$ is fixed, there is no $\sigg=\pm$ distinction. 
Part \eqref{Bproc.BM}
says that the Busemann process gives a coupling of the known invariant measures.
Parts \eqref{Bproc.cont} and \eqref{Bproc.a} say that the process is continuous in the time-space parameters and monotone in the $\lambda$ and $\sigg$ parameters. Part \eqref{Bproc.c} gives a continuity in the latter two parameters. The possible jumps of the process when $\sigg=-$ is switched to $\sigg=+$ are captured by the set $\Bruno$, defined and studied further down (see \eqref{Bruno}). Part \eqref{Bproc.b} says the process is an additive \emph{cocycle} and implies in particular that $\Bus^{\lambda\sig}(t,y,s,x)=-\Bus^{\lambda\sig}(s,x,t,y)$.

Comparison of \eqref{NEind} and 
\eqref{NEBus2} shows that, for $\P$-almost every $\w$, for each $\lambda\in\R$ and each $\sigg\in\{-,+\}$,
\begin{align}\label{Shela}
\Glob^{\lambda\sig}(t,x)=e^{\Bus^{\lambda\sig}(0,0,t,x)}
\end{align}
is a solution of the SHE \eqref{SHE} 
 defined for all times $t\in\R$. 
Then  $\Bus^{\lambda\sig}(0,0,t,x)$ is a physical solution of the KPZ equation \eqref{KPZ}   for all time $t$. Such solutions are called \emph{eternal}.
The next section explains 
how  these processes yield pullback attractors in the sense described in the introduction. See \eqref{eq:buslim} and \eqref{eq:buslim2}.

The Busemann process is initially constructed with a weak convergence argument  in Section \ref{sec:stat-coc} 
on an extended probability space and Theorem \ref{main:Bus} is proved in that setting. Later, in Corollary \ref{cor:what->w} in Section \ref{sec:erg}, we revert to the original probability space with the help of the almost sure limits   \eqref{eq:buslim} in the next section.

The Busemann process inherits  symmetries from the Green's function, which were previously recorded as \eqref{eq:cov.shift}, \eqref{eq:cov.ref}, and \eqref{eq:cov.shear}. 
These properties are proved in Section \ref{sec:erg}. 


\begin{theorem}\label{thm:bcov}
The process $\bigl\{\Bus^{\lambda\sig}(s,x,t,y):s,x,t,y,\lambda\in\R,\sigg\in\{-,+\}\bigr\}$ satisfies the following covariance properties.
\begin{enumerate} [label={\rm(\roman*)}, ref={\rm\roman*}]   \itemsep=3pt  
\item\label{thm:bcov.i} {\rm(Shift)} For any $r,z\in\R$ there exists an event $\Omega_{r,z}$ with $\P(\Omega_{r,z})=1$ such that for all $s,x,t,y,\lambda\in\R$, $\sigg\in\{-,+\}$, and $\w\in\Omega_{r,z}$,
	\[\Bus^{\lambda\sig}(s+r,x+z,t+r,y+z) \circ \shiftd{-r}{-z}=\Bus^{\lambda\sig}(s,x,t,y).\]
\item\label{thm:bcov.ii} {\rm(Reflection)} There exists an event $\Omega_0$ with $\P(\Omega_0)=1$ such that for all $s,x,t,y,\lambda\in\R$, all $\sigg\in\{-,+\}$, and $\w\in\Omega_0$,
	\[\Bus^{(-\lambda)(-\sig)}(s,-x,t,-y) \circ \refd_2=\Bus^{\lambda\sig}(s,x,t,y).\]
\item\label{thm:bcov.iii} {\rm(Shear)}   For any $r,c\in\R$ there exists an event $\Omega_{r,c}$ with $\P(\Omega_{r,c})=1$ such that for all $x,y,\lambda\in\R$, $\sigg\in\{-,+\}$, and $\w\in\Omega_{r,c}$, 
	\begin{align}\label{bus-shear}
	\Bus^{(\lambda+c)\sig}(s,x-c(s-r),t,y-c(t-r)) \circ\sheard{r}{c}=\Bus^{\lambda\sig}(s,x,t,y;\w)+c(y-x)-\frac{c^2}2(t-s).
	\end{align}
\end{enumerate}
\end{theorem}

 Theorem \ref{main:Bus}\eqref{Bproc.d}, Theorem \ref{thm:bcov}\eqref{thm:bcov.i}, 
and the invariance of $\P$ under the action of $\shiftd{t}{0}$ imply that the Busemann process itself is invariant under the SHE evolution. Applying the temporal reflection $\refd_1$ corresponds to working with terminal conditions instead of initial conditions in \eqref{KPZ} and results in a different Busemann process, coming from sending the terminal time to $\infty$ instead of the initial time to $-\infty$, with the same distribution as the one we study. See Remark \ref{forwardBus}.

Our goal is to describe attractors simultaneously for all slopes $\lambda$. For this we have to identify the set of exceptional slopes at which the Busemann process jumps: 
	\begin{align}\label{Bruno}
	\Bruno=\{\lambda\in\R:\exists(s,x,t,y)\in\R^4\text{ with }\Bus^{\lambda-}(s,x,t,y)\ne\Bus^{\lambda+}(s,x,t,y)\}.
	\end{align}

When $\lambda\not\in\Bruno$ we write $\Bus^{\lambda}$ to denote the common function $\Bus^{\lambda+}=\Bus^{\lambda-}$ and similarly   $\Glob^\lambda=\Glob^{\lambda+}=\Glob^{\lambda-}$.
If $\lambda\in\Bruno$, there are  two different
eternal solutions 
having  the same conserved quantity $\lambda$.
The next result implies that in this case   $\Glob^{\lambda-}(0,x)\ne\Glob^{\lambda+}(0,x)$ for all $x\ne0$. 

\begin{theorem}\label{L-char}
The following hold $\P$-almost surely.
\begin{enumerate} [label={\rm(\roman*)}, ref={\rm\roman*}]   \itemsep=3pt  
\item\label{L-char:i} For each $\lambda\in\Bruno$ and for each $t,x,y\in\R$ with $x\ne y$, $\Bus^{\lambda-}(t,x,t,y)\ne\Bus^{\lambda+}(t,x,t,y)$.  
\item\label{L-char:ii} For each $t\in\R$, $x<y$, and $\kappa<\mu$, $\Bus^{\kappa+}(t,x,t,y)<\Bus^{\mu-}(t,x,t,y)$.
\end{enumerate}
\end{theorem}

Part \eqref{L-char:i} above reduces the definition of $\Bruno$ to checking $\lambda$-continuity at a single spatial point:

\begin{corollary}\label{L-char2}
With $\P$-probability one, for any $a\ne0$,
\[\Bruno=\{\lambda:\Bus^{\lambda-}(0,0,0,\aabullet)\ne\Bus^{\lambda+}(0,0,0,\aabullet)\}=\{\lambda:\Bus^{\lambda-}(0,0,0,a)\ne\Bus^{\lambda+}(0,0,0,a)\}.\]
\end{corollary}  

The next theorem describes properties of the set $\Bruno$.


\begin{theorem}\label{th:La}
The following statements hold.
\begin{enumerate}[label={\rm(\alph*)}, ref={\rm\alph*}] \itemsep=3pt
\item\label{th:La:sym} For any $t,x,c\in\R$ we have for $\P$-almost every $\w$,
$\Bruno=\Lambda^{\shiftd{t}{x}\w}=\Lambda^{\sheard{t}{c}\w}-c=-\Lambda^{\refd_2\w}.$
\item\label{th:La:b} For each  $\lambda\in\R$, $\P(\lambda\in\Bruno)=0$. 
\item\label{th:La:e} Either $\P\{\Bruno=\varnothing\}=1$ or $\P\{\Bruno\text{ is countable and dense in $\R$}\}=1$.
\end{enumerate}
\end{theorem}

\begin{remark}\label{rk:Bruno-dense}
Regarding the dichotomy in part \eqref{th:La:e},   in all the solvable models where the distribution of the Busemann process has been described explicitly, the set $\Bruno$ is a countable dense subset of the parameter space \cite{Fan-Sep-20,Sep-Sor-21-,Bus-Sep-Sor-22-,Bat-Fan-Sep-22-}. 
In particular, \cite{Bus-Sep-Sor-22-} and the forthcoming \cite{Bat-Fan-Sep-22-} 
prove this for the KPZ fixed point and the log-gamma polymer models, respectively. Solutions to the KPZ equation converge to the KPZ fixed point under the KPZ scalings \cite{Vir-20-,Qua-Sar-20-} and the log-gamma polymer free energy converges to a solution to the KPZ equation \eqref{KPZ} upon appropriate rescaling \cite{Alb-Kha-Qua-14-aop}. 
This suggests Open Problem \ref{prob1} in Section \ref{sec:OP}.
\end{remark}

The proofs of Theorems \ref{L-char} and \ref{th:La} are in Section \ref{sec:Bruno}.
\begin{remark}\label{rk:Bus-prop}
We close this section by collecting the properties of the Busemann process as a process in $\lambda$. Giving an explicit description of the process $\Bus^\bbullet(0,0,0,1)$ 
is left as Open Problem \ref{prob1}.

\begin{enumerate}[label={\rm(\alph*)}, ref={\rm\alph*}] \itemsep=3pt
\item Theorems \ref{main:Bus}\eqref{Bproc.BM} and \ref{L-char}\eqref{L-char:ii} say that for any $t\in\R$, $\bigl\{\Bus^{\lambda\sig}(t,0,t,\cbullet):\lambda\in\R,\,\sigg\in\{-,+\}\bigr\}$ is a coupling of strictly ordered two-sided Brownian motions with linear drifts and  unit diffusion coefficient. 
If one fixes $\lambda\in\R$, then $\Bus^{\lambda\pm}(t,x,t,y)$ are both normally distributed with mean $\lambda(y-x)$ and variance $y-x$. 
Then the ordering \eqref{Busmono} implies that the two $\pm$ processes match and this is consistent with Theorem \ref{th:La}\eqref{th:La:b}.

\item  The Busemann process of the KPZ fixed point, called the {\it stationary horizon} (SH), is also a coupling of   ordered two-sided Brownian motions with drifts, parameterized by their drifts \cite{Bus-21-, Bus-Sep-Sor-22-}. In SH, each pair of Brownian motions coincide in a nondegenerate neighborhood of the origin. Thus the  Busemann process we study is not the same as SH. 

\item By the shear-covariance in Theorem \ref{thm:bcov}\eqref{thm:bcov.iii} and the shear-invariance of $\P$, for each $t$ and $x<y$ in $\R$ and for each $\sigg\in\{-,+\}$, the process $\lambda\mapsto\Bus^{\lambda\sig}(t,x,t,y)-\lambda(y-x)$ is stationary and, therefore, $\lambda\mapsto\Bus^{\lambda\sig}(t,x,t,y)$ has stationary increments. 
\item\label{not-indep} 
The central limit theorem shows that for any $t$ and $x<y$ in $\R$ and any $\sigg\in\{-,+\}$, the process $\lambda\mapsto\Bus^{\lambda\sig}(t,x,t,y)$ does not have independent increments.
See Appendix \ref{app:misc}.
\item By Theorem \ref{L-char}\eqref{L-char:ii}, the process $\lambda\mapsto\Bus^{\lambda\sig}(0,0,0,1)$ is almost surely strictly increasing.
\item Theorem \ref{thm:bcov}(\ref{thm:bcov.i})--(\ref{thm:bcov.ii}) and the reflection and translation invariance of $\bbP$ imply that the process $\lambda\mapsto\Bus^{\lambda\sig}(0,0,0,1)$ has the same distribution as the process $\lambda\mapsto-\Bus^{(-\lambda)(-\sig)}(0,0,0,1)$.
\end{enumerate}
\end{remark}

\subsection{Shape theorems and stochastic homogenization}\label{subsec:shape}

We show the following shape theorem for the fundamental solution $\She$,  proved in Section \ref{sec:shape}.

\begin{theorem}\label{thm:Z-cont} 
On an event of $\P$-probability one, the following holds: for any $C>0$, \begin{align}\label{Z-cont}
\lim_{n\to\infty}n^{-1}\sup_{\substack{(s,x,t,y)\tspb\in\tspb\varset: \\[2pt] s,x,t,y\tspb\in\tspb[-Cn,Cn]}}\,\Bigl|\,\log \She(t,y\viiva s,x)+\frac{t-s}{24}-\log\heat(t-s,y-x)\,\Bigr|=0.
\end{align}
\end{theorem}

\begin{remark}   We allow $t=s$ in \eqref{Z-cont} by taking the convention that $\log \She(t,y\viiva s,x)-\log\heat(t-s,y-x)\!=\!\log\Shenorm(t,y\viiva s,x)$ for all parameter values and recalling that $\Shenorm(s,y\viiva s,x)=1$. %
\end{remark}

As a corollary of Theorem \ref{thm:Z-cont} and easier versions of tail bounds which appear in various proofs below (see, e.g., the  proof of Theorem \ref{int-shape}), we obtain the following result. Remark \ref{rem:StHomKPZ} explains how this result implies stochastic homogenization of KPZ. 

\begin{corollary}\label{cor:StHom}
On an event of $\P$-probability one, the following holds: for all bounded uniformly continuous $U \in \sC(\R,\R)$, all $t>0$, and all compact $K \subset \R$,
\[\lim_{\e\searrow0}\;\sup_{x\in K}\;\big|\tspb\e \tspa\KPZ(\e^{-1}t, \e^{-1}x\viiva 0,-U_\e) +\overline{u}(t,x)\tspa \big|=0\]
where $U_\e(x) = \e^{-1}U(\e x)$, $h(\aabullet,\aabullet\viiva\aabullet,\aabullet)$ is defined via \eqref{eq:genKPZ}, and
\[
\overline{u}(t,x) = \frac{t}{24} + \inf_{z\in\R}\biggl\{\frac{(x-z)^2}{2t} + U(z)\biggr\}.
\]
\end{corollary}

\begin{remark}\label{rem:StHomKPZ}\textup{(Stochastic homogenization of KPZ)}
In the setting of Corollary \ref{cor:StHom}, call $u_\e(t,x)$ $= -\e \tspa\KPZ(\e^{-1}t, \e^{-1}x\viiva 0,-U_\e)$. Then for each $\e>0$ and each such $U$, $u_\e$ is indistinguishable from the Hopf-Cole solution to the viscous stochastic Hamilton-Jacobi 
equation
\[\partial_t u_\e(t,x)-\tfrac{\e}{2}\partial_{xx}u_\e(t,x)+H(\e^{-1}t, \e^{-1}x,\partial_x u_\e)=0, \qquad u_\e(0,x)=U(x),\] 
with Hamiltonian $H(t,x,p)=\frac{p^2}2+W(t,x)$. 
Corollary \ref{cor:StHom} implies that on a single event of full probability, for all bounded uniformly continuous $U \in \sC(\R,\R)$ and all $u_\e$ defined as above, 
$u_\e$ converges locally uniformly as $\e\to0$ to the viscosity solution of the \textit{effective} Hamilton-Jacobi equation
\[\partial_t\overline u+{\overline H}(\partial_x \overline u)=0, \qquad \overline{u}(0,x) = U(x)\]
with effective Hamiltonian ${\overline H}(p)=\frac{p^2}2-\frac1{24}$. This is the usual definition of stochastic homogenization as in \cite{Jin-Sou-Tra-17,Kos-Var-08}. 
\end{remark}

We also prove a shape theorem for $\Bus^\lambda$.  Due to Theorem \ref{main:Bus}\eqref{Bproc.-=+}, 
there is no need for a $\Bus^{\lambda\pm}$ distinction.

\begin{theorem}\label{thm:b-shape}
Fix $\lambda\in\R$. The following holds on an event of  $\P$-probability one.  For all $C>0$,
	\begin{align}\label{b-shape}
	\lim_{n\to\infty}n^{-1}\!\!\!\sup_{s,x,t,y\tspb\in\tspb[-Cn, Cn]}\,\Bigl|\tspb\Bus^\lambda(s,x,t,y)-\Bigl(\frac{\lambda^2}2-\frac1{24}\Bigr)(t-s)-\lambda(y-x)\Bigr|=0.
	\end{align}
\end{theorem}

We also have the following version of the shape theorem that works simultaneously for all $\lambda$.

\begin{theorem}\label{bus:exp'}
The following holds $\P$-almost surely: for all $t, \lambda\in\R$ and $\sigg\in\{-,+\}$,
	\begin{align}
\lim_{\abs{x}\to\infty}\abs{x}^{-1}\abs{\Bus^{\lambda\sig}(0,0,t,x)-\lambda x}=0.
\label{bus:exp3'}
	\end{align}
\end{theorem}
Recall from \eqref{eq:cons} that the SHE \eqref{SHE} conserves the spatial exponential growth rate. The limit \eqref{bus:exp3'} says that for $\P$-almost every $\w$, for all $\lambda\in\R$
the conserved quantity for both solutions 
$\Glob^{\lambda-}$ and $\Glob^{\lambda+}$, defined in \eqref{Shela},
is $\lambda$. 

These results also combine to show that the Busemann process defines correctors in the language of stochastic homogenization.

\begin{remark}\label{rem:corrector}\textup{(Busemann functions as correctors)}
Consider $u_\e$ defined as in Corollary \ref{cor:StHom}. We follow the corrector derivation in \cite{Lio-Sou-03} and assume that we may expand around $\overline u$ by writing  $u_\e(t,x)=\overline u(t,x)+\e v(t/\e,x/\e)+O(\e^2)$. 
Substitute in the equations for $u_\e$ and $\overline u$ 
and set the coefficients of the different powers of $\e$ equal to $0$.
This leads to the \emph{corrector equation}
\[\partial_t v-\frac12\partial_{xx}v+\frac12(p+\partial_x v)^2+W=\frac{p^2}2-\frac1{24},\]
for each fixed $p$. 
\eqref{NEBus2} implies that $\Bus^{-p}(0,0,\aabullet,\aabullet)$ is a solution to the KPZ equation \eqref{KPZ}. Therefore, the above corrector equation has a solution given by \[v(t,x)=-\Bus^{-p}(0,0,t,x)-px+\Bigl(\frac{p^2}2-\frac1{24}\Bigr)t,\]
on all of $\R^2$. Theorem \ref{thm:b-shape} says that $\e v(t/\e,x/\e)\to0$ locally uniformly as $\e\searrow0$, which is consistent with the convergence $u_\e\to\overline u$. 
\end{remark}


Combining the control \eqref{bus:exp3'} gives on the exponential growth rate of the Busemann process with results from   \cite{Alb-etal-22-spde-} gives  the following regularity of the Busemann process. 

\begin{theorem}\label{b cont}
The following holds $\P$-almost surely: for all $\lambda\in\R$ and $\sigg\in\{-,+\}$, for all $\alpha\in(0,1/4)$ and $\gamma\in(0,1/2)$, 
$\Bus^{\lambda\sig}(s,x,t,y)$ is locally $\alpha$-H\"older-continuous in $s$ and in $t$ and locally $\gamma$-H\"older-continuous in $x$ and in $y$.
\end{theorem}

 These shape theorems are first proved in Section \ref{sec:shape} for the Busemann process on an extended probability space. The proofs, and the proof of Theorem \ref{b cont}, are completed in Section \ref{sec:erg} where we return back  to the original probability space. 

\subsection{Busemann limits}\label{sec:Buslims} 
In Section \ref{sub:Bus}, we saw how the Busemann process produces solutions of the SHE \eqref{SHE} and KPZ equation \eqref{KPZ} that are defined for all time $t\in\R$. 
In this section we show how these solutions can be recovered as almost sure limits as $r\to-\infty$ of solutions to \eqref{SHE} that start at time $r$ with appropriate initial conditions.
Then Section \ref{sub:1F1S} explains how to interpret these results from a random dynamical systems point of view.
The main results of this section, Theorems \ref{thm:buslim} and \ref{l2p-buslim}, are proved first on the extended probability space in Section \ref{sec:BusLim} and then in their final form on the original probability space in Section \ref{sec:erg}.  The first result Theorem \ref{thm:buslim} treats  Dirac $\delta$, or narrow-wedge, initial conditions.

\begin{theorem}\label{thm:buslim}
The following hold $\P$-almost surely. Let  $\lambda\in\R$,   $C>0$, $\e>0$, and $\tau>0$. 
Then  there exist {\rm(}possibly random{\rm)} $R<0$ and $\delta>0$ 
such that for all  $r\le R$,  $z$ such that $\abs{\frac{z}r+\lambda}<\delta$, 
and for all $s,x,t,y\in[-C,C]$ with $t-s\ge\tau$, 
	\[\frac{\She(t,y\viiva r,z)}{\She(s,x\viiva r,z)}
	\le (1+\e)^2\int_x^\infty\She(t,y\viiva s,w)e^{\Bus^{\lambda+}(s,x,s,w)}\,dw+(1+\e)^2\int_{-\infty}^x\She(t,y\viiva s,w)e^{\Bus^{\lambda-}(s,x,s,w)}\,dw
	\]
and 
	\[\frac{\She(t,y\viiva r,z)}{\She(s,x\viiva r,z)}
	\ge (1+\e)^{-3}\int_x^\infty\She(t,y\viiva s,w)e^{\Bus^{\lambda-}(s,x,s,w)}\,dw+(1+\e)^{-3}\int_{-\infty}^x\She(t,y\viiva s,w)e^{\Bus^{\lambda+}(s,x,s,w)}\,dw.
	\]
In particular, with $\P$-probability one, for any $\lambda\not\in\Bruno$, 
	\begin{align}\label{eq:buslim}
	\lim_{\substack{r\to-\infty\\ z/r\to-\lambda}}\frac{\She(t,y\viiva r,z)}{\She(s,x\viiva r,z)}=e^{\Bus^\lambda(s,x,t,y)}\quad
	\text{locally uniformly in $(s,x,t,y)\in\R^4$.}
	\end{align}
\end{theorem}

 The gap $\tau>0$  is a technical artifact of the proof, where shrinking $\tau$ drives $R$ towards $-\infty$.  The  explicit $\e$ bounds are included in the theorem to cover all slopes, as the  limit \eqref{eq:buslim} works simultaneously only  for all  slopes $\lambda\not\in\Bruno$.  If  the conclusion is relaxed to a $\lambda$-dependent full-probability event, we can assert the limit for each fixed slope.  
This  comes from Theorem \ref{thm:buslim} and Theorem \ref{th:La}\eqref{th:La:b}.

\begin{corollary}\label{cor:Busp2p}
For each $\lambda\in\R$, there exists an event $\Omega_\lambda$ with $\P(\Omega_\lambda)=1$ such that \eqref{eq:buslim} holds for all $\w\in\Omega_\lambda$.
\end{corollary}

\begin{remark}
Equation \eqref{eq:buslim} gives $\Bus^\lambda$ as a limit of differences of the height function $\log\She$. In the zero temperature (or zero viscosity) models of first- or last-passage percolation, the analogous expression is a limit of differences of passage times. In the particular example of first-passage percolation, which corresponds to a random pseudo-metric, this limit coincides with the classical definition of \emph{Busemann functions}. See \cite[p. 131]{Bus-55} and \cite{New-95}.
\end{remark}

\begin{remark}
\cite{Das-Zhu-22-} proved that as $t\to\infty$,  $\log\She(t,\aabullet\tsp\viiva0,0)-\log\She(t,0\viiva0,0)$ converges in distribution to a standard two-sided Brownian motion $B$. 
The shear-covariance of the Green's function then implies that for any $\lambda\in\R$,
the process
\[\bigl\{\log\She(0,y\viiva r,-\lambda r)-\log\She(0,x\viiva r,-\lambda r):x,y\in\R\bigr\}\]
converges in distribution to $B(y)-B(x)+\lambda(y-x)$, as $r\to-\infty$. 
This weak limit also follows from our Corollary \ref{cor:Busp2p} and Theorem \ref{main:Bus}\eqref{Bproc.BM}.
\end{remark}

%

Next we treat  function-valued initial conditions. For  $\lambda\in\R$,      $\initF_\lambda$ will be a space  of initial conditions \emph{attracted} to the solution $\Glob^\lambda$ defined in \eqref{Shela}.  At a minimum, the basin of attraction should contain those functions whose logarithms  satisfy the limit in \eqref{infBusIC}. Furthermore, the attractor itself should be a member of the basin of attraction. To accommodate this,  $\initF_\lambda$ must admit time-dependent functions.

\begin{definition}\label{def:initF}
For $\lambda\in\R$, define the spaces $\initF_\lambda$ as follows.
\begin{enumerate}[label={\rm(\roman*)}, ref={\rm\roman*}]  
\item For $\lambda\ne0$, $\initF_\lambda$ is the space of functions $f:\R_-\times\R\to(0,\infty)$ such that $f(r,\aabullet)$ is Borel-measurable  for each $r\in\R_-$,  there exists a $\delta_0\in(0,\abs{\lambda})$ such that  
\begin{align}\label{W1}
\varliminf_{r\to-\infty}\frac1{\abs{r}}\inf_{x\tsp:\tsp\abs{\frac{x}r+\lambda}\le \delta_0}\bigl(\log f(r,x)-\lambda x\bigr)\;\in\;[0,\infty]
\end{align}
\begin{align}\label{W2}
\text{and}\quad \varlimsup_{r\to-\infty}\sup_{\substack{x\tsp:\tsp\lambda x\ge0\\[2pt]  \abs{\frac{x}r+\lambda}\ge\delta}}\frac{\log f(r,x)-\lambda x}{\abs{r}+\abs{x}}\;\in\;[-\infty,0]
\qquad\text{for all }\delta\in(0,\delta_0],
\end{align}
and there exists  $\mu\in\R$  such that $\mu/\lambda>-1$ and
\begin{align}\label{W3}
\varlimsup_{r\to-\infty}\sup_{x\tsp:\tsp\lambda x\le0}\frac{\log f(r,x)-\mu x}{\abs{r}+\abs{x}}\;\in\;[-\infty,0].
\end{align}

\item For $\lambda=0$, $\initF_0$ is the space of functions $f:\R_-\times\R\to(0,\infty)$ such that $f(r,\aabullet)$ is Borel-measurable  for each $r\in\R_-$ and for which there exist a $\delta_0>0$ and a $c>0$ such that 
for all $\delta\in(0,\delta_0]$,
\begin{align}\label{W''}
\varlimsup_{r\to-\infty}\sup_{\abs{x}\ge\delta\abs{r}}\frac{\log f(r,x)}{\abs{r}+\abs{x}}\;\in\;[-\infty,0]
\quad\text{and}\quad
\varliminf_{r\to-\infty}\abs{r}^{-1}\log\int_{\abs{x}\le c}f(r,x)\,dx\;\in\;[0,\infty].
\end{align}
\end{enumerate}
\end{definition}

\begin{remark}
Theorem \ref{thm:b-shape} (more precisely,  \eqref{bus:exp1} and \eqref{bus:exp2} in Theorem \ref{bus:exp} below) 
 implies 
that $\P$-almost surely, for each $\lambda\in\R$ and $\sigg\in\{-,+\}$,   $\initF_\lambda$ contains the function 
\begin{align}\label{def:fla}
f_{\lambda\sig}(r,x)=\frac{\Glob^{\lambda\sig}(r,x)}{\Glob^{\lambda\sig}(r,0)}=
e^{\Bus^{\lambda\sig}(r,0,r,x)}. 
\end{align}
We will see later in Theorem \ref{thm:1f1s}\eqref{thm:1f1s.att} that this is a pullback attractor for the SHE \eqref{SHE}.
\end{remark}

When the time variable is not present, conditions for membership in $\initF_\lambda$ simplify to the ones in 
Lemma \ref{specialW} below.  
This lemma is proved in Appendix \ref{app:misc}.
In particular, 
for each $\lambda\in\R$, $\initF_\lambda$ contains the function
$f(r,x)= e^{\lambda x}$. 

\begin{lemma}\label{specialW}
Let $g:\R\to(0,\infty)$ be a
Borel function such that $\log g$ is locally bounded. For  $(r,x)\in\R_-\times\R$, set $f(r,x)=g(x)$.  Then $f\in\initF_\lambda$ if and only if 
\begin{align}
-\infty\le\varlimsup_{x\to-\infty}\frac{\log g(x)}{\abs{x}}&<\lambda=\lim_{x\to\infty}\frac{\log g(x)}{x}\qquad\text{when $\lambda>0$,}\label{G1}\\
\lim_{x\to-\infty}\frac{\log g(x)}{\abs{x}}&=\abs{\lambda}>\varlimsup_{x\to\infty}\frac{\log g(x)}{x}\ge-\infty\qquad\text{when $\lambda<0$, and}\label{G2}\\
-\infty\le\varlimsup_{\abs{x}\to\infty}\frac{\log g(x)}{\abs{x}}&\le0\qquad\text{when $\lambda=0$.}\label{G3}
\end{align}
\end{lemma}

 The next theorem is the function-to-point version of Theorem \ref{thm:buslim} and the last main result in this section.


\begin{theorem}\label{l2p-buslim}
The following hold $\P$-almost surely.  For each $\lambda\in\R$, $f\in\initF_\lambda$, $C>0$, $\e>0$, and $\tau>0$ 
there exists an $R<0$ such that for all  $r\le R$, 
for all $s,x,t,y\in[-C,C]$ with $t-s\ge\tau$, we have
	\begin{align*}
	\frac{\int_\R\She(t,y\viiva r,z)\,f(r,z)\,dz}{\int_\R\She(s,x\viiva r,z)\,f(r,z)\,dz}
	&\le (1+\e)^3\int_x^\infty\She(t,y\viiva s,w)e^{\Bus^{\lambda+}(s,x,s,w)}\,dw\\
	&\qquad\qquad +(1+\e)^3\int_{-\infty}^x\She(t,y\viiva s,w)e^{\Bus^{\lambda-}(s,x,s,w)}\,dw
	\end{align*} 
and 
	\begin{align*}
	\frac{\int_\R\She(t,y\viiva r,z)\,f(r,z)\,dz}{\int_\R\She(s,x\viiva r,z)\,f(r,z)\,dz}
	&\ge (1+\e)^{-4}\int_x^\infty\She(t,y\viiva s,w)e^{\Bus^{\lambda-}(s,x,s,w)}\,dw\\
	&\qquad\qquad
	+(1+\e)^{-4}\int_{-\infty}^x\She(t,y\viiva s,w)e^{\Bus^{\lambda+}(s,x,s,w)}\,dw.
	\end{align*} 
In particular, on a single event of  $\P$-probability one, for each $\lambda\not\in\Bruno$ and  $f\in\initF_\lambda$,
	\begin{align}\label{eq:buslim2}
	\lim_{r\to-\infty}\frac{\int_\R\She(t,y\viiva r,z)\,f(r,z)\,dz}{\int_\R\She(s,x\viiva r,z)\,f(r,z)\,dz}=e^{\Bus^\lambda(s,x,t,y)} \quad\text{locally uniformly in $(s,x,t,y)\in\R^4$.}
	\end{align}
\end{theorem}

The following comes from Theorem \ref{l2p-buslim} and Theorem \ref{th:La}\eqref{th:La:b}.

\begin{corollary}\label{cor:Busp2l}
For each $\lambda\in\R$ there exists an event $\Omega_\lambda$ with $\P(\Omega_\lambda)=1$ and such that \eqref{eq:buslim2} holds for all $\w\in\Omega_\lambda$ and $f\in\initF_\lambda$.  
\end{corollary}



\begin{remark}\label{forwardBus}
The $\refd_1$ reflection symmetry of the white noise and the Green's function \eqref{eq:cov.ref} implies that, as $r\to\infty$, 
analogues of Theorems \ref{thm:buslim} and \ref{l2p-buslim} hold for the ratios 
\[\frac{\She(r,z\viiva t,y)}{\She(r,z\viiva s,x)}\quad\text{and}\quad\frac{\int_\R\She(r,z\viiva t,y)\,f(r,z)\,dz}{\int_\R\She(r,z\viiva s,x)\,f(r,z)\,dz}.\]
\end{remark}

\medskip

 \subsection{Ergodicity}\label{subsec:ergodic}

We now turn to the structure of temporally stationary and ergodic initial conditions. We begin with  the SHE \eqref{SHE} and then address the KPZ equation \eqref{KPZ}.  

\subsubsection{Ergodicity of SHE} The cocycle property \eqref{CK8} and the independence structure of white noise imply that SHE \eqref{SHE} defines a Markov process.  There are several natural choices of state space for this process.  The most general space is the collection of all locally finite positive Borel measures with an additional cemetery state $\infimeas$ that  accounts for the possibility of finite-time blowup described in \eqref{pr:IC:ICMsh}.  
Denote this space by 
  $\overline\sM_+=\sM_+(\R)\cup\{\infimeas\}$.  It is Polish with the vague topology on $\sM_+(\R)$ and with $\infimeas$ as an isolated point.  By definition, for any $t\ge0$ such that   the solution $\She(t,dx\viiva 0,f)$ started from $f\in\overline\sM_+$ is not a locally finite positive measure,   we stipulate that  $\She(t,dx\viiva 0,f)=\infimeas$. 
  


 


Introduce 
an equivalence relation on $\overline\sM_+$:  $f\sim g$ if $f=cg$ for some $c>0$. 
Again, $f$ can be either a measure or the density function of a measure, since multiplication by a constant preserves the relationship of a measure and its density.   
Let $\eqcl{f}=\{g\in\overline\sM_+: g\sim f\}$ denote the equivalence class of $f\in\overline\sM_+$.  The cemetery state is alone in its equivalence class:   $\eqcl{\infimeas}=\{\infimeas\}$. Linearity ensures that the SHE evolution \eqref{SHE} is well-defined on  equivalence classes, that is,   $\eqcl{f}=\eqcl{g}$ implies $\eqcl{\She(t,\aabullet\tsp\viiva s,f)}=\eqcl{\She(t,\aabullet\tsp\viiva s,g)}$. 

To define a convenient Polish state space of equivalence classes, we restrict from $\sM_+(\R)$ 
to the space
	\be\label{MP509}  \MP=\{\eta\in\sM_+(\R):\supp(\eta)=\R\}\ee
of positive Radon measures on $\R$ whose closed support equals the entire real line, and add in the cemetery state to define  $\radon = \MP \cup\{\infimeas\}$. $\radon$ is a Polish space in its subspace topology inherited from $\overline\sM_+$ (see Appendix \ref{sec:quotient}).   The space $\radon$ is sufficient for studying stationarity of nonzero solutions because, as can be seen from \eqref{eq:allpos}, for all $f\in\overline\sM_+$ other than the zero measure, either (i) $\She(t,dx\viiva 0,f)=\infimeas$ for all sufficiently large $t$ or (ii) $f \in \ICM$ and the support of $\She(t,dx\viiva 0,f)$  is all of $\R$ for all $t>0$. 

The space of equivalence classes is the quotient space 
\be\label{Radon509} \Radon=\radon/\!\!\sim\;=\{\eqcl f:  f\in\radon\}.  \ee 
  We show in Appendix \ref{sec:quotient} that $\Radon$ is a Polish space in its quotient  topology.  Denote generic elements of   $\Radon$  by $\clf$, so that $f\in\clf$ is equivalent to $\eqcl f=\clf$.  
 We take $\Radon$ as 
the state space of the SHE evolution $\Sols_{0,t}\clf=\Sols^\w_{0,t}\clf=\eqcl{\She^\w(t,\aabullet\tsp\viiva 0,f)}$ 
on equivalence classes, well-defined for any representative $f\in\clf$ of the initial state $\clf\in\Radon$.  

We recall briefly the standard notions of invariance, ergodicity, and total ergodicity. 
Given  an initial probability measure $\PMPinit$ on $\Radon$, let  the initial state $\clf$ have distribution $\PMPinit$ and denote  by  $\PMP$ the  probability distribution  of the Markov process $\{\Sols_{0,t}\clf: t\ge0\}$.  $\PMP$ is a probability measure  on the product space $\Radon^{\R_+}$ equipped with its product $\sigma$-algebra.    
 $\PMPinit$ is a {\it stationary measure} or {\it invariant distribution} for the evolution if the distribution of $\{\Sols_{0,t}\clf: t\ge0\}$ is invariant under time shifts, that is, if the processes $\{\Sols_{0,t}\clf: t\ge0\}$ and $\{\Sols_{0,s+t}\clf: t\ge0\}$ have the same distribution for all $s\ge0$. 
 In this case   $\PMP$ extends to   the  space $\Radon^{\R}$ of bi-infinite paths and $\PMP$ is invariant under the  time shift group $\{\shiftp_t: t\in\R\}$.     
 Time shifts act on elements $\clf_{\tspb\bbullet}=(\clf_{\tspa s})_{s\in\R}$ of $\Radon^{\R}$ by $(\shiftp_t\clf_{\,\bbullet})_s=\clf_{s+t}$. 
 
A stationary measure $\PMPinit$ is {\it ergodic} if  $\PMP$ is ergodic under the group of time shifts. 
This means that   $\PMP(A)\in\{0,1\}$ for every measurable set $A\subset\Radon^{\R}$ that satisfies  for each $t\in\R$ that  
$\shiftp_t^{-1}A=A$ $\PMP$-almost surely.    Stationary measures form a convex set  whose  extreme points are precisely the ergodic measures.  
Finally,  $\PMPinit$ is {\it totally ergodic} if  $\PMP$ is ergodic under each individual shift 
$\shiftp_t$ with $t \neq 0$. That is, for each $t\in\R\setminus\{0\}$ separately, $\PMP(A)\in\{0,1\}$ for every measurable set $A\subset\Radon^{\R}$ that satisfies 
$\shiftp_t^{-1}A=A$ $\PMP$-almost surely. 

The product space $\Radon^{\R}$ can be replaced by the space $\cC(\R,\Radon)$ of continuous $\Radon$-valued paths when blowup does not happen.  That is, 
 if the initial distribution $\PMPinit$ is  supported on equivalence classes of measures from the subspace  $\ICM$ of \eqref{eq:ICM},  
then the process $\Sols_{0,t}\clf$ has continuous paths. See Lemma \ref{lm:SHE-MC}. Indeed, this continuity holds in a much more restrictive topology.

It is shown in Theorem 2.6 of \cite{Alb-etal-22-spde-} that when blowups are ruled out by choosing an initial measure $f\in\ICM$, for any $t>0$,
the process $\She(t,\aabullet\viiva 0,f)$ is a strictly positive continuous function.  Hence perhaps the most natural smaller space that supports invariant distributions on locally finite measures is $\CICM$, which was previously introduced in \eqref{CICM},
\[  \CICM=\Bigl\{ f\in \sC(\R,(0,\infty)):       \forall a >0,  \int_{\bbR} e^{-a x^2} f(x)\tspb dx <\infty \Bigr\},  \]
  the space  of  strictly positive continuous densities of measures in $\ICM$. Measures represented by a density in  
  $\CICM$ form a Borel subset of $\sM_+(\R)$ (Lemma \ref{lm:density}). $\CICM$ is Polish in its natural topology, in which convergence is equivalent to uniform convergence on compact sets combined with convergence of the integrals appearing in the definition of $\CICM$. We completely metrize this topology explicitly in equation \eqref{d_CICM}.   If started from $f\in\CICM$, the process $t\mapsto\She(t,\aabullet\viiva 0,f)$ has continuous $\CICM$-valued paths  
  by Theorem 2.9 of  \cite{Alb-etal-22-spde-}. 

The space of equivalence classes of functions in $\CICM$  is denoted by  $\CICP=\{ \eqcl{f}: f\in\CICM\}$.
 With its quotient topology, $\CICP$ is homeomorphic to the closed subspace $\{f\in\CICM: f(0)=1\}$ of $\CICM$ and hence is itself Polish.  It follows that with initial state  $\clf\in\CICP$, paths of the equivalence class process $t\mapsto\Sols_{0,t}\clf$ are continuous and hence reside  in the space $\sC(\R_+, \CICP)$.  The Markov process defined in this way is Feller in the sense that the finite dimensional marginals are weakly continuous in the initial condition. Indeed, the structure of our coupling shows that the full path distribution is weakly continuous in the initial condition. See Remark \ref{rem:Feller} below.


\begin{theorem}\label{thm:unique}
Let $B$ denote a standard two-sided Brownian motion.    
\begin{enumerate} [label={\rm(\roman*)}, ref={\rm\roman*}]   \itemsep=3pt  
\item\label{thm:unique.i} For each  $\lambda\in\R$, the distribution of $\eqcl{e^{B(\aabullet)+\lambda\aabullet}}$  is stationary and  totally ergodic  for the $\CICP$-valued process $t\mapsto\Sols_{0,t}\clf$. 

\item\label{thm:unique.ii}   
Let $\PMPinit$ be a probability measure on $\Radon$ that is ergodic for the $\Radon$-valued process $t\mapsto\Sols_{0,t}\clf$.    Then either   $\PMPinit=\delta_{\eqcl\infimeas}$  or statements 
{\rm\ref{P.iia}--\ref{P.iib}} below hold:   
\begin{enumerate} [label={\rm(ii.\alph*)}, ref={\rm(ii.\alph*)}]   \itemsep=3pt  
\item \label{P.iia}   $\PMPinit$ is supported on $\CICP$ and there are deterministic finite constants $\PMPslopem\le \PMPslopep$ such that for $\PMPinit$-almost every $\clf\in\CICP$ and all $f\in\clf$,   \begin{align}\label{growth-rates-SHE}
\ddd\lim_{x\to-\infty} x^{-1}\log f(x) =\PMPslopem\quad\text{and}\quad \ddd\lim_{x\to\infty} x^{-1}\log f(x) =\PMPslopep.
\end{align}
  
  
 \item  \label{P.iib}  Assume that either $\PMPslopem\ne-\PMPslopep$  or  $\PMPslopem=\PMPslopep=0$. 
 Then $\PMPslopem=\PMPslopep$ and $\PMPinit$ is the distribution of $\eqcl{e^{B(\aabullet)+\lambda\aabullet}}$ with $\lambda=\PMPslopem=\PMPslopep$.   
 \end{enumerate} 
\end{enumerate}
\end{theorem}

The proof of this theorem in Section \ref{sec:unique} relies crucially on the results on the Busemann process.   We add some remarks to it.  

Concretely, part (\ref{thm:unique.i})  says that with initial function $f(x)=e^{B(x)+\lambda x}$ defined in terms of a Brownian motion $B(\aabullet)$ independent of the white noise, the   process $t\mapsto\frac{\She(t,\aabullet\viiva 0,f)}{\She(t,0\viiva 0,f)}$, indexed by $\R_+$ and with values in the Polish space $\{f\in \CICM: f(0)=1\}$, is stationary and ergodic under each individual  nonzero time shift.




An important corollary of part (\ref{thm:unique.ii})  is the characterization of stationary  distributions that are invariant under spatial translations, and more generally, those whose left and right asymptotic logarithmic growth rates $\PMPslopem$ and $\PMPslopep$ agree.  


 Spatial translations of measures and their equivalence classes are defined in the obvious way.   For $a\in\R$ and $f\in\MP$ let $\shiftx_a\eqcl{f}$ be the equivalence class $\eqcl{\shiftx_af}$  of the measure $\shiftx_af$ defined by $\shiftx_af(A)=f(A+a)$ for Borel $A\subset\R$.  This applies to functions similarly:  if $f$ is the density function  of the measure $\eta$, then $\shiftx_af(x)=f(x+a)$ is the density function  of  $\shiftx_a\eta$.   The distribution of 
$ \eqcl{e^{B(\acbullet) +\lambda\acbullet}}$
  is invariant under every spatial shift, as seen by recentering the Brownian motion: 
 \begin{align*}
\shiftx_a \eqcl{e^{B(\acbullet) +\lambda\acbullet}} 
= \eqcl{e^{B(a+\acbullet)+\lambda(a+\acbullet)}} 
  = \eqcl{e^{B(a+\acbullet)-B(a) +\lambda\acbullet}} 
 \deq  \eqcl{e^{B(\acbullet) +\lambda\acbullet}} . 
\end{align*}
 The corollary below follows from Theorem \ref{thm:unique}\eqref{thm:unique.ii} and the fact that spatial invariance implies equal  left and right growth rates. 
 
\begin{corollary}\label{cor:erg}     Let $\PMPinit\ne\delta_{\eqcl\infimeas}$ be a probability measure on $\Radon$ that is stationary and ergodic for the $\Radon$-valued process $t\mapsto\Sols_{0,t}\clf$.    
\begin{enumerate} [label={\rm(\roman*)}, ref={\rm\roman*}]   \itemsep=3pt  
\item\label{cor:erg.i}
 Suppose  $\PMPinit$ has equal  left and right  growth rates $\lambda=\PMPslopem=\PMPslopep$.  Then  $\PMPinit$ is the distribution of $\eqcl{e^{B(\acbullet)+\lambda\acbullet}}$. 
 
 \item\label{cor:erg.ii}   Suppose $\PMPinit$  is invariant under at least one spatial translation $\shiftx_a$ for some  $a\ne 0$.   Then $\PMPslopep=\PMPslopem$   and $\PMPinit$ is the distribution of $\eqcl{e^{B(\acbullet)+\lambda\acbullet}}$ for $\lambda=\PMPslopep=\PMPslopem$. 

 \end{enumerate}  
\end{corollary}

We generalize these results to the joint   evolution of multiple initial measures under a common realization of the white noise. For  $n\in\N$  let $\clf^{1:n}=(\clf^1,\dotsc,\clf^n)$ denote an element of the $n$-fold Cartesian product $\Radon^n$.  On the space $\Omega\times\Radon^n$ define the evolution 
\be\label{Sols-n}   \Sols^{(n)}_{0,t}\clf^{1:n}
=( \Sols_{0,t}\clf^1, \dotsc, \Sols_{0,t}\clf^n)
= \bigl( \tspb \eqcl{\She^\w(t,\aabullet\tsp\viiva 0,f^1)}, \dotsc, \eqcl{\She^\w(t,\aabullet\tsp\viiva 0,f^n)}\tspb\bigr) 
\ee 
where  $f^i\in\clf^i$ and $t>0$. The common superscript  $\w$ signals that each initial condition $f^i$ is updated with the same  $\She^\w$ in \eqref{NEind:m9}. 

\begin{theorem}\label{thm:inv-n}
Let $n\in\N$.    
\begin{enumerate} [label={\rm(\roman*)}, ref={\rm\roman*}]   \itemsep=3pt  
\item\label{thm:inv-n.i} Fix    $\lambda_1,\dotsc,\lambda_n\in\R$. Then  the distribution of $\bigl(\tspb\eqcl{e^{\Bus^{\lambda_1}(0,0,0,\aabullet)}}, \dotsc, \eqcl{e^{\Bus^{\lambda_n}(0,0,0,\aabullet)}}\tspb\bigr) $  is stationary and  totally ergodic  for the $\nCICP$-valued process $t\mapsto\Sols^{(n)}_{0,t}\clf^{1:n}$ of \eqref{Sols-n}.  In particular, the $\nCICM$-valued process
$t\mapsto(\tspb e^{\Bus^{\lambda_1}(t,0,t,\aabullet)}, \dotsc, e^{\Bus^{\lambda_n}(t,0,t,\aabullet)}\tspb) $ 
  is stationary and  totally ergodic. 

\item\label{thm:inv-n.ii}   
Let $\PMPinit^{(n)}$ be a probability measure on $\Radon^n$ that is ergodic for the  process of \eqref{Sols-n} and suppose that     $\PMPinit^{(n)}\{ \clf^{1:n}: \clf^i={\eqcl\infimeas}\}=0$ for each $i$.  Then  
{\rm\ref{inv-n.iia}--\ref{inv-n.iib}} below hold: \\[-6pt]   
\begin{enumerate} [label={\rm(ii.\alph*)}, ref={\rm(ii.\alph*)}]   \itemsep=4pt  
\item \label{inv-n.iia}   $\PMPinit^{(n)}$ is supported on $\nCICP$ and the deterministic finite asymptotic slopes exist: 
\be\label{inv-n.9}   \PMPslopem^i=\ddd\lim_{x\to-\infty} x^{-1}\log f^i(x) \quad\text{ and }\quad  \PMPslopep^i=\ddd\lim_{x\to\infty} x^{-1}\log f^i(x)\ee 
  for $\PMPinit^{(n)}$-almost every $\clf^{1:n}\in\nCICP$, each $i$  and   all $f^i\in\clf^i$. Moreover, for each $i$, $-\infty <\PMPslopem^i \leq \PMPslopep^i <\infty$.

  
 \item  \label{inv-n.iib}  Assume that for each $i$,  either $\PMPslopem^i\ne-\PMPslopep^i$  or  $\PMPslopem^i=\PMPslopep^i=0$. 
 Then $\PMPslopem^i=\PMPslopep^i$ for each $i$ and $\PMPinit^{(n)}$ is the distribution of $\bigl(\tspb\eqcl{e^{\Bus^{\lambda_1}(0,0,0,\aabullet)}}, \dotsc, \eqcl{e^{\Bus^{\lambda_n}(0,0,0,\aabullet)}}\tspb\bigr) $  with $\lambda_i=\PMPslopem^i=\PMPslopep^i$.
 \end{enumerate} 
\end{enumerate}
\end{theorem}

There is an immediate corollary analogous to Corollary \ref{cor:erg}. We omit the statement. 

\subsubsection{Ergodicity of KPZ}\label{sec:KPZer} Next, we explain some immediate consequences of the above results for the KPZ equation. Recall the space
\[
\CKPZ=\bigg\{f \in \sC(\R,\R) : \int_{\R} e^{f(x) - a x^2}dx <\infty \text{ for all }a>0 \bigg\},
\]
which was introduced in Section \ref{sec:KPZ}. $\CICM$ is homeomorphic to $\CKPZ$ through the bijection $f \in \CICM \mapsto \log f \in \CKPZ$. Convergence in $\CKPZ$ is equivalent to uniform convergence on compact sets combined with convergence of the integrals appearing in the definition of $\CKPZ$. Recall also the space $\CKPZt$, which is $\CKPZ$ modulo equivalence up to additive constants. We denote equivalence classes under this identification via $\langle f \rangle = \{g \in \CKPZ : \exists c \in \R \text{ such that } f(\abullet) = g(\abullet) + c\}$. It is straightforward to see that the map $f \in \CICM \mapsto \log f \in \CKPZ$ induces a homeomorphism between $\CICP$ and $\CKPZt$. 

Let $\clf \in \CKPZt$ and let $f \in \clf$ be arbitrary. The KPZ evolution is defined through \eqref{eq:genKPZ} via
\[
\Solk_{0,t} \clf (\abullet) = \langle \KPZ(t, \abullet \viiva 0, f) \rangle \in \CKPZt
\]
It follows from the definition that this evolution is well-defined. The resulting Markov process taking values in $\sC([0,\infty),\CKPZt)$ is Feller. Indeed, this is just the Markov process obtained from the one constructed above through $\Sols$ on $\CICP$ by applying the homeomorphism between $\CICP$ and $\CKPZt$ induced by the homeomorphism  $f \in \CICM \mapsto \log f \in \CKPZ$. By this observation, the following is an immediate consequence of Theorem \ref{thm:unique}.
\begin{theorem}\label{thm:KPZunique}
Let $B$ denote a standard two-sided Brownian motion.    
\begin{enumerate} [label={\rm(\roman*)}, ref={\rm\roman*}]   \itemsep=3pt  
\item\label{thm:uniqueKPZ.i} For each  $\lambda\in\R$, the distribution of $\langle B(\aabullet)+\lambda\aabullet\rangle$  is stationary and  totally ergodic  for the $\CKPZt$-valued process $t\mapsto\Solk_{0,t}\clf$. 
\item\label{thm:uniqueKPZ.ii}   
Let $\PMPinit$ be a probability measure on $\CKPZt$ that is ergodic for the $\CKPZt$-valued process $t\mapsto\Solk_{0,t}\clf$.    Then  statements 
{\rm\ref{P.KPZ.iia}--\ref{P.KPZ.iib}} below hold:   
\begin{enumerate} [label={\rm(ii.\alph*)}, ref={\rm(ii.\alph*)}]   \itemsep=3pt  
\item \label{P.KPZ.iia}   There are deterministic finite constants $\PMPslopem\le \PMPslopep$ such that for $\PMPinit$-almost every $\clf\in\CKPZt$,    $\ddd\lim_{x\to-\infty} x^{-1}f(x) =\PMPslopem$ and $\ddd\lim_{x\to\infty} x^{-1}f(x) =\PMPslopep$ for  all $f\in\clf$.
  
 \item  \label{P.KPZ.iib}  Assume that either $\PMPslopem\ne-\PMPslopep$  or  $\PMPslopem=\PMPslopep=0$. 
 Then $\PMPslopem=\PMPslopep$ and $\PMPinit$ is the distribution of $\langle B(\aabullet)+\lambda\aabullet\rangle$ with $\lambda=\PMPslopem=\PMPslopep$.   
 \end{enumerate} 
\end{enumerate}
\end{theorem} 
We have the following consequence of the previous result. 
\begin{corollary}\label{cor:KPZerg}     Let $\PMPinit$ be a probability measure on $\CKPZt$ that is stationary and ergodic for the $\CKPZt$-valued process $t\mapsto\Solk_{0,t}\clf$.    
\begin{enumerate} [label={\rm(\roman*)}, ref={\rm\roman*}]   \itemsep=3pt  
\item\label{cor:KPZPerg.i}
 Suppose  $\PMPinit$ has equal  left and right  growth rates $\lambda=\PMPslopem=\PMPslopep$.  Then  $\PMPinit$ is the distribution of $\langle B(\acbullet)+\lambda\acbullet\rangle$. 
 
 \item\label{cor:KPZerg.ii}   Suppose $\PMPinit$  is invariant under at least one spatial translation.   Then $\PMPslopep=\PMPslopem$   and $\PMPinit$ is the distribution of $\langle B(\acbullet)+\lambda\acbullet\rangle$ for $\lambda=\PMPslopep=\PMPslopem$.
 \end{enumerate}  
\end{corollary}  

Note that the previous result implies the prediction implicit in the KPZ scaling theory \cite{Kru-Mea-Hal-92} (and more explicit in \cite{Spo-14}) that the spatially translation invariant and temporally ergodic stationary distributions form a one parameter family indexed, for example, by the mean of a fixed increment. 

\begin{remark}\label{rem:wesuck}
Theorem \ref{thm:KPZunique} leaves open the question of whether or not there exist stationary and ergodic measures $\PMPinit$ for the KPZ equation (modulo additive constants) supported on (equivalence classes of) functions for which $-\PMPslopem= \PMPslopep$ for some $\PMPslopep>0$. Proving that such measures either do or do not exist is Open Problem \ref{prob:Ergodic} below. If such measures are shown not to exist, this would confirm the prediction in \cite[Remark 1.1]{Fun-Qua-15}.
\end{remark}

By the same logic as above, we have the following analogue (and consequence) of Theorem \ref{thm:inv-n}. For  $n\in\N$  let $\clf^{1:n}=(\clf^1,\dotsc,\clf^n)$ denote an element of the $n$-fold Cartesian product $\CKPZt^n$.  On the space $\Omega\times\CKPZt^n$ define the evolution 
\be\label{Solk-n}   \Solk^{(n)}_{0,t}\clf^{1:n}
=( \Solk_{0,t}\clf^1, \dotsc, \Solk_{0,t}\clf^n)
= \bigl( \tspb \langle\KPZ^\w(t,\aabullet\tsp\viiva 0,f^1)\rangle, \dotsc, \langle \KPZ^\w(t,\aabullet\tsp\viiva 0,f^n)\rangle\tspb\bigr) 
\ee 
where  $f^i\in\clf^i$ and $t>0$. The common superscript  $\w$ signals that each initial condition $f^i$ is updated with the same $\She^\w$ in \eqref{eq:genKPZ}. 

\begin{theorem}\label{thm:KPZinv-n}
Let $n\in\N$.    
\begin{enumerate} [label={\rm(\roman*)}, ref={\rm\roman*}]   \itemsep=3pt  
\item\label{thm:KPZinv-n.i} Fix    $\lambda_1,\dotsc,\lambda_n\in\R$. Then  the distribution of $\bigl(\tspb \langle \Bus^{\lambda_1}(0,0,0,\aabullet)\rangle, \dotsc, \langle\Bus^{\lambda_n}(0,0,0,\aabullet)\rangle\tspb\bigr) $  is stationary and  totally ergodic  for the $\nCKPZt$-valued process $t\mapsto\Solk^{(n)}_{0,t}\clf^{1:n}$ of \eqref{Solk-n}.  In particular, the $\nCKPZ$-valued process
$t\mapsto(\tspb \Bus^{\lambda_1}(t,0,t,\aabullet), \dotsc, \Bus^{\lambda_n}(t,0,t,\aabullet)\tspb) $ 
  is stationary and  totally ergodic. 

\item\label{thm:KPZinv-n.ii}   
Let $\PMPinit^{(n)}$ be a probability measure on $\CKPZt^n$ that is ergodic for the  process of \eqref{Solk-n}.  Then  
{\rm\ref{KPZinv-n.iia}--\ref{KPZinv-n.iib}} below hold: \\[-6pt]   
\begin{enumerate} [label={\rm(ii.\alph*)}, ref={\rm(ii.\alph*)}]   \itemsep=4pt  
\item \label{KPZinv-n.iia}  Deterministic finite asymptotic slopes  exist: 
\be\label{KPZinv-n.9}   \PMPslopem^i=\ddd\lim_{x\to-\infty} x^{-1}f^i(x) \quad\text{ and }\quad  \PMPslopep^i=\ddd\lim_{x\to\infty} x^{-1} f^i(x)\ee 
  for $\PMPinit^{(n)}$-almost every $\clf^{1:n}\in\nCKPZt$, each $i$  and   all $f^i\in\clf^i$. Moreover, for each $i$, $-\infty <\PMPslopem^i \leq \PMPslopep^i <\infty$.

  
 \item  \label{KPZinv-n.iib}  Assume that for each $i$,  either $\PMPslopem^i\ne-\PMPslopep^i$  or  $\PMPslopem^i=\PMPslopep^i=0$. 
 Then $\PMPslopem^i=\PMPslopep^i$ for each $i$ and $\PMPinit^{(n)}$ is the distribution of $\bigl(\tspb\langle \Bus^{\lambda_1}(0,0,0,\aabullet)\rangle, \dotsc, \langle \Bus^{\lambda_n}(0,0,0,\aabullet)\rangle \tspb\bigr) $  with $\lambda_i=\PMPslopem^i=\PMPslopep^i$.
 \end{enumerate} 
\end{enumerate}
\end{theorem}
 In words, the previous result implies in particular that the joint law of the Busemann process gives the unique ergodic jointly stationary coupling of the invariant measures (modulo additive constants) for the KPZ equation given by Brownian motion with drift.
 
\subsection{Synchronization}\label{sub:1F1S}
In this section, we recast some of our results in the language of random dynamical systems (RDS). We state these reformulations in terms of the SHE. Similar to Section \ref{sec:KPZer}, the results below imply analogous results for the KPZ equation. See Remark \ref{rem:KPZSynch}.

Recall the Polish quotient space $\CICP$ and the operators $\Sols_{s,t}$, $s\le t$, defined in Section \ref{subsec:ergodic}. Define the mapping 
	\begin{align}\label{phi-def}
	\varphi:[0,\infty)\times\Omega\times\CICP\to\CICP,\quad (t,\w,\clf)\mapsto \varphi(t,\w,\clf)=\Sols_{0,t}\clf=\eqcl{\She(t,\aabullet\tsp\viiva 0,f)}
	\end{align}
	for $f\in\clf$. 
By \eqref{NEind:m9} we have $\varphi(0,\w,\clf)=\clf$ for all $\clf\in\CICP$ and all $\w$.
For each $(t,\clf)$, $\varphi$ is measurable in the $\w$ variable and by Theorem 2.9 of \cite{Alb-etal-22-spde-}, there is a full $\P$-probability event $\Omega_0$ such that     $(t,\clf)\mapsto\varphi(t,\w,\clf)$ is continuous for $\w\in\Omega_0$.

We abbreviate temporal shifts as $\shiftp_t=\shiftd{t}{0}$. 
By the cocycle property \eqref{CK8}  and the shift-covariance \eqref{eq:cov.shift} of $\She$, for each $s\ge0$ the exists an event $\Omega^{(s)}$ such that $\P(\Omega^{(s)})=1$ and for each  $\w\in\Omega^{(s)}$ the cocycle property holds:
for all $t\in [s,\infty)$ and $\clf\in\CICP$, 
\be\label{fi-coc} \varphi\bigl(t-s,\shiftp_s\w,\varphi(s,\w,\clf)\bigr)=\varphi(t,\w,\clf).\ee
$\varphi$ is called a \emph{crude cocycle} and it defines a (crude) measurable random dynamical system (RDS) on $\CICP$ over $(\Omega,\sF,\P,\{\shiftp_t:t\ge0\})$. See 
Definition 1.1.1 in \cite{Arn-98} or Definition 6.2 in \cite{Cra-02}. 
$\varphi$ defines a continuous RDS 
under 
both Definition 1.1.2 in \cite{Arn-98} and the definition in the paragraph following Remark 6.5 in \cite{Cra-02}.



A random variable $\w\mapsto\clf^\w$  
from $\Omega$ into $\CICP$  
is said to be (strictly) {\it $\varphi$-invariant}  if 
\be\label{fi-inv} 
\text{for each $t>0$ we have   $\varphi(t,\w,\clf^\w)=\clf^{\shiftp_t\w}$ $\P$-almost surely. } 
\ee
See Definition 6.9(iii) in \cite{Cra-02}.
$\clf$ is {\it Markovian}  if  
$\clf^\w$ is $\fil_{-\infty:0}$-measurable. If $\clf$ is $\varphi$-invariant and Markovian, then the distribution of $\clf^\w$ under $\P$ is invariant and totally ergodic for the Markov process $t\mapsto\Sols_{0,t}\clf$.  See Lemma \ref{lm:attr2}. 

\begin{lemma}\label{lm:u-f} 
There is a one-to-one correspondence between $\varphi$-invariant $\CICP$-valued random variables $\clf^\w$ and 
$\sC(\R, \CICM)$-valued random variables 
  $\Glob^\w$  that satisfy these properties: 
\begin{enumerate} [label={\rm(\roman*)}, ref={\rm\roman*}]   \itemsep=3pt
    \item\label{u-f.i}  $\P$-almost surely, for all pairs $t\ge s$ in $\R^2$ and all $y\in\R$, 
$\She^\w\bigl(t,y\viiva s,\Glob^\w(s,\aabullet)\bigr)=\Glob^\w(t,y)$.
\item $\Glob^\w(0,0)=1$.
\item The equivalence class is  time-shift covariant in the following sense:  for each $r\in\R$ there exists an event $\Omega_r$ such that $\P(\Omega_r)=1$ and  $\eqcl{\Glob^{\shiftp_r\w}(t,\aabullet)}=\eqcl{\Glob^{\w}(t+r,\aabullet)}$ for all $\w\in\Omega_r$ and $t\in\R$. 
\end{enumerate}  
This  $\clf^\w\leftrightarrow u^\w$ correspondence is given by the  equations
   \begin{align}\label{f->u1}
    \Glob^\w(t,x)=\frac{\She^\w(t,x\viiva s,f^{\shiftp_s\w})}{\She^\w(0,0\viiva s,f^{\shiftp_s\w})}
    \qquad\text{and}\qquad 
    \clf^\w=\eqcl{\Glob^\w(0,\aabullet)}
    \end{align}
where on the left $s$ can be any rational $s<\min(t,0)$ and $f:\Omega\to\CICM$ is the function defined by $f^\w\in\clf^\w$ and 
$f^\w(0)=1$.      
\end{lemma}

Item \eqref{u-f.i} states that 
   $\Glob$ is an eternal physical solution of SHE \eqref{SHE}. 
The definition of $\Glob$ in \eqref{f->u1} does not depend on the choice of the representative from $\clf^{\shiftp_s\w}$ or on the rational  $s<\min(t,0)$.  The lemma is proved in Section \ref{sec:1F1S}.



By Lemma \ref{lm:u-f}, if 
one had an almost surely unique $\varphi$-invariant random variable, then there would be a unique way to measurably map almost every realization of the white noise forcing to a covariant eternal physical solution. 
In such a setting, \cite{E-etal-00} (page 879), among others, say that the \emph{one force--one solution principle} (1F1S) holds.

In contrast, in a situation such as the SHE \eqref{SHE} with  multiple ergodic invariant measures, one expects  distinct $\varphi$-invariant random variables corresponding to distinct  ergodic measures. 
Given a probability measure $\PMPinit$ on $\CICP$ that is invariant and ergodic for the Markov process $t\mapsto\Sols_{0,t}\clg$ we  say that the \textit{1F1S principle holds for $\PMPinit$} if there exists a $\P$-almost surely unique $\varphi$-invariant $\CICP$-valued random variable $\clf$ such that the distribution of $\w\mapsto\clf^\w$ under $\P$ is $\PMPinit$. This means that there is a unique way to measurably map almost every realization of the white noise forcing $\w$ to a covariant eternal physical solution $\Glob^\w$ 
such that the distribution of $\w\mapsto\eqcl{\Glob^\w(0,\aabullet)}$ under $\P$ is $\PMPinit$.

\emph{$\PMPinit$-synchronization} occurs when there exists a $\varphi$-invariant random variable $\clf:\Omega\to\CICP$, a countable set $\ttset\subset\R$ with $\sup\ttset=\infty$, and two events $\wt\cC_0\subset\CICP$ and $\Omega_0\subset\Omega$ such that $\PMPinit(\wt\cC_0)=\P(\Omega_0)=1$ and
\begin{align}\label{attractor-prob}
\lim_{\ttset\ni t\to\infty}\dCICP\bigl(\varphi(t,\shiftp_{-t}\w,\clg),\clf^\w\bigr)=0 
\qquad\text{for $\w\in\Omega_0$ and $\clg\in\wt\cC_0$.}  
\end{align}
Here and in the sequel, $\dCICP$ is a complete metric for the topology of $\CICP$. 
The idea is  that for initial conditions  $\clg\in\wt\cC_0$, the solution started from $\clg$ in the remote past ``synchronizes'' with $\clf^\w$. The random variable $\clf$ is a \emph{random \textup{(}point\textup{)} $\PMPinit$-attractor}. 

\begin{remark}
When \eqref{attractor-prob} holds for \textit{all} $\clg\in\CICP$, $\clf$ is said to be a \emph{global} attractor. 
This does not happen in our setting. 
\end{remark}

$\PMPinit$-synchronization 
implies the 1F1S holds for $\PMPinit$ by Proposition \ref{syn=>uniq}. 
Showing 
$\PMPinit$-synchronization is a common approach for proving 1F1S, as in
\cite{E-etal-00,Ros-21-,Bak-Cat-Kha-14,Bak-16,Bak-Li-19,Bak-Kha-18}.  



For $\lambda\in\R$ and $\sigg\in\{-,+\}$  (and $\w$ from the full $\P$-probability event on which the Busemann process is defined) let
    \be\label{f-la6}f_{\lambda\sig}^\w(x)=e^{\Bus^{\lambda\sig}(0,0,0,x)}=\Glob^{\lambda\sig}(0,x)
    \qquad\text{and}\qquad 
    \clf^\w_{\lambda\sig}=\eqcl{f^\w_{\lambda\sig}},\ee
where $\Glob^{\lambda\pm}$ are the solutions defined in \eqref{Shela}.
%
%
%
Recall the set $\Bruno$ of exceptional slopes defined in \eqref{Bruno}. By Corollary \ref{L-char2} we have
    \[\Bruno=\{\lambda\in\R:f^\w_{\lambda-}\ne f^\w_{\lambda+}\}=\{\lambda\in\R:\clf^\w_{\lambda-}\ne\clf^\w_{\lambda+}\}.\] 
As for the Busemann process, when $\lambda\not\in\Bruno$ we write
$\clf^\w_\lambda=\clf^\w_{\lambda+}=\clf^\w_{\lambda-}$ and $f^\w_\lambda=f^\w_{\lambda+}=f^\w_{\lambda-}$. 
This is the case with probability one when $\lambda$ is fixed, since $\P(\lambda\in\Bruno)=0$ by 
 Theorem \ref{th:La}\eqref{th:La:b}.

The following theorem implies that for any fixed value of the slope $\lambda$, $\PMPinit_\lambda$-synchronization and the corresponding 1F1S principle hold,  where $\PMPinit_\lambda$ is the distribution of $\eqcl{e^{B(\abullet)+\lambda \abullet}}$ and $B$ is a standard two-sided Brownian motion.

\begin{theorem}\label{cor:1f1s}
Fix $\lambda\in\R$. Then $\clf_\lambda$ of \eqref{f-la6} is the {\rm(}$\P$-almost surely{\rm)} unique $\varphi$-invariant $\CICP$-valued random variable such that,  with $\P$-probability one, any representative  $f\in\clf_\lambda$ satisfies  \eqref{G1}--\eqref{G3}. Furthermore, the following synchronization statements hold.
\begin{enumerate}[label={\rm(\alph*)}, ref={\rm\alph*}] \itemsep=3pt
\item For any Borel function $g:\R\to(0,\infty)$ such that $\log g$ is locally bounded and satisfies \eqref{G1}--\eqref{G3}, we have 	\begin{align}\label{1f1s-weak}
\lim_{t\to\infty} \dCICP\bigl(\varphi(t,\w,\eqcl{g}),\clf_\lambda^{\shiftp_t\w}\bigr)=0\quad\text{in $\P$-probability}.
\end{align}
\item\label{cor:1f1s.b} For any countable subset $\ttset\subset[0,\infty)$ with $\sup\ttset=\infty$, there exists an event $\Omega_{\lambda,\ttset}$ such that $\P(\Omega_{\lambda,\ttset})=1$ and for any $\w\in\Omega_{\lambda,\ttset}$ and any Borel function $g:\R\to(0,\infty)$ such that $\log g$ is locally bounded and satisfies \eqref{G1}--\eqref{G3}, 
\begin{align}\label{1f1s-tmp}
\lim_{\ttset\ni t\to\infty} \dCICP\bigl(\varphi(t,\shift_{-t}\w,\eqcl{g}),\clf^\w_\lambda\bigr)=0.
\end{align}
\end{enumerate}
\end{theorem}

\begin{remark}
Results analogous to Theorem \ref{cor:1f1s} have been shown previously in non-compact settings for the Burgers equation with Poisson forcing \cite{Bak-Cat-Kha-14} and smooth random kick forcing \cite{Bak-Li-19,Bak-16}. 
 It is expected that this form of synchronization and 1F1S hold for a general class of stochastic Hamilton-Jacobi equations which includes the KPZ equation \eqref{KPZ}.
See Conjecture 1 in \cite{Bak-Kha-18}.
\end{remark}



\begin{remark}\label{rem:discont}
The almost sure limit in \eqref{1f1s-tmp} is restricted to a fixed sequence $\ttset$ because we cannot define uncountably many time-shifts of $\She$ on a single event of full probability on our probability space $(\Omega,\sF,\bbP)$. Recalling \eqref{eq:cov.shift}, the essential issue here is that our Green's function was originally built using (the chaos expansion of) the mild formulation of \eqref{eq:FSol} in \eqref{eq:genmild} and such objects are only defined up to $L^2(\Omega,\sF,\bbP)$ equivalence classes. Remark \ref{rk:ttset} discusses options for removing this restriction.
\end{remark}

Theorem \ref{cor:1f1s} describes the behavior of the RDS for a class of initial conditions that supports a given ergodic measure.
A natural follow-up question is to 
describe the (quenched) long-term behavior of the RDS for a fixed,  typical realization $\w$ of the random environment.   
Therefore, we drop the idea of working with one ergodic measure at a time and instead describe a family of random attractors 
 and basins of attraction,  
for each $\w$ outside a single $\P$-null event. 

 The basins of attraction considered in Theorem \ref{cor:1f1s}  include only functions that depend on the space variable $x$ but not on $\w$, since the focus there is on one ergodic measure at a time. 
While these spaces include important cases such as 
all (equivalence classes of) locally bounded Borel-measurable functions $f:\R\to(0,\infty)$
satisfying \eqref{growth-rates-SHE} with $\PMPslopem=\PMPslopep=\lambda$, 
they are not rich enough to study the quenched problem in the previous paragraph for the following reasons.
\begin{enumerate}
\item The pullback attractors $\clf_{\lambda}$ in Theorem \ref{cor:1f1s} do depend on $\w$ and a basin of attraction should contain its pullback attractor. 
\item Studying the RDS for a ($\P$-almost surely) fixed realization of the forcing $\w$ means the basins of attraction should be allowed to 
contain initial conditions that 
depend on $\w$. See Remark 3.2(iv) in \cite{Cra-01} for a discussion of a similar point for general RDS. 
\end{enumerate}

Let $\HH$ denote the space of all functions  $\w\mapsto\clg^\w$ from $\Omega$ into $\CICP$, with no measurability requirement. 
For a subset $\HH'\subset\HH$ and $\w\in\Omega$, let $\HH'(\w)=\{\clg^\w:\clg\in\HH'\}\subset\CICP$. 
For $\lambda\in\R$ and a time set $\ttset\subset[0,\infty)$ with $\sup\ttset=\infty$, define $\initF_{\lambda,\ttset}$ exactly as  $\initF_\lambda$ was defined 
in Definition \ref{def:initF} except that $f\in\initF_{\lambda,\ttset}$ requires that  the limits $r\to-\infty$ in \eqref{W1}--\eqref{W''} are taken only along times    $r\in-\ttset$. 
Define the  subspace $\HH_{\lambda,\ttset}$ of $\HH$ by 
	\begin{align}\label{def:HHla}
	\HH_{\lambda,\ttset}=\bigl\{\eqcl{f}:f: \Omega\to\CICM\text{ and }\forall\w\in\Omega,\ (-\ttset)\times\R\ni(r,x)\mapsto f^{\shiftp_r\w}(x)\text{ is in }\initF_{\lambda,\ttset}\bigr\}. 
	\end{align}

%

The next theorem reformulates results from Theorems 
\ref{L-char}, \ref{l2p-buslim}, and \ref{bus:exp} in the context of RDS.

\begin{theorem}\label{thm:1f1s}
The stochastic process
$\bigl\{f_{\lambda\sig}:\lambda\in\R,\sigg\in\{-,+\}\bigr\}$  of \eqref{f-la6} 
is $\sigma\{\She(t,\aabullet \viiva s,\aabullet) : s<t \leq 0\}$ measurable.
For any countable set $\ttset\subset[0,\infty)$ with $\sup\ttset=\infty$, there exists an event $\Omega_{\ttset}$ such that $\P(\Omega_{\ttset})=1$ and for all $\w\in\Omega_{\ttset}$, all $\lambda\in\R$, and both signs  $\sigg\in\{-,+\}$:
\begin{enumerate} [label={\rm(\roman*)}, ref={\rm\roman*}]   \itemsep=3pt  
\item\label{thm:1f1s.finH}  $\clf_{\lambda\sig}^{\w}\in\HH_{\lambda,\ttset}(\w)$. 
\item\label{thm:1f1s.f>0} $f^\w_{\lambda\sig}$ is  strictly positive, continuous, and $f^\w_{\lambda\sig}(0)=1$.
\item\label{thm:1f1s.phi-inv} For all $t\ge s$ in $\ttset$, 
$\varphi(t-s,\shiftp_s\w,\clf_{\lambda\sig}^{\shiftp_s\w})=\clf_{\lambda\sig}^{\shiftp_t\w}$. 
\item\label{thm:1f1s.att} If $\lambda\not\in\Bruno$, then for any $\clg\in\HH_{\lambda,\ttset}$,
	\begin{align}\label{1f1s-cv}
	\lim_{\ttset\ni t\to\infty} \dCICP\bigl(\varphi(t,\shiftp_{-t}\w,\clg^{\shiftp_{-t}\w}),\clf^\w_\lambda\bigr)=0.
	\end{align}
\item\label{thm:1f1s.different} If $\lambda\in\Bruno$, then $f^\w_{\lambda-}(x)<f^\w_{\lambda+}(x)$ for all $x>0$ and $f^\w_{\lambda-}(x)>f^\w_{\lambda+}(x)$ for all $x<0$.
\end{enumerate}
\end{theorem}

\begin{remark}
The metric $\dCICP$ controls convergence only on 
compact sets, so there is no contradiction between (i) the range of growth rates at $x\to\infty$ and $x\to-\infty$ permitted by
$\clg\in\HH_{\lambda,\ttset}$, (ii) the conservation law \eqref{eq:cons}, which says for example that if $\clg^{\w}=\clg$ does not depend on $\w$ and every $g\in\clg$ has exponential growth rates which satisfy \eqref{G1}--\eqref{G3}, then  $\varphi(t,\shiftp_{-t}\w,\clg)$ has the same growth rates as $\clg$ for all $t\in\ttset$, and (iii) the convergence in \eqref{1f1s-cv} to $f_\lambda^\w(x)$, which has exponential growth rate $\lambda$ in both directions, as $x\to\infty$ and as $x\to-\infty$, by \eqref{bus:exp3'}.
\end{remark}

Part \eqref{thm:1f1s.phi-inv} says that for each $\lambda\in\R$ and $\sigg\in\{-,+\}$, $\clf^\w_{\lambda\sig}$ is a random attractor and Theorem \ref{main:Bus}\eqref{Bproc.BM} 
the definition  \eqref{f-la6} imply that it satisfies \eqref{growth-rates-SHE} with $\PMPslopem=\PMPslopep=\lambda$.
Parts \eqref{thm:1f1s.finH} and \eqref{thm:1f1s.att} say
that for $\lambda\not\in\Bruno$, $\HH_{\lambda,\ttset}(\w)$ is a \emph{basin of attraction} of $\clf^\w_\lambda$.
See Definition 9.3.1 in \cite{Arn-98}.
Since  $\P(\lambda\in\Bruno)=0$ for each given $\lambda\in\R$, Theorem \ref{cor:1f1s} is in fact a corollary of Theorem \ref{thm:1f1s}. Both theorems are proved in Section \ref{sec:1F1S}.

Recall the dichotomy from Theorem \ref{th:La}\eqref{th:La:e}:  
either $\Bruno$ is $\P$-almost surely empty,  or $\Bruno$ is  $\P$-almost surely countable and dense in $\R$. If it is the case that $\P(\Bruno=\varnothing)=1$, then, by \eqref{Buscont}, $\lambda\mapsto f^\w_\lambda$ 
is continuous, $\P$-almost surely.
On the other hand, if $\P(\Bruno\text{ is countable and dense})=1$, then 
part \eqref{thm:1f1s.different} implies that the process has discontinuities in $\lambda$ on every  interval $(\lambda',\lambda'')\in\R$. 
This suggests Open Problems \ref{prob3}-\ref{prob8} in Section \ref{sec:OP}.

\begin{remark}\label{rk:ttset}
(Dependence on $\ttset$) The dependence on the countable   time set  $\ttset$ in the full probability events in Theorems  \ref{cor:1f1s} and \ref{thm:1f1s} can be removed in two ways. One option is to avoid working with shifts $\shiftp_t$ and work directly with the operators $\Sols_{s,t}$. This involves changing some definitions, e.g.\ using the definition of an RDS in Section 2.1 of \cite{Cra-Deb-Fla-97} and, similarly to the definitions in Section 1 of that paper,
using the notion of attractors $\clf^\w_t\in\CICP$ at time $t$, instead of $\clf^{\shiftp_t\w}$, and replacing  $\varphi$-invariance \eqref{fi-inv} with the condition $\Sols_{s,t}\clf^\w_s=\clf^\w_t$ for all $t>s$. 
Alternatively, one could push forward the distribution of $\She(\aabullet,\aabullet\viiva\aabullet,\aabullet)$ and work on its state space $\OmZ=\cC(\varsets,\R_+)$ instead of $\Omega$. On $\OmZ$, the temporal shift map is continuous and so the issue mentioned in Remark \ref{rem:discont} is not present.  Either of these approaches will result in a continuous RDS and $\varphi$ becomes a \emph{perfect cocycle} (see e.g.\ Definition 6.2 in \cite{Cra-02}). Theorems \ref{cor:1f1s} and \ref{thm:1f1s} then hold with $\ttset=[0,\infty)$.
\end{remark}

\begin{remark}\label{rem:KPZSynch}
(Synchronization of KPZ) Because the Hopf-Cole transformation $\CICM\ni f \mapsto \log f \in \CKPZ$ defines a homeomorphism between $\CICP$ and $\CKPZt$, each result above is equivalent to an analogous result for the KPZ equation, with minimal notational changes. In particular, one would replace all instances of $\CICP$ with $\CKPZt$, $\varphi$ with $\psi :[0,\infty)\times\Omega\times\CKPZt\to\CKPZt$ defined by $\psi(t,\w,\clf)=\langle \KPZ(t,\aabullet\viiva 0,f)\rangle$, and would need to take logarithms of the expressions in equations like \eqref{f->u1}  and \eqref{f-la6}. A handful of other similar changes are needed which we do not enumerate.
\end{remark}

\begin{remark}\label{rmk:hyp} 
\textup{(Hyperbolicity)} Stochastic Hamilton-Jacobi equations like KPZ are expected to be associated to certain generalized directed polymer measures. See Section 3 of \cite{Bak-Kha-18}. In this correspondence, $\varphi$-invariant random variables are associated to infinite-volume polymer measures. Fairly generally, one expects that when synchronization occurs, the paths (with different terminal points) under these measures should coalesce either at a finite time or else asymptotically. This property is known as \emph{hyperbolicity}. See Conjecture 2 and Section 5 of \cite{Bak-Kha-18} for further discussion of what is expected to hold.  We introduce the infinite-volume measures in Section 9 and prove locally uniform (in the pair of space-time terminal points) quenched hyperbolicity in total variation norm in Theorem \ref{thm:hyp}.
\end{remark}


\section{Open problems}
\label{sec:OP}
\begin{problem}\label{prob:Ergodic}
  Does there exist $\lambda>0$ and a stationary and ergodic probability measure $\PMPinit$ on $\CKPZt$ for the  $\CKPZt$-valued process $t\mapsto\Solk_{0,t}\clf$ with the property that for $\PMPinit$-almost every $\clf\in\CKPZt$,    $\ddd\lim_{x\to-\infty} x^{-1}f(x) =-\lambda$ and $\ddd\lim_{x\to\infty} x^{-1}f(x) =\lambda$ for  all $f\in\clf$? 
\end{problem}

\begin{problem}\label{prob1}
Determine whether the set $\Bruno$ of discontinuities  is empty or dense. It suffices to prove or disprove that 
$\P\{\exists\lambda\in\R:\Bus^{\lambda+}(0,0,0,1)>\Bus^{\lambda-}(0,0,0,1)\}>0$.
More generally, describe the distribution of the process $\lambda\mapsto\Bus^{\lambda+}(0,0,0,1)$. 
\end{problem}



\begin{problem}\label{prob3}
For $\lambda\in\Bruno$, are $\Glob^{\lambda\pm}$ from \eqref{Shela} the only continuous functions $u:\R^2\to(0,\infty)$ such that $u(0,\aabullet)$ satisfies \eqref{G1}--\eqref{G3} and $u(t,x)=\She\bigl(t,x\viiva s,u(s,\aabullet)\bigr)$ for all $x$ and $t>s$?
If not, are there finitely, countably, or uncountably many such solutions? 
\end{problem}


\begin{problem}\label{prob:Domain}
For $\lambda\in\Bruno$, what is the basin of attraction of $\clf^\w_{\lambda\sig}$? In particular, is it only the singleton $\{\clf^\w_{\lambda\sig}\}$?  
The same questions can be asked about the solutions in Problem \ref{prob3}, if they do exist.
\end{problem}

\begin{problem}\label{prob:limpts}
When $\lambda\in\Bruno$ and $f\in\initF_\lambda$, find the limit points of $\frac{\She(0,\aabullet\viiva r,f(r,\aabullet))}{\She(0,0\viiva r,f(r,\aabullet))}$ as $r\to-\infty$.
\end{problem}

Take $\PMPslopem,\PMPslopep\in\R$ and consider a locally bounded
Borel function $g:\R\to(0,\infty)$ such that 
\[\lim_{x\to-\infty}x^{-1}\log g(x)=\PMPslopem\quad\text{and}\quad
\lim_{x\to-\infty}x^{-1}\log g(x)=\PMPslopep.\]
Define $f:\R_-\times\R\to(0,\infty)$ by $f(r,x)=g(x)$.
Then it follows from Lemma \ref{specialW} that 
either $\PMPslopep>0\vee(-\PMPslopem)$ and $f\in\initF_{\PMPslopep}$,
or $\PMPslopem<0\vee(-\PMPslopep)$ and $f\in\initF_{\PMPslopem}$,
or $\PMPslopep\le0\le\PMPslopem$ and $f\in\initF_0$,
or $\PMPslopep=-\PMPslopem>0$. In the last case, there is no $\lambda\in\R$ such that $f\in\initF_\lambda$. This suggests the following.

\begin{problem}
Take $\lambda>0$ and suppose $g$ is a locally bounded strictly positive Borel function satisfying  
$\lim_{\abs{x}\to\infty}\abs{x}^{-1}\log g(x)=\lambda$.
Find the limit points of
\begin{align}\label{pr:lim}
\frac{\int_\R\She(t,y\viiva r,z)\,g(z)\,dz}{\int_\R\She(s,x\viiva r,z)\,g(z)\,dz}
\quad\text{as $r\to-\infty$.} 
\end{align}
\end{problem}


\begin{problem}
Find the limit points in \eqref{pr:lim}
for $\lambda\in\Bruno$ and $g$ as in the previous problem.
\end{problem}

\begin{problem}\label{prob8}
Find the limit points of $\frac{\She(0,\aabullet\viiva r,f(r,\aabullet))}{\She(0,0\viiva r,f(r,\aabullet))}$ as $r\to-\infty$, for $f:\R_-\times\R\to(0,\infty)$ such that $f(r,\aabullet)$ is Borel-measurable  for each $r\in\R_-$ but
$f\not\in\bigcup_{\lambda\in\R}\initF_\lambda$.
\end{problem}

\begin{problem}\label{prob:LDP}
Show that for all $t,y,\lambda \in \R$ and $\sig \in \{+,-\}$, the distribution of $\{r^{-1}X_{\tau r}:\tau\ge0\}$ under the semi-infinite path measure $\Poly_{(t,y)}^{\lambda\sig}$ introduced in Section \ref{sec:poly} satisfies a large deviation principle with rate function
	\[I^\lambda(f)=\frac12\int_0^\infty (f'(\tau)+\lambda)^2\,d\tau,\]
when $f:[0,\infty)\to\R$ is such that $f(0)=0$  and $(f'+\lambda)\in L^2([0,\infty))$ and $I^\lambda(f) =\infty$ otherwise.
\end{problem}



\section{Construction of the Busemann process}\label{sec:stat-coc}
This section constructs the Busemann process.
The starting point  is to build a monotone coupling of the known (ratio-)stationary   distributions of the SHE \eqref{SHE} given by geometric Brownian motion with drift $\lambda\in\R$, all of which are started in the infinite past. To construct the process, we start the equation from a countable collection of coupled stationary distributions driven by the same Brownian motion and then take subsequential weak limits of these measures as the initial time tends to the infinite past.

Because the Busemann process is initially defined through a weak limit, its construction will initially require working on an extended probability space $(\Omhat,\cFhat,\bfP)$. This is a temporary convenience. Once the limit \eqref{eq:buslim} is proved in Section \ref{sec:BusLim}, we know that the Busemann process is measurable with respect to the Green's function and hence the driving white noise. Using this fact, in Section \ref{sec:erg}, we revert to   $(\Omega,\sF,\P)$ and the  technical details concerning $\Omhat$  will become immaterial. 

Let $\Udense$ be a countable dense subset of $\R$. 
Define the product space
	\be\label{Omhat2}   \Omhat=\Omega\times\cC(\R^4,\R)^{\Udense}.\ee
$\cC(\R^4,\R)$ comes with its Polish topology of uniform convergence on compact sets and   $\cC(\R^4,\R)^{\Udense}$ with the product topology and the corresponding Borel $\sigma$-algebra.   
A generic element of $\Omhat$ is denoted by $\what$.
The projection from $\Omhat$ onto $\Omega$ is denoted by $\w$ and the projection onto $\cC(\R^4,\R)^{\Udense}$ by
$\{\Bus^\lambda(s,x,t,y):\lambda\in\Udense,s,x,t,y\in\R\}$.
Define the white noise $W$ on the new space $\Omhat$ by 
$W(\what)=W(\w)$. 
We equip $\Omhat$ with the completion $\cFhat$ of the product $\sigma$-algebra under $\bfP$.

Denote by $\widehat\nulls$ the $\sigma$-algebra generated by the $\bfP$-null sets. 
Let $\filthat{W,0}_{s,t} = \sigma(W(f) : f \in \timefun{s}{t})\vee \widehat\nulls$. For each $s \leq t$, we define $\filhat_{s,t} = \filthat{W,0}_{s-,t+} = \bigcap_{(a,b): \tspa a < s \leq t < b} \filthat{W,0}_{a,b}$ to be the associated natural augmented filtration, which satisfies the ``usual conditions''.   

\begin{proposition}\label{Bproc1}  
There exists a probability measure $\bfP$ on $(\Omhat, \cFhat)$ that satisfies the following.
\begin{enumerate}[label={\rm(\alph*)}, ref={\rm\alph*}]
\item The $\Omega$-marginal of $\bfP$ is $\P$.
\item 
For each  $r,z\in\R$, 
\[ \{\w,\Bus^\lambda(s,x,t,y):\lambda\in\Udense,s,x,t,y\in\R\}\deq\{\shiftd{r}{z}\w,\Bus^\lambda(r+s,z+x,r+t,z+y):\lambda\in\Udense,s,x,t,y\in\R\}.\]
\item For any $T\in\R$, 
$\{\Bus^\lambda(s,x,t,y):\lambda\in\Udense,x,y\in\R,s,t\le T\}$ is independent of 
$\filhat_{T:\infty}$.
\end{enumerate}
\end{proposition}

In the course of the proof, we establish a monotonicity property (see Lemma \ref{lm:mono}) that allows us to define the full Busemann process 
$\bigr\{\Bus^{\lambda\sig}(s,x,t,y):s,x,t,y,\lambda\in\R,\,\sigg\in\{-,+\}\bigl\}$ on $(\Omhat,\cFhat)$
by taking left and right limits in the slope parameter:
  	\[\Bus^{\lambda-}(s,x,t,y)
	=\lim_{\Udense\ni\mu\nearrow\lambda}\Bus^{\mu}(s,x,t,y)
	\quad\text{and}\quad 
	\Bus^{\lambda+}(s,x,t,y)=\lim_{\Udense\ni\mu\searrow\lambda}\Bus^{\mu}(s,x,t,y),\quad s,x,t,y,\lambda\in\R.\]
Lemma \ref{b-=b+} says that for $\lambda\in\Udense$, $\Bus^{\lambda\pm}\equiv\Bus^\lambda$,
i.e.,\ the new definitions extend the old one.
The next result gives the properties of the above process.

\begin{proposition}\label{Bproc2}  
The measure $\bfP$ satisfies:
\begin{enumerate}[label={\rm(\alph*)}, ref={\rm\alph*}]
\item\label{Bproc.-=+hat} For each $\lambda\in\R$, $\bfP\{\Bus^{\lambda-}(s,x,t,y)=\Bus^{\lambda+}(s,x,t,y)\ \forall s,x,t,y\}=1$. 
\item\label{Bproc.BMhat} For each $t,\lambda\in\R$ and $\sigg\in\{-,+\}$,  the process  $\{\Bus^{\lambda\sig}(t,x,t,y):x,y\in\R\}$
has the same distribution under $\bfP$ as $B(y)-B(x)+\lambda(y-x)$, where $B$ is a 
two-sided standard Brownian motion.
\end{enumerate}
There exists an event $\Omhat_0\in\, \cFhat$ such that $\bfP(\Omhat_0)=1$
and the following hold for all $\what\in\Omhat_0$.
\begin{enumerate}[resume,label={\rm(\alph*)}, ref={\rm\alph*}]
\item\label{Bproc.ahat} For all $x<y$, all $t$, and $\mu<\lambda$, 
	\begin{align}\label{Busmonohat}
	\begin{split}
	&\Bus^{\mu-}(t,x,t,y)\le\Bus^{\mu+}(t,x,t,y)\le\Bus^{\lambda-}(t,x,t,y)\le\Bus^{\lambda+}(t,x,t,y)\quad\text{and}\\
	&\Bus^{\mu-}(t,y,t,x)\ge\Bus^{\mu+}(t,y,t,x)\ge\Bus^{\lambda-}(t,y,t,x)\ge\Bus^{\lambda+}(t,y,t,x).
	\end{split}
	\end{align} 
\item\label{Bproc.bhat} For all $r,s,x,t,y,z$ and all $\sigg\in\{-,+\}$
	\begin{align}\label{cocyclehat}
	\Bus^{\lambda\sig}(r,x,s,y)+\Bus^{\lambda\sig}(s,y,t,z)=\Bus^{\lambda\sig}(r,x,t,z).
	\end{align}
 \item\label{Bproc.chat} For all $s,x,t,y,\lambda$  and all $\sigg\in\{-,+\}$
	\begin{align}\label{Busconthat}
	\Bus^{\lambda-}(s,x,t,y)
	=\lim_{\mu\nearrow\lambda}\Bus^{\mu\sig}(s,x,t,y)
	\quad\text{and}\quad 
	\Bus^{\lambda+}(s,x,t,y)=\lim_{\mu\searrow\lambda}\Bus^{\mu\sig}(s,x,t,y).
	\end{align}
\item\label{Bproc.dhat} For all  $t>r$, all $s,x,y$, and all $\sigg\in\{-,+\}$
	\begin{align}\label{NEBus2hat}
	e^{\Bus^{\lambda\sig}(s,x,t,y)}
	&= \int_{-\infty}^\infty  \She(t,y\viiva r,z)e^{\Bus^{\lambda\sig}(s,x,r,z)}\, dz.
	\end{align}
\end{enumerate}
\end{proposition}


The rest of this section is dedicated to the construction of $\bfP$ and the proofs of Propositions \ref{Bproc1} and \ref{Bproc2}.\medskip
%

Let $\BM$ be the distribution on $\cC(\R,\R)$ of two-sided standard Brownian motion $\{B(x):x\in\R\}$ with $B(0)=0$. 
Let $\P\otimes\BM$ be the product probability measure on the space $\Omega\times\cC(\R,\R)$, equipped with the completion of the product Borel $\sigma$-algebra.
On this space, for $z,\lambda\in\R$, define 
\be\label{f-la}  f_\lambda(z)=e^{B(z)+\lambda z}
\quad\text{for }  z\in\R. \ee
Then  $f_\lambda\in\ICM$ almost surely. For $t\ge S$  and $x \in \R$,
define $\She(t,x\viiva S,f_\lambda)$ on $\Omega\times\cC(\R,\R)$ as per \eqref{NEind}.   
By \cite[Theorem 2.6]{Alb-etal-22-spde-} and \eqref{pr:IC:1},
we have that $\P\otimes\BM$-almost surely, 
$\She(t,y\viiva S,f_\lambda)$ is positive and continuous on $(t,y)\in[S,\infty)\times\R$.
For $x,y\in\R$ and $s,t\ge S$, let 
	\begin{align}\label{def:bS}
	\Bus^\lambda_S(s,x,t,y)=\log \She(t,y\viiva S,f_\lambda)-\log \She(s,x\viiva S,f_\lambda).
	\end{align}
Note that \eqref{CK} implies that for $S$ fixed, $\P\otimes\BM$-almost surely,
simultaneously for all  $t>r>S$, $s>S$, and $x,y$,
	\begin{align}
	\Bus^\lambda_S(s,x,t,y)
	&=\log \frac{ \int_{-\infty}^\infty  \She(t,y\viiva r,z) \She(r,z\viiva S,f_\lambda)\,dz}{\She(s,x\viiva S,f_\lambda)} = \log\int_{-\infty}^\infty  \She(t,y\viiva r,z)e^{\Bus^\lambda_S(s,x,r,z)}\, dz.	\label{NEBus}
	\end{align}
This implies that for any $r\ge S$, one can compute $\{\Bus^\lambda_S(s,x,t,y):s,t\in(r,\infty),x,y\in\R\}$ from 
$\{\Bus^\lambda_S(r,x,r,y):x,y\in\R\}$ and $\{\She(t,y\viiva r,z): t\in(r,\infty), y,z\in\R\}$.

\begin{lemma}\label{lm:mono}
Take $S\in\R$ and $\lambda>\mu$. Then $\P\otimes\BM$-almost surely, 
for all real $x<y$ and $t\ge S$, $\Bus^\lambda_S(t,x,t,y)\ge\Bus^\mu_S(t,x,t,y)$. 
\end{lemma}

\begin{proof}

By \eqref{NEind},  
\begin{align*}   
\Bus^\lambda_S(t,x,t,y) &= 
\log \She(t,y\viiva S,f_\lambda)-\log \She(t,x\viiva S,f_\lambda)  \\
&=\log  \int_{-\infty}^\infty e^{B(z)+\lambda z} \She(t,y\viiva S,z)\, dz
- \log \int_{-\infty}^\infty e^{B(z)+\lambda z} \She(t,x\viiva S,z)\, dz. 
\end{align*} 
Differentiate with respect to  $\lambda$ to get 
	\[\partial_\lambda\Bus^\lambda_S(t,x,t,y)
	=\frac{\int_{-\infty}^\infty z \,e^{B(z)+\lambda z}\She(t,y\viiva S,z)\,dz}
	{\int_{-\infty}^\infty e^{B(z)+\lambda z}\She(t,y\viiva S,z)\,dz}
	-\frac{\int_{-\infty}^\infty z \,e^{B(z)+\lambda z}\She(t,x\viiva S,z)\,dz}
	{\int_{-\infty}^\infty e^{B(z)+\lambda z}\She(t,x\viiva S,z)\,dz}\,.\]
The differentiation is justified by \eqref{pr:IC:1}: 
since $B$ does not grow faster than linearly, $\P\otimes\BM$ almost surely,
\[   \int_{-\infty}^\infty   \She(t,y\vert S,z) \,\abs z e^{B(z)+c\abs{z}}  \,dz <\infty\]
for all $c\ge0$ and all $S,t,y$ with $t>S$.

Consider the Markov process $\Poly_{(t,x),(S,f_\lambda)}$ defined in Section \ref{sec:CDRP}.
Denote the position of the Markov process at times $s$ by $X_s$. 
By Proposition 2.18 in \cite{Alb-etal-22-spde-}, we have that if $x<y$, then $\Poly_{(t,x),(S,f_\lambda)}$ is stochastically dominated by 
$\Poly_{(t,y),(S,f_\lambda)}$. Therefore, $E^{\Poly_{(t,y),(S,f_\lambda)}}[X_S]\ge E^{\Poly_{(t,x),(S,f_\lambda)}}[X_S]$.

Now compute 
\begin{align*}
E^{\Poly_{(t,x),(S,f_\lambda)}}[X_S] &=\int_{-\infty}^\infty z \,\pi_{S,f_\lambda}(S,z\viiva t,x)\,dz
=\int_{-\infty}^\infty z \,\frac{\She(t,x\viiva S,z)\She(S,z\viiva S,f_\lambda)}{\She(t,x\viiva S,f_\lambda)}\,dz \\
&=\int_{-\infty}^\infty z \,\frac{\She(t,x\viiva S,z)e^{B(z)+\lambda z}}{\She(t,x\viiva S,f_\lambda)}\,dz 
= \frac{\int_{-\infty}^\infty z \,e^{B(z)+\lambda z}\She(t,x\viiva S,z)\,dz}
	{\int_{-\infty}^\infty e^{B(z)+\lambda z}\She(t,x\viiva S,z)\,dz}\,. 
\end{align*} 
From this, 
	\[\partial_\lambda\Bus^\lambda_S(t,x,t,y)
	=E^{\Poly_{(t,y),(S,f_\lambda)}}[X_S]- E^{\Poly_{(t,x),(S,f_\lambda)}}[X_S]\ge0.\]
%
This proves that $\Bus^\lambda_S(t,x,t,y)$ is nondecreasing in $\lambda$ and the proof is complete.
\end{proof}

We turn to extending the distribution of $(\w, \Bus_S)$ as we let $S\searrow-\infty$. Note that the distributional equality in \eqref{PS=PT} below  is not valid without restricting $\She$ to $[T, \infty)$ because $\Bus^\lambda_S$ on $[T,\infty)$ depends on $\She$ in the time interval $(S,T)$. We get around this inconsistency  by averaging over $S\in(-\infty,T)$.

\begin{lemma}\label{lm:stat}
Let $\lambda\in\R$ and $S\le T$. Then, under $\P\otimes\BM$,
\begin{align}\label{PS=PT}
\bigl(\She\big|_{([T,\infty)\times\R\times[T,\infty)\times\R)\cap\varsets}\,,\,\Bus^\lambda_S\big|_{[T,\infty)\times\R\times[T,\infty)\times\R} \bigr)\overset{d}=
\bigl(\She\big|_{([T,\infty)\times\R\times[T,\infty)\times\R)\cap\varsets}\,,\,\Bus^\lambda_T\bigr).
\end{align} 
In particular,
\begin{align}\label{blam-dist}
\{ \Bus^{\lambda}_S(T,x,T,y): x,y\in\R\}  \;\overset{d}=  \;
\{  B(y)-B(x)+\lambda(y-x):  x,y\in\R\}  . 
\end{align}
\end{lemma}


\begin{proof}
By Lemma A.5 in \cite{Alb-etal-22-spde-},
for each fixed $S$ and $\lambda$,   $\She(\aabullet,\aabullet\tsp\viiva S,f_\lambda)$ is indistinguishable from the unique adapted continuous solution to the mild equation 
\[\She(t,y\viiva S,f_\lambda) = \int_{-\infty}^\infty \heat(t-S,y-x) f_\lambda(x)\,(dx) + \int_S^t \int_{-\infty}^\infty \heat(t-u, y-z) \She(u,z\viiva S,f_\lambda) W(du\, dz).\]
%

By Proposition 3.19 on page 218 in \cite{Fun-Qua-15} or Remark 8.3 in \cite{Gub-Per-17}, we have that under $\P\otimes\BM$, 
\[\Bigl\{e^{\Bus^\lambda_S(T,0,T,y)}=\frac{\She(T,y\viiva S,f_\lambda)}{\She(T,0\viiva S,f_\lambda)}:y\in\R\Bigr\}
\;\overset{d}=\;\{f_\lambda(y):y\in\R\}\]
and the processes on both sides of the above equality are independent of $\fil_{T:\infty}$.
Furthermore, by \eqref{NEBus} with $s=r=T$ and $x=0$,
\[\Bigl\{e^{\Bus^\lambda_S(T,0,t,y)}=\frac{\She(t,y\viiva S,f_\lambda)}{\She(T,0\viiva S,f_\lambda)}:t\ge T,y\in\R\Bigr\}\]
satisfies 
	\begin{align*}
	e^{\Bus^\lambda_S(T,0,t,y)}
	&= \int_{-\infty}^\infty  \She(t,y\viiva T,z)e^{\Bus^\lambda_S(T,0,T,z)}\, dz.
	\end{align*}
By \eqref{NEind}, $\{\She(t,y\viiva T,f_\lambda):t\ge T,y\in\R\}$ satisfies the same formula, with
$e^{\Bus^\lambda_S(T,0,T,\aabullet)}$ replaced by $f_\lambda$. 
Consequently, for each $S\le T$,
\begin{align*}
\bigl(\She\big|_{([T,\infty)\times\R\times[T,\infty)\times\R)\cap\varsets}\,,\,e^{\Bus^\lambda_S(T,0,\aabullet,\aabullet)}\big|_{[T,\infty)\times\R}\bigr)
\overset{d}=\bigl(\She\big|_{([T,\infty)\times\R\times[T,\infty)\times\R)\cap\varsets}\,,\,\She(\aabullet,\aabullet\tsp\viiva T,f_\lambda)\big|_{[T,\infty)\times\R}\bigr).
\end{align*} 
In particular, the distribution of the left-hand side does not depend on $S$ and \eqref{PS=PT} follows from this and from
	\[\Bus^\lambda_S(s,x,t,y)=\Bus^\lambda_S(T,0,t,y)-\Bus^\lambda_S(T,0,s,x).\]
Claim \eqref{blam-dist} follows because
 $\Bus^\lambda_T(T,x,T,y)=\log f_\lambda(y)-\log f_\lambda(x)$ by \eqref{NEind} and \eqref{def:bS}.
\end{proof}

Recall that $\Udense$ is a countable dense subset of $\R$. Equip $\cC\bigl([S,\infty)\times\R\times[S,\infty)\times\R,\R\bigr)$
with the Polish topology of uniform convergence on compact sets.
Let $\bfP_S$ be the joint distribution of $\w$ 
 and $\{\Bus^\lambda_S(s,x,t,y):x,y\in\R,s,t\in[S,\infty),\lambda\in\Udense\}$, induced  by $\P\otimes\BM$, \eqref{NEind}, and \eqref{def:bS}, 
 on the product space 
 	\[\Omhat_S=\Omega\times\cC\bigl([S,\infty)\times\R\times[S,\infty)\times\R,\R\bigr)^{\Udense},\] 
equipped with the product topology and Borel $\sigma$-algebra.  Recall that $\Omega$ is a Polish space. Therefore, $\Omhat_S$ is also Polish.

\begin{lemma}\label{lm:tight}
There exists a measure $\bfP$ on $(\Omhat,\cFhat)$ and a sequence $(S_j : j \in \N)$ satisfying $S_j \to -\infty$ with the property that for every $T \in \R$,
\[\frac1{T-S_j}\int_{S_j}^T\bfP_R{|}_{\Omhat_T}\,dR \underset{j\to\infty}\longrightarrow  \bfP{|}_{\Omhat_T} \text{ on the space $\Omhat_T$}. \] 
\end{lemma}


\begin{proof}
Our first claim is that for each $T\in\R$, the family $\bigl\{\frac1{T-S}\int_S^T\bfP_R{|}_{\Omhat_T}\,dR:S\in(-\infty,T) \bigr\}$ is tight on $\Omhat_T$. In order to show tightness on a countable product space, it is enough to prove the tightness on each of the factors.
Lemma  \ref{lm:stat} implies that for any $\lambda\in\Udense$, the distribution of 
$\{\Bus^\lambda_R(s,x,t,y):x,y\in\R,s,t\ge T\}$ under $\P\otimes\BM$ does not depend on $R\in(-\infty, T]$. Since $\P$ is a probability measure on a Polish space, it is also tight. Hence, the family is tight on $\Omhat_T$.

Using Prohorov's theorem and the diagonal trick, we may find a sequence $S_j\searrow-\infty$ such that for all $T\in\Z$, there exists a weak limit 
\be\label{eq:weaklim}\begin{aligned}\frac1{T-S_j}\int_{S_j}^T\bfP_R{|}_{\Omhat_T}\,dR \underset{j\to\infty}\longrightarrow \wh \bfP_T \end{aligned}\ee on the space $\Omhat_T$. Define for any $T'\in\R$ and $T\in \Z$ with  $T<T'$ a measure on $\Omhat_{T'}$ via $\wh\bfP_{T'}=\wh\bfP_T\vert_{\Omhat_{T'}}$.  This definition is consistent because for $T<T'$ in $\Z$, 
\[    \wh\bfP_T\vert_{\Omhat_{T'}} = \lim_{j\to\infty} \frac1{T-S_j}\int_{S_j}^T\bigl(\bfP_R{|}_{\Omhat_T}\bigr)\vert_{\Omhat_{T'}}\,dR   
 = \lim_{j\to\infty} \frac1{T'-S_j}\int_{S_j}^{T'}  \bfP_R\vert_{\Omhat_{T'}}\,dR = \wh\bfP_{T'}. 
\]
A similar computation shows that \eqref{eq:weaklim} holds for all $T\in\R$. The above consistency and the projective version of Kolmogorov's extension theorem, \cite[Corollary 8.22]{Kal-21}, 
then imply the existence of a measure $\bfP$ on $(\Omhat,\cFhat)$ with the property that $\wh \bfP_T = \bfP{|}_{\Omhat_T}$ as measures on $\Omhat_T$.
\end{proof}

Recall that the coordinate projection onto $\cC(\R^4,\R)^{\Udense}$ is denoted by
$\{\Bus^\lambda(s,x,t,y):\lambda\in\Udense,s,x,t,y\in\R\}$. 
%

\begin{lemma}\label{lm:marginal}
For each real pair  $S\le T$ and $\lambda\in\Udense$,
\begin{align*}
&\text{the joint distribution of  } \ \bigl(\She\big|_{([T,\infty)\times\R\times[T,\infty)\times\R)\cap\varsets}\,,\,\Bus^\lambda\big|_{[T,\infty)\times\R\times[T,\infty)\times\R}\bigr) \ \text{ under $\bfP$ } \\
&\quad =  
\text{the joint distribution of  } \  \bigl(\She\big|_{([T,\infty)\times\R\times[T,\infty)\times\R)\cap\varsets}\,,\,\Bus^\lambda_S\big|_{[T,\infty)\times\R\times[T,\infty)\times\R}\bigr) \ \text{ under $\P\otimes\BM$.  }  
\end{align*} 
In particular, the joint distribution of  $\bigl(\She\big|_{([T,\infty)\times\R\times[T,\infty)\times\R)\cap\varsets}\,,\,\Bus^\lambda\big|_{[T,\infty)\times\R\times[T,\infty)\times\R}\bigr)$ under $\bfP$  is ergodic under shifts in space.
\end{lemma}

At this point, we cannot claim any ergodicity of $\bfP$ under shifts in time. This  comes  in Section \ref{sec:erg} as a consequence of the almost sure limit \eqref{eq:buslim}.

\begin{proof}[Proof of Lemma \ref{lm:marginal}]
From \eqref{PS=PT}, we have that for $R<T$, the distribution of 
\begin{align}\label{auxvar1}
\bigl(\She\big|_{([T,\infty)\times\R\times[T,\infty)\times\R)\cap\varsets}\,,\,\Bus^\lambda\bigr)
\end{align}
under $\bfP_R{|}_{\Omhat_T}$ 
is the same as that of 
\begin{align}\label{auxvar2}
\bigl(\She\big|_{([T,\infty)\times\R\times[T,\infty)\times\R)\cap\varsets}\,,\,\Bus^\lambda_T\bigr)
\end{align} 
under $\P\otimes\BM$. 
In the expression above, the projection $\Bus^\lambda$ is restricted to  $\Bus^\lambda\in\cC\bigl([T,\infty)\times\R\times[T,\infty)\times\R,\R\bigr)$
because $\bfP_R{|}_{\Omhat_T}$ is a probability measure on $\Omhat_T$.
Consequently, for $S<T$, the distribution of \eqref{auxvar1}  under 
$(T-S)^{-1}\int_S^T\bfP_R{|}_{\Omhat_T}\,dR$  is the same as
 that of \eqref{auxvar2} under $\P\otimes\BM$. 
 Since $\bfP{|}_{\Omhat_T}$ is a limit point of these Ces\`aro averages,  the distribution of  \eqref{auxvar1} 
 under $\bfP{|}_{\Omhat_T}$ 
is the same as  that of \eqref{auxvar2} under $\P\otimes\BM$. 

Recall that Brownian motion is the integral of a one-dimensional white noise, which is mixing under non-trivial shifts. We have assumed that for $a \neq 0$, $\shiftd{0}{a}$ is mixing on $(\Omega,\fil,\bbP)$. For $b \neq 0$, call $\tau_b f(\cdot) = f(b+\cdot)$ the shift by $b$ on $\sC(\R,\R)$. Since independent mixing processes are jointly mixing \cite[Theorem 5.1(a)]{Bra-05}, 
it follows that for $a,b\neq0$, $\shiftd{0}{a}\,\otimes\,\tau_b$ is ergodic on $(\Omega\,\times\, \sC(\R,\R),$ $\fil\, \otimes\,\sB(\sC(\R,\R)),\bbP\,\otimes\,\PMPinit)$. The claimed ergodicity then follows from the $\fil\, \otimes\,\sB(\sC(\R,\R))$ measurability of $\bigl(\She\big|_{([T,\infty)\times\R\times[T,\infty)\times\R)\cap\varsets}\,,\,\Bus^\lambda_S\big|_{[T,\infty)\times\R\times[T,\infty)\times\R}\bigr)$.
%
\end{proof}

\begin{proof}[Proof of Propositions \ref{Bproc1} and \ref{Bproc2}]
Proposition \ref{Bproc1} follows from the construction of $\bfP$. We prove Proposition \ref{Bproc2}.

As a consequence of Lemma \ref{lm:marginal}, the monotonicity proved in Lemma \ref{lm:mono} transfers to a monotonicity $\bfP$-almost surely and now for all times $t\in\R$. 
This allows us to take left and right limits in the parameter $\lambda$ to extend $\{\Bus^\lambda(t,x,t,y):t,x,y\in\R,\lambda\in\Udense\}$ to $\lambda\in\R$ by defining
	\begin{align}\label{foofifoo}
	\Bus^{\lambda-}(t,x,t,y)
	=\lim_{\Udense\ni\mu\nearrow\lambda}\Bus^\mu(t,x,t,y)
	\quad\text{and}\quad 
	\Bus^{\lambda+}(t,x,t,y)=\lim_{\Udense\ni\mu\searrow\lambda}\Bus^\mu(t,x,t,y),
	\end{align}
for all $t,x,y\in\R$ with $x<y$ and setting $\Bus^{\lambda\pm}(t,x,t,x)=0$ and 
$\Bus^{\lambda\pm}(t,x,t,y)=-\Bus^{\lambda\pm}(t,x,t,y)$ when $x>y$.
Then \eqref{Busmonohat} is satisfied and part \eqref{Bproc.ahat} is proved. 
We also have for all $s,x,y,\lambda\in\R$ and $\sigg\in\{-,+\}$,
	\begin{align}\label{Buscont-tmp}
	\Bus^{\lambda-}(s,x,s,y)
	=\lim_{\mu\nearrow\lambda}\Bus^{\mu\sig}(s,x,s,y)
	\quad\text{and}\quad 
	\Bus^{\lambda+}(s,x,s,y)=\lim_{\mu\searrow\lambda}\Bus^{\mu\sig}(s,x,s,y).
	\end{align}

For $s<t$ and $\lambda\in\R$, define
	\begin{align}\label{b:def2}
	\begin{split} 
	\Bus^{\lambda\pm}(s,x,t,y)
	&=\log\int_{-\infty}^\infty\She(t,y\viiva s,z)e^{\Bus^{\lambda\pm}(s,x,s,z)}\,dz\\
	&=\log\Bigl(\;\int_{-\infty}^x\She(t,y\viiva s,z)e^{\Bus^{\lambda\pm}(s,x,s,z)}\,dz+\int_x^\infty\She(t,y\,|\,s,z)e^{\Bus^{\lambda\pm}(s,x,s,z)}\,dz\Bigr).
	\end{split}
	\end{align}
\eqref{Busconthat} follows from monotone convergence and \eqref{Buscont-tmp}.
For $s>t$, let $\Bus^{\lambda\pm}(s,x,t,y)=-\Bus^{\lambda\pm}(t,y,s,x)$. The limits in \eqref{Busconthat} still hold for this case.   
Part \eqref{Bproc.chat} is proved.

Part \eqref{Bproc.bhat} follows from the definition \eqref{def:bS}, Lemma \ref{lm:marginal} (which transfers the property to $\Bus^\lambda$, $\lambda\in\Udense$), and the limits \eqref{Busconthat}. 

For $t>r>q$ \eqref{b:def2} and 
\eqref{eq:CK} imply
	\begin{align*}
	e^{\Bus^{\lambda\sig}(q,x,t,y)}
	&=\int_{-\infty}^\infty\int_{-\infty}^\infty\She(t,y\viiva r,z)\She(r,z\viiva q,w)e^{\Bus^{\lambda\sig}(q,x,q,w)}\,dz\,dw\\	
	&=\int_{-\infty}^\infty\She(t,y\viiva r,z)\int_{-\infty}^\infty\She(r,z\viiva q,w)e^{\Bus^{\lambda\sig}(q,x,q,w)}\,dw\,dz\\
	&=\int_{-\infty}^\infty\She(t,y\viiva r,z)e^{\Bus^{\lambda\sig}(q,x,r,z)}\,dz.
	\end{align*}
Multiplying both sides by $e^{\Bus^{\lambda\sig}(s,x,q,x)}$ and using the cocycle property \eqref{cocyclehat} gives \eqref{NEBus2hat} and part \eqref{Bproc.dhat}.

Next, we prove part \eqref{Bproc.-=+hat}. 
By monotone convergence, \eqref{foofifoo}, and Lemma \ref{lm:marginal} with $S=T=r$, 
we have, for $r,x,y,\lambda\in\R$,
	\[\bfE[\Bus^{\lambda-}(r,x,r,y)]
	=\lim_{\Udense\ni\mu\nearrow\lambda}\bfE[\Bus^\mu(r,x,r,y)]
	=\lim_{\Udense\ni\mu\nearrow\lambda}\mu(y-x)=\lambda(y-x)\]
and similarly
	\[\bfE[\Bus^{\lambda+}(r,x,r,y)]=\lambda(y-x).\]
Thus, for any given $r,x,y,\lambda\in\R$ with $x<y$ we have $\bfP$-almost surely 
	\[\int_x^y (\Bus^{\lambda+}(r,x,r,z)-\Bus^{\lambda-}(r,x,r,z))\,dz\ge0\]
and 	\[\bfE\Bigl[\int_x^y (\Bus^{\lambda+}(r,x,r,z)-\Bus^{\lambda-}(r,x,r,z))\,dz\Bigr]
	=\int_x^y \bfE\bigl[\Bus^{\lambda+}(r,x,r,z)-\Bus^{\lambda-}(r,x,r,z)\bigr]\,dz=0.\]
Consequently, for all $r\in\R$, $\bfP$-almost surely,
	\[\Bus^{\lambda+}(r,x,r,y)=\Bus^{\lambda-}(r,x,r,y)\quad\text{for Lebesgue-almost every }x,y\in\R.\]

But now, \eqref{NEBus2hat}  implies that for all $r,\lambda\in\R$, there exists an event $\Omhat_{r,\lambda}$ such that $\bfP(\Omhat_r)=1$ and for each $\what\in\Omhat_{r,\lambda}$, $\Bus^{\lambda-}(r,x,t,y)=\Bus^{\lambda+}(r,x,t,y)$ 
for all $t>r$ and $x,y\in\R$. By the cocycle property \eqref{cocyclehat} we have that $\Bus^{\lambda\sig}(s,x,t,y)=\Bus^{\lambda\sig}(r,0,t,y)-\Bus^{\lambda\sig}(r,0,s,x)$ for all $\sigg\in\{-,+\}$, $r\in\Z$, and $s,t>r$. This implies that for each $\lambda\in\R$, 
the following holds $\bfP$-almost surely:
	\[\Bus^{\lambda-}(s,x,t,y)=\Bus^{\lambda+}(s,x,t,y)\quad\text{for all $s,x,t,y\in\R$}.\]
Part \eqref{Bproc.-=+hat} is proved.

Lemma \ref{b-=b+} below shows that when $\lambda\in\Udense$, we actually have 
$\bfP(\Bus^{\lambda-}\equiv\Bus^{\lambda+}\equiv\Bus^{\lambda})=1$.  This and Lemma \ref{lm:marginal} with $S=T=t$ imply the claim in part \eqref{Bproc.BMhat} when $\lambda\in\Udense$. Since $\Udense$ is arbitrary and any given $\lambda\in\R$ can be thrown into $\Udense$, the distributional 
claim in fact holds for all $\lambda\in\R$.  The proposition is proved.
\end{proof}

The next lemma shows that when $\lambda\in\Udense$, the newly defined  $\Bus^{\lambda\pm}$ are the same as the old $\Bus^\lambda$.

\begin{lemma}\label{b-=b+}
If $\lambda\in\Udense$, then $\bfP$-almost surely  $\Bus^{\lambda-}(s,x,t,y)=\Bus^{\lambda+}(s,x,t,y)=\Bus^{\lambda}(s,x,t,y)$ for all $s,x,t,y\in\R$.
\end{lemma}

\begin{proof}
%
%
If $\lambda\in\Udense$, then monotonicity and the limits \eqref{foofifoo} give 
	\[\Bus^{\lambda-}(r,x,r,y)
	=\lim_{\Udense\ni\mu\nearrow\lambda}\Bus^\mu(r,x,r,y)
	\le\Bus^{\lambda}(r,x,r,y)
	\le \lim_{\Udense\ni\mu\searrow\lambda}\Bus^\mu(r,x,r,y)
	=\Bus^{\lambda+}(r,x,r,y)\]
for all $r,x,y\in\R$ with $x<y$.   This and Proposition \ref{Bproc2}\eqref{Bproc.-=+hat} give that $\bfP$-almost surely, $\Bus^{\lambda-}(s,x,t,y)=\Bus^{\lambda+}(s,x,t,y)=\Bus^{\lambda}(s,x,t,y)$ for all $s,x,t,y\in\R$.
\end{proof}

\section{Shape theorems}\label{sec:shape}

The next item on the way to the limits \eqref{eq:buslim} and \eqref{eq:buslim2} are shape theorems for the Green's function and the Busemann process. These are Theorem \ref{thm:Z-cont} above and Proposition \ref{pr:b-shape} below. The proofs begin with the next preliminary version of the shape theorem for the Green's function, where some of the variables are restricted to a lattice.
The distinction between statements \eqref{Z-grid} and \eqref{Z-grid2} below lies in which spatial variable, $x$ or $y$, is restricted to a discrete set.

\begin{lemma}\label{lm:Z-grid} 
There exists a finite constant $c_0>0$ such that  the following holds.  If the 
 sets $\Sbset_n\subset\R$, $\Tbset_n\subset\R$, and  $\Vbset_n\subset\R$ have no accumulation points and the constant  $C>0$ satisfies  
\begin{align}\label{decay}
\sum_{n\in\N}n\tspb\bigl|\{(s,t,v)\in\Sbset_n\times\Tbset_n\times\Vbset_n:s<t,\abs{s}+\abs{t}+\abs{v}\le Cn\}\bigr|\tspb e^{-c_0n}<\infty,
\end{align}
then the following hold $\P$-almost surely: for any $M>0$,
\begin{align}\label{Z-grid}
\lim_{n\to\infty}n^{-1}\sup_{\substack{(s,x,t,y)\,\in\,\Sbset_n\times\R\times\Tbset_n\times\Vbset_n:\, t>s,\\[1pt] s,t,y\,\in\,[-Cn,Cn],\,\abs{x}\le Mn}}\,\Bigl|\,\log \She(t,y\viiva s,x)+\frac{t-s}{24}+\frac{(y-x)^2}{2(t-s)}\,\Bigr|=0
\end{align}
and
\begin{align}\label{Z-grid2}
\lim_{n\to\infty}n^{-1}\sup_{\substack{(s,x,t,y)\in\Sbset_n\times\Vbset_n\times\Tbset_n\times\R:t>s,\\ s,t,x\in[-Cn,Cn],\abs{y}\le Mn}}\,\Bigl|\,\log \She(t,y\viiva s,x)+\frac{t-s}{24}+\frac{(y-x)^2}{2(t-s)}\,\Bigr|=0.
\end{align}
\end{lemma}

\begin{proof}
We prove \eqref{Z-grid} and \eqref{Z-grid2} comes similarly or by using the reflection invariance recorded as \eqref{eq:cov.ref}.

Recall the renormalized Green's function $\rnShe(t,y\viiva s,x) = \She(t,y\viiva s,x)/\heat(t-s,y-x)$ from \eqref{women}. As recorded in Proposition 1.4 in \cite{Ami-Cor-Qua-11} or Corollary 2.4 in \cite{Alb-etal-22-spde-}, the process $x\mapsto \rnShe(t,y\viiva s,x)$ is stationary. Using this in the first equality and then and the translation and reflection invariance in \eqref{eq:cov.shift} and \eqref{eq:cov.ref} as part of the second inequality, we have the following:
\begin{align*}
&\P\Bigl\{\sup_{\substack{(s,x,t,y)\,\in\,\Sbset_n\times\R\times\Tbset_n\times\Vbset_n:\,t>s,\\[1pt] s,t,y\in[-Cn,Cn],\,\abs{x}\le Mn}}\,\Bigl|\,\log \She(t,y\viiva s,x)+\frac{t-s}{24}+\frac{(y-x)^2}{2(t-s)}\,\Bigr|\ge\e n\Bigr\}\\
&\le\sum_{\abs{m}\le Mn+1}\sum_{\substack{\\[1pt] (s,t,y)\,\in\,\Sbset_n\times\Tbset_n\times\Vbset_n:\\[1pt] s<t,\, s,t,y\in[-Cn,Cn]}}\!\!\!\!\P\Bigl\{\sup_{m\le x\le m+1}\Bigl|\,\log\She(t,y\viiva s,x)+\frac{t-s}{24}+\frac{(y-x)^2}{2(t-s)}\,\Bigr|\ge\e n\Bigr\}\\
&= \sum_{\abs{m}\le Mn+1}\sum_{\substack{\\[1pt] (s,t,y)\,\in\,\Sbset_n\times\Tbset_n\times\Vbset_n:\\[1pt] s<t,\, s,t,y\in[-Cn,Cn]}}\!\!\!\!\P\Bigl\{\sup_{0\le x\le1}\Bigl|\,\log\rnShe(t,y\viiva s,x)+\frac{t-s}{24}+\log \heat(t-s,0)\,\Bigr|\ge\e n\Bigr\}\\
&\le2(Mn+1)\sum_{\substack{\\[1pt] (s,t,y)\,\in\,\Sbset_n\times\Tbset_n\times\Vbset_n:\\[1pt] s<t,\, s,t,y\in[-Cn,Cn]}}\!\!\!\!\biggr(\P\Bigl\{\Bigl|\,\log\She(t-s,0\viiva 0,0)+\frac{t-s}{24}\,\Bigr|\ge\e n/2\Bigr\}\\
&\qquad\qquad\qquad\qquad+
\P\Bigl\{\sup_{0\le w\le 1}\Bigl|\,\log\She(t-s,w\viiva 0,0)+\frac{w^2}{2(t-s)}-\log\She(t-s,0\viiva0,0)\,\Bigr|\ge\e n/2\Bigr\}\biggl).\\
&\le2(Mn+1)\sum_{\substack{\\[1pt] (s,t,y)\,\in\,\Sbset_n\times\Tbset_n\times\Vbset_n:\\[1pt] s<t,\, s,t,y\in[-Cn,Cn]}}\!\!\!\!\biggl(\P\Bigl\{\Bigl|\,\log\She(t-s,0\viiva 0,0)+\frac{t-s}{24}\,\Bigr|\ge\e n/2\Bigr\}\\
&\qquad\qquad\qquad\qquad+
\sum_{i=-1}^2\P\Bigl\{\Bigl|\,\log\She(t-s,0\viiva i,0)+\frac{t-s}{24}\,\Bigr|\ge\e n/8\Bigr\}\\
&\qquad\qquad\qquad\qquad+
\P\Bigl\{\sup_{0\le w\le 1}\Bigl|\,\log\She(t-s,w\viiva 0,0)+\frac{w^2}{2(t-s)}-\log\She(t-s,0\viiva0,0)\,\Bigr|\ge\e n/2,\\
&\qquad\qquad\qquad\qquad\qquad\qquad\qquad\qquad
\Bigl|\,\log \She(t-s,0\viiva i,0)+\frac{t-s}{24}\,\Bigr|<\e n/8 \text{ for } i\in\{-1,0,1,2\}\Bigr\}\biggr).
\end{align*}

Proposition 4.3 of \cite{Cor-Gho-Ham-21} says that the last probability on the right-hand side is bounded above by $C'e^{-c'n^{4/3}}$ for some finite strictly positive $C'$ and $c'$. 
By Theorems 1.11 of \cite{Cor-Gho-20-ejp} and Theorem 1.1 of \cite{Cor-Gho-20-duke}, the other  two probabilities on the final right-hand side are  bounded above by $C''e^{-c''n}$, for some finite strictly positive constants $C''$ and $c''$. Since we assume \eqref{decay}, claim  \eqref{Z-grid} follows from the above and the Borel-Cantelli lemma.
\end{proof}

\begin{proof}[Proof of Theorem \ref{thm:Z-cont}] At various points in this proof it is convenient to use \eqref{women} to switch between studying $\She$ and $\Shenorm$. 
We first prove the statement over bounded time differences:
\begin{align}\label{Z-cont1}
\lim_{n\to\infty}n^{-1}\sup_{\substack{(s,x,t,y)\in\R^4:0< t-s\le1\\s,x,t,y\in[-Cn,Cn]}}\,\Bigl|\,\log \She(t,y\viiva s,x)+\frac{t-s}{24}-\log\heat(t-s,y-x)\,\Bigr|=0.
\end{align}
Since $t-s\le1$ we can ignore the term $(t-s)/24$.
Take $\delta>0$ and write
\begin{align*}
&\P\Bigl\{\sup_{\substack{(s,x,t,y)\in\R^4:0\le t-s\le1\\s,x,t,y\in[-Cn,Cn]}}\abs{\log \Shenorm(t,y\viiva s,x)}\ge\delta n\Bigr\}\\
&\qquad\le
\sum_{m\in[-Cn-1,Cn]\cap\Z}\P\Bigl\{\sup_{\substack{m\le s\le m+1,s\le t\le s+1\\x,y\in[-Cn,Cn]}}\abs{\log \Shenorm(t,y\viiva s,x)}\ge\delta n\Bigr\}\\
&\qquad\le
C'n\P\Bigl\{\sup_{\substack{0\le s\le t\le 2\\x,y\in[-Cn,Cn]}}\abs{\log \Shenorm(t,y\viiva s,x)}\ge\delta n\Bigr\}\\
&\qquad\le
C'n\P\Bigl\{\sup_{\substack{0\le s\le t\le 2\\x,y\in[-Cn,Cn]}}\Shenorm(t,y\viiva s,x)\ge e^{\delta n}\Bigr\}
+C'n\P\Bigl\{\sup_{\substack{0\le s\le t\le 2\\x,y\in[-Cn,Cn]}}\Shenorm(t,y\viiva s,x)^{-1}\ge e^{\delta n}\Bigr\}\\
&\qquad\le C'' n^4 e^{-\delta n}.
\end{align*}
In the last inequality uses Corollary 3.10 in \cite{Alb-etal-22-spde-}. 
\eqref{Z-cont1} follows from this and the Borel-Cantelli lemma.

With \eqref{Z-cont1} at hand \eqref{Z-cont} follows if we show 
\begin{align}\label{Z-cont2}
\lim_{n\to\infty}n^{-1}\sup_{\substack{(s,x,t,y)\in\R^4:t-s>1\\s,x,t,y\in[-Cn,Cn]}}\,\Bigl|\,\log \She(t,y\viiva s,x)+\frac{t-s}{24}+\frac{(y-x)^2}{2(t-s)}\,\Bigr|=0.
\end{align}
Since the above event is monotone in $C$, it is enough to work with a fixed $C>0$. We prove the lower and upper bounds separately.

\medskip 

{\it Step 1:  Lower bound.} 
For $n\in\N$ let $\Sbset_n=\Tbset_n=n^{-2}\Z$ and $\Vbset_n=n^{-1}\Z$. Then \eqref{decay} is satisfied for any strictly positive $C$ and $c_0$.
Consequently, \eqref{Z-grid} and \eqref{Z-grid2} hold $\P$-almost surely, for any strictly positive $C$ and $M$.


For $s\in\R$ let $k=\ce{n^2s}+1$ and $s'=n^{-2}k$ and for $t\in\R$ let $\ell=\fl{n^2t}-1$ and $t'=n^{-2}\ell$. 
Note that if $t-s>1$ and $n^2\ge8$, then 
	\[ s<s'\le s+2n^{-2}<t-1+2n^{-2}\le t-2n^{-2}-1/2\le t'-1/2.\] 
For $x\in\R$ let $m=\fl{nx}$ and $x'=m/n$ and for $y\in\R$ let $m'=\fl{ny}$ and $y'=m'/n$. Then
	\be\label{aux569} \begin{aligned}
	&\She(t,y\viiva s,x) \\
	&=\int_{-\infty}^\infty\int_{-\infty}^\infty\Shenorm(t,y\viiva t',w)\heat(t-t',y-w)\She(t',w\viiva s',u)
	\heat(s'-s,u-x)\Shenorm(s',u\viiva s,x)\,dw\,du\\[3pt] 
	&\ge\int\limits_{x'}^{x'+n^{-1}}\int\limits_{y'}^{y'+n^{-1}}\Shenorm(t,y\viiva t',w)\heat(t-t',y-w)\She(t',w\viiva s',u)
	\heat(s'-s,u-x)\Shenorm(s',u\viiva s,x)\,dw\,du\\[5pt] 
	&\ge\Bigl\{ \inf_{\substack{y'\le v,w\le y'+n^{-1}\\ t'+n^{-2}\le r\le t'+2n^{-2}}}\Shenorm(r,v\viiva t',w)\Bigr\} \cdot  
	\Bigl\{ \inf_{\substack{x'\le u,v\le x'+n^{-1}\\ s'-2n^{-2}\le r\le s'-n^{-2}}}\Shenorm(s',u\viiva r,v)\Bigr\} \\[3pt] 
	&\qquad\qquad\cdot\;\int_{x'}^{x'+n^{-1}}\int_{y'}^{y'+n^{-1}}\heat(t-t',y-w)\She(t',w\viiva s',u)\heat(s'-s,u-x)\,dw\,du.
	\end{aligned}\ee

 By \cite[Corollary 3.11]{Alb-etal-22-spde-}, we have
	\begin{align}\label{tempo2234}
	\E\Bigr[\sup_{\substack{0\le u,v\le 1\\ 0\le r\le n^{-2}}}\Shenorm(2n^{-2},u\viiva r,v)^{-1}\Bigr]\le \E\Bigr[\sup_{\substack{0\le u,v\le 1\\ 0\le r\le s\le 2}}\Shenorm(s,u\viiva r,v)^{-1}\Bigr]=C<\infty.
	\end{align}
	
A simple calculus computation 
shows that 
\be\label{a689} 
\forall \alpha\ge 1 \;\exists c_\alpha<\infty: \;  \forall x\ge 0\;   (\log(1+x))^\alpha \le c_\alpha x.  
\ee
 Next, let $A$ be an arbitrary set and  $g:A\to(0,\infty)$.   Then  we have this bound $\forall p\ge q>0$: 
 \be\label{a704}  \begin{aligned}
 \bigl( \tsp [\inf_x \log g(x)]^- \tsp \bigr)^p &=  \bigl( \tsp \bigl[\sup_x \log \bigl(g(x)^{-1}\bigr)\bigr]^+ \tsp \bigr)^p   
\le  \bigl(  \sup_x \log \bigl(1+g(x)^{-1}\bigr) \tsp \bigr)^p\\
&=   \sup_x \bigl(  \log \bigl(1+g(x)^{-1}\bigr) \tsp \bigr)^p
\le   c_{p/q}^q \cdot  \sup_x  g(x)^{-q},
\end{aligned}\ee
where for the last inequality we used \eqref{a689} with $\alpha=p/q$.

By \eqref{tempo2234} and \eqref{a704} we have
	\[\E\Bigl[\Bigl(\Bigl[\inf_{\substack{0\le u,v\le n^{-1}\\ 0\le r\le n^{-2}}}\log\Shenorm(2n^{-2},u\viiva r,v)\Bigr]^-\Bigr)^p\Bigr]\le
	\E\Bigl[\Bigl(\Bigl[\inf_{\substack{0\le u,v\le 1\\ 0\le r\le n^{-2}}}\log\Shenorm(2n^{-2},u\viiva r,v)\Bigr]^-\Bigr)^p\Bigr]\le Cc_{p}\]
for any $1\le p<\infty$. 
Taking $p>6$ and using a union bound then the Borel-Cantelli lemma gives
	\begin{align}\label{tempo-foo}
	\lim_{n\to\infty} n^{-1}\sup_{\substack{\abs{m}\le Cn^2+1\\\abs{k}\le Cn^3+2}}\Bigl(\inf_{\substack{n^{-1}m\le u,v\le n^{-1}(m+1)\\ (k-2)n^{-2}\le r\le (k-1)n^{-2}}}\log\Shenorm(n^{-2}k,u\viiva r,v)\Bigr)^-=0 \quad\text{a.s.} 
	\end{align}
Similarly,
	\begin{align}\label{tempo-fifi}
	\lim_{n\to\infty} n^{-1}\sup_{\substack{\abs{m'}\le Cn^2+1\\\abs{\ell}\le Cn^3+2}}\Bigl(\inf_{\substack{n^{-1}m'\le v,w\le n^{-1}(m'+1)\\ (\ell+1)n^{-2}\le r\le (\ell+2)n^{-2}}}\log\Shenorm(r,v\viiva n^{-2}\ell,w)\Bigr)^-=0 \quad\text{a.s.}
	\end{align}
	
Next we treat the last double integral in \eqref{aux569}.  
Note that  $\abs{y-w}\vee\abs{u-x}\le n^{-1}$ and both $t-t'$ and $s'-s$ are between $n^{-2}$ and $2n^{-2}$.
Therefore, the double integral is bounded below by 
	\[(4\pi e)^{-1}n^2\int_{x'}^{x'+n^{-1}}\int_{y'}^{y'+n^{-1}}\She(t',w\viiva s',u)\,dw\,du\ge 
	\int_{x'}^{x'+n^{-1}}\int_{y'}^{y'+n^{-1}}\She(t',w\viiva s',u)\,dw\,du\]
for $n$ large enough. Apply \eqref{crossing} to write
	\[\She(t',w\viiva s',u)\ge\frac{\She(t',w\viiva s',x')\She(t',y'\viiva s',u)}{\She(t',y'\viiva s',x')}.\]
Factoring the above double integral gives 
	\begin{align*}
		&\log\int_{x'}^{x'+n^{-1}}\int_{y'}^{y'+n^{-1}}\She(t',w\viiva s',u)\,dw\,du\\
		&\quad\ge\log\int_{y'}^{y'+n^{-1}}\!\!\!\!\She(t',w\viiva s',x')\,dw
		+\log\int_{x'}^{x'+n^{-1}}\!\!\!\!\She(t',y'\viiva s',u)\,du-\log\She(t',y'\viiva s',x').
	\end{align*}
	
	To derive a lower bound, restrict $s,x,t,y$ to $[-Cn,Cn]$. Recall that $y'\le y\le y'+n^{-1}$ and $x'\le x\le x'+n^{-1}$ and,  as in the integrals above, consider $u\in[x',x'+n^{-1}]$ and $w\in[y', y'+n^{-1}]$.  
	In the first inequality below, for the cross term in the numerator use 
$\abs{w-u+y-x}\le 4Cn+2n^{-1}\le 4(Cn+1)$, $\abs{w-y}\le n^{-1}$ and $\abs{x-u}\le n^{-1}$. 
Recall also that $t'-s'\ge1/2$, $t-s>1$, 
$n^{-2}\le t-t'\le 2n^{-2}$, 
and $n^{-2}\le s'-s\le 2n^{-2}$.
	\be\label{aux561} \begin{aligned}
	&\frac{(w-x')^2+(y'-u)^2-(y'-x')^2}{2(t'-s')}\\
	&\qquad=\frac{2(u-x')(w-y')+(w-u+y-x)(w-u-y+x)+(y-x)^2}{2(t'-s')}\\
	&\qquad\le 2n^{-2}+8(Cn+1)n^{-1}+\frac{(y-x)^2}{2(t'-s')}
	\le2+8(C+1)+\frac{(y-x)^2}{2(t'-s')}\\
	&\qquad\le10+8C+\frac{4C^2n^2(\abs{t-t'}+\abs{s-s'})}{2(t-s)(t'-s')}+\frac{(y-x)^2}{2(t-s)}\\
	&\qquad\le10+8C+16C^2+\frac{(y-x)^2}{2(t-s)}\,.
	\end{aligned}\ee
	
	Return to \eqref{aux569} to collect  the bounds. Use $t'-s'=(\fl{n^2 t}-\ce{n^2s}-2)n^{-2}\le t-s$ and use \eqref{aux561} to bound $\frac{(y-x)^2}{2(t-s)}$ from below. 
\begin{align*}
&\inf_{\substack{(s,x,t,y)\in\R^4:t-s>1,\\ s,x,t,y\in[-Cn,Cn]}}\,\Bigl(\,\log \She(t,y\viiva s,x)+\frac{t-s}{24}+\frac{(y-x)^2}{2(t-s)}\,\Bigr)\\
&\qquad\ge
\inf_{\substack{\abs{m'}\le Cn^2+1\\ \abs{\ell}\le Cn^3+2}} \ \inf_{\substack{n^{-1}m'\le v,w\le n^{-1}(m'+1)\\ (\ell+1)n^{-2}\le r\le(\ell+2)n^{-2}}}\log\Shenorm(r,v\viiva n^{-2}\ell,w)\\
&\qquad\qquad +\inf_{\substack{\abs{m}\le Cn^2+1\\ \abs{k}\le Cn^3+2}} \ \inf_{\substack{n^{-1}m\le u,v\le n^{-1}(m+1)\\ (k-2)n^{-2}\le r\le(k-1)n^{-2}}}\log\Shenorm(n^{-2}k,u\viiva r,v)\\
&\qquad\qquad+\inf_{\substack{(s',x',t',w)\in\Sbset_n\times\Vbset_n\times\Tbset_n\times\R:t'>s',\\ s',x',t',w\in[-Cn-2, Cn+2]}}\Bigl(\log\She(t',w\viiva s',x')+\frac{t'-s'}{24}+\frac{(w-x')^2}{2(t'-s')}\Bigr)\\
&\qquad\qquad+\inf_{\substack{(s',u,t',y')\in\Sbset_n\times\R\times\Tbset_n\times\Vbset_n:t'>s',\\ s',u,t',y'\in[-Cn-2, Cn+2]}}\Bigl(\log\She(t',y'\viiva s',u)+\frac{t'-s'}{24}+\frac{(y'-u)^2}{2(t'-s')}\Bigr)\\
&\qquad\qquad-\sup_{\substack{(s',x',t',y')\in\Sbset_n\times\Vbset_n\times\Tbset_n\times\Vbset_n:t'>s',\\ s',x',t',y'\in[-Cn-2, Cn+2]}}\Bigl(\log\She(t',y'\viiva s',x')+\frac{t'-s'}{24}+\frac{(y'-x')^2}{2(t'-s')}\Bigr)\\
&\qquad\qquad-(10+8C)-16C^2.
\end{align*}
Then \eqref{tempo-foo}, \eqref{tempo-fifi}, \eqref{Z-grid}, and \eqref{Z-grid2} give
\be\label{aux580}
\varliminf_{n\to\infty}n^{-1}\inf_{\substack{(s,x,t,y)\in\R^4:t-s>1,\\ s,x,t,y\in[-Cn,Cn]}}\,\Bigl(\,\log \She(t,y\viiva s,x)+\frac{t-s}{24}+\frac{(y-x)^2}{2(t-s)}\,\Bigr)\ge0.\ee

\medskip 

{\it Step 2:  Upper bound.}   
Take $M>C>0$.  Decompose as 
\be\label{aux638} \begin{aligned}
\She(t,y\viiva s,x)
&= \iint_{\R^2\setminus[-Mn,Mn]^2}
	\!\!\!\!\!\!\!\!\!\!\She(t,y\viiva \ell/4,w)\She(\ell/4,w\viiva k/4,u)\She(k/4,u\viiva s,x)\,dw\,du\\
&\qquad\qquad+\iint_{[-Mn,Mn]^2}\She(t,y\viiva \ell/4,w)\She(\ell/4,w\viiva k/4,u)\She(k/4,u\viiva s,x)\,dw\,du. 
\end{aligned}\ee 

We address the first  integral on the right.  
  For any $\ell,k,m,m'\in\Z$ with $\ell>k$, $\abs{m}\vee\abs{m'}\le Cn+1$, and $(\ell-k)/4\le 2Cn$, we have
	\begin{align*}
	&\P\Biggl\{\sup_{\substack{m\le x\le m+1\\ (k-2)/4\le s\le(k-1)/4\\ m'\le y\le m'+1\\ (\ell+1)/4\le t\le(\ell+2)/4}}\log\!\!\!\!\!\iint_{\R^2\setminus[-Mn,Mn]^2}
	\!\!\!\!\!\!\!\!\!\!\She(t,y\viiva \ell/4,w)\She(\ell/4,w\viiva k/4,u)\She(k/4,u\viiva s,x)\,dw\,du\ge-a n^2\Biggr\}\\
	&\quad\le e^{an^2}\!\!\!\!\!\!\!
	\iint_{\R^2\setminus[-Mn,Mn]^2}\!\!\!\!\!\E\Biggl[\sup_{\substack{m'\le y\le m'+1\\ (\ell+1)/4\le t\le(\ell+2)/4}}\!\!\!\!\!\!\!\!\!\!\She(t,y\viiva \ell/4,w)\Biggr]\,
	\heat\bigl((\ell-k)/4,w-u\bigr)\\
	&\qquad\qquad\qquad\qquad\qquad
	\times
	\E\Biggl[\sup_{\substack{m\le x\le m+1\\ (k-2)/4\le s\le(k-1)/4}}\!\!\!\!\!\!\!\!\!\!\She(k/4,u\viiva s,x)\Biggr]\,dw\,du\\
	&\quad= e^{an^2}\!\!\!\!\!\!\!\iint_{\R^2\setminus[-Mn,Mn]^2}\!\!\!\!\!\E\Biggl[\sup_{\substack{m'-w\le z\le m'-w+1\\ 1/4\le r\le1/2}}\!\!\!\!\!\!\!\!\!\!\She(r,z\viiva 0,0)\Biggr]\,
	\heat\bigl((\ell-k)/4,w-u\bigr)
	\E\Biggl[\sup_{\substack{u-m-1\le z\le u-m\\ 1/4\le r\le 1/2}}\!\!\!\!\!\!\!\!\!\!\She(r,z\viiva 0,0)\Biggr]\,dw\,du\\
	&\quad\le C e^{an^2}\!\!\!\!\!\!\!\iint_{\R^2\setminus[-Mn,Mn]^2}\!\!\!\!\!\!\!\!(m'-w)^3(u-m)^3\!\!\sup_{\substack{m'-w\le z\le m'-w+1\\ 1/4\le r\le1/2}}\!\!\!\!\!\!\!\!\heat(r,z)\,\heat\bigl((\ell-k)/4,w-u\bigr)\!\!\!\sup_{\substack{u-m-1\le z\le u-m\\ 1/4\le r\le 1/2}}\!\!\!\!\!\!\!\!\!\!\heat(r,z)\,dw\,du\\
	&\quad\le C'e^{an^2}\iint_{\R^2\setminus[-Mn,Mn]^2}e^{-c(w-m')^2}e^{-\frac{(w-u)^2}{4Cn}}e^{-c(u-m)^2}\,dw\,du\\
	&\quad\le C'e^{an^2}\int_{-\infty}^\infty e^{-c(w-m')^2}\,dw\cdot\int_{\R\setminus[-Mn,Mn]}e^{-c(u-m)^2}\,du\\
	&\qquad\qquad\qquad+C'e^{an^2}\int_{-\infty}^\infty e^{-c(u-m)^2}\,du\cdot\int_{\R\setminus[-Mn,Mn]}e^{-c(w-m')^2}\,dw\\
	&\quad\le C''e^{an^2}\int_{(M-C)n-1}^\infty e^{-cv^2}\,dv.
	\end{align*}
For the equality we used shift invariance and reflection symmetry from \cite[Proposition 2.3]{Alb-etal-22-spde-}. 
For the second inequality we used Corollary 3.10 in \cite{Alb-etal-22-spde-} and for the third 
inequality we used the bounds  $1/4<(\ell-k)/4\le t-s\le 2Cn$.
For the last inequality we used $\abs{m}\vee\abs{m'}\le Cn+1$.    

The above bounds, a union bound, and the Borel-Cantelli lemma tell us that $\P$-almost surely, for any $a>0$, and  for any $M$ large enough relative to $a$ and $C$,
\begin{align}\label{aux0001}
\begin{split}
&\lim_{n\to\infty}n^{-2}\sup_{\substack{(s,x,t,y)\in\R^4:t-s>1,\\ s,x,t,y\in[-Cn,Cn]}}
\log\!\!\!\!\!\iint_{\R^2\setminus[-Mn,Mn]^2}
	\!\!\!\!\!\!\!\!\!\!\She\bigl(t,y\bviiva (\fl{4t}-1)/4,w\bigr)\\
	&
	\qquad
	\qquad
	\times\She\bigl((\fl{4t}-1)/4,w\bviiva (\ce{4s}+1)/4,u\bigr)
	\, \She\bigl((\ce{4s}+1)/4,u\bviiva s,x\bigr)\,dw\,du \; \le \; -a.
\end{split}
\end{align}
This takes care of the first integral in \eqref{aux638}.

\medskip 

For the second integral in \eqref{aux638} we start by recording bounds on the discrete lattice from 
 \eqref{Z-grid} and \eqref{Z-grid2}:  $\P$-almost surely, for any  $\delta>0$, if $n$ is large enough, then for any 
$(s,u,t,w)\in\R^4$ with $t-s>1$, $\abs{s}+\abs{t}\le Cn$, and $\abs{u}\vee\abs{w}\le Mn$, 
	\begin{align*}
	&\log \She\bigl(\fl{4t}-1)/4,w\bviiva (\ce{4s}+1)/4,\fl{u}+1\bigr)\le\delta n+\frac{1}{24}-\frac{t-s}{24}-\frac{2(w-\fl{u}-1)^2}{\bigl(\fl{4t}-\ce{4s}-2\bigr)}\,,\\
	&\log \She\bigl(\fl{4t}-1)/4,\fl{w}\bviiva (\ce{4s}+1)/4,u\bigr)\le\delta n+\frac{1}{24}-\frac{t-s}{24}-\frac{2(\fl{w}-u)^2}{\bigl(\fl{4t}-\ce{4s}-2\bigr)}\,,\\
	&\log \She\bigl(\fl{4t}-1)/4,\fl{w}\bviiva (\ce{4s}+1)/4,\fl{u}+1\bigr)\ge-\delta n+\frac{1}{48}-\frac{t-s}{24}-\frac{2(\fl{w}-\fl{u}-1)^2}{\bigl(\fl{4t}-\ce{4s}-2\bigr)}\,.
	\end{align*}
Take $m=\fl{x}$, $m'=\fl{y}$, $k=\ce{4s}+1$, and $\ell=\fl{4t}-1$. As we integrate over $[-Mn,Mn]^2$, apply first the comparison inequality \eqref{crossing} to the middle term.  Then bound the integral by the maximum of the integrals over squares $[i,i+1]\times[j,j+1]$ 
 and  use the above bounds on the middle ratio. 
\begin{align}
&\iint_{[-Mn,Mn]^2}\She(t,y\viiva \ell/4,w)\She(\ell/4,w\viiva k/4,u)\She(k/4,u\viiva s,x)\,dw\,du\notag\\
&\le\iint_{[-Mn,Mn]^2}\She(t,y\viiva \ell/4,w)\cdot\frac{\She(\ell/4,w\viiva k/4,\fl{u}+1)\She(\ell/4,\fl{w}\viiva k/4,u)}{\She(\ell/4,\fl{w}\viiva k/4,\fl{u}+1)}
\cdot\She(k/4,u\viiva s,x)\,dw\,du\notag\\
&\le (2Mn)^2e^{3\delta n+\frac{1}{16}-\frac{t-s}{24}}\max_{\abs{i}\vee\abs{j}\le Mn+1}\Biggl(\Biggl[\sup_{\substack{j\le w\le j+1\\ m'\le v\le m'+1\\(\ell+1)1/4\le r\le (\ell+2)/4}}\frac{\She(r,v\viiva \ell/4,w)}{\heat(r-\ell/4,v-w)}\Biggr]\notag\\
&\qquad\qquad\qquad
\times\Biggl[\sup_{\substack{i\le u\le i+1\\ m\le v\le m+1\\(k-2)/4\le r\le (k-1)/4}}\frac{\She(k/4,u\viiva r,v)}{\heat(k/4-r,u-v)}\Biggr]\notag\\
&\qquad\qquad\qquad
\times\int_i^{i+1}\int_{j}^{j+1}\heat(t-\ell/4,y-w)\heat(k/4-s,u-x)e^{-\frac{2[(w-i-1)^2+(j-u)^2-(j-i-1)^2]}{(\ell-k)}}\,dw\,du\Biggr)\notag\\
&\le \tfrac{8M^2n^2}{\pi}e^{3\delta n+\frac{1}{16}-\frac{t-s}{24}}\max_{\abs{i}\vee\abs{j}\le Mn+1}\Biggl(\Biggl[\sup_{\substack{j\le w\le j+1\\ m'\le v\le m'+1\\(\ell+1)/4\le r\le (\ell+2)/4}}\Shenorm(r,v\viiva \ell/4,w)\Biggr]\notag\\
&\qquad\qquad\qquad\qquad\qquad
\times\Biggl[\sup_{\substack{i\le u\le i+1\\ m\le v\le m+1\\(k-2)/4\le r\le (k-1)/4}}\Shenorm(k/4,u\viiva r,v)\Biggr]\notag\\
&\qquad\qquad\qquad\qquad\qquad
\times\int_i^{i+1}\int_{j}^{j+1} e^{-\frac{(y-w)^2}{2(t-\ell/4)}}e^{-\frac{(u-x)^2}{2(k/4-s)}}
e^{-\frac{2[(w-i-1)^2+(j-u)^2-(j-i-1)^2]}{(\ell-k)}}\,dw\,du\Biggr)\notag\\
&\le \tfrac{8M^2n^2}{\pi}e^{3\delta n+\frac{1}{16}+4-\frac{t-s}{24}-\frac{(y-x)^2}{2(t-s)}}
\max_{\abs{j}\le Mn+1}\sup_{\substack{j\le w\le j+1\\ m'\le v\le m'+1\\(\ell+1)/4\le r\le (\ell+2)/4}}\Shenorm(r,v\viiva \ell/4,w)\notag\\
&\qquad\qquad\qquad\qquad\qquad\qquad
\times\max_{\abs{i}\le Mn+1}\sup_{\substack{i\le u\le i+1\\ m\le v\le m+1\\(k-2)/4\le r\le (k-1)/4}}\Shenorm(k/4,u\viiva r,v)\,.\label{aux0003}
\end{align}
For the last inequality we used two facts. First, 
	\[\abs{(w-i-1)^2+(j-u)^2-(j-i-1)^2-(w-u)^2}\le 2\]
and second, for $s<a<b<t$ and any $x,y\in\R$, the minimum over $(u,w)$ of $\frac{(y-w)^2}{2(t-b)}+\frac{(u-x)^2}{2(a-s)}+\frac{(w-u)^2}{2(b-a)}$ 
equals $\frac{(y-x)^2}{2(t-s)}$. 

Next, use the same $\delta>0$ and write for $\abs{j}\le Mn+1$, $\abs{m'}\le Cn+1$, and $\abs{\ell}\le 4Cn$,
\begin{align*}
&\P\Biggl\{\sup_{\substack{j\le w\le j+1\\ m'\le v\le m'+1\\(\ell+1)/4\le r\le (\ell+2)/4}}\log\Shenorm(r,v\viiva \ell/4,w)\ge\delta n\Biggr\}
=\P\Biggl\{\sup_{\substack{0\le w\le 1\\ m'-j\le v\le m'-j+1\\ 1/4\le r\le 1/2}}\log\Shenorm(r,v\viiva 0,w)\ge\delta n\Biggr\}\\
&\qquad\le e^{-\delta n}\tspb\E\Biggl[\;\sup_{\substack{0\le w\le 1\\ m'-j\le v\le m'-j+1\\ 1/4\le r\le 1/2}}\Shenorm(r,v\viiva 0,w)\Biggr] \le C'n^3 e^{-\delta n}
\end{align*}
for a finite strictly positive constant $C'$. The last bound 
used \cite[Corollary 3.11]{Alb-etal-22-spde-}.
 A union bound and the Borel-Cantelli lemma imply then that 
\begin{align}\label{aux0004}
\varlimsup_{n\to\infty}n^{-1}\max_{\substack{\abs{j}\le Mn+1\\ \abs{m'}\le Cn+1\\ \abs{\ell}\le 4Cn}}
\sup_{\substack{j\le w\le j+1\\ m'\le v\le m'+1\\(\ell+1)/4\le r\le (\ell+2)/4}}\log\frac{\She(r,v\viiva \ell/4,w)}{\heat(r-\ell/4,v-w)}\le\delta.
\end{align}
Similarly,
\begin{align}\label{aux0006}
\varlimsup_{n\to\infty}n^{-1}\max_{\substack{\abs{i}\le Mn+1\\ \abs{m}\le Cn+1\\ \abs{k}\le 4Cn+2}}
\sup_{\substack{i\le u\le i+1\\ m\le v\le m+1\\(k-2)/4\le r\le (k-1)/4}}\log\frac{\She(k/4,u\viiva r,v)}{\heat(k/4-r,u-v)}\le\delta.
\end{align}
Return to \eqref{aux638}. Take $a>2C^2$ and $M$ large enough for \eqref{aux0001} to hold. The choice of $a$ and the $n^{-2}$ normalization in \eqref{aux0001} control the term $\frac{(y-x)^2}{2(t-s)}$ which is at most ${2C^2n^2}$.
Putting \eqref{aux0004}--\eqref{aux0006} and \eqref{aux0001}--\eqref{aux0003} together and taking $\delta\to0$ gives 
\[\varlimsup_{n\to\infty}n^{-1}\sup_{\substack{(s,x,t,y)\in\R^4:t-s>1,\\ s,x,t,y\in[-Cn,Cn]}}\,\Bigl(\,\log \She(t,y\viiva s,x)+\frac{t-s}{24}+\frac{(y-x)^2}{2(t-s)}\,\Bigr)\le0.\]
The theorem is proved.
\end{proof}

We turn to the shape theorem for the cocycles $\Bus^\lambda$. Recall that, for the moment, the Busemann process is defined on the extended probability space $(\Omhat,\cFhat,\bfP)$. 

\begin{proposition}\label{pr:b-shape}
Fix $\lambda\in\R$. The following holds on an event of  $\bfP$-probability one.  For all $C>0$
	\begin{align}\label{b-shapehat}
	\lim_{n\to\infty}n^{-1}\!\!\!\sup_{s,x,t,y\tspb\in\tspb[-Cn, Cn]}\,\Bigl|\Bus^\lambda(s,x,t,y)-\Bigl(\frac{\lambda^2}2-\frac1{24}\Bigr)(t-s)-\lambda(y-x)\Bigr|=0.
	\end{align}
\end{proposition}

Let $\cI_{1,0}$ be the $\sigma$-algebra of events that are invariant under the shift by one unit in time.
Let $\cI_{0,1}$ be the $\sigma$-algebra of events that are invariant under the shift by one unit in space.  

Recall that we do not know yet if the distribution of $\Bus^{\lambda}(\aabullet,\aabullet,\aabullet,\aabullet)$ under $\bfP$ is ergodic under shifts in the time direction.  For $\lambda\in\R$ define the random variable 
	\[a_\lambda=\bfE[\Bus^{\lambda}(0,0,1,0)\,|\,\cI_{1,0}].  \]
%
Since $x \mapsto \Bus^\lambda(0,0,0,x)$ has the same distribution, under $\bfP$, as a standard Brownian motion with  drift $\lambda$, we have
	\[\bfE[\Bus^\lambda(0,0,0,1)\,|\,\cI_{0,1}]=\lambda\quad\bfP\text{-a.s.}\]

\begin{lemma}\label{a_mu}
For all $\lambda\in\R$ we have with $\bfP$-probability one, $a_\lambda=\lambda^2/2-1/24$.
\end{lemma}

\begin{proof}
The cocycle property of $\Bus^\lambda$, time-stationarity of $\bfP$, and   Birkhoff's ergodic theorem imply that $\bfP$-almost surely
\begin{align*}  a_\lambda=\lim_{n\to\infty}n^{-1}\Bus^\lambda(0,0,n,0).
	\end{align*} 
By the shear \eqref{bus-shear}, $\Bus^\lambda(0,0,n,0)$ has the same distribution as $\Bus^0(0,0,n,\lambda n)+\lambda^2/2$.
Birkhoff's ergodic theorem implies that 
$n^{-1}\Bus^0(0,0,0,\lambda n)$ converges to $0$, $\bfP$-almost surely, and 
 therefore $n^{-1}\Bus^0(n,0,n,\lambda n)$ (which has the same distribution as 
 $n^{-1}\Bus^0(n,0,n,\lambda n)$) converges to $0$ in $\bfP$-probability.
Theorem 1.2 in \cite{Bor-etal-15} implies that $n^{-1}\Bus^0(0,0,n,0)$ converges in $\bfP$-probability to $-1/24$. Therefore, 
	\[n^{-1}\Bus^0(0,0,n,\lambda n)=n^{-1}\Bus^0(0,0,n,0)+n^{-1}\Bus^0(n,0,n,\lambda n)\]
converges in $\bfP$-probability to $-1/24$. The claim follows.
%
\end{proof}

\begin{remark}\label{rk:shear}
The above lemma is another place where integrable probability is used. 
However, combining \eqref{bus-shear} with \eqref{b-shape} we get
$a_{\lambda+c}=a_\lambda-c\lambda+c^2/2$,
which gives $a_c=a_0+c^2/2$. This is more than enough for our results in this paper and knowing the exact value $a_0=-1/24$ is not necessary.
\end{remark}

\begin{proof}[Proof of Proposition \ref{pr:b-shape}]
By Lemma \ref{lm:Bus54} it suffices to prove that for some $p>2$
    \begin{align}\label{sup b tmp}
    \Bus^\lambda(0,0,0,1)\text{ and }\Bus^\lambda(0,0,1,0)\text{ are in }L^p(\bfP)\quad\text{and}\quad
    \sup_{\substack{1\le t\le 2\\ 0\le x\le1}}\abs{\Bus^\lambda(0,0,t,x)}\in L^2(\bfP).
    \end{align}
Since $\Bus^\lambda(0,0,0,1)$ is a normal random variable, all its moments are finite. 
That $\abs{\Bus^\lambda(0,0,1,0)}$ is in $L^p(\bfP)$ (and in fact has an exponential moment)
and the  last part of  \eqref{sup b tmp} come from verifying  that
	\begin{align}\label{Eeb}
	\bfE\Bigl[\,\sup_{\substack{1\le t\le2\\0\le y\le1}}e^{\Bus^\lambda(0,0,t,y)}\Bigr]
	<\infty
	\end{align}
	and 
	\begin{align}\label{Eebinv}
	\bfE\Bigl[\,\sup_{\substack{1\le t\le2\\0\le y\le1}}e^{-\Bus^\lambda(0,0,t,y)}\Bigr]<\infty.
	\end{align}

For \eqref{Eebinv},  use  \eqref{NEBus2}  to write
	\begin{align}\label{ebinv-comp}
	\begin{split}
	\bfE\Bigl[\,\sup_{\substack{1\le t\le2\\0\le y\le1}}e^{-\Bus^\lambda(0,0,t,y)}\Bigr]
	&=\bfE\Bigl[\,\sup_{\substack{1\le t\le2\\0\le y\le1}}\Bigl(\;\int_{-\infty}^\infty e^{\Bus^\lambda(0,0,0,x)}\She(t,y\viiva 0,x)\,dx\Bigr)^{-1}\Bigr]\\
	&\le\bfE\Bigl[\,\sup_{\substack{1\le t\le2\\0\le y\le1}}\Bigl(\;\int_{0}^1 e^{\Bus^\lambda(0,0,0,x)}\She(t,y\viiva 0,x)\,dx\Bigr)^{-1}\Bigr]\\
	&\le\bfE\Bigl[\,\sup_{0\le x\le 1}e^{-\Bus^\lambda(0,0,0,x)}\Bigr]\,\E\Bigl[\,\sup_{\substack{1\le t\le2\\0\le x,y\le1}}\She(t,y\viiva 0,x)^{-1}\Bigr].
	\end{split}
	\end{align}
The first expectation 
on the right-hand side is finite because $\Bus^\lambda(0,0,0,x)-\lambda x$ is a Brownian motion. 
The second expectation is also seen to be finite by applying \cite[Corollary 3.11]{Alb-etal-22-spde-}.

For \eqref{Eeb},  apply \eqref{NEBus2} again to write
\begin{align*}
\bfE\Bigl[\,\sup_{\substack{1\le t\le2\\0\le y\le1}}e^{\Bus^\lambda(0,0,t,y)}\Bigr]
&=\bfE\Bigl[\,\sup_{\substack{1\le t\le2\\0\le y\le1}}\int_{-\infty}^\infty \She(t,y\viiva 0,x)e^{\Bus^\lambda(0,0,0,x)}\,dx\Bigr]\\
&=\bfE\Bigl[\,\sup_{\substack{1\le t\le2\\0\le y\le1}}\int_{-\infty}^\infty \Shenorm(t,y\viiva 0,x)\heat(t,y-x)e^{\Bus^\lambda(0,0,0,x)}\,dx.\Bigr]\\
&\le\bfE\Bigl[\,\int_{-\infty}^\infty \sup_{\substack{1\le t\le2\\0\le y\le1}}\Shenorm(t,y\viiva 0,x)\cdot\sup_{\substack{1\le t\le2\\0\le y\le1}} \heat(t,y-x)\cdot e^{\Bus^\lambda(0,0,0,x)}\,dx\Bigr]\\
&=\int_{-\infty}^\infty \E\Bigl[\sup_{\substack{1\le t\le2\\0\le y\le1}}\Shenorm(t,y\viiva 0,x)\Bigr]\cdot\sup_{\substack{1\le t\le2\\0\le y\le1}} \heat(t,y-x)\cdot\bfE[e^{\Bus^\lambda(0,0,0,x)}]\,dx\\
&\le C\int_{-\infty}^\infty \abs{x}^3\cdot e^{-\frac{x^2-2\abs{x}}4}\cdot e^{\frac{x}2+\lambda x}\,dx<\infty.
\end{align*}
In the second inequality we applied Corollary 3.10 in \cite{Alb-etal-22-spde-} and used the facts that $\Bus^\lambda(0,0,0,x)-\lambda x$ is a two-sided standard Brownian motion and $\heat(t,y-x)\le \frac1{\sqrt{2\pi}}e^{-\frac{x^2-2\abs{x}}4}$ for all $t\in[1,2]$ and $y\in[0,1]$.
\end{proof}


The following theorem gives versions of \eqref{b-shapehat} that hold for all $\lambda$ simultaneously.  

\begin{theorem}\label{bus:exp}
The following holds $\bfP$-almost surely: for all $\lambda\in\R$ and $\sigg\in\{-,+\}$
	\begin{align}
	&\lim_{r\to-\infty}\abs{r}^{-1}\sup_{\abs{x}\le C\abs{r}}\abs{\Bus^{\lambda\sig}(r,0,r,x)-\lambda x}=0,\quad\text{for all $C>0$,}\label{bus:exp1}\\
	&\lim_{r\to-\infty}\sup_{x\in\R}\frac{\abs{\Bus^{\lambda\sig}(r,0,r,x)-\lambda x}}{\abs{r}+\abs{x}}=0,\quad\text{and}\label{bus:exp2}\\
	&\lim_{\abs{x}\to\infty}\abs{x}^{-1}\abs{\Bus^{\lambda\sig}(t,0,t,x)-\lambda x}=0\quad\text{for all $t\in\R$.}\label{bus:exp3}
	\end{align}
\end{theorem}

\begin{proof}
Due to the monotonicity in $C$, it is enough to work with a fixed value of $C$. 
Let $\Omhat_1$ be the intersection of the full $\bfP$-probability event on which 
\eqref{Busmonohat} holds with the full $\bfP$-probability event on which
\eqref{b-shapehat} holds simultaneously for all $\lambda\in\Udense$.

For any $\what\in\Omhat_1$, any $\kappa<\mu$ in $\Udense$, and any $\lambda\in(\kappa,\mu)$, $r<0$, and $x>0$, by \eqref{Busmonohat}, 
	\begin{align*}
	\Bus^{\kappa}(r,0,r,x)-\kappa x+(\kappa-\lambda)x
	\le \Bus^{\lambda\sig}(r,0,r,x)-\lambda x
	\le \Bus^{\mu}(r,0,r,x)-\mu x+(\mu-\lambda)x.
	\end{align*}
The same bounds with reversed inequalities hold for $x<0$.
Divide by $r$, take it to $-\infty$, and use \eqref{b-shapehat} to get that
	\[\varlimsup_{r\to-\infty}\abs{r}^{-1}\sup_{\abs{x}\le C\abs{r}}\abs{\Bus^{\lambda\sig}(r,0,r,x)-\lambda x}\le 
	C(\mu-\kappa).\]
Take $\mu$ down to $\lambda$ and $\kappa$ up to $\lambda$ to get the first limit. 

For the second limit one can again use the above bounds to reduce the problem to one with $\lambda\in\Udense$ and thus dispense with the $\sigg$. 
For $\lambda\in\Udense$ we have for any $m\in\Z$, $n\in\N$, and $\delta\in(0,1/2)$
	\begin{align*}
	&\bfP\Bigl\{\sup_{\substack{-n-1\le r\le -n\\ m\le x\le m+1}}\abs{\Bus^\lambda(r,0,r,x)-\lambda x}\ge\delta(n+\abs{m})\Bigr\}\\
	&\qquad=\bfP\Bigl\{\sup_{\substack{1\le r\le 2\\ m\le x\le m+1}}\abs{\Bus^\lambda(r,0,r,x)-\lambda x}\ge\delta(n+\abs{m})\Bigr\}\\
	&\qquad\le e^{-\delta^2(n+\abs{m})/3}\bfE\Bigl[\exp\Bigl\{\frac{\delta}3\sup_{\substack{1\le r\le 2\\ m\le x\le m+1}}\abs{\Bus^\lambda(r,0,r,x)-\lambda x}\Bigr\}\Bigr]\\
	&\qquad\le e^{-\delta^2(n+\abs{m})/3}\bfE\Bigl[\exp\Bigl\{\delta\sup_{1\le r\le 2}\abs{\Bus^\lambda(0,0,r,0)}\Bigr\}\Bigr]^{1/3}
	\bfE\Bigl[\exp\Bigl\{\delta\sup_{m\le x\le m+1}\abs{\Bus^\lambda(0,0,0,x)-\lambda x}\Bigr\}\Bigr]^{1/3}\\
	&\qquad\qquad\qquad\qquad\times\bfE\Bigl[\exp\Bigl\{\delta\sup_{\substack{1\le r\le 2\\ m\le x\le m+1}}\abs{\Bus^\lambda(0,x,r,x)}\Bigr\}\Bigr]^{1/3}\\
	&\qquad\le e^{-\delta^2(n+\abs{m})/3}
	\bfE\Bigl[\exp\Bigl\{\delta\sup_{m\le x\le m+1}\abs{\Bus^\lambda(0,0,0,x)-\lambda x}\Bigr\}\Bigr]^{1/3}
	\bfE\Bigl[\exp\Bigl\{\delta\sup_{\substack{1\le r\le 2\\ 0\le x\le 1}}\abs{\Bus^\lambda(0,x,r,x)}\Bigr\}\Bigr]^{2/3}\\
	&\qquad\le e^{-\delta^2(n+\abs{m})/3}
	\bfE\Bigl[\exp\Bigl\{\delta\sup_{m\le x\le m+1}\abs{\Bus^\lambda(0,0,0,x)-\lambda x}\Bigr\}\Bigr]^{1/3}
	\bfE\Bigl[\exp\Bigl\{2\delta\sup_{\substack{1\le r\le 2\\ 0\le x\le 1}}\abs{\Bus^\lambda(0,0,r,x)}\Bigr\}\Bigr]^{1/3}\\
	&\qquad\qquad\qquad\qquad\times\bfE\Bigl[\exp\Bigl\{2\delta\sup_{0\le x\le 1}\abs{\Bus^\lambda(0,0,0,x)}\Bigr\}\Bigr]^{1/3}\\
	&\qquad\le C e^{-\delta^2(n+\abs{m})/3}e^{\delta^2\abs{m}/6}.
	\end{align*}
We used here \eqref{Eeb} and \eqref{Eebinv} (since $2\delta<1$) and the fact
that  $\Bus^\lambda(0,0,0,x)-\lambda x$ is a standard Brownian motion.
The desired limit comes then by applying 
Borel-Cantelli lemma. The last limit is an easier version of the first two.
\end{proof}


The next result is a convex dual of \eqref{Z-cont}, which can be interpreted as a variant of Varadhan's theorem from the theory of large deviations.

\begin{lemma}\label{int-shape2}
The following holds $\P$-almost surely.
For all $\mu\in\R$ and all $-\infty\le \lambda_1<\lambda_2\le \infty$, for any $C>0$,
\begin{align}\label{dual-shape}
\lim_{r\to-\infty}\,\sup_{t,x\in[-C,C]}\,\Bigr|\,\frac1{\abs{r}}\log\int_{\lambda_1\abs{r}}^{\lambda_2\abs{r}}\She(t,x\viiva r,w)\,e^{\mu w}\,dw-\sup_{\lambda_1<\lambda< \lambda_2}\Bigl\{\mu\lambda-\frac{\lambda^2}2-\frac1{24}\Bigr\}\Bigr|=0.
\end{align}
\end{lemma}

We will need a variant of this lemma
that links $\She$ and the Busemann process.
This is given in the next theorem. The proof of this theorem comes at the end of the section. 
The proof of Lemma \ref{int-shape2} is an easier version of that of Theorem \ref{int-shape} and is therefore omitted.

\begin{theorem}\label{int-shape}
The following holds $\bfP$-almost surely.
For all $\mu\in\R$, $\sigg\in\{-,+\}$, for all $-\infty\le \lambda_1<\lambda_2\le \infty$, 
and for any $C>0$,
\begin{align}\label{eq:Zb}
\lim_{r\to-\infty}\,\sup_{t,x\in[-C,C]}\,\Bigr|\,\frac1{\abs{r}}\log\int_{\lambda_1\abs{r}}^{\lambda_2\abs{r}}\She(t,x\,|\,r,w)e^{\Bus^{\mu\sig}(r,0,r,w)}\,dw-\sup_{\lambda_1<\lambda< \lambda_2}\Bigl\{\mu\lambda-\frac{\lambda^2}2-\frac1{24}\Bigr\}\Bigr|=0.
\end{align}
\end{theorem}

\begin{proof}
First, we prove the result for fixed $\mu\in\R$ and  $-\infty\le \lambda_1<\lambda_2\le \infty$. In this case, there is no need for the $\pm$ distinction.
We begin by treating the case of finite $\lambda_1$ and $\lambda_2$. In this case, \eqref{b-shapehat} and \eqref{Z-cont} imply that $\P$-almost surely, for  $\delta>0$, $C>0$, and $r<0$ with $\abs{r}$ large enough, we have for any 
$t,x\in[-C,C]$
	\begin{align*}
	&\frac1{\abs{r}}\log\int_{\lambda_1\abs{r}}^{\lambda_2\abs{r}}\She(t,x\viiva r,w)e^{\Bus^\mu(r,0,r,w)}\,dw
	\le \delta+\frac1{\abs{r}}\log\int_{\lambda_1\abs{r}}^{\lambda_2\abs{r}}e^{-\frac{t-r}{24}-\frac{(x-w)^2}{2(t-r)}+\mu w}\,dw\\
	&\qquad\le\delta+\frac{C}{24\abs{r}}-\frac1{24}+\frac{\log\abs{r}}{\abs{r}}+\frac{C\abs{\mu}}{\abs{r}}+\frac1{\abs{r}}\log\int_{\lambda_1-C/\abs{r}}^{\lambda_2+C/\abs{r}}e^{-\frac{\abs{r}^2 u^2}{2(C-r)}}e^{\mu\abs{r}u}\,du\\
	&\qquad\le\delta+\frac{C}{24\abs{r}}-\frac1{24}+\frac{\log\abs{r}}{\abs{r}}+\frac{C\abs{\mu}}{\abs{r}}+\frac{\log\bigl(\lambda_2-\lambda_1+2C/\abs{r}\bigr)}{\abs{r}}\\
	&\qquad\qquad+\sup_{\lambda_1-C/\abs{r}<\lambda<\lambda_2+C/\abs{r}}\Bigl\{-\frac{\abs{r}\lambda^2}{2(C+\abs{r})}+\mu\lambda\Bigr\}\\
	&\qquad\le\delta+\frac{C}{24\abs{r}}-\frac1{24}+\frac{\log\abs{r}}{\abs{r}}+\frac{C\abs{\mu}}{\abs{r}}+\frac{\log\bigl(\lambda_2-\lambda_1+2C/\abs{r}\bigr)}{\abs{r}}+\frac{C(\lambda_1^2\vee\lambda_2^2)}{2(C+\abs{r})}\\
	&\qquad\qquad+\sup_{\lambda_1-C/\abs{r}<\lambda<\lambda_2+C/\abs{r}}\Bigl\{-\frac{\lambda^2}{2}+\mu\lambda\Bigr\}\\
	&\qquad\le\delta+\frac{C}{24\abs{r}}+\frac{\log\abs{r}}{\abs{r}}+\frac{C\abs{\mu}}{\abs{r}}+\frac{\log\bigl(\lambda_2-\lambda_1+2C/\abs{r}\bigr)}{\abs{r}}+\frac{C(\lambda_1^2\vee\lambda_2^2)}{2(C+\abs{r})}\\
	&\qquad\qquad+\sup_{\lambda_1-C/\abs{r}<\lambda<\lambda_2+C/\abs{r}}\Bigl\{\frac{\mu^2}2-\frac1{24}-\frac{(\mu-\lambda)^2}2\Bigr\}.
	\end{align*}
The lower bound comes similarly. Taking $\delta\to0$ shows that
\[\lim_{r\to-\infty}\sup_{t,x\in[-C,C]}\Bigl|\frac1{\abs{r}}\log\int_{\lambda_1\abs{r}}^{\lambda_2\abs{r}}\She(t,x\,|\,r,w)e^{\Bus^\mu(r,0,r,w)}\,dw-\sup_{\lambda_1<\lambda<\lambda_2}\Bigl\{\mu\lambda-\frac{\lambda^2}2-\frac1{24}\Bigr\}\Bigr|=0.\]

Next, we treat the case where $\lambda_1=-\infty$ but $\lambda_2<\infty$. 
By \cite[Lemma 3.1]{Alb-etal-22-spde-},
there exists a $c>0$ such that 
$\Mom{4}{2(n+1)}^{1/4}n^3\le e^{cn}$ for all $n\in\N$. 
Fix
	\[M<\sup_{\lambda<\lambda_2}\Bigl\{\mu\lambda-\frac{\lambda^2}2-\frac1{24}\Bigr\}\] 
and $\lambda_1'<0\wedge\lambda_2$ such that 
	\begin{align}\label{lam1}
	(\mu+2)^2-\bigl(\lambda_1'-2(\mu+2)\bigr)/4\le M-c-1
	\end{align}
and 
	\begin{align}\label{sup=sup}
	\sup_{\lambda<\lambda_2}\Bigl\{\mu\lambda-\frac{\lambda^2}2-\frac1{24}\Bigr\}=\sup_{\lambda_1'<\lambda<\lambda_2}\Bigl\{\mu\lambda-\frac{\lambda^2}2-\frac1{24}\Bigr\}.
	\end{align}
Then for $n\in\N$ 
	\begin{align*}
	&\bfP\Bigl\{\frac1{n}\log\int_{-\infty}^{\lambda_1' n}\sup_{\substack{-n-1\le r\le-n\\ t,x\in[-C,C]}}\She(t,x\,|\,r,w)\sup_{-n-1\le s\le-n}e^{\Bus^\mu(s,0,s,w)}\,dw\ge M\Bigr\}\\
	&\qquad\le e^{-Mn}\int_{-\infty}^{\lambda_1'n}\E\Bigl[\,\sup_{\substack{-n-1\le r\le-n\\ t,x\in[-C,C]}}\She(t,x\,|\,r,w)^4\Bigr]^{1/4}\,\bfE\Bigl[\,\sup_{-n-1\le s\le -n}e^{4\Bus^\mu(s,0,-n-2,0)}\Bigr]^{1/4}\\
	&\qquad\qquad\qquad\qquad\qquad\times\bfE\bigl[e^{4\Bus^\mu(-n-2,0,-n-2,w)}\bigr]^{1/4}\,\bfE\Bigl[\,\sup_{-n-1\le s\le-n}e^{4\Bus^\mu(-n-2,w,s,w)}\Bigr]^{1/4}\,dw\\
	&= e^{-Mn}\bfE\Bigl[\,\sup_{1\le s< 2}e^{-4\Bus^\mu(0,0,s,0)}\Bigr]^{1/4}\,\bfE[\sup_{1\le s< 2}e^{4\Bus^\mu(0,0,s,0)}\Bigr]^{1/4}\\
	&\qquad\qquad\qquad\qquad\qquad\times\int_{-\infty}^{\lambda_1'n}\E\Bigl[\,\sup_{\substack{n\le r\le n+1\\ t,x\in[-C,C]}}\She(r,-w\,|\,t,x)^4\Bigr]^{1/4}\,\bfE[e^{4\Bus^\mu(0,0,0,w)}]^{1/4}\,dw\\
	&= e^{-Mn}\bfE\Bigl[\,\sup_{1\le s< 2}e^{-4\Bus^\mu(0,0,s,0)}\Bigr]^{1/4}\,\bfE\Bigl[\sup_{1\le s< 2}e^{4\Bus^\mu(0,0,s,0)}\Bigr]^{1/4}\\
	&\qquad\qquad\qquad\qquad\qquad\times\int_{-\infty}^{\lambda_1'n}\E\Bigl[\,\sup_{\substack{n\le r\le n+1\\ t,x\in[-C,C]}}\She(r,-w\,|\,t,x)^4\Bigr]^{1/4}e^{(\mu+2)w}\,dw.
	\end{align*}
	
By similar arguments to the ones giving \eqref{Eeb} and \eqref{ebinv-comp}, the expectations in front of the last integral are finite. Applying Corollary 3.10 in \cite{Alb-etal-22-spde-}, we have
\begin{align*}
\E\Bigl[\,\sup_{\substack{n\le r\le n+1\\ t,x\in[-C,C]}}\She(r,-w\,|\,t,x)^4\Bigr]^{1/4} &\le 
\E\Bigl[\,\sup_{\substack{n\le r\le n+1\\ t,x\in[-C,C]}}\Shenorm(r,-w\,|\,t,x)^4\Bigr]^{1/4}\cdot
\sup_{\substack{n\le r\le n+1\\ t,x\in[-C,C]}}\heat(r-t,w+x)\\
&\le C' \Mom{4}{2(n+1)}^{1/4}n^3 \abs{w}^3 e^{-\frac{w^2}{2n}},
\end{align*} 
for some constant $C'>0$.
Recall that $\Mom{4}{2(n+1)}^{1/4}\le e^{cn}$ for all $n\in\N$.
Therefore, we can continue the above bounds with:
	\begin{align*}
	&\le C'n^3e^{-(M-c)n}\int_{-\infty}^{\lambda_1' n}\abs{w}^3e^{-\frac{w^2}{2n}}e^{(\mu+2)w}\,dw\\	
	&= C'e^{-(M-c) n}n^5\int_{-\infty}^{\lambda_1' \sqrt{n}}\abs{u}^3e^{-\frac{u^2}{2}}e^{(\mu+2)u\sqrt{n}}\,du\\	
	&\le C''e^{-(M-c) n}n^5\int_{-\infty}^{\lambda_1' \sqrt{n}}e^{-\frac{u^2}{4}}e^{(\mu+2)u\sqrt{n}}\,du\\	
	&= C''e^{-(M-c) n}e^{(\mu+2)^2n}n^5\int_{-\infty}^{\lambda_1' \sqrt{n}}e^{-\frac{(u-2(\mu+2)\sqrt n\tspb)^2}{4}}\,du\\	
	&= 4C''e^{-(M-c) n}e^{(\mu+2)^2n}n^5\sqrt{\pi}\cdot\frac1{\sqrt{2\pi}}\int_{-\infty}^{(\lambda_1' -2(\mu+2))\sqrt{n/2}}e^{-\frac{v^2}{2}}\,dv\\	
	&\le 4C''e^{-(M-c) n}e^{(\mu+2)^2n-(\lambda_1' -2(\mu+2))^2 n/4}n^5\sqrt{\pi}.	
	\end{align*}
	By \eqref{lam1}, the right-hand side is bounded by $4C''e^{-n}n^5\sqrt{\pi}$.
Consequently, by the Borel-Cantelli lemma, $\bfP$-almost surely, for $n$ large enough,
	\[\frac1{n}\log\int_{-\infty}^{\lambda_1' n}\sup_{\substack{-n-1\le r\le -n\\ t,x\in[-C,C]}}\She(t,x\,|\,r,w)\sup_{-n-1\le s\le-n}e^{\Bus^\mu(s,0,s,w)}\,dw<\sup_{\lambda<\lambda_2}\Bigl\{\mu\lambda-\frac{\lambda^2}2-\frac1{24}\Bigr\}.\]
This implies that $\bfP$-almost surely, for $r<0$ with $\abs{r}$ large enough,  
	\begin{align}\label{exp-tight}
	\sup_{t,x\in[-C,C]}\frac1{\abs{r}}\log\int_{-\infty}^{\lambda_1' \abs{r}}\She(t,x\viiva r,w)e^{\Bus^\mu(r,0,r,w)}\,dw<\sup_{\lambda<\lambda_2}\Bigl\{\mu\lambda-\frac{\lambda^2}2-\frac1{24}\Bigr\}.
	\end{align}
	Recalling \eqref{sup=sup}, we by now know that 
	\begin{align*}
	\lim_{r\to-\infty}\sup_{t,x\in[-C,C]}\Bigl|\frac1{\abs{r}}\log\int_{\lambda_1' \abs{r}}^{\lambda_2\abs{r}}\She(t,x\viiva r,w)e^{\Bus^\mu(r,0,r,w)}\,dw
	&-\sup_{\lambda<\lambda_2}\Bigl\{\mu\lambda-\frac{\lambda^2}2-\frac1{24}\Bigr\}\Bigr|=0.
	\end{align*}
	Therefore, we have
	\[\lim_{r\to-\infty}\sup_{t,x\in[-C,C]}\Bigl|\frac1{\abs{r}}\log\int_{-\infty}^{\lambda_2 \abs{r}}\She(t,x\viiva r,w)e^{\Bus^\mu(r,0,r,w)}\,dw
	-\sup_{\lambda<\lambda_2}\Bigl\{\mu\lambda-\frac{\lambda^2}2-\frac1{24}\Bigr\}\Bigr|=0.\]
 The case where $\lambda_1>-\infty$ and $\lambda_2=\infty$ is similar and the case $\lambda_1=-\infty$ and $\lambda_2=\infty$ follows.
 
 To prove the claim on one full $\bfP$-probability event, simultaneously for all $\mu\in\R$, $\sigg\in\{-,+\}$, and $-\infty\le\lambda_1<\lambda_2\le\infty$, note that if $\lambda_1'<\lambda_2'\in\Q\cup\{-\infty,\infty\}$ and $\mu\in(\kappa_1,\kappa_2)$ with $\kappa_1<\kappa_2$ in $\Udense$, then one can bound
 \begin{align*}
 &\int_{\lambda_1\abs{r}}^{\lambda_2\abs{r}}\She(t,x\,|\,r,w)e^{\Bus^{\mu\sig}(r,0,r,w)}\,dw
 \le\int_{\lambda_1'\abs{r}}^{\lambda_2'\abs{r}}\She(t,x\,|\,r,w)e^{\Bus^{\mu\sig}(r,0,r,w)}\,dw\\
&\qquad \le \int_{(0\wedge\lambda_1')\abs{r}}^{(0\wedge\lambda_2')\abs{r}}\She(t,x\,|\,r,w)e^{\Bus^{\kappa_1}(r,0,r,w)}\,dw
 +\int_{(0\vee\lambda_1')\abs{r}}^{(0\vee\lambda_2')\abs{r}}\She(t,x\,|\,r,w)e^{\Bus^{\kappa_2}(r,0,r,w)}\,dw
\end{align*}
and applying the already proved result to get that 
\begin{align*}
&\varlimsup_{r\to-\infty}\,\sup_{t,x\in[-C,C]}\frac1{\abs{r}}\log\int_{\lambda_1\abs{r}}^{\lambda_2\abs{r}}\She(t,x\,|\,r,w)e^{\Bus^{\mu\sig}(r,0,r,w)}\,dw\\
&\qquad\le\sup_{0\wedge\lambda_1'<\lambda< 0\wedge\lambda_2'}\Bigl\{\kappa_1\lambda-\frac{\lambda^2}2-\frac1{24}\Bigr\}
+\sup_{0\vee\lambda_1'<\lambda< 0\vee\lambda_2'}\Bigl\{\kappa_2\lambda-\frac{\lambda^2}2-\frac1{24}\Bigr\}.
\end{align*}
Taking $\lambda_1'\to\lambda_1$, $\lambda_2'\to\lambda_2$, and both $\kappa_1$ and $\kappa_2$ to $\mu$ we get the desired upper bound. The lower bound comes similarly. 
 \end{proof}
 
 \begin{remark}\label{rk:amu}
The formula for $a_\mu$ that appears in Lemma \ref{a_mu} can also be seen from \eqref{eq:Zb}.
Namely, use 
\eqref{NEBus2hat} and \eqref{b-shapehat} to write
	\begin{align*}
	\frac1{\abs{r}}\log\int_{-\infty}^{\infty}\She(0,0\,|\,r,w)e^{\Bus^\mu(r,0,r,w)}\,dw
	=\frac{\Bus^\mu(r,0,0,0)}{\abs{r}}\mathop{\longrightarrow}_{r\to-\infty}a_\mu.
	\end{align*}
Consequently,
	\[a_\mu=\sup_\lambda\Bigl\{\mu\lambda-\frac{\lambda^2}2-\frac1{24}\Bigr\}=\frac{\mu^2}2-\frac1{24}.\]
\end{remark}

\section{Busemann limits}\label{sec:BusLim}

With the Busemann process constructed in Section \ref{sec:stat-coc} and the shape theorems proved in Section \ref{sec:shape}, we can prove the limits 
claimed \eqref{eq:buslim} and \eqref{eq:buslim2}.
However, we are still working on the  extended probability space $(\Omhat,\cFhat,\bfP)$ of \eqref{Omhat2}.
Define the set
	\begin{align}\label{Brunohat}
	\Brunohat=\{\lambda\in\R:\exists(s,x,t,y)\in\R^4\text{ with }\Bus^{\lambda-}(s,x,t,y)\ne\Bus^{\lambda+}(s,x,t,y)\}.
	\end{align}

\begin{remark}\label{rk:Bruno}
When we switch back to $(\Omega,\sF,\P)$, in Section \ref{sec:erg}, we will denote the above set by $\Bruno$. See \eqref{Bruno}.
\end{remark}

When $\lambda\not\in\Brunohat$ we write $\Bus^{\lambda}$ for $\Bus^{\lambda+}$ and $\Bus^{\lambda-}$.  This convention is consistent with the earlier use of  $\Bus^\lambda$ for  $\lambda\in\Udense$
because   by Lemma \ref{b-=b+} we can always include  the condition  $\Udense\subset\R\setminus\Brunohat$ in any  $\bfP$-full probability event we work on. 

The main goal of this section is to prove the following proposition.

\begin{proposition}\label{pr:buslim}
There exists an event $\Omhat_0\in\cFhat$ such that $\bfP(\Omhat_0)=1$ and we have the following for all $\w\in\Omhat_0$, $\lambda\in\R$, $C>0$, $\e>0$, and $\tau>0$.
\begin{enumerate}[label={\rm(\alph*)}, ref={\rm\alph*}]
\item\label{buslimhat1} 
There exist {\rm(}possibly random{\rm)} $R<0$ and deterministic $\delta>0$ 
such that for all  $r\le R$,  $z$ such that $\abs{\frac{z}r+\lambda}<\delta$, 
and for all $s,x,t,y\in[-C,C]$ with $t-s\ge\tau$, 
	\[\frac{\She(t,y\viiva r,z)}{\She(s,x\viiva r,z)}
	\le (1+\e)^2\int_x^\infty\She(t,y\viiva s,w)e^{\Bus^{\lambda+}(s,x,s,w)}\,dw+(1+\e)^2\int_{-\infty}^x\She(t,y\viiva s,w)e^{\Bus^{\lambda-}(s,x,s,w)}\,dw
	\]
and 
	\[\frac{\She(t,y\viiva r,z)}{\She(s,x\viiva r,z)}
	\ge (1+\e)^{-3}\int_x^\infty\She(t,y\viiva s,w)e^{\Bus^{\lambda-}(s,x,s,w)}\,dw+(1+\e)^{-3}\int_{-\infty}^x\She(t,y\viiva s,w)e^{\Bus^{\lambda+}(s,x,s,w)}\,dw.
	\]
\item\label{buslimhat2} For each $f\in\initF_\lambda$,
there exists {\rm(}possibly random{\rm)}   $R<0$ such that for all  $r\le R$, 
for all $s,x,t,y\in[-C,C]$ with $t-s\ge\tau$, we have
	\begin{align*}
	\frac{\int_\R\She(t,y\viiva r,z)\,f(r,z)\,dz}{\int_\R\She(s,x\viiva r,z)\,f(r,z)\,dz}
	&\le (1+\e)^3\int_x^\infty\She(t,y\viiva s,w)e^{\Bus^{\lambda+}(s,x,s,w)}\,dw\\
	&\qquad\qquad +(1+\e)^3\int_{-\infty}^x\She(t,y\viiva s,w)e^{\Bus^{\lambda-}(s,x,s,w)}\,dw
	\end{align*} 
and 
	\begin{align*}
	\frac{\int_\R\She(t,y\viiva r,z)\,f(r,z)\,dz}{\int_\R\She(s,x\viiva r,z)\,f(r,z)\,dz}
	&\ge (1+\e)^{-4}\int_x^\infty\She(t,y\viiva s,w)e^{\Bus^{\lambda-}(s,x,s,w)}\,dw\\
	&\qquad\qquad
	+(1+\e)^{-4}\int_{-\infty}^x\She(t,y\viiva s,w)e^{\Bus^{\lambda+}(s,x,s,w)}\,dw.
	\end{align*} 
\end{enumerate}
In particular, with $\P$-probability one, for any $\lambda\not\in\Bruno$ and any $f\in\initF_\lambda$, we have the limits
	\begin{align}
	&\lim_{\substack{r\to-\infty\\ z/r\to-\lambda}}\frac{\She(t,y\viiva r,z)}{\She(s,x\viiva r,z)}=e^{\Bus^\lambda(s,x,t,y)}\quad\text{and}	\label{eq:buslimhat}\\
	&\lim_{r\to-\infty}\frac{\int_\R\She(t,y\viiva r,z)\,f(r,z)\,dz}{\int_\R\She(s,x\viiva r,z)\,f(r,z)\,dz}=e^{\Bus^\lambda(s,x,t,y)},
	\label{eq:buslim2hat}
\end{align}
locally uniformly in $(s,x,t,y)\in\R^4$.
\end{proposition}
We begin with the following preliminary lemma.

\begin{lemma}\label{Bus-bnds}
The following holds with $\bfP$-probability one: for all $\kappa<\mu$ in $\Udense$, 
$\e>0$,  and $C>0$,
there exist 
{\rm(}possibly random{\rm)}
$R<-C$ and 
deterministic
$\delta>0$ such that for all $\lambda\in[\kappa+\e,\mu-\e]$ and $t,x\in[-C,C]$, for all $r\le R$, and for all
$z$ such that $\abs{\frac{z}r+\lambda}<\delta$, 
	\begin{align}\label{ratio-bnd1}
	\frac{\She(t,y\viiva r,z)}{\She(t,x\viiva r,z)}
	\le (1+\e)e^{\Bus^{\mu}(t,x,t,y)}\quad\text{for all $y\in(x,\infty)$}
	\end{align}
and
	\begin{align}\label{ratio-bnd2}
	\frac{\She(t,y\viiva r,z)}{\She(t,x\viiva r,z)}\le (1+\e)e^{\Bus^{\kappa}(t,x,t,y)}\quad\text{for all $y\in(-\infty,x)$}.
	\end{align}
\end{lemma}

\begin{proof}
We prove \eqref{ratio-bnd1}, the other bound being similar.
Take the full $\bfP$-probability event to be the intersection of the events in Proposition \ref{Bproc2} and Lemmas \ref{b-=b+} and \ref{lm:comp}, with the event on which Theorem \ref{int-shape} holds for all
$\mu\in\Udense$ and $\lambda_1,\lambda_2\in\Udense\cup\{\pm\infty\}$. 
By \eqref{comp}, for all $y>x$, all $r<t$, and all $f$ as in that result,
	\begin{align*}
	\frac{\She(t,y\viiva r,z)}{\She(t,x\viiva r,z)}
	&\le 	\frac{\int_{z}^\infty\She(t,y\viiva r,w)f(w)\,dw}{\int_{z}^\infty\She(t,x\viiva r,w)f(w)\,dw}\\
	&=\frac{\int_{-\infty}^\infty\She(t,y\viiva r,w)f(w)\,dw}{\int_{-\infty}^\infty\She(t,x\viiva r,w)f(w)\,dw}\cdot\frac{\int_{-\infty}^\infty\She(t,x\viiva r,w)f(w)\,dw}{\int_{z}^\infty\She(t,x\viiva r,w)f(w)\,dw}
	\cdot\frac{\int_{z}^\infty\She(t,y\viiva r,w)f(w)\,dw}{\int_{-\infty}^\infty\She(t,y\viiva r,w)f(w)\,dw}\\
	&\le \frac{\int_{-\infty}^\infty\She(t,y\viiva r,w)f(w)\,dw}{\int_{-\infty}^\infty\She(t,x\viiva r,w)f(w)\,dw}
	\cdot\frac{\int_{-\infty}^\infty\She(t,x\viiva r,w)f(w)\,dw}{\int_{z}^\infty\She(t,x\viiva r,w)f(w)\,dw}\,.
	\end{align*}
Take  $\mu\in\Udense$ and use $f(w)=e^{\Bus^{\mu}(r,0,r,w)}$ to get
	\begin{align*}
	\frac{\She(t,y\viiva r,z)}{\She(t,x\viiva r,z)}
	&\le \frac{\int_{-\infty}^\infty \She(t,y\viiva r,w)e^{\Bus^{\mu}(r,0,r,w)}\,dw}{\int_{-\infty}^\infty \She(t,x\viiva r,w)e^{\Bus^{\mu}(r,0,r,w)}\,dw}
	\cdot\frac{\int_{-\infty}^\infty\She(t,x\viiva r,w)e^{\Bus^{\mu}(r,0,r,w)}\,dw}{\int_{z}^\infty\She(t,x\viiva r,w)e^{\Bus^{\mu}(r,0,r,w)}\,dw}\\	
	&=e^{\Bus^{\mu}(t,x,t,y)}
	\cdot\frac{\int_{-\infty}^\infty\She(t,x\viiva r,w)e^{\Bus^{\mu}(r,0,r,w)}\,dw}{\int_{z}^\infty\She(t,x\viiva r,w)e^{\Bus^{\mu}(r,0,r,w)}\,dw}\\
	&=e^{\Bus^{\mu}(t,x,t,y)}
	\cdot\Biggl(1-\frac{\int_{-\infty}^{z}\She(t,x\viiva r,w)e^{\Bus^{\mu}(r,0,r,w)}\,dw}{\int_{-\infty}^\infty\She(t,x\viiva r,w)e^{\Bus^{\mu}(r,0,r,w)}\,dw}\,\Biggr)^{\!-1}.
	\end{align*}
For the first equality we applied \eqref{NEBus2} and the cocycle property \eqref{cocycle}. Take 
\[\lambda_2\in(\mu-3\e/4,\mu-\e/4)\cap\Udense\quad\text{and}\quad\delta\in(0,\e/4).\] 
Then for all $r<0$ and $\lambda\le\mu-\e$, $\frac{z}r+\lambda\ge-\delta$ implies $z\le -\lambda_2 r$ 
and 
	\[\frac{\She(t,y\viiva r,z)}{\She(t,x\viiva r,z)}
	\le e^{\Bus^{\mu}(t,x,t,y)}
	\cdot\Biggl(1-\frac{\int_{-\infty}^{\lambda_2\abs{r}}\She(t,x\viiva r,w)e^{\Bus^\mu(r,0,r,w)}\,dw}{\int_{-\infty}^\infty\She(t,x\viiva r,w)e^{\Bus^\mu(r,0,r,w)}\,dw}\,\Biggr)^{\!-1}.\]
Since $\lambda_2<\mu$, Theorem \ref{int-shape} implies that the ratio of integrals converges to $0$ as $r\to-\infty$.
Indeed, take $\epsilon_0=(\mu-\lambda_2)^2/8$. Then there exists and $R_1<0$ depending on $\mu$ and $\lambda_2$, which in turn is determined by $\mu$ and $\e$, such that for $r\le R_1$ the bottom integral is at least $e^{(a_\mu-\epsilon_0)\abs{r}}$ while the top one is at most $e^{(a_\mu-(\mu-\lambda_2)^2/2+\epsilon_0)\abs{r}}$ and the ratio is at most $e^{-(\mu-\lambda_2)^2\abs{r}/4}\le e^{-\e^2\abs{r}/64}$, uniformly in $t,x\in\R^2$ with $\abs{t}+\abs{x}\le C$. Now choose $R\le R_1$ so that $e^{-\e^2\abs{R}/64}\le\e/(1+\e)$ and \eqref{ratio-bnd1} follows.
\end{proof}	

\begin{lemma}\label{int-cont}
Fix $M\in [0,\infty]$. Then with $\bfP$-probability one, for any $\lambda\in\R$ and $\sigg\in\{-,+\}$,
$\int_{x}^M \She(t,y\viiva s,w)e^{\Bus^{\lambda\sig}(s,x,s,w)}\,dw$ and 
$\int_{-M}^x \She(t,y\viiva s,w)e^{\Bus^{\lambda\sig}(s,x,s,w)}\,dw$ 
are jointly  continuous in $(s,x,t,y)$, with $t>s$.
\end{lemma}

\begin{proof}
First, recall that we have $\bfP$-almost surely, $\Bus^{\lambda\sig}(s,x,s,w)\le \Bus^{\mu}(s,x,s,w)$ for all $\sigg\in\{-,+\}$, $\lambda\in\R$, $\mu\in\Udense$ with $\mu>\lambda$, and all $x\le w$ and $s$ in $\R$. Similarly,
we have $\bfP$-almost surely, $\Bus^{\lambda\sig}(s,x,s,w)\le \Bus^{\kappa}(s,x,s,w)$ for all $\sigg\in\{-,+\}$, $\lambda\in\R$, $\kappa\in\Udense$ with $\kappa<\lambda$, and all $x\ge w$ and $s$ in $\R$. 
Thus, the claim follows from the dominated convergence theorem if we show that for each fixed $\lambda\in\R$ and each $C>0$ and $\e>0$,
	\begin{align}\label{dom-conv}
	\int_{-\infty}^\infty\E\Bigl[\sup_{\substack{s,t,y\in[-C,C]\\ t-s>\e}}\She(t,y\viiva s,w)\Bigr]\cdot\bfE\Bigl[\sup_{s,x\in[-C,C]}e^{\Bus^\lambda(s,x,s,w)}\Bigr]\,dw<\infty.
	\end{align}
By Corollary 3.10 in \cite{Alb-etal-22-spde-}, the first expectation is bounded by $C'e^{-cw^2}$ for some strictly positive finite constants $c$ and $C'$.
 	For the other expectation, write
		\begin{align*}
		\bfE\Bigl[\sup_{s,x\in[-C,C]}e^{\Bus^\lambda(s,x,s,w)}\Bigr]
		&=\bfE\Bigl[\sup_{\substack{1\le s\le 2C+1\\ \abs{x}\le C}}e^{\Bus^\lambda(s,x,s,w)}\Bigr]
		\le\bfE\Bigl[\sup_{\substack{1\le s\le 2C+1\\ \abs{x}\le C}}e^{-\Bus^\lambda(0,0,s,x)}
		\sup_{1\le s\le 2C+1}e^{\Bus^\lambda(0,0,s,w)}\Bigr]\\
		&\le\bfE\Bigl[\sup_{\substack{1\le s\le 2C+1\\ \abs{x}\le C}}e^{-2\Bus^\lambda(0,0,s,x)}\Bigr]^{1/2}
		\E\Bigl[\sup_{1\le s\le 2C+1}e^{2\Bus^\lambda(0,0,s,w)}\Bigr]^{1/2}.
		\end{align*}
	By a minor change to \eqref{ebinv-comp}, the first expectation on the right-hand side is bounded by a constant.
	Similarly, we get that the last expectation is bounded by $C''e^{c'\abs{w}}$. Here are the details:
	\begin{align*}
	&\bfE\Bigl[\sup_{1\le s\le K}e^{2\Bus^\lambda(0,0,s,w)}\Bigr]
	\le\bfE\Bigl[\Bigl(\;\int_{-\infty}^\infty e^{\Bus^\lambda(0,0,0,u)}\sup_{1\le s\le K}\She(s,w\viiva 0,u)\,du\Bigr)^2\Bigr]\\
	&\qquad=    \int_{\R^2}  
	\bfE\bigl[e^{\Bus^\lambda(0,0,0,u)}e^{\Bus^\lambda(0,0,0,v)}\bigr]\cdot
	\E\Bigl[\sup_{1\le s\le K}\She(s,w\viiva 0,u)\cdot\sup_{1\le s\le K}\She(s,w\viiva 0,v)\Bigr]\,du\,dv\\
	&\qquad\le \int_{\R^2}  
	\bfE\bigl[e^{2\Bus^\lambda(0,0,0,u)}\bigr]^{1/2}\cdot
	\bfE\bigl[e^{2\Bus^\lambda(0,0,0,v)}\bigr]^{1/2}\\
	&\qquad\qquad\qquad\qquad\times
	\E\Bigl[\Bigl(\sup_{1\le s\le K}\She(s,w\viiva 0,u)\Bigr)^2\Bigr]^{1/2}\cdot\E\Bigl[\Bigl(\sup_{1\le s\le K}\She(s,w\viiva 0,v)^2\Bigr]^{1/2}\,du\,dv\\
	&\qquad\le C'   \int_{\R^2} 
	e^{c'\abs{u}}e^{c'\abs{v}}e^{-c''(u-w)^2}e^{-c''(v-w)^2}\,du\,dv\\
	&\qquad= C'\Bigl(\;\int_{-\infty}^\infty e^{c'\abs{u}}e^{-c''(u-w)^2}\,du\Bigr)^2 \le C''e^{c'\abs{w}}.
	\end{align*}
In the second inequality, we used Corollary 3.10 in \cite{Alb-etal-22-spde-}  and the fact that $\Bus^\lambda(0,0,0,x)$ is a Gaussian random variable with mean $\lambda x$ and variance $\abs{x}$.
Thus \eqref{dom-conv} holds and the lemma is proved.
\end{proof}

\begin{lemma}\label{Bus-bnds2}
The following holds with $\bfP$-probability one. 
For all $\kappa<\mu$ in $\Udense$,  $\e>0$, $C>0$, and $\tau>0$, there exist {\rm(}possibly random{\rm)}
$R<-C$ and 
deterministic
$\delta>0$ such that for all $\lambda\in[\kappa+\e,\mu-\e]$ these statements hold. 
\begin{enumerate}[label={\rm(\alph*)}, ref={\rm\alph*}]
\item For all $s,x\in[-C,C]$, for all $t,y\in\R$ with $t>s$, for all $r\le R$, and for all $z$ such that 
$\abs{z/r+\lambda}<\delta$,
	\begin{align}\label{ratio-bnd3}
	\frac{\She(t,y\viiva r,z)}{\She(s,x\viiva r,z)}
	\le (1+\e)\Bigl(\;\int_x^\infty\She(t,y\viiva s,w)e^{\Bus^{\mu}(s,x,s,w)}\,dw+\int_{-\infty}^x\She(t,y\viiva s,w)e^{\Bus^{\kappa}(s,x,s,w)}\,dw\Bigr).
	\end{align}
\item For all $s,x,t,y\in[-C,C]$ with $t-s\ge\tau$, for all $r\le R$, and for all $z$ such that 
$\abs{z/r+\lambda}<\delta$,
	\begin{align}\label{ratio-bnd4}
	\frac{\She(t,y\viiva r,z)}{\She(s,x\viiva r,z)}
	\ge (1+\e)^{-2}\Bigl(\;\int_x^\infty\She(t,y\viiva s,w)e^{\Bus^{\kappa}(s,x,s,w)}\,dw+\int_{-\infty}^x\She(t,y\viiva s,w)e^{\Bus^{\mu}(s,x,s,w)}\,dw\Bigr).
	\end{align}
\end{enumerate}
\end{lemma}

\begin{proof}
Apply \eqref{ratio-bnd1} and \eqref{ratio-bnd2} with $s$ in place of $t$ and $w$ in place of $y$, multiply all sides by $\She(t,y\viiva s,w)$ and integrate 
the first inequality over $w\in(x,\infty)$ and the second one over $w\in(-\infty,x)$, then add the two inequalities to get \eqref{ratio-bnd3}.

For the other bound take an integer $M>C$. 
Apply \eqref{ratio-bnd1} and \eqref{ratio-bnd2} to get that there exist an $R<-M$ and $\delta>0$ such that for all $s\in[-C,C]$ and $w\in[-M,M]$, for all $r\le R$, and all $z$ such that $\abs{z/r+\lambda}<\delta$, 
	\[\frac{\She(s,w\viiva r,z)}{\She(s,x\viiva r,z)}
	\ge (1+\e)^{-1}e^{\Bus^{\mu}(s,x,s,w)}\quad\text{for all $x\in(w,\infty)$}\]
and
	\[\frac{\She(s,w\viiva r,z)}{\She(s,x\viiva r,z)}\ge (1+\e)^{-1}e^{\Bus^{\kappa}(s,x,s,w)}\quad\text{for all $x\in(-\infty,w)$}.\]
Multiply all sides by $\She(t,y\viiva s,w)$ and integrate over $w$ to get	
	\[\frac{\int_{-M}^x\She(t,y\viiva s,w)\She(s,w\viiva r,z)\,dw}{\She(s,x\viiva r,z)}
	\ge (1+\e)^{-1}\int_{-M}^x \She(t,y\viiva s,w)e^{\Bus^{\mu}(s,x,s,w)}\,dw\]
and
	\[\frac{\int_{x}^M\She(t,y\viiva s,w)\She(s,w\viiva r,z)\,dw}{\She(s,x\viiva r,z)}\ge (1+\e)^{-1}\int_x^M \She(t,y\viiva s,w)e^{\Bus^{\kappa}(s,x,s,w)}\,dw.\]
	Add the two and enlarge the integrals on the left-hand side to get
	\begin{align}\label{aux-bnd-Bus}
	\frac{\She(t,y\viiva r,z)}{\She(s,x\viiva r,z)}\ge (1+\e)^{-1}\Bigl(
	\int_x^M \She(t,y\viiva s,w)e^{\Bus^{\kappa}(s,x,s,w)}\,dw
	+\int_{-M}^x \She(t,y\viiva s,w)e^{\Bus^{\mu}(s,x,s,w)}\,dw\Bigr).
	\end{align}
Next, note that for $t>s$, $\int_{-\infty}^\infty\She(t,y\viiva s,w)e^{\Bus^{\kappa}(s,x,s,w)}\,dw= e^{\Bus^{\kappa}(s,x,t,y)}<\infty$, with a similar bound for $\Bus^{\mu}$. Therefore, 
	\begin{align}\label{integrals}
	\frac{\int_x^M\She(t,y\viiva s,w)e^{\Bus^{\kappa}(s,x,s,w)}\,dw}{\int_x^\infty\She(t,y\viiva s,w)e^{\Bus^{\kappa}(s,x,s,w)}\,dw}\quad\text{and}\quad
	\frac{\int_{-M}^x\She(t,y\viiva s,w)e^{\Bus^{\mu}(s,x,s,w)}\,dw}{\int_x^\infty\She(t,y\viiva s,w)e^{\Bus^{\mu}(s,x,s,w)}\,dw}
	\end{align}
both increase to $1$ as $M\nearrow\infty$.	
By Lemma \ref{int-cont}, these are continuous functions. Since they are monotonically converging (pointwise) to a 
continuous function, 
Dini's theorem \cite[Theorem 7.13]{Rud-76} implies the convergence is uniform on the compact set $K=\{(s,x,t,y)\in[-C,C]:t-s\ge\tau\}$.
Thus, for $M$ large enough and all $(s,x,t,y)\in K$,
	\[\int_x^M \She(t,y\viiva s,w)e^{\Bus^{\kappa}(s,x,s,w)}\,dw\ge 
	(1+\e)^{-1}\int_x^\infty \She(t,y\viiva s,w)e^{\Bus^{\kappa}(s,x,s,w)}\,dw\]
and 
	\[\int_{-M}^{x} \She(t,y\viiva s,w)e^{\Bus^{\mu}(s,x,s,w)}\,dw\ge
	(1+\e)^{-1}\int_{-\infty}^x \She(t,y\viiva s,w)e^{\Bus^{\mu}(s,x,s,w)}\,dw.\]
\eqref{ratio-bnd3} follows from this and \eqref{aux-bnd-Bus}.
\end{proof}


\begin{proof}[Proof of Proposition \ref{pr:buslim}]
{\it Part \eqref{buslimhat1}.}
We work on the full $\bfP$-probability event that is the intersection of the one on which Theorem 2.6 
in \cite{Alb-etal-22-spde-}, \eqref{pr:IC:1}, \eqref{pr:IC:ICMsh},
Theorem \ref{thm:Z-cont}, Proposition \ref{Bproc2}, Lemma \ref{b-=b+}, 
and Lemmas \ref{int-shape2}, \ref{Bus-bnds},
and \ref{Bus-bnds2} hold, with the event on which Lemma \ref{int-cont} holds for all $M\in\Z_+$. 

By monotone convergence, 
\[ \int_x^\infty\She(t,y\viiva s,w)e^{\Bus^{\kappa}(s,x,s,w)}\,dw \ \to\  
\int_x^\infty\She(t,y\viiva s,w)e^{\Bus^{\lambda-}(s,x,s,w)}\,dw  \quad\text{as } \ \Udense\ni\kappa \nearrow\lambda \]
and 
\[  \int_{-\infty}^x\She(t,y\viiva s,w)e^{\Bus^{\mu}(s,x,s,w)}\,dw\ \to \  
\int_{-\infty}^x\She(t,y\viiva s,w)e^{\Bus^{\lambda+}(s,x,s,w)}\,dw \quad\text{as } \ \Udense\ni\mu\searrow\lambda. 
\] 
%
By Lemma \ref{int-cont}, these are continuous functions that are monotonically converging (pointwise) to a continuous function.
Thus, Dini's theorem \cite[Theorem 7.13]{Rud-76} implies the convergence is uniform on the compact set $K=\{(s,x,t,y)\in[-C,C]:t-s\ge\tau\}$.
Consequently,  given $\e>0$ and $\lambda\in\R$, there exists an $\e_0>0$ such that if  $\kappa<\mu$ in $\Udense$ are such that $\mu-\kappa<\e_0$ and $\kappa<\lambda<\mu$, then for all $(s,x,t,y)\in K$
	\[\int_x^\infty\She(t,y\viiva s,w)e^{\Bus^{\kappa}(s,x,s,w)}\,dw\ge(1+\e)^{-1}\int_x^\infty\She(t,y\viiva s,w)e^{\Bus^{\lambda-}(s,x,s,w)}\,dw\]
and	
	\[\int_{-\infty}^x\She(t,y\viiva s,w)e^{\Bus^{\mu}(s,x,s,w)}\,dw\ge(1+\e)^{-1}\int_{-\infty}^x\She(t,y\viiva s,w)e^{\Bus^{\lambda+}(s,x,s,w)}\,dw.\]
Take $\e_1\in(0,\e\wedge(\mu-\lambda)\wedge(\lambda-\kappa))$ and take $R$ large enough (and negative), depending on $C$, $\tau$, $\delta$, $\kappa$, $\mu$, and $\e_1$ (and hence also depending on $\lambda$), 
so that  \eqref{ratio-bnd4} holds for any $\lambda'\in[\kappa+\e_1,\mu-\e_1]$ and $z$ such that $\abs{z/r+\lambda'}<\delta$, 
with $\e_1$ in place of $\e$. 
The second inequality of part \eqref{buslimhat1} now follows. 
The first inequality comes similarly. The limit \eqref{eq:buslimhat}, locally uniformly in $(s,x,t,y)\in\R^4$ with $t>s$ comes immediately. To extend the statement to all of $\R^4$ take any $C>0$ and take $s'<-C-1$. Then we have that 
	\[\lim_{\substack{r\to-\infty\\ z/r\to-\lambda}}\frac{\She(t,y\viiva r,z)}{\She(s',0\viiva r,z)}=e^{\Bus^\lambda(s',0,t,y)}\quad\text{and}\quad \lim_{\substack{r\to-\infty\\ z/r\to-\lambda}}\frac{\She(s,x\viiva r,z)}{\She(s',0\viiva r,z)}=e^{\Bus^\lambda(s',0,s,x)},\]
	uniformly in $(s,x,t,y)\in[-C,C]^4$. Take a ratio and use the cocycle property of $\Bus^\lambda$ to conclude.
\smallskip

{\it Part \eqref{buslimhat2}.}  
Consider first the case $\lambda>0$.
If $f\in\initF_\lambda$, then for any $\delta>0$ small enough there exists a $\mu>-\lambda$ such that for 
any $\epsilon>0$  there exists an $R_1<0$ such that for $r\le R_1$ we have for all $t>r$ and all $y$
\begin{align*}
\int_{z\ge0: \tspa\abs{\frac{z}{r}+\lambda}\ge\delta}\She(t,y\viiva r,z)\,f(r,z)\,dz 
&\le e^{\epsilon\abs{r}}\int_{z\ge0: \tspa\abs{\frac{z}{r}+\lambda}\ge\delta}\She(t,y\viiva r,z)\,e^{\lambda z+\epsilon z}\,dz,\\
\int_{-\infty}^0\She(t,y\viiva r,z)\,f(r,z)\,dz 
&\le e^{\epsilon\abs{r}}\int_{-\infty}^0\She(t,y\viiva r,z)\,e^{\mu z+\epsilon\abs{z}}\,dz,\quad\text{and}\\
\int_{\abs{\frac{z}{r}+\lambda}\le\delta}\She(t,y\viiva r,z)\,e^{\lambda z}\,dz 
&\le e^{\epsilon\abs{r}}\int_{\abs{\frac{z}{r}+\lambda}\le\delta}\She(t,y\viiva r,z)\,f(r,z)\,dz.
\end{align*}
Note that reducing $\mu$ makes the right-hand side in the second inequality larger. Therefore, one can assume $\mu\in(-\lambda,0]$.
Lemma \ref{int-shape2} implies that for any $\e>0$,  for any $\delta>0$ small enough, for $\epsilon\in(0,\delta)$ such that
$\epsilon<\min\bigl(\frac{\delta^2}{2(\lambda+\delta+4)},\frac{\lambda^2-\mu^2}{9-2\mu}\bigr)$,  for any $C>0$, 
there exists an $R_2\le R_1$ such that for any $t,y\in[-C,C]$, and any $r\le R_2$, 
\begin{align*}
&e^{\epsilon\abs{r}}\Bigl(\;\int_{(\lambda+\delta)\abs{r}}^\infty\She(t,y\viiva r,z)\,e^{(\lambda+\epsilon)z}\,dz
+ \int_0^{(\lambda-\delta)\abs{r}}\She(t,y\viiva r,z)\,e^{(\lambda+\epsilon)z}\,dz
+\int_{-\infty}^0\She(t,y\viiva r,z)\,e^{(\mu-\epsilon) z}\,dz\Bigr)\\
&\qquad\le e^{-\abs{r}/24}e^{2\epsilon\abs{r}}\Bigl(e^{(\lambda+\delta)(\lambda+2\epsilon-\delta)\abs{r}/2}+e^{(\lambda-\delta)(\lambda+2\epsilon+\delta)\abs{r}/2}
+e^{(\mu-\epsilon)^2\abs{r}/2}\Bigr)\\
&\qquad= e^{-\abs{r}/24}e^{-2\epsilon\abs{r}}e^{\lambda^2\abs{r}/2}\Bigl(e^{[-\delta^2+2\epsilon(\lambda+\delta+4)]\abs{r}/2}+e^{[-\delta^2+2\epsilon(\lambda-\delta+4)]\abs{r}/2}
+e^{[-(\lambda^2-\mu^2)-2\mu\epsilon+\epsilon^2+8\epsilon]\abs{r}/2}\Bigr)\\
&\qquad\le \e e^{-\abs{r}/24}e^{-2\epsilon\abs{r}}e^{\lambda^2\abs{r}/2}\\
&\qquad\le \e e^{-\epsilon\abs{r}}\int_{\abs{\frac{z}{r}+\lambda}\le\delta}\She(t,y\viiva r,z)\,e^{\lambda z}\,dz.
\end{align*}
For the penultimate inequality, after fixing $\epsilon$, $\abs r$ is increased further if necessary so that the sum of three exponentials is $\le\e$. The last inequality is another instance of Lemma \ref{int-shape2}.  

Together, the above bounds give 
\be\label{Zf-bnd1}
		\begin{aligned}
	\int_\R\She(t,y\viiva r,z)\,f(r,z)\,dz 
	&=\int_{\abs{\frac{z}{r}+\lambda}>\delta}\She(t,y\viiva r,z)\,f(r,z)\,dz +\int_{\abs{\frac{z}{r}+\lambda}\le\delta}\She(t,y\viiva r,z)\,f(r,z)\,dz\\
	&\le (1+\e) \int_{\abs{\frac{z}{r}+\lambda}\le\delta}\She(t,y\viiva r,z)\,f(r,z)\,dz.	\end{aligned}\ee
The already proved part \eqref{buslimhat1} says that for $\delta>0$ small enough there exists an $R\le R_2$ such that
for all $r\le R$ and $s,x,t,y\in[-C,C]$ with $t-s\ge\tau>0$
\begin{align*}
&\int_{\abs{\frac{z}{r}+\lambda}\le\delta}\She(t,y\viiva r,z)\,f(r,z)\,dz\\
&\quad\le \Bigl( (1+\e)^2\int_x^\infty\She(t,y\viiva s,w)e^{\Bus^{\lambda+}(s,x,s,w)}\,dw+(1+\e)^2\int_{-\infty}^x\She(t,y\viiva s,w)e^{\Bus^{\lambda-}(s,x,s,w)}\,dw\Bigr)\\
&\qquad\qquad\qquad\qquad\qquad\qquad\times\int_{\abs{\frac{z}{r}+\lambda}\le\delta}\She(s,x\viiva r,z)\,f(r,z)\,dz\\
&\le \Bigl( (1+\e)^2\int_x^\infty\She(t,y\viiva s,w)e^{\Bus^{\lambda+}(s,x,s,w)}\,dw+(1+\e)^2\int_{-\infty}^x\She(t,y\viiva s,w)e^{\Bus^{\lambda-}(s,x,s,w)}\,dw\Bigr)\\
&\qquad\qquad\qquad\qquad\qquad\qquad\times\int_\R\She(s,x\viiva r,z)\,f(r,z)\,dz.
\end{align*}
The upper bound of part \eqref{buslimhat2} follows. The lower  bound is similar, again using \eqref{Zf-bnd1}. The case $\lambda<0$ comes the same way. 

It remains to prove the case  $\lambda=0$. 
If $f\in\initF_0$, then for any $\delta>0$ small enough, for 
any $\epsilon>0$,  there exists an $R_1<0$ such that for $r\le R_1$ we have for all $t>r$ and all $y$
\[\int_{\abs{z}\ge\delta\abs{r}}\She(t,y\viiva r,z)\,f(r,z)\,dz 
\le e^{\epsilon\abs{r}}\int_{\abs{z}\ge\delta\abs{r}}^\infty\She(t,y\viiva r,z)\,e^{\epsilon\abs{z}}\,dz.\]
Lemma \ref{int-shape2} implies that for any $\delta>0$ small enough, for $\epsilon\in(0,\delta)$, 
for any $C>0$, 
there exists an $R_2\le R_1$ such that for any $t,y\in[-C,C]$, and any $r\le R_2$, 
\[\int_{\abs{z}\ge\delta\abs{r}}\She(t,y\viiva r,z)\,e^{\epsilon\abs{z}}\,dz\le 2e^{(\epsilon+\epsilon\delta-\frac{\delta^2}2-\frac1{24})\abs{r}}.\]
By Theorem \ref{thm:Z-cont} and condition \eqref{W''} there exists an $R_3<\min(R_2,-C)$ such that for all $s,x\in[-C,C]$ and $r\le R_3$, 
\begin{align*}
\int_{\abs{z}\le\delta\abs{r}}\She(s,x\viiva r,z)f(r,z)\,dz
&\ge \frac1{\sqrt{2\pi(s-r)}}e^{-\epsilon\abs{r}}e^{-\frac{s-r}{24}}\int_{\abs{z}\le c}e^{-\frac{(x-z)^2}{2(s-r)}}f(r,z)\,dz\\
&\ge \tfrac12\,e^{-\frac{C}{24}}e^{-2\epsilon\abs{r}}e^{-\frac{\abs{r}}{24}}\int_{\abs{z}\le c}f(r,z)\,dz\\
&\ge \tfrac12\,e^{-\frac{C}{24}}e^{-3\epsilon\abs{r}}e^{-\frac{\abs{r}}{24}}.
\end{align*}
This and the above two bounds give that for any $\e>0$, taking $\epsilon<\frac{\delta^2}{2(5+\delta)}$ we have for any $s,x,t,y\in[-C,C]$ and any $r$ large enough negative,
\begin{align}
\int_\R\She(t,y\viiva r,z)\,f(r,z)\,dz 
&=\int_{\abs{z}>\delta\abs{r}}\She(t,y\viiva r,z)\,f(r,z)\,dz+\int_{\abs{z}\le\delta\abs{r}}\She(t,y\viiva r,z)\,f(r,z)\,dz\notag\\
&\le 2e^{(2\epsilon+\epsilon\delta-\frac{\delta^2}2-\frac1{24})\abs{r}}
+\int_{\abs{z}\le\delta\abs{r}}\She(t,y\viiva r,z)\,f(r,z)\,dz\notag\\
&= e^{-3\epsilon\abs{r}}e^{-\frac{\abs{r}}{24}}\cdot e^{[-\frac{\delta^2}2+\epsilon(5+\delta)]\abs{r}}
+\int_{\abs{z}\le\delta\abs{r}}\She(t,y\viiva r,z)\,f(r,z)\,dz\notag\\
&\le \tfrac12\,e^{-\frac{C}{24}}e^{-3\epsilon\abs{r}}e^{-\frac{\abs{r}}{24}}\e+\int_{\abs{z}\le\delta\abs{r}}\She(t,y\viiva r,z)\,f(r,z)\,dz\notag\\
&\le (1+\e)\int_{\abs{z}\le\delta\abs{r}}\She(t,y\viiva r,z)\,f(r,z)\,dz.	\label{Zf-bnd2}
\end{align}
The claims of the theorem now follow as for the case $\lambda\ne0$.
\end{proof}

We close this section with a line-to-point version of the bounds \eqref{ratio-bnd1} and \eqref{ratio-bnd2}. The difference between these bounds and those of  Proposition \ref{pr:buslim}\eqref{buslimhat2} is that in   \eqref{l2p-bnd1} and \eqref{l2p-bnd2} below the terminal times are equal. 
 
\begin{lemma}\label{l2p-bnds}
The following holds $\bfP$-almost surely: for any $\kappa<\mu$ in $\Udense$, $\e>0$, $\lambda\in[\kappa+\e,\mu-\e]$, $f\in\initF_\lambda$, and $C>0$, 
there exists an $R<-C$ such that for all $t,x\in[-C,C]$ and all $r\le R$, 
	\begin{align}\label{l2p-bnd1}
	\frac{\int_\R\She(t,y\viiva r,z)\,f(r,z)\,dz}{\int_\R\She(t,x\viiva r,z)\,f(r,z)\,dz}
	\le (1+\e)^2e^{\Bus^{\mu}(t,x,t,y)}\quad\text{for all $y\in(x,\infty)$}
	\end{align}
and
	\begin{align}\label{l2p-bnd2}
	\frac{\int_\R\She(t,y\viiva r,z)\,f(r,z)\,dz}{\int_\R\She(t,x\viiva r,z)\,f(r,z)\,dz}\le (1+\e)^2e^{\Bus^{\kappa}(t,x,t,y)}\quad\text{for all $y\in(-\infty,x)$}.
	\end{align}
\end{lemma}

\begin{proof}
Apply \eqref{Zf-bnd1}, then \eqref{Zf-bnd2}, and then Lemma \ref{Bus-bnds}.
\end{proof}
%

\section{Ergodicity and symmetries of the Busemann process}\label{sec:erg}


The almost sure limit \eqref{eq:buslimhat} of the previous section implies that the Busemann process can be defined  on the original probability space $(\Omega,\sF,\P)$ of the white noise, as stated in the next 
corollary. 

\begin{corollary}\label{cor:what->w}
The process $\bigl\{\Bus^{\lambda\sig}(s,x,t,y):s,x,t,y,\lambda\in\R,\sigg\in\{-,+\}\bigr\}$ is a measurable function of the Green's function  $\She(\aabullet,\aabullet\tsp\viiva\aabullet,\aabullet)$ and hence is $\fil$-measurable.
\end{corollary}

\begin{proof}
For  $\lambda\in\Udense$, 
\eqref{eq:buslimhat} implies the claim for the process
$\{\Bus^\lambda(s,x,t,y):s,x,t,y\in\R\}$
and \eqref{Busconthat} extends this to the whole process
$\bigl\{\Bus^{\lambda\sig}(s,x,t,y):s,x,t,y, \lambda\in\R,\sigg\in\{-,+\}\bigr\}$.
\end{proof}

\begin{proof}[Proofs of Theorems \ref{main:Bus}, \ref{thm:b-shape}, \ref{bus:exp'}, \ref{thm:buslim}, and \ref{l2p-buslim}]
By Corollary \ref{cor:what->w}, the properties of $\Bus^{\aabullet}$ under $\bfP$ in Propositions \ref{Bproc1} and \ref{Bproc2} are now properties under $\P$. All the claims of Theorem \ref{main:Bus} follow. Similarly, Theorem \ref{thm:b-shape} follows from Proposition \ref{pr:b-shape} and Theorems \ref{thm:buslim} and \ref{l2p-buslim} follow from Proposition \ref{pr:buslim}.
Also, the claim in Theorem \ref{bus:exp} holds $\P$-almost surely and then
Theorem \ref{bus:exp'} follows from \eqref{bus:exp3} since $\Glob^{\lambda\sig}(t,x)=e^{\Bus^{\lambda\sig}(0,0,t,0)}e^{\Bus^{\lambda\sig}(t,0,t,x)}$. 
\end{proof}

\begin{proof}[Proof of Theorem \ref{thm:bcov}]
All the properties are direct consequences of the limits \eqref{eq:buslim} and \eqref{Buscont} 
and the properties of $\She$ and $\Shenorm$ in Proposition 2.3 of \cite{Alb-etal-22-spde-}. We spell out the proof of the shear-covariance property. 
For this, first use the shear  property of $\rnShe$ to get that for any $r,a\in\R$ there exists an event $\Omega'_{r,c}$ with $\P(\Omega'_{r,c})=1$ such that for all $\w\in\Omega'_{r,c}$, for all $R\in\R$, $t\in(R,\infty)$, and all $r,y,z\in\R$, 
\[e^{-c(y-z)+\frac{c^2}2(t-R)}\She(t,y-c(t-r)\viiva R,z-c(R-r))\circ\sheard{r}{c}=\She(t,y\viiva R,z).\]
Next, for $\lambda\in\Udense$, take $R\to-\infty$, $z/R\to-\lambda$ and use  $\P(\lambda\notin\Bruno)=\P(\lambda+c\notin\Bruno)=1$ (see Proposition \ref{Bproc2}\eqref{Bproc.-=+hat}) and \eqref{eq:buslim} to deduce the 
existence of  
an event $\Omega''_{r,c}\subset\Omega'_{r,c}$ such that $\P(\Omega''_{r,c})=1$ and for all $\w\in\Omega''_{r,c}$, all $s,x,t,y\in\R$, and all $\lambda\in\Udense$,
\begin{align*}
e^{\Bus^{\lambda+c}(s,x-c(s-r),t,y-c(t-r);\sheard{r}{c}\w)}
&=\lim_{R\to-\infty}\frac{\She(t,y-c(t-r)\viiva R,-\lambda R-c(R-r))\circ\sheard{r}{c}}{\She(s,x-c(s-r)\viiva R,-\lambda R-c(R-r))\circ\sheard{r}{c}}\\
&=\lim_{R\to-\infty}\frac{e^{c(y+\lambda R)-\frac{c^2}2(t-R)}}{e^{c(x+\lambda R)-\frac{c^2}2(s-R)}}\cdot
\frac{\She(t,y\viiva R,-\lambda R)}{\She(s,x\viiva R,-\lambda R)}\\
&=e^{\Bus^\lambda(s,x,t,y;\w)+c(y-x)-\frac{c^2}2(t-s)}.
\end{align*}
The shear-covariance claim follows  from this and \eqref{Buscont}.  
\end{proof}

\begin{proof}[Proof of Theorem \ref{b cont}]
The  limit in \eqref{bus:exp3'} implies that for any $r$, $f(z)=e^{\Bus^{\lambda\sig}(r,0,r,z)}$ does not grow in $z$ faster than exponentially. Then \eqref{NEBus2} and \cite[Theorem 2.6]{Alb-etal-22-spde-} imply that $e^{\Bus^{\lambda\sig}(r,0,t,y)}$ is locally H\"older-continuous with the claimed exponents, in $t\in(r,\infty)$ and $y\in\R$. 
Consequently, $e^{\Bus^{\lambda\sig}(s,x,t,y)}$ and therefore also $\Bus^{\lambda\sig}(s,x,t,y)$  are 
locally H\"older-continuous in all four variables. 
\end{proof}

In the sequel   we no longer need to refer to the extended space  $(\Omhat,\widehat\sF,\bfP)$.  We also note that the claims in 
Theorem \ref{int-shape}  and Lemmas \ref{Bus-bnds}, \ref{int-cont}, \ref{Bus-bnds2}, and \ref{l2p-bnds} all hold $\P$-almost surely. 

We close this section with an immediate consequence of Corollary \ref{cor:what->w}, Theorem \ref{thm:bcov}\eqref{thm:bcov.i},
and the ergodicity of $(\Omega,\fil,\P)$ under non-trivial shifts.

\begin{theorem}\label{thm:tot-erg}
The process $\bigl\{\Bus^{\lambda\sig}(s,x,t,y):s,x,t,y,\lambda\in\R,\sigg\in\{-,+\}\bigr\}$ is totally ergodic under non-trivial shifts. 
Precisely, for any $(r,z)\ne(0,0)$ and any Borel set $A\subset \cC(\R^4,\R)^{\R\times\{-,+\}}$ that is invariant under the simultaneous shift of all spatial coordinates by $z$ and all temporal coordinates by $r$, 
	\[\P\Bigl(\bigl\{\Bus^{\lambda\sig}(s,x,t,y):s,x,t,y,\lambda\in\R,\sigg\in\{-,+\}\bigr\}\in  A\Bigr)\in\{0,1\}.\] 
\end{theorem}

In particular, for any given $\lambda$, $\bigl\{\Bus^{\lambda}(s,x,t,y):s,x,t,y\bigr\}$ is ergodic under each  temporal shift. (Recall from Theorem \ref{main:Bus}\eqref{Bproc.-=+} that when $\lambda$ is fixed, the $\lambda\pm$ distinction disappears $\P$-almost surely.) 
This says that the known ratio-stationary solutions of \eqref{SHE} (see \cite{Fun-Qua-15,Ber-Gia-97,Gub-Per-17}) 
are ergodic under the time shift.
More details follow  in Section \ref{sec:1F1S}.

%

\section{Semi-infinite continuum polymer} 
\label{sec:poly} 



Our proof of Theorem \ref{thm:unique} relies on the analysis of a family of semi-infinite continuum directed polymer measures, which we discuss in this section. 

For each $(t,y)\in\R^2$, it is shown in Theorem 2.14 of \cite{Alb-etal-22-spde-}
that the polymer measures $\{\Poly_{(t,y),(s,x)}:t>s,y\in\R\}$ from Section \ref{sec:CDRP}
are consistent in the sense of Gibbs conditioning, or, equivalently, that they satisfy the domain Markov property. 
It is then natural to consider the question of existence of infinite length polymers, i.e., solutions to the Dobrushin–Lanford–Ruelle (DLR) equations.   A discussion of these equations in the context of planar lattice polymers appears in Sections 2.4 and 2.5 of 
\cite{Jan-Ras-18-arxiv}.  
Formulas \eqref{pi-rf} and \eqref{pi-rx} suggest   
that this question is tightly bound to 
the Busemann limits \eqref{eq:buslim2} and \eqref{eq:buslim} and that the limiting objects are tightly connected to the Busemann process $\Bus^{\aabullet\pm}$.  We thus now describe the limiting polymer measures.

For $t,y,\lambda\in\R$ and $\sigg\in\{-,+\}$, 
let $\Poly_{(t,y)}^{\lambda\sig}$ denote the distribution of the real-valued Markov process
$\{X_s:s\in (-\infty, t]\}$ that evolves backward in time from the initial point $X_t=y$ and 
whose move from $(s,x)$ to $(r,z)$ 
obeys the transition probability density  
	\begin{align}\label{pi}
	\pi^{\lambda\sig}(r,z\viiva s,x)=\She(s,x\viiva r,z)e^{\Bus^{\lambda\sig}(s,x,r,z)},\quad s>r\text{ and }x,z\in\R.
	\end{align}

These are indeed transition probability densities because \eqref{eq:CK} and the additivity \eqref{cocycle} 
 imply that they satisfy the Chapman-Kolmogorov equation and \eqref{NEBus2} implies that
	\[\int_\R \pi^{\lambda\sig}(r,z\viiva s,x)\,dz=\int_\R \She(s,x\viiva r,z)e^{\Bus^{\lambda\sig}(s,x,r,z)}\,dz
	=e^{\Bus^{\lambda\sig}(s,x,s,x)}=1.\]
One can immediately recognize in the expression above that this family of semi-infinite length continuum polymer measures arise by using the Busemann process to define a family of Doob transforms of the finite length measures.

Our first result in this section collects some basic properties of these semi-infinite polymers. 

\begin{theorem}\label{lm:sem-pol}
The following statements hold $\P$-almost surely.  
\begin{enumerate}  [label={\rm(\alph*)}, ref={\rm\alph*}]   \itemsep=3pt  
\item\label{sem-pol.a} \textup{(Existence)} For each $\lambda\in\R$, $\sigg\in\{-,+\}$, and initial point  $(t,y)\in\R^2$, 
the expression in \eqref{pi} defines a measure $\Poly_{(t,y)}^{\lambda\sig}$ on $\sC((-\infty,t],\R)$ with its Borel $\sigma$-algebra, under which the path is almost surely locally $\alpha$-H\"older-continuous for any $\alpha\in(0,1/2)$. 
 
\item \label{th:weak-cv} \textup{(Thermodynamic limits)}
For all $\lambda\not\in\Bruno$, $(t,y)\in\R^2$, terminal condition $f\in\initF_\lambda$, and  time-space paths  $\{(r,z_r) : r \tsp\in\tsp(-\infty,t]\}$ such that $z_r/r\to-\lambda$ as $r\to-\infty$, for any $s<t$, both $\Poly_{(t,y),(r,z_r)}(X_{s:t}\in\aabullet)$ and $\Poly_{(t,y),(r,f_r)}(X_{s:t}\in\aabullet)$ converge in the total variation distance to $\Poly_{(t,y)}^{\lambda}(X_{s:t}\in\aabullet)$ as $r\to-\infty$.

\item \label{Q-LLN}
\textup{(LLN)} For all $t,y,\lambda\in\R$ and $\sigg\in\{-,+\}$, 
\begin{align}\label{LLN}
\Poly_{(t,y)}^{\lambda\sig}\Bigl\{\lim_{r\to-\infty}\frac{X_r}r=-\lambda\Bigr\}=1.
\end{align}

\item\label{sem-pol.b}  
\textup{(Continuity)} Let  $t,y,\lambda\in\R$ and $\sigg\in\{-,+\}$. As $x\to y$,   $\Poly_{(t,x)}^{\lambda\sig}$ converges weakly to $\Poly_{(t,y)}^{\lambda\sig}$ as probability measures on $\cC((-\infty,t],\R)$. Furthermore, for any $s<t$, 
$\Poly_{(t,y)}^{\mu\sig}(X_{s:t}\in\aabullet)$ converges in the total variation distance to $\Poly_{(t,y)}^{\lambda-}(X_{s:t}\in\aabullet)$ as $\mu\nearrow\lambda$, and to $\Poly_{(t,y)}^{\lambda+}(X_{s:t}\in\aabullet)$ as $\mu\searrow\lambda$. 
 \end{enumerate} 
\end{theorem}

\begin{remark}
By Theorem \ref{lm:sem-pol}\eqref{Q-LLN}, for any $\lambda\ne\mu$ and $\sigg,\sigg'\in\{-,+\}$, the measures $\Poly_{(t,y)}^{\mu\sig}$ 
and $\Poly_{(t,y)}^{\lambda\sig'}$ are mutually singular.  Therefore their  total variation distance  is one and the total variation convergence of  Theorem \ref{lm:sem-pol}\eqref{sem-pol.b}  cannot hold on the semi-infinite  time interval $(-\infty,t]$. However, the total variation convergence   of  the projections onto bounded  time intervals does imply weak convergence  on the full time interval.
\end{remark}

Combining Theorem \ref{lm:sem-pol}\eqref{th:weak-cv} with Theorem \ref{main:Bus}\eqref{Bproc.-=+} gives the following statement for each fixed $\lambda$. 


\begin{corollary}
Fix $\lambda\in\R$. Then the following holds $\P$-almost surely.
For any $(t,y)\in\R^2$, $f\in\initF_\lambda$, and  time-space paths  $\{(r,z_r)\}_{r\tsp\in\tsp(-\infty,t]}$ such that $z_r/r\to-\lambda$ as $r\to-\infty$, for any $s<t$, both $\Poly_{(t,y),(r,z_r)}(X_{s:t}\in\aabullet)$ and $\Poly_{(t,y),(r,f_r)}(X_{s:t}\in\aabullet)$ converge in the total variation distance to $\Poly_{(t,y)}^{\lambda}(X_{s:t}\in\aabullet)$ as $r\to-\infty$.
\end{corollary}


Recall that in zero temperature or inviscid settings, semi-infinite polymer measures correspond to semi-infinite geodesics or characteristics. Fairly generally, one expects that when synchronization occurs (Section \ref{sub:1F1S}), these minimizing paths should coalesce either at a finite time or else asymptotically. This property is known as \emph{hyperbolicity}. Our next result shows that for   $\lambda\notin \Bruno$, the polymer measures with parameter $\lambda$ started from different initial conditions are hyperbolic in total variation norm. 

\begin{theorem}\label{thm:hyp}
The following statements hold $\P$-almost surely simultaneously  for all $\lambda\not\in\Bruno$.

\begin{enumerate} [label={\rm(\roman*)}, ref={\rm\roman*}]   \itemsep=3pt  
\item\label{hyp.i}    We have locally uniform total variation convergence:  for all $C<\infty$, 
\be\label{hyp59} 
\lim_{r\to-\infty} \; \sup_{s,x,t,y\in[-C,C]} \;  \bigl\lVert   \Poly_{(t,y)}^{\lambda}(X_{-\infty:r}\in \aabullet\tspb)  - \Poly_{(s,x)}^{\lambda}(X_{-\infty:r}\in \aabullet\tspb)  \bigr\rVert_{\rm TV}  =0.
\ee

\item\label{hyp.ii}   
The measures 
$\{\Poly_{(t,y)}^{\lambda}: (t,y)\in\R^2\}$ are equal and  trivial on the tail $\sigma$-algebra $\Paths_{\rm tail} \equiv\bigcap_{r<0}\Paths_{-\infty:r}$ of the path space. 

\item\label{hyp.iii}     Each $\Poly_{(t,y)}^{\lambda}$ is mixing in the following total variation sense:   for $B\in\Paths_{-\infty:t}$ such that   $\Poly_{(t,y)}^{\lambda}(B)>0$,  
\be\label{hyp61} 
\lim_{r\to-\infty} \;  \bigl\lVert   \Poly_{(t,y)}^{\lambda}(X_{-\infty:r}\in \aabullet\tspb)  - \Poly_{(t,y)}^{\lambda}(X_{-\infty:r}\in \aabullet\tspb\viiva B)  \bigr\rVert_{\rm TV}  =0.
\ee
\end{enumerate} 

\end{theorem}



Before turning to the proofs of these main results, we make two remarks concerning the forward-in-time version of the continuum polymers.

\begin{remark}
The invariance of $\P$ 
under temporal reflection $\refd_1$
 implies that the analogous results also hold
for the forward random polymer measure 
with terminal time $t$ and terminal condition $f$, which has
kernel
	\begin{align*}
	\pi^{t,f}(s',w'\viiva s,w)
	&=\frac{\She(t,f\viiva s',w')\She(s',w'\viiva s,w)}{\She(t,f\viiva s,w)} \\
	&=\frac{\She(s',w'\viiva s,w)\int_\R\She(t,z\viiva s',w')\,f(dz)}{\int_\R\She(t,z\viiva s,w)\,f(dz)},\quad \text{for } s<s'<r   \text{ and }  w,w'\in\R.
	\end{align*}
The forward semi-infinite polymer is then defined via the forward Busemann limits mentioned in Remark \ref{forwardBus}.
\end{remark}

\begin{remark}
The forward point-to-point and point-to-line polymer measures mentioned in the previous remark (with $f=\delta_y$ and $f(dz)=dz$, respectively) were originally introduced in \cite{Alb-Kha-Qua-14-jsp}.  It was shown in  that on an event of full probability depending on the initial and terminal conditions, the
finite-dimensional marginals of the path measures are absolutely continuous with respect to the corresponding finite-dimensional marginals of Brownian motion, but the two distributions are mutually singular at the process level.
\end{remark}



We begin with the proof of the parts of Theorem \ref{lm:sem-pol} other than \eqref{Q-LLN}. Part \eqref{sem-pol.a} shows existence and basic properties of the measures studied in this section,  part \eqref{th:weak-cv} shows that these measures arise as limits of finite volume measures, and part \eqref{sem-pol.b} shows that these measures satisfy basic continuity conditions. We return to prove part \eqref{Q-LLN} after proving a large deviation principle for the paths.

\label{page:lm:sem-pol}
\begin{proof}[Proof of Theorem \ref{lm:sem-pol} except part \eqref{Q-LLN}]
Initially, the probability measure $\Poly_{(t,y)}^{\lambda\sig}$ can be defined on the product space $\R^{(-\infty,t]}$, through Kolmogorov's extension.
But then one notes that for any $r<t$, $\Poly_{(t,y)}^{\lambda\sig}\big|_{\Paths_{r:t}}=\Poly_{(t,y),(r,f^{\lambda\sig}_r)}$, where $f^{\lambda\sig}_r(z)=e^{\Bus^{\lambda\sig}(r,0,r,z)}$. Indeed, for $s<s'$ in $(r,t]$,
	\[\pi^{\lambda\sig}(s,w\viiva s',w')=\frac{\She(s',w'\viiva s,w)e^{\Bus^{\lambda\sig}(r,0,s,w)}}{e^{\Bus^{\lambda\sig}(r,0,s',w')}}
	=\frac{\She(s',w'\viiva s,w)\She(s,w\viiva r,f^{\lambda\sig}_r)}{\She(s',w'\viiva r,f^{\lambda\sig}_r)}=\pi_{r,f^{\lambda\sig}_r}(s,w\viiva s',w').\]
By Theorem \ref{bus:exp'}, $f^{\lambda\sig}_r\in\ICM$ for any $r, \lambda\in\R$ and $\sigg\in\{-,+\}$. Existence and uniqueness of the claimed measure on $\sC([r,t],\R)$ and the claimed $\alpha$-H\"older-continuity on $[r,t]$ follow  from Theorem 2.15 in \cite{Alb-etal-22-spde-}. Consistency then shows furnishes a measure on $\sC((-\infty,t],\R)$, proving \eqref{sem-pol.a}.

Turning to part \eqref{th:weak-cv}, we begin with an observation similar to one in the proof of Theorem \ref{lm:sem-pol}. We have that if $r<s<t$, then $\Poly_{(t,y),(r,z_r)}\big|_{\Paths_{s:t}}=\Poly_{(t,y),(s,g_r)}$ and 
$\Poly_{(t,y),(r,f)}\big|_{\Paths_{s:t}}=\Poly_{(t,y),(s,h_r)}$, where the right-hand sides are the finite-length polymers \eqref{pi-rf} with terminal functions 
	\[g_r(x)=\frac{\She(s,x\viiva r,z_r)}{\She(s,0\viiva r,z_r)}\quad\text{and}\quad 
	h_r(x)=\frac{\She(s,x\viiva s,f(r,\aabullet))}{\She(s,0\viiva s,f(r,\aabullet))}\,.\]
The denominators can be cancelled. They are included for the next step. 

By Lemma \ref{Bus-bnds} (which we now know holds on the space $(\Omega,\sF,\P)$),  $\P$-almost surely: 
for any $\kappa<\mu$ in $\Udense$, $\e>0$, and $s\in\R$, there exist $R<s$ and $\delta>0$ such that
for any $r\le R$, $\lambda\in[\kappa-\e,\mu+\e]$, $z$ such that $\abs{\frac{z}r+\lambda}<\delta$, and all $x\in\R$,
	\[\frac{\She(s,x\viiva r,z)}{\She(s,0\viiva r,z)}
	\le (1+\e)\bigl(e^{\Bus^{\kappa}(s,0,s,x)}+e^{\Bus^{\mu}(s,0,s,x)}\bigr).\]

By Lemma \ref{l2p-bnds},  $\P$-almost surely: 
for any $\kappa<\mu$ in $\Udense$, $\e>0$, $s\in\R$, 
$\lambda\in[\kappa+\e,\mu-\e]$, and $f\in\initF_\lambda$, there exists an $R<s$ such that
for any $r\le R$ and $x\in\R$
	\[\frac{\int_\R\She(s,x\viiva r,z)\,f(r,z)\,dz}{\int_\R\She(s,0\viiva r,z)\,f(r,z)\,dz}
	\le (1+\e)^2\bigl( e^{\Bus^{\kappa}(s,0,s,x)}+e^{\Bus^{\mu}(s,0,s,x)}\bigr).\]

The claims now follow from Theorem \ref{bus:exp'}, the convergence in Theorem 2.16 of \cite{Alb-etal-22-spde-},
and the limits \eqref{eq:buslim} and \eqref{eq:buslim2}. This shows part \eqref{th:weak-cv}.

By the monotonicity \eqref{Busmono}, if $\mu\in[\eta-1,\eta+1]$, then $f^{\mu\sig}_r(z)\le f^{(\eta+1)+}_r(z)+f^{(\eta-1)-}_r(z)$, for all $z\in\R$. By \eqref{bus:exp3'},  
$\int_\R e^{-a^2z}(f^{(\eta+1)+}_r(z)+f^{(\eta-1)-}_r(z))\,dz<\infty$ for all $a>0$.  
 The convergence claims in \eqref{sem-pol.b} then follow from Theorem 2.16 \cite{Alb-etal-22-spde-} and the limits in \eqref{Buscont}. 
\end{proof}

Recall the notation $\Paths_{s:s'}$ for the $\sigma$-algebra on $\cC([t,t'],\R)$ generated by $X_{s:s'}$ from the notation section.

The law of large numbers in Theorem \ref{lm:sem-pol}\eqref{Q-LLN} arises in part through the following large deviation principle for the finite-dimensional distributions of the path.  This LDP has the same quadratic rate function as Brownian motion with drift $-\lambda$. Note that on every interval, the path measure should be expected to be singular with respect to a Brownian measure. See \cite[Section 4.4]{Alb-Kha-Qua-14-jsp}.

\begin{lemma}\label{Q-LDP}
The following holds $\P$-almost surely: for all $t,y,\lambda\in\R$, $\sigg\in\{-,+\}$, and $0=\tau_0<\tau_1<\cdots<\tau_k$,
the distribution of $(r^{-1}X_{\tau_1r},\dotsc,r^{-1}X_{\tau_kr})$ under $\Poly_{(t,y)}^{\lambda\sig}$ satisfies a large deviation principle, as $r\to-\infty$,
 with normalization $\abs r$ and rate function 
	\begin{align*}
	I^\lambda(u_1,\dotsc,u_k)
	&=\frac12\sum_{i=0}^{k-1}\Bigl(\frac{u_{i+1}-u_i}{\tau_{i+1}-\tau_i}+\lambda\Bigr)^2(\tau_{i+1}-\tau_i)
	\,,\quad u_0=0,\ (u_1,\dotsc,u_k)\in\R^k.
	\end{align*}
\end{lemma}

\begin{remark}\label{LLN-probab}
A special case of the above is the large deviation principle for the distribution of $X_r/r$ under $\Poly_{(t,y)}^{\lambda\sig}$, which has the rate function $I^\lambda(-\mu)=\frac{(\mu-\lambda)^2}2$. In particular, for any $\e>0$, 
$\Poly_{(t,y)}^{\lambda\sig}(\abs{X_r+\lambda r}>\e\abs{r})$ decays exponentially fast as $r\to-\infty$.
\end{remark}

\begin{remark}
Lemma \ref{Q-LDP} suggests that as $r\to-\infty$,  the distribution of the scaled path $\{r^{-1}X_{\tau r}:\tau\ge0\}$ under $\Poly_{(t,y)}^{\lambda\sig}$ satisfies a large deviation principle with rate function
	\[I^\lambda(f)=\frac12\int_0^\infty (f'(\tau)+\lambda)^2\,d\tau,\]
where $f:[0,\infty)\to\R$ is such that $f'+\lambda$ is in $L^2(\R)$ and $f(0)=0$.
The above lemma immediately implies the weak large deviation upper bound. The full LDP is Open Problem \ref{prob:LDP}.
\end{remark}

\begin{proof}[Proof of Lemma \ref{Q-LDP}]
%
The weak large deviation principle follows directly from the shape theorems \eqref{b-shape} and \eqref{Z-cont}. 
More precisely, consider the full $\P$-probability event that is 
the intersection of the events on which parts \eqref{Bproc.b} and \eqref{Bproc.d} of Theorem \ref{main:Bus}, Theorem \ref{thm:b-shape}, Theorem \ref{thm:Z-cont}, Theorem \ref{int-shape} (which holds on $(\Omega,\sF,\P)$),  
and \eqref{exp-tight} are satisfied. 
Take $\w$ in this event. 
Then for any bounded Borel set $A\subset\R^k$
and any $(t,y)\in\R^2$, 
$\sigg\in\{-,+\}$, $\kappa<\mu$ in $\Udense$, and $\lambda\in(\kappa,\lambda)$, 
		\begin{align*}
	&\frac1{\abs{r}}\log\Poly_{(t,y)}^{\lambda\sig}\{(r^{-1}X_{\tau_1 r},\dotsc,r^{-1}X_{\tau_k r})\in A\}\\
	&=\frac1{\abs{r}}\log\int_{rA}   e^{\Bus^{\lambda\sig}(t,y,\tau_kr,z_k)}
	\She(t,y\viiva \tau_1 r,z_1)  \prod_{i=1}^{k-1} \She(\tau_i r,z_i\viiva \tau_{i+1} r,z_{i+1}) 
	  \,dz_{1:k} \\
	&\le\frac{\Bus^{\lambda\sig}(t,y,\tau_kr,y)}{\abs{r}}\\
	&+\frac1{\abs{r}}\log\Bigl[\;\int_{rA}\one\{z_k\ge y\}  e^{\Bus^\mu(\tau_k r,y,\tau_kr,z_k)}
	\She(t,y\viiva \tau_1 r,z_1) \prod_{i=1}^{k-1} \She(\tau_i r,z_i\viiva \tau_{i+1} r,z_{i+1}) \,dz_{1:k} \\
	&\qquad\qquad 
	+\int_{rA}\one\{z_k\le y\}   e^{\Bus^\kappa(\tau_k r,y,\tau_kr,z_k)}
	\She(t,y\viiva \tau_1 r,z_1)\prod_{i=1}^{k-1} \She(\tau_i r,z_i\viiva \tau_{i+1} r,z_{i+1}) \,dz_{1:k} \Bigr]\\
	&=\frac{\Bus^{\lambda\sig}(t,y,\tau_kr,y)}{\abs{r}} +  \frac1{\abs{r}}\log\abs{r}^k\\
	&+\frac1{\abs{r}}\log\Bigl[\;\int_{-A}\one\{v_k\ge y\abs{r}^{-1}\}  e^{\Bus^\mu(\tau_k r,y,\tau_kr,v_k\abs{r})}
	\She(t,y\viiva \tau_1 r,v_1\abs{r}) \prod_{i=1}^{k-1} \She(\tau_i r,v_i\abs{r}\viiva \tau_{i+1} r,v_{i+1}\abs{r}) \,dv_{1:k} \\
	&\qquad 
	+\int_{-A}\one\{v_k\le y\abs{r}^{-1}\}   e^{\Bus^\kappa(\tau_k r,y,\tau_kr,v_k\abs{r})}
	\She(t,y\viiva \tau_1 r,v_1\abs{r})\prod_{i=1}^{k-1} \She(\tau_i r,v_i\abs{r}\viiva \tau_{i+1} r,v_{i+1}\abs{r}) \,dv_{1:k} \Bigr].
	\end{align*}
By Theorems  \ref{thm:b-shape} and   \ref{thm:Z-cont},  
  the  ${\abs{r}}^{-1}\log[\dotsm]$  converges to the maximum of
	\begin{align*}
	\sup_{-(v_1,\cdots,v_k)\in A,v_k\ge0}
	\Bigl\{\mu v_k -\frac{\tau_1}{24}-\frac{v_1^2}{2\tau_1}+\sum_{i=1}^{k-1}\frac{\tau_i-\tau_{i+1}}{24}
	-\sum_{i=1}^{k-1}\frac{(v_{i+1}-v_i)^2}{2(\tau_{i+1}-\tau_i)}\Bigr\}
	\end{align*}
and
	\begin{align*}
	\sup_{-(v_1,\cdots,v_k)\in A,v_k\le0}
	\Bigl\{\kappa v_k-\frac{\tau_1}{24}-\frac{v_1^2}{2\tau_1}+\sum_{i=1}^{k-1}\frac{\tau_i-\tau_{i+1}}{24}
	-\sum_{i=1}^{k-1}\frac{(v_{i+1}-v_i)^2}{2(\tau_{i+1}-\tau_i)}\Bigr\}.
	\end{align*}
Taking $\kappa$ and $\mu$ to $\lambda$ gives the limit 	
	\begin{align*}
	\sup_{-(v_1,\cdots,v_k)\in A}
	\Bigl\{\lambda v_k-\frac{\tau_1}{24}-\frac{v_1^2}{2\tau_1}-\sum_{i=1}^{k-1}\frac{\tau_i-\tau_{i+1}}{24}
	-\sum_{i=1}^{k-1}\frac{(v_{i+1}-v_i)^2}{2(\tau_{i+1}-\tau_i)}\Bigr\}
	=\Bigl(\frac{\lambda^2}2-\frac1{24}\Bigr)\tau_k-\inf_A I^\lambda.
	\end{align*}
Next, apply \eqref{NEBus2} to get
	\[\Bus^{\lambda\sig}(t,y,\tau_kr,y)
	=-\log\int_{-\infty}^\infty\She(t,y\viiva \tau_k r,z)e^{\Bus^{\lambda\sig}( \tau_k r,y, \tau_k r,z)}\,dz.\]
Theorem \ref{int-shape} implies that after dividing the right-hand side by $\abs{r}$ and taking $r\to-\infty$ it  converges to 
	\[-\tau_k\sup_{\nu}\Bigl\{\lambda\nu-\frac{\nu^2}2-\frac1{24}\Bigr\}=-\Bigl(\frac{\lambda^2}2-\frac1{24}\Bigr)\tau_k.\]
Thus, we get that
	\[\varlimsup_{r\to-\infty}
	\frac1{\abs{r}}\log\Poly_{(t,y)}^{\lambda\sig}\{(r^{-1}X_{\tau_1 r}, \dotsc,r^{-1}X_{\tau_k r})\in A\}
	\le-\inf_A I^\lambda.\]
The matching lower bound comes by switching around $\kappa$ and $\mu$:
	\[\varliminf_{r\to-\infty}
	\frac1{\abs{r}}\log\Poly_{(t,y)}^{\lambda\sig}\{(r^{-1}X_{\tau_1 r}, \dotsc,r^{-1}X_{\tau_k r})\in A\}
	\ge-\inf_A I^\lambda.\]
	
The weak large deviation principle is proved. (That is, the upper bound holds for compact instead of all closed sets.)   The full large deviation principle follows from exponential tightness.
See Theorem 2.19 in \cite{Ras-Sep-15-ldp}.
 For this, it is enough to show that $\P$-almost surely, for any $t,y,\lambda\in\R$, $\sigg\in\{-,+\}$, and $\tau\in(0,1]$, 
\[\lim_{C\to\infty}\varlimsup_{r\to-\infty}\abs{r}^{-1}\log\Poly_{(t,y)}^{\lambda\sig}\{\abs{X_{\tau r}}\ge C\abs{r}\}=-\infty.\]
This comes with the exact same argument as for \eqref{exp-tight}.
\end{proof}
With these large deviation results in-hand, we turn to the proof of Theorem \ref{lm:sem-pol}\eqref{Q-LLN}.

\begin{proof}[Proof of Theorem \ref{lm:sem-pol}\eqref{Q-LLN}]

Consider the full $\P$-probability event that is \label{page:Q-LLN}
the intersection of the events on which parts \eqref{Bproc.b} and \eqref{Bproc.d} of Theorem \ref{main:Bus} and Theorems \ref{bus:exp'}, \ref{thm:buslim}, and \ref{lm:sem-pol} are satisfied. 
For $x\in\R$ and $r<s\le t$ define
	\[M_r^{s,x,t,y}=\frac{\She(s,x\viiva r,X_r)}{\She(t,y\viiva r,X_r)}\,.\]
Then $\{M_r^{s,x,t,y}:r<s\}$ is a $\Poly_{(t,y)}^{\lambda\sig}$-backward martingale with respect to the path filtration $\Paths_{-\infty:r}$, which was defined in the notation section \ref{sec:notation}. 
Indeed,
for $r'<r''<r$ let $F$ be a bounded  $\Paths_{r':r''}$-measurable function.
Then
	\begin{align*}
	&E^{\Poly_{(t,y)}^{\lambda\sig}}[M_r^{s,x,t,y}F(X_{r':r''})]\\
	&\qquad=\int_{\R^2} \She(t,y\viiva r,u)e^{\Bus^{\lambda\sig}(t,y,r,u)}\cdot\frac{\She(s,x\viiva r,u)}{\She(t,y\viiva r,u)}
	\cdot \She(r,u\viiva r'',v)e^{\Bus^{\lambda\sig}(r,u,r'',v)}E^{\Poly_{(r'',v)}^{\lambda\sig}}[F(X_{r':r''})]\,dv\,du\\
	&\qquad=\int_{\R^2}\She(s,x\viiva r,u)\She(r,u\viiva r'',v)e^{\Bus^{\lambda\sig}(t,y,r'',v)}
	\cdot E^{\Poly_{(r'',v)}^{\lambda\sig}}[F(X_{r':r''})]\,du\,dv\\
	&\qquad=\int_{\R}\She(s,x\viiva r'',v)e^{\Bus^{\lambda\sig}(t,y,r'',v)}
	\cdot E^{\Poly_{(r'',v)}^{\lambda\sig}}[F(X_{r':r''})]\,dv\\
	&\qquad=\int_{\R}\She(t,y\viiva r'',v)e^{\Bus^{\lambda\sig}(t,y,r'',v)}\cdot\frac{\She(s,x\viiva r'',v)}{\She(t,y\viiva r'',v)}
	\cdot E^{\Poly_{(r'',v)}^{\lambda\sig}}[F(X_{r':r''})]\,dv\\
	&\qquad=E^{\Poly_{(t,y)}^{\lambda\sig}}[M_{r''}^{s,x,t,y}F(X_{r':r''})].
	\end{align*}
By the martingale convergence theorem (see e.g.\ \cite[Theorem 3.15, page 17]{Kar-Shr-91}), the limit
	\be\label{eq:polyMG}M_{-\infty}^{s,x,t,y}=\lim_{r\to-\infty}\frac{\She(s,x\viiva r,X_r)}{\She(t,y\viiva r,X_r)}\,\ee
exists $\Poly_{(t,y)}^{\lambda\sig}$-almost surely. 

Suppose now that, with strictly positive $\Poly_{(t,y)}^{\lambda\sig}$-probability, $r^{-1}X_r$ does not have a limit as $r\to-\infty$. We can then find $\kappa<\mu$ in $\Udense$, the countable dense subset of $\R$ from Section \ref{sec:stat-coc}, such that with strictly positive $\Poly_{(t,y)}^{\lambda\sig}$-probability, we have
	\[\varliminf_{r\to-\infty} r^{-1}X_r \le -\mu < -\kappa \le \varlimsup_{r\to-\infty} r^{-1}X_r.\]
By path-continuity the above event equals 
$\{-\kappa\text{ and }-\mu\text{ are limit points of }r^{-1}X_r\}$.  
Limit \eqref{eq:buslim} of Theorem \ref{thm:buslim} implies that on the intersection of this event and the event on which \eqref{eq:polyMG} holds,
 	\begin{align*}
	&\int_{-\infty}^\infty\She(t,y\viiva s,w)e^{\Bus^{\kappa}(s,x,s,w)}\,dw
	=e^{\Bus^{\kappa}(s,x,t,y)}=\frac1{M_{-\infty}^{s,x,t,y}}
	=e^{\Bus^{\mu}(s,x,t,y)}=\int_{-\infty}^\infty\She(t,y\viiva s,w)e^{\Bus^{\mu}(s,x,s,w)}\,dw.
	\end{align*}
	The previous equalities then hold $\Poly_{(t,y)}^{\lambda\sig}$-almost surely for all rational $s<t$ and $x$ . 
	Taking $s\to t$ we get that $\Bus^\kappa(t,x, t,y)=\Bus^\mu(t,x,t,y)$ 
	for all rational $x$. This contradicts \eqref{bus:exp3'},
	 since $\kappa\ne\mu$. Consequently, we have shown that $r^{-1}X_r$ has a limit, $\Poly_{(t,y)}^{\lambda\sig}$-almost surely.
Then Lemma \ref{Q-LDP} (or, more precisely, Remark \ref{LLN-probab}) implies that this limit must equal $-\lambda$.
\end{proof}
We turn to prove the second main result of this section, Theorem \ref{thm:hyp} on hyperbolicity.

\begin{proof}[Proof of Theorem \ref{thm:hyp}]   
Fix $\lambda\not\in\Bruno$. 

\smallskip 

  Part \eqref{hyp.i}.   
By symmetry, we can assume $s\le t$.  
 We prove first total variation convergence of the one-time marginals, that is, 
\be\label{hyp67} 
\lim_{r\to-\infty}  \; \sup_{s,x,t,y\in[-C,C]} \;     \norm{ \, \Poly_{(t,y)}^{\lambda}(X_r\in \aabullet\tspb)  - \Poly_{(s,x)}^{\lambda}(X_r\in \aabullet\tspb) \,  }_{\rm TV}   =0.
\ee
Let  $\Poly_{(t,y);r}^\lambda$ denote  the distribution of $X_r$ under $\Poly_{(t,y)}^\lambda$.
We have for all $s,x,t,y\in\R$, $r<s\wedge t$, and $z\in\R$
	\[g_{r,s,x,t,y}^\lambda(z)=\frac{d\Poly_{(s,x);r}^\lambda}{d\Poly_{(t,y);r}^\lambda}(z)=\frac{\She(s,x\viiva r,z)}{\She(t,y\viiva r,z)}\cdot e^{\Bus^\lambda(s,x,t,y)}.\]
Fix $\tau>0$.  

\smallskip 

\textit{Case 1 of \eqref{hyp67}: $t-s\ge\tau$.}  
By Theorem \ref{thm:buslim}, for any $\lambda\not\in\Bruno$,
any $C>0$, $\e>0$, and $\tau>0$, there exist $R<0$ and $\delta>0$ such that for all  $r\le R$, $z$ 
such that $\abs{\frac{z}r+\lambda}<\delta$, and all $s,x,t,y\in[-C,C]$ with $t-s\ge\tau$, 
	\[(1+\e)^{-2}\le g_{r,s,x,t,y}^\lambda(z)\le (1+\e)^3.\]
By the LDP of Lemma \ref{Q-LDP}   and the above uniform bound on the Radon-Nikodym derivative
 there exist $c=c(\lambda,\delta)>0$ and  $R<0$ such that 
	\[\sup_{t,y\in[-C,C]}\Poly_{(t,y)}^\lambda\{\abs{X_r/r+\lambda}\ge\delta\}\le e^{-c\abs{r}}\qquad\text{for all $r\le R$}.\]
	Take  $\e>0$ small enough so that 
	\[1-(1+\e)^{-2}\le (1+\e)^3-1\le 4\e.\]
By Lemma \ref{lm:var-dist},  
	\begin{align*}  
	\frac{1}{2}\norm{\Poly_{(s,x);r}^\lambda - \Poly_{(t,y);r}^\lambda}_{\rm TV} 
	&\le E^{\Poly_{(t,y)}^\lambda}\bigl[\tspa \abs{g_{r,s,x,t,y}^\lambda(X_r)-1}\cdot \one\{\abs{\tfrac{X_r}r+\lambda}<\delta\}\tspa\bigr]
	+\Poly_{(t,y)}^\lambda\bigl\{\abs{\tfrac{X_r}r+\lambda}\ge\delta\bigr\} \\
	&\le 4\e+e^{-c\abs{r}}.
	\end{align*}  
The bound above is uniform over   $s,x,t,y\in[-C,C]$ such that  $t-s\ge\tau$ and $r\le R$.  
Take $r\to-\infty$ and then $\e\to0$. We have  proved \eqref{hyp67}  for    $t-s\ge\tau$.

\smallskip 

\textit{Case 2 of \eqref{hyp67}: $0\le t-s\le\tau$.}  
Pick $u\le -C-\tau$. By  the Markov property,  for $B\in\sB(\R)$ and $K\in(0,\infty)$, 
	\begin{align*} 
	&\Poly_{(t,y);r}^\lambda(B)-\Poly_{(s,x);r}^\lambda(B)
	=E^{\Poly_{(s,x)}^\lambda}\bigl[\Poly_{(t,y);r}^\lambda(B)-\Poly_{(u,X_u);r}^\lambda(B)\bigr] \\
	&=E^{\Poly_{(s,x)}^\lambda}\bigl[ \bigl( \Poly_{(t,y);r}^\lambda(B)-\Poly_{(u,X_u);r}^\lambda(B)\bigr) \tspa\ind_{\abs{X_u}\le  C+K}  \bigr]  + \Poly_{(s,x)}^\lambda\{\abs{X_u}>  C+K\}. 
	\end{align*} 
From this,  for $K\ge C+\abs u$, 
	\begin{align*}   
	\sup_{s,x,t,y\in[-C,C]} \;    \frac{1}{2}\norm{\Poly_{(t,y);r}^\lambda - \Poly_{(s,x);r}^\lambda}_{\rm TV}
	&\le  
	\sup_{v,z,t,y\in[-C-K,\tspc C+K]: \, t-v\ge\tau} \;   
	 \norm{\Poly_{(t,y);r}^\lambda - \Poly_{(v,z);r}^\lambda}_{\rm TV}\\
	&\qquad \qquad 
	+ \sup_{s,x\in[-C,C]}  \Poly_{(s,x)}^\lambda\{\abs{X_u}>  C+K\}. 
	\end{align*} 
For a fixed $K$, the first term on the right  converges to $0$ as $r\to-\infty$ by the case already proved.  
The last term  converges to $0$ as $K\to\infty$.   This completes the proof of \eqref{hyp67}. 

\smallskip 

The Markov property takes us from \eqref{hyp67} to \eqref{hyp59}:  for  $r<s\wedge t$,
\begin{align*}
&\sup_{A\tspa\in\tspa \sB(\sC((-\infty,r],\R))} \bigl\lvert    \Poly_{(t,y)}^{\lambda}(X_{-\infty:r}\in A)   - \Poly_{(s,x)}^{\lambda}(X_{-\infty:r}\in A)  \bigr\rvert  \\
&= \sup_{A\tspa\in\tspa \sB(\sC((-\infty,r],\R))} \;  \biggl\lvert  \; \int \bigl[  \Poly_{(t,y)}^{\lambda}(X_r\in dz)  - \Poly_{(s,x)}^{\lambda}(X_r\in dz)     \bigr]   \Poly_{(r,z)}^{\lambda}(X_{-\infty:r}\in A)  \;  \biggr\rvert  \\[3pt] 
&\le 2 \tspb  \norm{ \, \Poly_{(t,y)}^{\lambda}(X_r\in \aabullet)  - \Poly_{(s,x)}^{\lambda}(X_r\in \aabullet) \,  }_{\rm TV} . 
\end{align*} 

\medskip 

Part \eqref{hyp.ii}.  By  \eqref{hyp59}   and the implication (iii)$\Rightarrow$(i) of   Theorem 25.25 on p.~576    of  Kallenberg \cite{Kal-21},  $\Poly_{(t,y)}^{\lambda} = \Poly_{(s,x)}^{\lambda}  
$  on the tail $\sigma$-algebra $\Paths_{\rm tail}$ for all time-space points $(s,x),(t,y)\in\R^2$.     
Let $A\in\Paths_{\rm tail}$ and $c=\Poly_{(t,y)}^{\lambda}(A)$  the common probability value for all $(t,y)\in\R^2$. 
Then $\Poly_{(t,y)}^{\lambda}$-almost surely, for $s<t$, first by the Markov property and then by martingale convergence, 
\begin{align*}
c =  \Poly_{(s, X_s)}^{\lambda} (A)  = \Poly_{(t,y)}^{\lambda} (\tspa A \viiva \Paths_{s:t} \tspa)   
\underset{s\to-\infty}\longrightarrow  \ind_A. 
\end{align*} 
Hence the common value $c=0$ or $1$.

\medskip 

Part \eqref{hyp.iii}.  The implication from part \eqref{hyp.ii}  to part \eqref{hyp.iii} is the implication  (i)$\Rightarrow$(ii) of  Theorem 26.10 on p.~595 of Kallenberg \cite{Kal-21}. 
\end{proof}

We close this section with some a result about stochastic monotonicity. Planar directed polymer measures with nearest-neighbor random walk (on the lattice) or continuous sample paths are typically stochastically monotone in their initial and terminal conditions. This is a straightforward consequence of the Karlin-McGregor theorem. See the proof of Proposition 5.2 in \cite{Alb-etal-22-spde-}.

To make this precise in our setting, recall that measures on continuous functions come with a natural partial order, defined as follows. Given $s<t$, a function $F:\cC([s,t],\R)\to\R$ is nondecreasing if $F(X_{s:t})\le F(Y_{s:t})$ whenever $X_r\le Y_r$ for all $r\in[s,t]$.
Given two probability measures $Q_1$ and $Q_2$ on $\cC([s,t],\R)$, we say $Q_1$ is stochastically dominated by $Q_2$, and write $Q_1\std Q_2$, if $E^{Q_1}[F]\le E^{Q_2}[F]$ for all bounded measurable nondecreasing functions $F:\cC([s,t],\R)\to\R$. 
Then the polymer measures are stochastically ordered. Proposition 2.18 in \cite{Alb-etal-22-spde-} shows that  for all $s,x,t,y,u,v\in\R$ with $s<t$, $u\le x$, $v\le y$, and for all $f\in\ICM$, we have the stochastic ordering, \begin{align}\label{eq:stochmon}\Poly_{(t,v),(s,u)}\std\Poly_{(t,y),(s,x)}\qquad\text{ and }\qquad \Poly_{(t,v),(s,f)}\std\Poly_{(t,y),(s,f)}.\end{align}

Because the infinite length polymers arise as limits of finite length polymers, they inherit this monotonicity from the finite volume measures.

\begin{lemma}\label{lm:Q-mono}
The following holds $\P$-almost surely: for all $\sigg\in\{-,+\}$ and all $t,y,v,\lambda,\mu\in\R$ with $v\le y$ and $\lambda<\mu$, we have the stochastic ordering
 	\begin{align*}
	&\Poly_{(t,y)}^{\lambda-}\std\Poly_{(t,y)}^{\lambda+}\std\Poly_{(t,y)}^{\mu-}\std\Poly_{(t,y)}^{\mu+}\,,\quad
	\Poly_{(t,v)}^{\lambda+}\std\Poly_{(t,y)}^{\lambda+},\quad\text{and}\quad\Poly_{(t,v)}^{\lambda-}\std\Poly_{(t,y)}^{\lambda-}.
	\end{align*}
\end{lemma}

\begin{proof}
The continuity of the paths implies that it is enough to prove the stochastic ordering claims at the level of the finite marginal distributions. Then an induction, using the Chapman-Kolmogorov property of the point-to-point measures and the Markov property of the semi-infinite measures, reduces this to the one-dimensional marginal. See the proof of Proposition 2.18 in \cite{Alb-etal-22-spde-} for a similar induction argument.

%
Start by observing that for $r<t$
	\[\frac{\int_{-\infty}^a e^{\Bus^{\lambda+}(r,a,r,u)}\She(t,y\viiva r,u)\,du}{\int_{-\infty}^a e^{\Bus^{\lambda-}(r,a,r,u)}\She(t,y\viiva r,u)\,du}\le 1\le \frac{\int_a^\infty e^{\Bus^{\lambda+}(r,a,r,u)}\She(t,y\viiva r,u)\,du}{\int_a^\infty e^{\Bus^{\lambda-}(r,a,r,u)}\She(t,y\viiva r,u)\,du}.\]
From this we get
	\[\frac{1-\int_a^\infty e^{\Bus^{\lambda+}(t,y,r,u)}\She(t,y\viiva r,u)\,du}{1-\int_a^\infty e^{\Bus^{\lambda-}(t,y,r,u)}\She(t,y\viiva r,u)\,du}=
	\frac{\int_{-\infty}^a e^{\Bus^{\lambda+}(t,y,r,u)}\She(t,y\viiva r,u)\,du}{\int_{-\infty}^a e^{\Bus^{\lambda-}(t,y,r,u)}\She(t,y\viiva r,u)\,du}\le \frac{\int_a^\infty e^{\Bus^{\lambda+}(t,y,r,u)}\She(t,y\viiva r,u)\,du}{\int_a^\infty e^{\Bus^{\lambda-}(t,y,r,u)}\She(t,y\viiva r,u)\,du}.\]
	Rearranging, one gets 
	\begin{align}\label{buspoly-mono1}
	\Poly_{(t,y)}^{\lambda-}(X_r\ge a)\le\Poly_{(t,y)}^{\lambda+}(X_r\ge a).
	\end{align} 
	The next two inequalities come similarly.
	
 For the second-to-last inequality apply \eqref{comp} with the function $f(w)=e^{\Bus^{\lambda+}(t,0,r,w)}$ to get
 	\[\frac{\int_{-\infty}^a\She(t,y\viiva r,w)e^{\Bus^{\lambda+}(t,0,r,w)}\,dw}{\int_{-\infty}^a\She(t,v\viiva r,w)e^{\Bus^{\lambda+}(t,0,r,w)}\,dw}
	\le\frac{\int_a^\infty\She(t,y\viiva r,w)e^{\Bus^{\lambda+}(t,0,r,w)}\,dw}{\int_a^\infty\She(t,v\viiva r,w)e^{\Bus^{\lambda+}(t,0,r,w)}\,dw}.\]
	Multiply both sides by $e^{\Bus^{\lambda+}(t,y,t,0)}/e^{\Bus^{\lambda+}(t,v,t,0)}$ and rearrange to get
$\Poly_{(t,v)}^{\lambda+}(X_r\ge a)\le\Poly_{(t,y)}^{\lambda+}(X_r\ge a)$. The last inequality is similar.	
\end{proof}

\section{Exceptional slopes}
\label{sec:Bruno}

This section contains the proofs of the properties of the set $\Bruno$ of the discontinuities of the Busemann process, defined in \eqref{Bruno}. 

\begin{proof}[Proof of Theorem \ref{L-char}]
Suppose that $\Bus^{\lambda-}(t,x,t,y)=\Bus^{\lambda+}(t,x,t,y)$ for some $t\in\R$ and $x<y$. 
Take $r<t$
and for $\sigg\in\{-,+\}$ let 
	\[F_{r,\sig}(v)=\Poly_{(t,x)}^{\lambda\sig}(X_r\le v)
	=\int_{-\infty}^v \She(t,x\viiva r,z)e^{\Bus^{\lambda\sig}(t,x,r,z)}\tspb dz.\]
By Lemma \ref{lm:Q-mono} we have $F_{r,+}(v)\le F_{r,-}(v)$ for all $v\in\R$. 
Let $U$ be a uniform random variable  on $[0,1]$ with distribution $\Poly$ and independent of everything else.
Let 
$X_r^{\lambda\sig}=F_{r,\sig}^{-1}(U)$.  
Then $X_r^{\lambda-}\le X_r^{\PMPslopep}$. Also,
the distribution of $X_r^{\lambda\sig}$ under $\Poly$ is $\Poly_{(t,x)}^{\lambda\sig}$.
The  comparison inequality \eqref{crossing} tells us  that
	\begin{align}\label{cross-aux}
	\frac{\She(t,y\viiva r,X_r^{\lambda-})}{\She(t,x\viiva r,X_r^{\lambda-})}\le \frac{\She(t,y\viiva r,X_r^{\lambda+})}{\She(t,x\viiva r,X_r^{\lambda+})}\,,
	\end{align}
with a strict inequality if $X_r^{\lambda-}<X_r^{\lambda+}$.

By \eqref{NEBus2} we have for $s<t$ and $x,y,\lambda\in\R$,
\begin{align*}
e^{\Bus^{\lambda\sig}(t,x,t,y)}
=\int_{-\infty}^\infty\She(t,y\viiva s,u)e^{\Bus^{\lambda\sig}(t,x,s,u)}\,du
=E^{\Poly_{(t,x)}^{\lambda\sig}}\Bigl[\frac{\She(t,y\viiva s,X_s^{\lambda\sig})}{\She(t,x\viiva s,X_s^{\lambda\sig})}\Bigr].
\end{align*}
Since $\Bus^{\lambda-}(t,x,t,y)=\Bus^{\lambda+}(t,x,t,y)$ we get that equality holds in \eqref{cross-aux}
$\Poly$-almost surely and therefore we have $\Poly(X_r^{\lambda-}=X_r^{\lambda+})=1$.
Consequently, $F_{r,\pm}$ are equal, which implies that 
$\Bus^{\lambda-}(t,x,r,z)=\Bus^{\lambda+}(t,x,r,z)$ for all $z\in\R$
because the probability densities  $\partial_zF_{r,\sig}(z)=\She(t,x\viiva r,z)e^{\Bus^{\lambda\sig}(t,x,r,z)}$ are continuous.
Since $r<t$  was arbitrary, we get that $\Bus^{\lambda\pm}(r,u,s,v)$ match for all $r,s\in(-\infty,t)$ and all $u,v\in\R$. By continuity, this extends to $r,s\in(-\infty,t]$.

To summarize, we have shown that if for some $t\in\R$ and some $x<y$ we have $\Bus^{\lambda-}(t,x,t,y)=\Bus^{\lambda+}(t,x,t,y)$, then $\Bus^{\lambda\pm}$ match on $(-\infty,t]\times\R\times(-\infty,t]\times\R$. 
A similar argument shows that if for some $t\in\R$, $x<y$, and $\kappa<\mu$, we have $\Bus^{\kappa+}(t,x,t,y)=\Bus^{\mu-}(t,x,t,y)$, then $\Bus^{\kappa+}$ and $\Bus^{\mu-}$ match on $(-\infty,t]\times\R\times(-\infty,t]\times\R$. 
We are now ready to prove the theorem.
	
	

%

 Take $\lambda\in\Bruno$.  Then there exist $s,w,s',w'\in\R$ such that $\Bus^{\lambda-}(s,w,s',w')\ne\Bus^{\lambda+}(s,w,s',w')$. Then by \eqref{NEBus2} it must be that for any $r<\min(s,s')$, 
there exist $u<v$ such that 
$\Bus^{\lambda-}(r,u,r,v)\ne\Bus^{\lambda+}(r,u,r,v)$. 
By the above, we get that for any $t>r$ and any $x,y\in\R$, we must have 
$\Bus^{\lambda-}(t,x,t,y)\ne\Bus^{\lambda+}(t,x,t,y)$. Since $r<\min(s,s')$ is arbitrary, $t$ is also arbitrary and part \eqref{L-char:i} is proved.

If $\kappa<\mu$, then \eqref{bus:exp3'}
 implies that for any $r\in\R$, there exists a $v>0$ such that
$\Bus^{\kappa+}(r,0,r,v)\ne\Bus^{\mu-}(r,0,r,v)$. 
This implies that $\Bus^{\kappa+}(t,x,t,y)\ne\Bus^{\mu-}(t,x,t,y)$ for any $t\in\R$ and any $x,y\in\R$. 
Part \eqref{L-char:ii} follows since we already know that when $x<y$, $\Bus^{\kappa+}(t,x,t,y)\le\Bus^{\mu-}(t,x,t,y)$.
\end{proof}


\begin{proof}[Proof of Theorem \ref{th:La}]
%
%
The symmetry claims in part \eqref{th:La:sym} follow from the definition \eqref{Bruno} and the properties  in Theorem \ref{thm:bcov}.

From Corollary \ref{L-char2}
	\begin{align}\label{Bruno-simple}
	\Bruno=\{\lambda:\Bus^{\lambda-}(0,0,0,1)<\Bus^{\lambda+}(0,0,0,1)\}.
	\end{align}
This shows that for a given $\lambda\in\R$, $\{\lambda\in\Bruno\}$ is a measurable event
and part \eqref{th:La:b} follows from Theorem \ref{main:Bus}\eqref{Bproc.-=+}. 

For $\kappa<\mu$, 
$\{[\kappa,\mu]\cap\Bruno\ne\varnothing\}=\{\Bus^{\kappa-}(0,0,0,1)<\Bus^{\mu+}(0,0,0,1)\}\in\fil$. 
%
By part \eqref{th:La:sym}, $\{[\kappa,\mu]\cap\Bruno\ne\varnothing\}$ is invariant under each shift $\shiftd{t}{x}$. The ergodicity of $(\Omega,\fil,\P)$ under these shifts implies 
    \begin{align}\label{Bruno01}
    \P\{[\kappa,\mu]\cap\Bruno=\varnothing\}\in\{0,1\}.
    \end{align}

	Combining the statement for shear in part  \eqref{th:La:sym} with the shear-invariance of $\P$, we have   for any $t,c\in\R$,
    \begin{align}\label{Bruno-shear}
    \P\{[\kappa,\mu]\cap\Bruno=\varnothing\}=\P\{[\kappa+c,\mu+c]\cap\Lambda^{\sheard{t}{c}\w}=\varnothing\}=\P\{[\kappa+c,\mu+c]\cap\Bruno=\varnothing\}.
    \end{align}

Suppose now that $\P\{\Bruno\ne\varnothing\}>0$. 
Then for any $n\in\N$ there exists a $m\in\Z$ such that $\P\{[m/n,(m+1)/n]\cap\Bruno\ne\varnothing\}>0$.
Then \eqref{Bruno01} implies that $\P\{[m/n,(m+1)/n]\cap\Bruno\ne\varnothing\}=1$ and \eqref{Bruno-shear} implies that this is in fact true for all 
$m\in\Z$. Thus, 
	\[\P\bigl\{[m/n,(m+1)/n]\cap\Bruno\ne\varnothing\ \ \forall n\in\N,\ \forall m\in\Z\bigr\}=1.\]
Part \eqref{th:La:e} of Theorem \ref{lm:sem-pol} is proved.
\end{proof}

\section{Ergodic solutions}\label{sec:unique}
 
This section   proves  
  the results on ergodic invariant distributions in Section \ref{subsec:ergodic}.  
  Recall  the space $\Radon$ of equivalence classes 
defined in \eqref{Radon509}, and  the random evolution defined  on it:    
for $t\ge s$ and $f\in\radon$,    
	\begin{align}\label{Sst}
	\Sols_{s,t}\eqcl{f}=\begin{cases}
	\eqcl{\She(t,\aabullet\tsp\viiva s,f)}&\text{if }\She(t,\aabullet\tsp\viiva s,f)\in\MP 
	\quad\text{and}\\
	\eqcl{\infimeas}&\text{if not.}
	\end{cases}
	\end{align}
By \eqref{CK8}, these operators satisfy the cocycle property: for $r\le s\le t$, 
\be\label{eq:Solsco}\begin{aligned}\Sols_{s,t}\Sols_{r,s}=\Sols_{r,t}.
\end{aligned}\ee
Define the transition kernel  $\piMP(t, \aabullet\tsp\viiva s,\aabullet)$ on the quotient space $\Radon$ between times $s<t$ by 
	\begin{align}\label{SHE-kernel}
	\int_{\Radon}\Phi(\clg)\tspb\piMP(t,d\clg\viiva s,\clf)=\E[\Phi(\Sols_{s,t}\clf)]=\int_{\Omega}\Phi(\Sols_{s,t}^\w\clf)\bbP(d\w),
	\end{align}
for $\clf\in\Radon$  and a  bounded Borel function $\Phi:\Radon\to\R$.  This is a time-homogeneous transition kernel, that is,  $\piMP(t,d\clg\viiva s,\clf)=\piMP(t-s,d\clg\viiva 0,\clf)$.

Two relevant subspaces of $\Radon$ require mention.   Recall definition \eqref{eq:ICM} of $\ICM$. 
Let 
	\be\label{def:tICP} 
	\tICP=\{\tspa\eqcl{f}:f\in\ICM\cap\MP\}=\{\tspa\eqcl{f}\in\Radon:f\in\ICM\}\ee 
denote the subspace of equivalence classes of measures  that do not blow up under the evolution.  $\tICP$ is not a closed subspace of $\Radon$.  It can be given its own Polish quotient topology but we have no need for this.   Recall also the space $\CICM$ of  strictly positive continuous densities of measures in $\ICM$ defined in \eqref{CICM} and its quotient space $\CICP$.  Both are Polish.  Parts \eqref{lm:SHE-MC.iii} and \eqref{lm:SHE-MC.iv} of the next lemma follow because  the appropriately initialized  equivalence class process  $\Sols_{0,t}\clf$  
 possesses paths in the spaces $\cC(\R_+,\Radon)$ and $\cC(\R_+,\CICP)$, as mentioned in Section \ref{subsec:ergodic}.   


\begin{lemma}\label{lm:SHE-MC}
The following hold.
\begin{enumerate} [label={\rm(\roman*)}, ref={\rm\roman*}]   \itemsep=3pt  
\item\label{lm:SHE-MC.i} The transition kernel  \eqref{SHE-kernel}  satisfies the Chapman-Kolmogorov equations. 

\item\label{lm:SHE-MC.ii} 
Let $\PMPinit$ be a probability measure  on $\Radon$. Then $\{\Sols_{0,t}\clf\}_{t\ge0} $ under $\P(d\w)\otimes \PMPinit(d\clf)$ is a realization of the Markov process  with initial distribution $P$ and transition kernel \eqref{SHE-kernel}.
 
 \item\label{lm:SHE-MC.iii}  
If the initial distribution $\PMPinit$ satisfies $\PMPinit(\tICP)=1$, then  
the Markov process    with transition kernel \eqref{SHE-kernel}  can be realized on the path space $\cC(\R_+,\Radon)$. 


 \item\label{lm:SHE-MC.iv}   Suppose the initial distribution $\PMPinit$ satisfies $\PMPinit(\CICP)=1$.  Then  the Markov process with transition kernel \eqref{SHE-kernel} can be realized on the path space 
 $\cC(\R_+,\CICP)$.

\end{enumerate}
\end{lemma}

\begin{proof}
Let $r<s<t$ and pick a state $\clf_{\tspa 0}\in\Radon$:
	\begin{align*}
	&\int_{\Radon}\Bigl(\;\int_{\Radon}\Phi(\clf_2)\,\piMP(t,d\clf_2\viiva s,\clf_1)\Bigr)\piMP(s,d\clf_1\viiva r,\clf_{\tspa 0})
	=\int_{\Radon}\Bigl(\;\int_\Omega\Phi(\Sols_{s,t}^\w \clf_1)\P(d\w)\Bigr)\piMP(s,d\clf_1\viiva r,\clf_{\tspa 0})\\
	&\qquad=\int_\Omega\Bigl(\;\int_\Omega\Phi(\Sols_{s,t}^\w S^{\w'}_{r,s}\clf_{\tspa 0})\P(d\w)\Bigr)\P(d\w')
	= \int_\Omega\Phi(\Sols_{s,t}^\w S^\w_{r,s}\clf_{\tspa 0})\P(d\w)\\
	&\qquad= \int_\Omega\Phi(S^\w_{s,t}\clf_{\tspa 0})\P(d\w)
	=\int_{\Radon}\Phi(\clf_2)\tspb\piMP(t,d\clf_2\viiva r,\clf_{\tspa 0}).
	\end{align*}
In the third equality we used the fact that the time intervals $(r,s]$ and $(s,t]$ are disjoint and hence their white noises are independent.
Thus the two integrals  $\P(d\w)\P(d\w')$ combine into one. The fourth equality  used the cocycle property of \eqref{Sst}.
Part \eqref{lm:SHE-MC.i} is proved. 

The Markov process on $\Radon^{\R_+}$ is then obtained by an application of Kolmogorov's consistency theorem. 
To prove the distributional claim of part \eqref{lm:SHE-MC.ii} it is sufficient to check that the finite-dimensional marginals coincide. This is done inductively through repeated use of the definition \eqref{SHE-kernel}  and \eqref{CK8}.  For $0=t_0<t_1<\cdots<t_n$, 
	\begin{align*}
	&\int_{\Radon^{n+1}}\Phi(\clf_{\tspa 0},\dotsc,\clf_n)\,\prod_{i=0}^{n-1}\piMP(t_{i+1},d\clf_{i+1}\viiva t_i,\clf_i)\,\PMPinit(d\clf_{\tspa 0})\\
	&\quad=\int_{\Radon^n}\Bigl(\int_\Omega\Phi(\clf_{\tspa 0},\dotsc,\clf_{n-1},\Sols^\w_{t_{n-1},t_n}\clf_{n-1})\,\P(d\w)\Bigr)
	\,\prod_{i=0}^{n-2}\piMP(t_{i+1},d\clf_{i+1}\viiva t_i,\clf_i)\,\PMPinit(d\clf_{\tspa 0})\\
	&\quad=\!\int\limits_{\Radon^{n-1}}\Bigl(\int\limits_\Omega\Phi\bigl(\clf_{\tspa 0},\dotsc,\clf_{n-2},\Sols^\w_{t_{n-2},t_{n-1}}\clf_{n-2},\Sols^\w_{t_{n-2},t_n}\clf_{n-2})\,\P(d\w)\Bigr)\prod_{i=0}^{n-3}\piMP(t_{i+1},d\clf_{i+1}\viiva t_i,\clf_i)\,\PMPinit(d\clf_{\tspa 0})\\
	&\quad=\cdots
		=\int\limits_{\Radon}\int\limits_\Omega\Phi\bigl(\clf_{\tspa 0},\Sols^\w_{t_0,t_1}\clf_{\tspa 0},\dotsc,\Sols^\w_{t_0,t_n}\clf_{\tspa 0})\,\P(d\w)\,\PMPinit(d\clf_{\tspa 0})
		=E^{\P\otimes\PMPinit}[\Phi(\clf_{\tspa 0},\Sols_{t_0,t_1}\clf_{\tspa 0},\dotsc,\Sols_{t_0,t_n}\clf_{\tspa 0})].
	\end{align*}

From an initial state $\eta\in\ICM$, $t\mapsto\She(t,dx\viiva 0,\eta)$ is continuous  on $\R_+$ in the vague topology of $\sM_+(\R)$ by Theorem 2.9(i) in \cite{Alb-etal-22-spde-}. Note that the topology on $\ICM$ in that result is finer than the vague topology.
Consequently, if $\clf\in\tICP$, then the path $\{\Sols_{0,t}\clf:t\in\R_+\}$ lies in $\cC(\R_+,\tspc\Radon)$ and its distribution defines a probability measure on $\cC(\R_+,\tspc\Radon)$.
 This and part \eqref{lm:SHE-MC.ii} imply part \eqref{lm:SHE-MC.iii}.
 
 Part \eqref{lm:SHE-MC.iv}.    By  Theorem 2.9(iii) in \cite{Alb-etal-22-spde-}, the process $[s,\infty)\ni t\mapsto\She(t,\aabullet\viiva s,f)$ is continuous in the space $\CICM$.   Hence so is the process $t\mapsto\eqcl{\She(t,\aabullet\viiva s,f)}$     in the space $\CICP$.  
\end{proof}

\begin{remark}\label{rem:Feller}
One can put complete and separable metrics on $\ICM\cap\MP$
as well as on the space $\CICM$ of strictly positive continuous functions that are Radon-Nikodym derivatives of  measures in $\ICM$.  See Lemmas D.1 and D.2 in \cite{Alb-etal-22-spde-} and the proof of Lemma \ref{lem:subspace} in this paper.
Then Theorem 2.9 in \cite{Alb-etal-22-spde-} says that $\P$-almost surely, for any $t>s$, the mappings $f\mapsto\She(t,x\tsp\viiva s,f)\,dx$ on $\ICM$ and  $f\mapsto\She(t,\aabullet\tsp\viiva s,f)$ on $\CICM$ are continuous in the induced topologies. This remains true on the quotient spaces with their quotient topologies.  
In particular, by the bounded convergence theorem, 
the Markov process of Lemma \ref{lm:SHE-MC}  
is Feller continuous in the sense that the finite dimensional marginal distributions are continuous in the weak topology if the initial conditions converge in the topologies mentioned above. See Remark 2.12 in \cite{Alb-etal-22-spde-} for a similar argument.
\end{remark}

\begin{proof}[Proof of part \eqref{thm:unique.i} of Theorem \ref{thm:unique}]  
Fix $\lambda\in\R$ and  recall that $\BBus^\lambda$  has no $\lambda\pm$ distinction $\P$-a.s. Call $f_t(\aabullet)=e^{\BBus^\lambda(t,0,t,\aabullet)}$. By the propagation in \eqref{NEBus2} and the additivity in \eqref{cocycle}, for $t\ge s$, 
\be\label{mc780} 
\Sols_{s,t}\clf_s=\eqcl{\She(t,\aabullet\tsp\viiva s,f_s)}
	=\eqcl{e^{\BBus^{\lambda}(s,0,t,\aabullet)}}
	=\eqcl{e^{\BBus^{\lambda}(s,0,t,0)}\cdot e^{\BBus^{\lambda}(t,0,t,\aabullet)}}
	=\eqcl{e^{\BBus^{\lambda}(t,0,t,\aabullet)}}
	=\clf_{\tspa t}.\ee
	Now, let $\PMPinit \in \sM_1(\CICP)$ denote the distribution of $[e^{B(\abullet)+\lambda \abullet}]$, where $B$ is standard Brownian motion. Recall from Section \ref{subsec:ergodic} the notation $\PMP$  for the induced measure on the product space.
	By Theorem \ref{main:Bus}\eqref{Bproc.BM} and the computation above, the finite-dimensional distributions of the $\CICP$-valued process $(\clf_{\tspa t} : t \in \R)$ are given by $\PMP$. On the other hand, by Theorem \ref{thm:tot-erg}, $f_t(\aabullet)=e^{\BBus^\lambda(t,0,t,\aabullet)}$   is a totally ergodic  $\CICM$-valued process under time shifts.  The continuous mapping $f_t\mapsto\clf_{\tspa t}=\eqcl{f_t}\in\CICP$ to equivalence classes implies that  $\clf_{\tspa\bbullet}$   is a totally ergodic  $\CICP$-valued process under time shifts. It follows that $\PMP$ is totally ergodic under time shifts.
	\end{proof} 
	
	\begin{proof}[Proof of  part \eqref{thm:inv-n.i} of Theorem \ref{thm:inv-n}]
	The   $\nCICM$-valued process $t\mapsto(\tspb e^{\Bus^{\lambda_1}(t,0,t,\aabullet)}, \dotsc, e^{\Bus^{\lambda_n}(t,0,t,\aabullet)}\tspb) $   is stationary and  totally ergodic by Theorem \ref{thm:tot-erg}. Mapping the components to their equivalence classes gives the stationarity and  total ergodicity of the process  
  $t\mapsto\bigl(\tspb\eqcl{e^{\Bus^{\lambda_1}(t,0,t,\aabullet)}}, \dotsc, \eqcl{e^{\Bus^{\lambda_n}(t,0,t,\aabullet)}}\tspb\bigr) $.   The argument in \eqref{mc780} applies to each component.  
\end{proof}

\medskip

We turn to develop a series of auxiliary results towards proving part \eqref{thm:unique.ii} of Theorem \ref{thm:unique}.  Part \eqref{thm:inv-n.ii} of Theorem \ref{thm:inv-n} comes from a small extension of the proof.  

Let $\PMPinit\ne\delta_{\eqcl{\infimeas}}$ be a probability distribution on $\Radon$ that is invariant and ergodic for the Markov kernel \eqref{SHE-kernel}. As recorded in \eqref{pr:IC:ICMsh}, any initial condition which does not lie in $\tICP$ reaches $\eqcl{\infimeas}$ in finite time. Hence $\PMPinit(\tICP)=1$.   As is remarked between \eqref{pr:IC:1} and \eqref{pr:IC:ICMsh} and proven in Theorem 2.6 of \cite{Alb-etal-22-spde-}, any initial condition in $\tICP$ instantaneously becomes a measure in $\CICP$ and therefore $\PMPinit(\CICP)=1$. 
  
  \medskip 
  
  The  proof  constructs a new cocycle $\bbus$ from $\PMPinit$ and couples it with the fundamental solution $\She$ and the Busemann process $\Bus^{\lambda\sig}$.  The polymers defined from the cocycle $\bbus$ have asymptotic velocities by the same martingale argument as used in Section \ref{sec:poly}.  This gives  the spatial  asymptotics of  $\PMPinit$ claimed in Theorem \ref{thm:unique}\ref{P.iia}. With  these asymptotics we appeal to the Busemann limits of Corollary \ref{cor:Busp2l} to conclude the proof.  
  
 \medskip 
 
Put the product measure $\P(d\w)\otimes\PMPinit(d\clf)$ on the space $\Omega\times\CICP$.  On this space, define the process 
 $\{\Sols_{s,t}\clf: s\le t\}$  
  as in \eqref{Sst}. 
 We construct an ergodic dynamical system that couples 
  a globally defined  cocycle built from $\PMPinit$
  with the fundamental solution $\She$.  
  An explicit coupling is defined on each time interval $[S,\infty)$   and then  the joint distribution extended to all times.

For $S\in\R$, define the parameter domains 
\be\label{UUdef}\begin{aligned}   
\bbU_S&=\{(s,x,t,y): t>s\ge S; \, x,y\in\R\} \quad\,\text{ for strictly ordered times and }\\
\R^4_S&=\{ (s,x,t,y): s,t\ge S; \, x,y\in\R\}  \qquad\tsp\text{ for unordered times. } 
\end{aligned} \ee
The joint distributions will be constructed on the Polish  spaces  $\Gamma_S=\sC(\bbU_S,\R)\times\sC(\R^4_S,\R)$  
and  extended to the space     
$\Gamma=\sC(\varsets,\R)\times\sC(\R^4,\R)$   with unrestricted time coordinates.  
 On the space $\Omega\times\CICP$,  define the $\Gamma_S$-valued random pair $(\coShe_S, \bbus_S)$ as follows.  
$\coShe_S$ is the restriction of $\She$ to a $\sC(\bbU_S,\R)$-valued random variable: 
\[  \coShe^\w_S  = \{ \She^\w(t,y\viiva s,x):  t>s\ge S;\, x,y\in\R\}  . 
\]  
For any $f\in\clf\in\CICP$, $(t,x)\mapsto\She(t,x\tsp\viiva S,f)$ is strictly positive  and 
continuous on $[S,\infty)\times\R$.
 This follows from  Theorem 2.6 in \cite{Alb-etal-22-spde-}.  
Define the $\sC(\R^4_S,\R)$-valued random variable   $\bbus_S$  as a function of $(\coShe^\w_S,  \clf)$:   for $s,t\ge S$ and $x,y\in\R$,   for any representative $f\in\clf$,  
\be\label{b379}\begin{aligned}
\exp\{\bbus_S(s, x, t, y; \coShe^\w_S, \clf)\}  &= \frac{\She^\w(t,y\viiva S, f)}{\She^\w(s,x\viiva S, f)}.  
\end{aligned}\ee 
 Then  
\be\label{431}   
[S,\infty)\ni t\mapsto \eqcl{\exp\{\bbus_S(t, 0, t, \aabullet)\}} 
=   \bggeqcl{ \frac{\She(t,\aabullet\viiva S, f)}{\She(t,0\viiva S, f)}}
=   \eqcl{\She(t,\aabullet\viiva S,f)} = \Sols_{S,t}\clf 
 \ee
is the stationary $\CICP$-valued Markov process with marginal distribution $\PMPinit$.

Let   $\bfP_S$ be the distribution on   $\Gamma_S$ of  $(\coShe_S, b_S)$    under $\P(d\w)\otimes\PMPinit(d\clf)$.  
By Proposition \ref{pr:bfP3} in  Appendix \ref{app:extend},   there is a unique  probability measure $\bfP$ on $\Gamma$ whose projection to $\Gamma_S$ agrees with $\bfP_S$ for each $S\in\R$.   
The coordinate functions on $\Gamma$ are denoted by $(\bShe, \bbus)$. 
 Proposition \ref{pr:bfP3} gives also   the invariance and ergodicity of    $\bfP$ under the time translation group $\{\shiftp_u\}_{u\tsp\in\tsp\R}$ that acts on   $\Gamma$ by 
\be\label{th8} 
 (\shiftp_u\bShe)(t,y\viiva s,x) = \bShe(t+u,y\viiva s+u,x)
\quad\text{and}\quad 
(\shiftp_u\bbus)(s,x,t,y) = \bbus(s+u,x,t+u,y).
\ee
  



The marginal of $\bfP$ on $\sC(\varsets,\R)$ is the distribution of the fundamental solution $\She$ of SHE.  Hence all the properties of the fundamental solution  transfer to the coordinate function $\bShe$ on $\GaBbfP$. In particular, utilizing Corollary \ref{cor:what->w}, we can define the  Busemann process $\BBus^{\lambda\sig}$ on $\GaBbfP$ as a function of $\bShe$.

 The update rule \eqref{NEBus} and the cocycle property that $\bbus_S$ satisfies by its definition \eqref{b379} transfer to $\bbus$ and
so we have $\bfP$-almost surely:
	\begin{align}\label{NEBus-tmp}
	&e^{\bbus(s,x,t,y)}
	= \int_{-\infty}^\infty  \bShe(t,y\viiva r,z)e^{\bbus(s,x,r,z)}\, dz 
	\qquad \forall  \tspa t>r \text{  and  }  s,x,y\in \R   \qquad\text{and}  \\
	\label{NEBus245}   &  \bbus(r,x,s,y)+\bbus(s,y,t,z)=\bbus(r,x,t,z)
	\qquad \forall \tspa (r,x),(s,y),(t,z)\in\R^2. 
\end{align} 
 From \eqref{431} follows that $\eqcl{e^{\bbus(t,0,t,\aabullet)}}$ has the distribution of  the stationary Markov process with marginal distribution $\PMPinit$ and transition kernel \eqref{SHE-kernel}, now defined for all $t\in\R$.

 
Next we define quenched semi-infinite backward polymer distributions as functions of the environment $(\bShe, \bbus)$ on the probability space $(\Gamma, \cB_\Gamma, \bfP)$.   
For $(t,y)\in\R^2$ let $\Poly^{\bShe,\bbus}_{(t,y)}$ denote the distribution of the Markov process
$\{X_s:s\in(-\infty, t]\}$ with initial  point $X_t=y$ and the transition probability density from $(s,x)$ to $(r,z)$, $r<s\le t$, given by 
	\[\pi(r,z\viiva s,x)=\bShe(s,x\viiva r,z)e^{\bbus(s,x,r,z)}.\]
	That $\pi$ is a transition probability that satisfies the Chapman-Kolmogorov equations comes from \eqref{NEBus-tmp},  \eqref{NEBus245}, and identity \eqref{eq:CK} for $\bShe$. 
Since $\bfP\{  e^{\bbus(r,0,r,\aabullet)}\in\CICM \, \forall r\in\R\}=1$, 
the same proof as in 
Theorem \ref{lm:sem-pol} on page \pageref{page:lm:sem-pol} implies that $\bfP$-almost surely, $\Poly^{\bShe,\bbus}_{(t,y)}$ is supported on paths that are 
locally $\alpha$-H\"older for any $\alpha\in(0,1/2)$. 
Thus we consider  $\Poly^{\bShe,\bbus}_{(t,y)}$ as a probability distribution on the path space $\sC((-\infty, t],\R)$. We will compare measures $\Poly^{\bShe,\bbus}_{(t,y)}$ from different starting points $(t,y)$ even though they may a priori be  defined on different path spaces. We can always put them all on $\sC(\R,\R)$   by stipulating that $\Poly^{\bShe,\bbus}_{(t,y)}( X_r=y \; \forall r\in(t,\infty)) =1$. 
Without further comment,  we restrict consideration to the $\bfP$-almost every environment  $(\bShe,\bbus)$ under which the polymer measures are well-defined and all the further properties we use below are valid. 

By the proof of Theorem \ref{lm:sem-pol}\eqref{Q-LLN} on page \pageref{page:Q-LLN},   $\bfP$-almost surely, for any $s,x,t,y\in\R$, 
	\[M_r^{s,x,t,y}=\frac{\bShe(s,x\viiva r,X_r)}{\bShe(t,y\viiva r,X_r)},\quad r<s\wedge t,\]
 is a $\Poly^{\bShe,\bbus}_{(t,y)}$-backward martingale with respect to the filtration $\Paths_{-\infty:r}$. 
By martingale convergence, 
	\[M_{-\infty}^{s,x,t,y}=\lim_{r\to-\infty}\frac{\bShe(s,x\viiva r,X_r)}{\bShe(t,y\viiva r,X_r)}\,\]
exists $\Poly^{\bShe,\bbus}_{(t,y)}$-almost surely and in $L^1$. Next,  the proof of Theorem \ref{lm:sem-pol}\eqref{Q-LLN} gives that 
 \be\label{chi6}  \Xvelo=\lim_{s\to-\infty} s^{-1}X_s\;\in\;[-\infty,\infty]  \ee
 exists as  a  possibly random    
 limit   $\Poly^{\bShe,\bbus}_{(t,y)}$-almost surely, for $\bfP$-almost every $(\bShe,\bbus)$.   Lemma \ref{lm:lim<infty} below shows that $\Xvelo$  is  finite. We also use 
 \be\label{chi9} \Xsvelo=  \varlimsup_{s\to-\infty} s^{-1}{X_s}\quad\text{and}\quad     \Xivelo=  \varliminf_{s\to-\infty} s^{-1}{X_s} \ee
 to state events without reference to a particular $\Poly^{\bShe,\bbus}_{(t,y)}$.

For $\lambda\in\R$  define the events 
\be\label{defA} 
\underline A_\lambda = \{\Xivelo \ge - \lambda \}
\quad\text{ and }\quad 
  \underline\Gamma_\lambda=\bigl\{(\bShe,\bbus)\in\Gamma:  \Poly^{\bShe,\bbus}_{(0,0)}(\underline A_\lambda)>0\bigr\} . 
\ee
Note that $\underline A_\lambda$ can be viewed as a tail event on the  path space $\sC((-\infty, t],\R)$  for each $t\in\R$.
For any tail event $G\in\bigcap_{r<0}\Paths_{-\infty:r}$,  
  for all $s<t$ and $y\in\R$, 
	\begin{align}\label{harmonic}
	\Poly^{\bShe,\bbus}_{(t,y)}(G)=\int_\R\She(t,y\viiva s,x)\tspb e^{\bbus(t,y,s,x)}\Poly^{\bShe,\bbus}_{(s,x)}(G)\,dx.
	\end{align}
	Thus in a given environment $(\bShe,\bbus)$, all the polymer measures $\{\Poly^{\bShe,\bbus}_{(t,y)} : (t,y)\in\R^2\}$ are pairwise mutually absolutely continuous on the tail $\sigma$-algebra.   This applies in particular to $\underline A_\lambda$, and hence 
	 $\Poly^{\bShe,\bbus}_{(0,0)}(\underline A_\lambda)>0$ implies that  $\Poly^{\bShe,\bbus}_{(t,y)}(\underline A_\lambda)>0$ for all $(t,y)\in\R^2$.   Thereby 
	 $\shiftp^{-1}_t\underline\Gamma_\lambda=\underline\Gamma_\lambda$ $\forall t\in\R$ and by ergodicity $\bfP(\underline\Gamma_\lambda)\in\{0,1\}$. 

On the event $\underline A_\lambda$,    for any rational $\e>0$ and  large enough negative $r$, 
 $X_r \le -r(\lambda+\e)$.  Then the  comparison inequality \eqref{crossing} gives for each $x<0$, and all sufficiently large negative $r$,
	\[\frac{\bShe(t,x \viiva r,-(\lambda+\e) r)}{\bShe(t,0 \viiva r,-(\lambda+\e)r)} \cdot \ind_{\underline A_\lambda}
	\le
	\frac{\bShe(t,x\viiva r,X_r)}{\bShe(t,0\viiva r,X_r)} \cdot \ind_{\underline A_\lambda} = M_r^{t,x,t,0} \cdot \ind_{\underline A_\lambda}\le M_r^{t,x,t,0}.\]
From this and by a direct computation of the expectation,
\be\label{M7800} 
\frac{\She(t,x \viiva r,-(\lambda+\e) r)}{\She(t,0 \viiva r,-(\lambda+\e)r)}
\,\Poly_{(t,0)}^{\bShe,\bbus}(\underline A_\lambda)\le E^{\Poly^{\bShe,\bbus}}_{(t,0)}[M_r^{t,x,t,0}]=e^{\bbus(t,0,t,x)}.\ee
Consider now $\lambda$ fixed so that by Theorem \ref{main:Bus}\eqref{Bproc.-=+} there is no $\lambda\pm$ distinction. 
Take $r\to-\infty$ and apply \eqref{eq:buslim}, then
take $\e\to0$ and apply \eqref{Buscont}   to get
	\begin{align}\label{fofofo}
	e^{\BBus^{\lambda}(t,0,t,x)} \Poly_{(t,0)}^{\bShe,\bbus}(\underline A_\lambda) \le e^{\bbus(t,0,t,x)}\quad\text{for all }x<0.
	\end{align}
By the shape theorem \eqref{b-shape} for Busemann functions for a  fixed  $\lambda\in\R$,  $\bfP$-almost surely, 
\begin{align}\label{fififo}
&(\bShe,\bbus)\in\underline\Gamma_\lambda \ \Longrightarrow\forall \tspa  t\in\R:\Poly_{(t,0)}^{\bShe,\bbus}(\underline A_\lambda)>0\ \Longrightarrow \ \forall\tspa  t\in\R:\varlimsup_{x\to-\infty} x^{-1} \bbus(t,0,t,x)\le \lambda. 
\end{align}
 By the invariance of $\PMPinit$, we have established the following statement: 
\begin{align}\label{conf1}
\bfP\bigl\{\Poly_{(0,0)}^{\bShe,\bbus}(\tspb\Xvelo\ge - \lambda)>0 \bigr\} >0
\quad\text{implies}\quad
\PMPinit\Bigl\{ \clf: \varlimsup_{x\to-\infty} x^{-1} \log f(x) \le \lambda \Bigr\} = 1.
\end{align}
In the last event, $f$ can be any representative of $\clf\in\CICP$. 

 
For the next step, let $\overline A_\lambda=\{ \Xsvelo \le - \lambda \}$.
On this event, 
  for any rational $\e>0$ and $r$ large enough and negative,
$X_r\ge -(\lambda-\e) r$. 
The comparison lemma \eqref{crossing} gives for each $x<0$ and all sufficiently large negative $r,$
	\[M_r^{t,x,t,0} 
	\cdot \ind_{\overline A_\lambda}
	= \frac{\bShe(t,x\viiva r,X_r)}{\bShe(t,0\viiva r,X_r)}
	\cdot \ind_{\overline A_\lambda}
	\le \frac{\bShe(t,x \viiva r,-(\lambda-\e) r)}{\bShe(t,0 \viiva r,-(\lambda-\e) r)}
	\cdot \ind_{\overline A_\lambda}
	\,.\]
Taking $r\to-\infty$ then $\e\to0$, we see that $\Poly^{\bShe,\bbus}_{(t,0)}$-almost surely, 
	\[\lim_{r\to-\infty} M_r^{t,x,t,0} \cdot \ind_{\overline A_\lambda}
	\le e^{\BBus^{\lambda}(t,0,t,x)}\cdot \ind_{\overline A_\lambda}.\]
Compute
    \begin{align*}
    E^{\Poly^{\bShe,\bbus}_{(t,0)}}[M_r^{t,x,t,0}\tspb\ind_{\overline A_\lambda}] 
    &=E^{\Poly^{\bShe,\bbus}_{(t,0)}}\bigl[M_r^{t,x,t,0}\Poly^{\bShe,\bbus}_{(r,X_r)}(\overline A_\lambda)\bigr]\\ 
    &=\int\She(t,0\viiva r,z)e^{\bbus(t,0,r,z)}\cdot\frac{\She(t,x\viiva r,z)}{\She(t,0\viiva r,z)}\cdot\Poly^{\bShe,\bbus}_{(r,z)}(\overline A_\lambda)\,dz\\
    &=e^{\bbus(t,0,t,x)}\int \She(t,x\viiva r,z) e^{\bbus(t,x,r,z)}\Poly^{\bShe,\bbus}_{(r,z)}(\overline A_\lambda)\,dz\\
    &=e^{\bbus(t,0,t,x)}\Poly^{\bShe,\bbus}_{(t,x)}(\overline A_\lambda).
    \end{align*}
Therefore for all $x<0$,
\begin{align}\label{b<bla}
\begin{split}
e^{\bbus(t,0,t,x)} \Poly^{\bShe,\bbus}_{(t,x)}(\overline A_\lambda)
 &= \lim_{r\to-\infty} E^{\Poly^{\bShe,\bbus}_{(t,0)}}[M_r^{t,x,t,0}\tspb\ind_{\overline A_\lambda}] 
 =  E^{\Poly^{\bShe,\bbus}_{(t,0)}}\bigl[ \, \lim_{r\to-\infty} M_r^{t,x,t,0}\tspb\ind_{\overline A_\lambda}\bigr] \\
 &\le e^{\BBus^{\lambda}(t,0,t,x)}\Poly_{(t,0)}^{\bShe,\bbus}(\overline A_\lambda).
 \end{split}
\end{align}
If $\Poly^{\bShe,\bbus}_{(0,0)}(\overline A_\lambda)=1$, then   \eqref{harmonic} implies  $\Poly^{\bShe,\bbus}_{(t,x)}(\overline A_\lambda)=1$ for all $(t,x)\in\R^2$. Then, recalling that $\Xvelo=\lim_{s\to-\infty} X_s/s$ exists,  \eqref{b<bla} and ergodicity 
  lead to 
\begin{align}\label{conf2}
\bfP\bigl\{\Poly_{(0,0)}^{\bShe,\bbus}(\tspb\Xvelo\le - \lambda)=1 \bigr\}>0
\quad\text{implies}\quad\PMPinit\Bigl\{ \clf:  \varliminf_{x\to-\infty} x^{-1} \log f(x) \ge \lambda \Bigr\} = 1.
\end{align}

Repeating analogous reasoning twice more shows that 
\begin{align}\label{conf4}
\bfP\bigl\{\Poly_{(0,0)}^{\bShe,\bbus}(\tspb\Xvelo\le - \lambda)>0 \bigr\} >0
\quad\text{implies}\quad
\PMPinit\Bigl\{ \clf:  \varliminf_{x\to\infty} x^{-1} \log f(x) \ge \lambda \Bigr\} = 1
\end{align}
and that 
\begin{align}\label{conf5}
\bfP\bigl\{\Poly_{(0,0)}^{\bShe,\bbus}(\tspb\Xvelo\ge - \lambda)=1 \bigr\} >0
\quad\text{implies}\quad
\PMPinit\Bigl\{ \clf:  \varlimsup_{x\to\infty} x^{-1} \log f(x) \le \lambda \Bigr\} = 1.
\end{align}

The  left-hand sides in \eqref{conf4} and \eqref{conf5} cannot fail together.  This forces 
\[
\forall\lambda\in\R : \quad  \PMPinit\bigl\{ \clf:  \varliminf_{x\to\infty} x^{-1} \log f(x) \ge \lambda \bigr\} = 1\quad\text{or}\quad
\PMPinit\bigl\{ \clf:  \varlimsup_{x\to\infty} x^{-1} \log f(x) \le \lambda \bigr\} = 1.\]
The same works for \eqref{conf1} and \eqref{conf2}. 
Thus there exist deterministic constants     $\PMPslopep, \PMPslopem\in[-\infty,\infty]$ such that
\be\label{la-pm} \PMPinit\Bigl\{ \clf:  \lim_{x\to\infty} x^{-1} \log f(x) =\PMPslopep\Bigr\} =\PMPinit\Bigl\{ \clf:  \lim_{x\to-\infty} x^{-1} \log f(x) =\PMPslopem\Bigr\} = 1.\ee

From \eqref{conf1} and \eqref{conf4} we get that 
    \begin{align}\label{la<lim<la}
    \bfP\bigl\{\Poly_{(0,0)}^{\bShe,\bbus}(-\PMPslopep\le\Xvelo\le - \PMPslopem)=1 \bigr\} =1.
    \end{align}
In particular, $\PMPslopem\le\PMPslopep$. 

\begin{lemma}\label{lm:lim<infty}
We have
    \begin{align}\label{lim<inft}
    \bfP\bigl\{\forall (t,y)\in\R^2:\Poly_{(t,y)}^{\bShe,\bbus}(-\infty<\Xvelo<\infty)=1\bigr\}=1.
    \end{align}
\end{lemma}
\begin{proof} Suppose $\Poly_{(t,y)}^{\bShe,\bbus}(\Xvelo=-\infty)>0$ for some $(t,y)\in\R^2$.   By the mutual  absolute continuity on the tail $\sigma$-algebra,   
    \[\Poly_{(0,-1)}^{\bShe,\bbus}(\overline A_\lambda)\ge\Poly_{(0,-1)}^{\bShe,\bbus}(\Xvelo=-\infty)=\delta>0,\]
for all $\lambda\in\R$. This and \eqref{b<bla} imply that 
    \[\bbus(0,0,0,-1)+\log\delta\le\BBus^\lambda(0,0,0,-1),\]
for all $\lambda\in\R$. Since $\BBus^\lambda(0,0,0,-1)$ is normally distributed   with mean $-\lambda$ and variance $1$, letting $\lambda\to\infty$ contradicts the finiteness of the random variable $\bbus(0,0,0,-1)$. The fact that the limit of $s^{-1}X_s$ cannot be $\infty$ comes similarly, using \eqref{fofofo}.
\end{proof}

\bigskip 
 
By  \eqref{la<lim<la} and  
  \eqref{lim<inft}  we know that  $\PMPslopem\le\PMPslopep$, $\PMPslopem <\infty$ and $\PMPslopep>-\infty$.  
  Next we show   in Lemma \ref{la<infty} that  $\PMPslopem>-\infty$ and $\PMPslopep<\infty$, after several auxiliary lemmas.   Recall this identity from \eqref{M7800}, valid for $r < s \wedge t$ and $x,y \in \R$: 
\be\label{Q132}
E^{\Poly^{\She,\Bus}_{(s,x)}} [\,  M_r^{t,y,s,x} ] =  e^{\Bus(s,x,t,y)}.  
\ee

\begin{lemma} \label{lm:bb555} 
For $\bfP$-almost every $(\bShe, \bbus)$, the following holds for all $t\in\R$ and $x<y$: 
 \be\label{bb555} 
  \int_{\R} e^{\Bus^{\lambda-}(t,x,t,y)}\tspb \Poly^{\bShe,\bbus}_{(t,x)}(-\Xvelo\in d\lambda)
 \; \le \; 
e^{\bbus(t,x,t,y)}
\;  \le \; \int_{\R} e^{\Bus^{\lambda+}(t,x,t,y)}\tspb \Poly^{\bShe,\bbus}_{(t,x)}(-\Xvelo\in d\lambda). 
\ee
Furthermore,   
\be\label{bb556}   \PMPslopep =  \Poly^{\bShe,\bbus}_{(0,0)}\text{-}\esssup (-\Xvelo)   
\quad\text{and}\quad 
  \PMPslopem=  \Poly^{\bShe,\bbus}_{(0,0)}\text{-}\essinf (-\Xvelo)  . 
\ee  
\end{lemma} 

\begin{proof} 
   In a given environment $(\bShe, \bbus)$, the at most countable set of atoms of $\Xvelo$ is common to all polymer distributions $\Poly^{\bShe,\bbus}_{(t,z)}$ due to their mutual absolute continuity on the tail $\sigma$-algebra (recall \eqref{harmonic}).  
Let $t\in\R$ and $x<y$.   Fix $m\in\N$.   Choose  $\lambda_{-m}<\dotsm<\lambda_m$ so that no $\lambda_i$ is a jump of the Busemann process $\Bus^{\lambda\sig}$  and no $-\lambda_i$ is an atom of $\Xvelo$.  
 For $r<0$ let $z^r_i=-\lambda_i r=\lambda_i \abs r$. Use the comparison inequality \eqref{crossing} to write 
 \begin{align*}  
   &\ind\{ X_r<z^r_{m}\}   \tspb M_r^{t,y,t,x} 
\; = \;   \ind\{ X_r<z^r_{-m}\}   \tspb  \frac{\She(t,y\viiva r, X_r)}{\She(t,x\viiva r, X_r)} 
\   +   \sum_{i=-m+1}^m   \ind\{ z^r_{i-1} \le X_r<z^r_{i}\}   \tspb  
 \frac{\She(t,y\viiva r, X_r)}{\She(t,x\viiva r, X_r)}    \\[4pt] 
&\qquad
\le \; \ind\{ -r^{-1} X_r< \lambda_{-m}\}   \tspb  \frac{\She(t,y\viiva r, z^r_{-m})}{\She(t,x\viiva r, z^r_{-m})} \   +   \sum_{i=-m+1}^m   \ind\{ \lambda_{i-1} \le  -r^{-1}X_r<  \lambda_{i}\}   \tspb  
 \frac{\She(t,y\viiva r, z^r_{i})}{\She(t,x\viiva r, z^r_{i})}\,. 
\end{align*} 
 Under $\Poly^{\She,\Bus}_{(t,x)}$,   $X_r$ has a continuous distribution,  the finite limit $\Xvelo=\lim_{r\to-\infty}r^{-1}X_r$ exists, and no $-\lambda_i$ is an atom of the limit. Hence the indicators converge.  
Since $z^r_i/r=-\lambda_i$, we can take the limit as $r\to-\infty$ to get, for $\bfP$-almost every $(\bShe,\bbus)$ and   $\Poly^{\She,\Bus}_{(t,x)}$-almost surely,   
\begin{align*}
 \ind\{ -\Xvelo<  \lambda_{m}\}   \tspb   M_{-\infty}^{t,y,t,x}   
 \; &\le \;      \ind\{ -\Xvelo< \lambda_{-m}\}   \tspb  e^{\Bus^{\lambda_{-m}}(t,x,t,y)}  +   \sum_{i=-m+1}^m   \ind\{    \lambda_{i-1} \le - \Xvelo < \lambda_{i} \}   \tspb  
e^{\Bus^{\lambda_{i}}(t,x,t,y)} . 
 \end{align*} 
 Fix $\mu>0$.  Keep $\lambda_m=-\lambda_{-m}=\mu$  while letting  $m\nearrow\infty$ and refining  the partition so that  $\max_i(\lambda_i-\lambda_{i-1})\to0$.  For each value of $-\Xvelo\in[-\mu,\mu)$, as the partition refines, the unique $\lambda_i$ that satisfies $ \lambda_{i-1} \le - \Xvelo < \lambda_{i}$ converges:  $\lambda_i\searrow-\Xvelo$.    In the $m\nearrow\infty$ limit  we get 
 \begin{align*}
 \ind\{ -\Xvelo<  \mu\}   \tspb   M_{-\infty}^{t,y,t,x}   
 \; &\le \;      \ind\{ -\Xvelo< -\mu\}   \tspb  e^{\Bus^{-\mu}(t,x,t,y)}  +     \ind\{    -\mu \le - \Xvelo < \mu \}   \tspb    e^{\Bus^{(-\Xvelo)+}(t,x,t,y)}  \\
  \; &\le \;      \ind\{ -\Xvelo< -\mu\}   \tspb  e^{\Bus^{-\mu}(t,x,t,y)}  +     e^{\Bus^{(-\Xvelo)+}(t,x,t,y)}. 
 \end{align*} 
Let $\mu\nearrow\infty$.  For $x<y$, $  e^{\Bus^{-\mu}(t,x,t,y)}  \to 0$ as $-\mu\searrow-\infty$. 
In the limit we have 
  \be\label{bb549} 
 M_{-\infty}^{t,y,t,x} \le    e^{\Bus^{(-\Xvelo)+}(t,x,t,y)}.  
 \ee
 For the lower bound, begin with the following and take limits:
  \begin{align*}  
    M_r^{t,y,t,x} 
\; &\ge \;     \sum_{i=-m+1}^m   \ind\{ z^r_{i-1} < X_r \le z^r_{i}\}   \tspb  
 \frac{\She(t,y\viiva r, X_r)}{\She(t,x\viiva r, X_r)}    \\[4pt] 
&\ge \;    \sum_{i=-m+1}^m   \ind\{ \lambda_{i-1} <  -r^{-1}X_r \le  \lambda_{i}\}   \tspb  
 \frac{\She(t,y\viiva r, z^r_{i-1})}{\She(t,x\viiva r, z^r_{i-1})}\,. 
\end{align*} 
Let $r\to-\infty$, 
refine the partition, and  let $\lambda_{m}=-\lambda_{-m}\nearrow\infty$ to get the lower bound 
  \be\label{bb549.8} 
 M_{-\infty}^{t,y,t,x} \ge    e^{\Bus^{(-\Xvelo)-}(t,x,t,y)}.  
 \ee
The conclusion \eqref{bb555} comes from \eqref{Q132} and  by taking  $E^{\Poly^{\bShe,\bbus}_{(t,x)}}$ expectation over \eqref{bb549} and \eqref{bb549.8}.   
Identity  \eqref{bb556} follows from \eqref{bb555}.  
 \end{proof}

\begin{lemma} \label{lm:bb535}  There exists a $\bfP$-almost surely  finite random $\lambda_0>0$ such that $\P$-almost surely
\be\label{bb535}
\forall \lambda\ge \lambda_0,  \, \sigg\in\{-,+\}, \,   y\ge 0:  \; \abs{\tspb\Bus^{\lambda\sig}(0,0,0,y)-\lambda y\tspb}  \; \le \; 1+(\log\lambda + 1) y. 
\ee
\end{lemma} 

\begin{proof} 
Pick slopes  $\lambda_j\nearrow\infty$ and constants $\alpha_j>0$.  Define standard Brownian motions $B_j(y)=\Bus^{\lambda_j}(0,0,0,y)-\lambda_j y$. 
Using the exact formula for the probability that Brownian motion ever reaches a line (Eq.~(5.13) on page 197 of  \cite{Kar-Shr-91}), 
\begin{align*} 
\bfP\{  \exists y\ge 0:  \abs{B_j(y)} \ge 1+\alpha_j y \} 
&\le  \bfP\{  \exists y\ge 0:   B_j(y) \ge 1+\alpha_j y \} 
 + \bfP\{  \exists y\ge 0:  B_j(y) \le -1-\alpha_j y \}  \\
  &= 2 e^{-2\alpha_j}. 
\end{align*} 
We conclude that 
\be\label{bb530.5} \begin{aligned} 
\sum_j e^{-2\alpha_j}<\infty  \ \implies\   &\exists \text{ random }  j_0<\infty \text{ such that } \\[-8pt]   &\forall j\ge j_0, y\ge 0:   \abs{\Bus^{\lambda_j}(0,0,0,y)-\lambda_j y} \le 1+\alpha_j y. 
\end{aligned} \ee
Then for  $j\ge j_0+1$, $\lambda\in(\lambda_{j-1},  \lambda_j)$,  $\sigg\in\{-,+\}$,  $y\ge0$:  
\begin{align*}
\lambda y -(\lambda_j-\lambda_{j-1})y -\alpha_{j-1}y -1  
&\le  \lambda_{j-1} y -\alpha_{j-1}y -1 \le  \Bus^{\lambda_{j-1}}(0,0,0,y) \\
& \le  \Bus^{\lambda\sig}(0,0,0,y) \le \Bus^{\lambda_j}(0,0,0,y) 
 \le  \lambda_j y +\alpha_j y + 1  \\
 &\le \lambda y +(\lambda_j-\lambda_{j-1}) y+\alpha_jy +1 
\end{align*} 
from which 
\be\label{bb532}
 \abs{\Bus^{\lambda\sig}(0,0,0,y)-\lambda y} \le 1+(\alpha_{j-1}\vee\alpha_j + \lambda_j-\lambda_{j-1}) y. 
\ee
To get  \eqref{bb535}   choose 
  $\lambda_j=j$ and  $\alpha_j=\log (j-1)=\log\lambda_{j-1}$ for $j\ge 2$.   
\end{proof}


\begin{lemma} \label{lm:bb559}
 Suppose    $\Poly^{\bShe,\bbus}_{(0,0)}(-\Xsvelo>\lambda)>0$ for all $\lambda<\infty$.   Let $\mu\in(0,\infty)$.   Then 
 \be\label{bb559} 
 \lim_{t\to\infty}  \; \inf_{x\tspb\in\tspb[ 0, \mu], \; y\tspb\in\tspb[2\mu,\infty) } \; \frac{\bbus(0,0,0,ty)-\bbus(0,0,0,tx)}t   = \infty.  
 \ee
\end{lemma} 

\begin{proof}  
 Let $\lambda_0>0$ be large enough to satisfy Lemma \ref{lm:bb535} and also so that $\lambda\ge 4(\log\lambda+1)$ for $\lambda\ge\lambda_0$.   Then  for any $y\ge 2\mu > \mu \ge x> 0$ and $\lambda\ge\lambda_0$,  
\be\label{bb735}\begin{aligned}    \lambda (y-x)  -(x+y)(\log\lambda+1)  &\ge    \lambda (y-x) - \tfrac14  \lambda(x+y)  
=    \tfrac34   \lambda y  - \tfrac54    \lambda x \\
& \ge  \lambda(  \tfrac64  \mu   - \tfrac54  \mu) 
\ge  \tfrac14  \lambda_0 \mu. 
\end{aligned}\ee 
 Abbreviate  $\Qbar(d\lambda)= \Poly^{\bShe,\bbus}_{(0,0)}(-\Xsvelo\in d\lambda)$. 
  Since  $\lambda\mapsto\Bus^{\lambda+}(0,0,0,x)$  is increasing for  $x>0$, we have the positive correlation 
\be\label{bb709}  
\int_{(\lambda_0,\infty) } e^{\Bus^{\lambda+}(0,0,0,x) }\, \Qbar(d\lambda)   
\ge    \Qbar(\lambda_0,\infty)  \cdot  \int_{\R} e^{\Bus^{\lambda+}(0,0,0,x)}\, \Qbar(d\lambda)  . 
\ee
In the calculation below, begin and end with \eqref{bb555}.  
\begin{align*} 
e^{ \bbus(0,0,0,ty) }
&\ge    \int_{\R} e^{\Bus^{\lambda-}(0,0,0,ty)}\tspb \Qbar(d\lambda) \\[3pt] 
&\ge   \int_{(\lambda_0,\infty) } e^{\Bus^{\lambda+}(0,0,0,tx) + \Bus^{\lambda-}(0,0,0,ty)-\Bus^{\lambda+}(0,0,0,tx)}\tspb \Qbar(d\lambda)    \\[3pt] 
(\text{by \eqref{bb535}}) \qquad &\ge    \int_{(\lambda_0,\infty) } e^{\Bus^{\lambda+}(0,0,0,tx) + \lambda t(y-x) - 2 -t(x+y)(\log\lambda+1)  }\tspb \Qbar(d\lambda)     \\[3pt] 
(\text{by \eqref{bb735}}) \qquad &\ge     e^{-2 +  \frac14  \lambda_0 \mu t}  \int_{(\lambda_0,\infty) } e^{\Bus^{\lambda+}(0,0,0,tx)   }\tspb \Qbar(d\lambda)   \\[3pt] 
(\text{by \eqref{bb709}}) \qquad &\ge   e^{-2 +  \frac14  \lambda_0 \mu t}  \; \Qbar(\lambda_0,\infty)\,  \int_{\R} e^{\Bus^{\lambda+}(0,0,0,tx)}\tspb \Qbar(d\lambda)   \\[3pt] 
 &\ge   e^{-2 +  \frac14  \lambda_0 \mu t}  \; \Qbar(\lambda_0,\infty)\,   e^{\bbus(0,0,0,tx)}.  
\end{align*} 
Since $\lambda_0$ can be taken arbitrarily large, the conclusion follows. 
  \end{proof}

 \begin{lemma}\label{Qb-LDP}
The following large deviation bounds  hold $\bfP$-almost surely.
\begin{enumerate}[label={\rm(\alph*)}, ref={\rm\alph*}] \itemsep=3pt
\item\label{Qb-LDP1} If $\PMPslopep=\lambda$ and $\PMPslopem=-\lambda$ for $\lambda\in(0,\infty)$, then
for $\mu<\nu$ in $(-\lambda,\lambda)$
    \[\varlimsup_{t\to\infty}t^{-1}\log\Poly_{(t,0)}^{\bShe,\bbus}(\mu t\le X_0\le\nu t)<0,\]
and for $\nu'>\nu>\lambda$ and $\mu'<\mu<-\lambda$,  
    \begin{align*}
    \varlimsup_{t\to\infty}t^{-1}\log\Poly_{(t,0)}^{\bShe,\bbus}(\mu' t\le X_0\le\mu t)<0\quad\text{and}\quad
    \varlimsup_{t\to\infty}t^{-1}\log\Poly_{(t,0)}^{\bShe,\bbus}(\nu t\le X_0\le\nu' t)<0.
    \end{align*}
\item\label{Qb-LDP2} If $\PMPslopep=\infty$ and $\PMPslopem\in\R$, then 
for any $\mu<0$,
    \[\varlimsup_{t\to\infty}t^{-1}\log\Poly_{(t,0)}^{\bShe,\bbus}(\mu t\le X_0\le 0)<0.\]
    Similarly, if $\PMPslopem=-\infty$ and $\PMPslopep\in\R$, then for any $\nu>0$,
    \[\varlimsup_{t\to\infty}t^{-1}\log\Poly_{(t,0)}^{\bShe,\bbus}(0\le X_0\le \nu t)<0.\]
\item\label{Qb-LDP3} If $\PMPslopep=\infty$, then 
for any $\nu>0$,
    \[\varlimsup_{t\to\infty}t^{-1}\log\Poly_{(t,0)}^{\bShe,\bbus}(0\le X_0\le \nu t)=-\infty.\]
    Similarly, if $\PMPslopem=-\infty$, then for any $\mu<0$,
    \[\varlimsup_{t\to\infty}t^{-1}\log\Poly_{(t,0)}^{\bShe,\bbus}(\mu t\le X_0\le 0)=-\infty.\]
\end{enumerate}
\end{lemma}

\begin{proof}
For $t>0$ and  Borel $A\subset\R$,
    \begin{align*}
    \Poly_{(t,0)}^{\bShe,\bbus}(X_0\in A)
    &=\int_A\bShe(t,0\viiva 0,x)e^{\bbus(t,0,0,x)}\,dx
    =e^{\bbus(t,0,0,0)}\int_A\bShe(t,0\viiva 0,x)e^{\bbus(0,0,0,x)}\,dx\\
    &=\frac{\int_A\bShe(t,0\viiva 0,x)e^{\bbus(0,0,0,x)}\,dx}{\int_\R\bShe(t,0\viiva 0,x)e^{\bbus(0,0,0,x)}\,dx}
    =\frac{\int_{t^{-1}A}\bShe(t,0\viiva 0,tx)e^{\bbus(0,0,0,tx)}\,dx}{\int_\R\bShe(t,0\viiva 0,tx)e^{\bbus(0,0,0,tx)}\,dx}\,.
    \end{align*}
Recall from \eqref{la-pm}  the limits of  $x^{-1}\bbus(0,0,0,x)$ as $x\to\pm\infty$.  
When these limits are finite, we utilize them in this form: for any $C>0$,
\[\lim_{t\to\infty}t^{-1}\sup_{0\le x\le Ct}\abs{\bbus(0,0,0,x)-\PMPslopep x}=\lim_{t\to\infty}t^{-1}\sup_{-Ct\le x\le 0}\abs{\bbus(0,0,0,x)-\PMPslopem x}=0.\]

We start with the case  $\PMPslopep=-\PMPslopem=\lambda\in(0,\infty)$. Take $\nu'>\nu>\lambda$, $\e>0$, and $0<\delta<\nu-\lambda$. Then for $t$ large enough, we have $\bbus(0,0,0,tx)\le \PMPslopep tx+\e t$ for all $x\in[\nu ,\nu']$ and $\bbus(0,0,0,tx)\ge\PMPslopep tx-\e t$ for all $x\in[\lambda,\lambda+\delta]$. From the shape theorem \eqref{Z-cont} we also have that for $t$ large enough, 
$\bShe(t,0\viiva 0,tx)\le e^{\e t-t/24-x^2t/2}$ for all 
$x\in[\nu ,\nu']$ and $\bShe(t,0\viiva 0,tx)\ge e^{-\e t-t/24-x^2t/2}$ for all $x\in[\lambda,\lambda+\delta]$.  Putting all this together gives 
\begin{align*}
\Poly_{(t,0)}^{\bShe,\bbus}(\nu t\le X_0\le\nu' t)
&\le \frac{e^{4\e t}\int_\nu^{\nu'}e^{-x^2t/2+\lambda xt}\,dx}{\int_\lambda^{\lambda+\delta}e^{-x^2t/2+\lambda xt}\,dx}\\
&\le\exp\Bigl\{4\e t-\frac{\nu^2t}2+\lambda\nu t+\frac{(\lambda+\delta)^2t}2-\lambda(\lambda+\delta)t\Bigr\}\cdot\delta^{-1}(\nu'-\nu).
\end{align*}
Take logarithms, divide by $t$ and take it to $\infty$, then take $\delta\to0$ then $\e\to0$ to get
    \[\varlimsup_{t\to\infty}
    t^{-1}\log\Poly_{(t,0)}^{\bShe,\bbus}(\nu t\le X_0\le\nu' t)\le -\tspb\frac{(\nu-\lambda)^2}2<0.\]
Similarly, if $\mu'<\mu<-\lambda$, then
    \[\varlimsup_{t\to\infty}
    t^{-1}\log\Poly_{(t,0)}^{\bShe,\bbus}(\mu' t\le X_0\le\mu t)\le -\tspb\frac{(\mu-\lambda)^2}2<0.\]
Similar arguments give that if $-\lambda\le\mu<\nu<\lambda$, then
    \begin{align*}
    &\varlimsup_{t\to\infty}
    t^{-1}\log\Poly_{(t,0)}^{\bShe,\bbus}(0\le X_0\le\nu t)\le -\tspb\frac{(\nu-\lambda)^2}2<0\quad\text{and}\\
    &\varlimsup_{t\to\infty}
    t^{-1}\log\Poly_{(t,0)}^{\bShe,\bbus}(\mu t\le X_0\le0)\le -\tspb\frac{(\mu-\lambda)^2}2<0.
    \end{align*}
Part \eqref{Qb-LDP1} is proved.

\smallskip 

Next, consider the case $\PMPslopep=\infty$ and $\PMPslopem\in\R$. Take $\gamma\ge1$ such that $\gamma^2-\PMPslopem^2>6$ and take $\e\in(0,1/3)$. Then,  for $t$ large enough, we have
    \[\sup_{\mu \le x\le 0}\abs{\bbus(0,0,0,xt)-\PMPslopem xt}\le\e t\quad\text{and}\quad \inf_{x\ge1}x^{-1}\bbus(0,0,0,xt)\ge \gamma t.\]
Below, replace the integrand in the numerator by its global maximum taken at $x=\PMPslopem$.
\begin{align*}
\Poly_{(t,0)}^{\bShe,\bbus}(\mu t\le X_0\le0)
&\le \frac{e^{3\e t}\int_\mu^0 e^{-x^2t/2+\PMPslopem xt}\,dx}{\int_\gamma^{\gamma+1}e^{-x^2t/2+\gamma xt}\,dx}\\
&\le\exp\Bigl\{3\e t+\frac{\PMPslopem^2t}2+\frac{(\gamma+1)^2t}2-\gamma(\gamma+1)t\Bigr\}\cdot(-\mu)
\le \abs{\mu}e^{-t}.
\end{align*}
 A similar reasoning works for the case $\PMPslopem=-\infty$ and $\PMPslopep\in\R$ and part \eqref{Qb-LDP2} is proved.

\smallskip 

To prove part \eqref{Qb-LDP3}, assume $\PMPslopep=\infty$ without any assumptions on $\PMPslopem$.
Let  $\nu>0$.  
    \begin{align*}
      &  t^{-1}\log  \Poly_{(t,0)}^{\bShe,\bbus}(0\le X_0\le \nu t)
\le   t^{-1}\log \frac{e^{2\e t}\int_0^\nu e^{-x^2t/2+\bbus(0,0,0,tx)}\,dx}{\int_{2\nu}^{2\nu+1}e^{-y^2t/2+\bbus(0,0,0,ty)}\,dy}\\
&\qquad  \le   2\e + \sup_{\substack{0\le x\le\nu\\ 2\nu\le y\le2\nu+1}}
        \Bigl( -\frac{x^2}2+t^{-1}\bbus(0,0,0,tx)+\frac{y^2}2-t^{-1}\bbus(0,0,0,ty)\Bigr) + t^{-1}\log \nu\\
        &\qquad \le   2\e +  \frac{(2\nu+1)^2}2-
        \inf_{\substack{0\le x\le\nu\\ 2\nu\le y\le2\nu+1}}
        \frac{\bbus(0,0,0,ty)-\bbus(0,0,0,tx)}t + t^{-1}\log \nu \; \underset{t\to\infty} \longrightarrow-\infty  .
    \end{align*}
   The last limit follows from  Lemma \ref{lm:bb559} because,  
  by Lemma \ref{lm:bb555},    $\PMPslopep=\infty$ implies that   $-\Xvelo$ is unbounded above.  
  
Again, a similar reasoning works when $\PMPslopem=-\infty$ and the lemma is proved.
\end{proof}


 \begin{lemma}\label{la<infty}
Both $\PMPslopem$ and $\PMPslopep$ are finite.
\end{lemma}

\begin{proof}
Recall that $\PMPslopem\le\PMPslopep$, $\PMPslopem<\infty$, and $\PMPslopep>-\infty$.
We prove that $\PMPslopep<\infty$, the proof of $\PMPslopem>-\infty$ being similar.
So assume  $\PMPslopep=\infty$. Lemma \ref{Qb-LDP}\eqref{Qb-LDP2}--\eqref{Qb-LDP3} give  $\bfP$-almost surely, for any $\mu<0<\nu$, 
    \[\varlimsup_{t\to\infty}t^{-1}\log\Poly_{(t,0)}^{\bShe,\bbus}(\mu t\le X_0\le \nu t)<0.\]
Use the temporal shift invariance of $\bfP$ to write, for  $\mu<0<\nu$,
    \begin{align*}
    &\bfE[\Poly_{(0,0)}^{\bShe,\bbus}(-\nu< \Xvelo<-\mu)]
    \le\bfE \oE^{\Poly_{(0,0)}^{\bShe,\bbus}}\Bigl[\,\varliminf_{r\to-\infty}\one\{-\mu r\le X_r\le-\nu r\}\Bigr]\\
    &\quad\le\varliminf_{r\to-\infty}\bfE \Poly_{(0,0)}^{\bShe,\bbus}(-\mu r\le X_r\le-\nu r)
    =\varliminf_{t\to\infty}\bfE \Poly_{(t,0)}^{\bShe,\bbus}(\mu t\le X_0\le \nu t)=0.
    \end{align*}
This contradicts the finiteness of the limiting velocity $\Xvelo$, proved in Lemma \ref{lm:lim<infty}.
\end{proof}

We are ready to complete the proofs of Theorems \ref{thm:unique} and \ref{thm:inv-n}. 

\begin{proof}[Proof of part \eqref{thm:unique.ii} of Theorem \ref{thm:unique}]  
The existence of left and right asymptotic slopes was established in \eqref{la-pm}.
The combination of \eqref{la<lim<la} and Lemmas 
  \ref{lm:lim<infty}  and \ref{la<infty} show that $-\infty<\PMPslopem\le\PMPslopep<\infty$.  
   Part \ref{P.iia} of  Theorem \ref{thm:unique} has been  proved. 
 
\medskip 


To prove   part \ref{P.iib} of  Theorem \ref{thm:unique}, assume that  either $\PMPslopep+
\PMPslopem\ne0$  or  $\PMPslopem=\PMPslopep=0$.   Define $\lambda\in\R$ as follows:
\begin{itemize} 
\item If $\PMPslopep+\PMPslopem>0$  and therefore $\PMPslopep>0$, 
 let $\lambda=\PMPslopep>0$.   
\item If $\PMPslopep+\PMPslopem<0$ and therefore  $\PMPslopem<0$, 
  let $\lambda=\PMPslopem<0$.
\item If $\PMPslopem=\PMPslopep=0$, let $\lambda=0$. 
\end{itemize}

 
In  these cases   all  $f\in\clf$ satisfy the conditions of Lemma \ref{specialW} and are thus in $\initF_\lambda$. 
Theorem \ref{l2p-buslim} implies that $\P\otimes\PMPinit$-almost surely, for any $f\in\clf$,
	\be\label{bb407}
	\lim_{r\to-\infty}\frac{\int_\R\She(t,y\viiva r,z)\,f(z)\,dz}{\int_\R\She(s,x\viiva r,z)\,f(z)\,dz}=e^{\BBus^\lambda(s,x,t,y)}
	\ee
locally uniformly in $(s,x,t,y)\in\R^4$. The above ratio does not depend on the choice of  $f\in\clf$.
Since the distribution of $f(\aabullet)/f(0)$ under $\PMPinit$ is the same as that of $e^{\bbus(r,0,r,\aabullet)}$ under $\bfP$ and 
$e^{\bbus(r,0,r,\aabullet)}$ is independent of $\{\bShe(t,\aabullet\viiva r,\aabullet): t>r\}$
we get that for any $m>0$, as $r\to-\infty$, 
	\be\label{bb409}
	\biggl\{\frac{\int_\R\bShe(t,y\viiva r,z)\,e^{\bbus(r,0,r,z)}\,dz}{\int_\R\bShe(s,x\viiva r,z)\,e^{\bbus(r,0,r,z)}\,dz}:s,x,t,y\in[-m,m]\biggr\}\ee
converges weakly under $\bfP$, on the space  $\sC([-m,m]^4,\R)$,  to the distribution of $\{e^{\BBus^\lambda(s,x,t,y)}:s,x,t,y\in[-m,m]\}$.
But \eqref{NEBus-tmp} 
implies that the ratio of integrals is actually equal to $e^{\bbus(s,x,t,y)}$. Consequently, the distribution of 
$e^{\bbus}$ under $\bfP$ is the same as that of $e^{\BBus^\lambda}$ under $\P$. In particular, $\PMPinit$ is   the distribution of $\eqcl{e^{\BBus^\lambda(0,0,0,\aabullet)}}$. Since $\lim_{\abs y\to\infty} y^{-1}\BBus^\lambda(0,0,0,y) =\lambda$,  we must have   $\PMPslopem=\PMPslopep=\lambda$. 
\end{proof}

\begin{proof}[Proof of part  \eqref{thm:inv-n.ii} of Theorem \ref{thm:inv-n}]
The multivariate case follows the same reasoning. 
Start with the product measure $\P(d\w)\otimes\PMPinit^{(n)}(d\clf^{1:n})$ on the space $\Omega\times\nCICP$ and construct the time-ergodic measure $\bfP^{(n)}$ on the space 
$\Gamma^{(n)}=\sC(\varsets,\R)\times\sC(\R^4,\R)^n$ with coordinate variables 
$(\bShe, \bbus^{1:n})=(\bShe, \bbus^1,\dotsc, b^n)$.   The previous arguments apply to each $(\bShe, \bbus^i)$ marginal. 

Part  \ref{inv-n.iia} of Theorem \ref{thm:inv-n} follows from  part \ref{P.iia} of  Theorem \ref{thm:unique}. 

For part  \ref{inv-n.iib}, 
the limit \eqref{bb407} becomes now 
$\P\otimes\PMPinit^{(n)}$-almost sure locally uniform  convergence of the $n$-tuple: 
	\be\label{bb407n}
	\lim_{r\to-\infty}\biggl\{\frac{\int_\R\She(t,y\viiva r,z)\,f^i(z)\,dz}{\int_\R\She(s,x\viiva r,z)\,f^i(z)\,dz}\biggr\}_{i=1}^n=\bigl\{e^{\BBus^{\lambda_i}(s,x,t,y)}\bigr\}_{i=1}^n
	\ee
	for any $f^i\in\clf^i$. 
In the next step, as in the proof of part \eqref{thm:unique.ii} of Theorem \ref{thm:unique}, we get distributional convergence of $n$-tuples of continuous functions: 
	\begin{align*}
	&\bigl\{e^{\bbus^i(s,x,t,y)}:s,x,t,y\in[-m,m]\bigr\}_{i=1}^n
	=  \biggl\{\frac{\int_\R\bShe(t,y\viiva r,z)\,e^{\bbus^i(r,0,r,z)}\,dz}{\int_\R\bShe(s,x\viiva r,z)\,e^{\bbus^i(r,0,r,z)}\,dz}:s,x,t,y\in[-m,m]\biggr\}_{i=1}^n \\
	&\qquad\qquad 
	\underset{r\to-\infty}{\overset{d}\longrightarrow}  \bigl\{e^{\BBus^{\lambda_i}(s,x,t,y)}:s,x,t,y\in[-m,m]\bigr\}_{i=1}^n. 
\end{align*} 
Since $\PMPinit^{(n)}$ is the distribution of 
 $\{e^{\bbus^i(0,0,0,\aabullet)}\}_{i=1}^n$, we have identified $\PMPinit^{(n)}$ as the distribution of 
 $\{e^{\BBus^{\lambda_i}(0,0,0,\aabullet)}\}_{i=1}^n$. 
\end{proof}

\medskip

\section{Synchronization and one force--one solution principle}\label{sec:1F1S}

\begin{proof}[Proof of Lemma \ref{lm:u-f}]
Given a $\varphi$-invariant random variable $\clf:\Omega\to\CICP$ let $f:\Omega\to\CICM$ be the function defined by $f^\w\in\clf^\w$ and 
$f^\w(0)=1$.   
Define $\Glob:\R^2\times\Omega\to\R$ as follows: for $t\in\R$ take a rational $s<\min(t,0)$ and let
    \begin{align}\label{f->u}
    \Glob^\w(t,x)=\frac{\She^\w(t,x\viiva s,f^{\shiftp_s\w})}{\She^\w(0,0\viiva s,f^{\shiftp_s\w})}\,.
    \end{align}

Observe that by a combination of a temporal translation  \eqref{CK8} and the $\varphi$-invariance \eqref{fi-inv} of $\clf$, the equalities 
   \be\label{u999} 
\begin{aligned}    
   \frac{\She^\w(t,x\viiva r,f^{\shiftp_r\w})}{\She^\w(0,0\viiva r,f^{\shiftp_r\w})}
    &=\frac{\She^\w\bigl(t,x\bviiva s,\She^\w(s,\aabullet\viiva r,f^{\shiftp_r\w})\bigr)}{\She^\w\bigl(0,0\bviiva s,\She^\w(s,\aabullet\viiva r,f^{\shiftp_r\w})\bigr)}
    =\frac{\She^\w\bigl(t,x\bviiva s,\She^{\shiftp_r\w}(s-r,\aabullet\viiva 0,f^{\shiftp_r\w})\bigr)}{\She^\w\bigl(0,0\bviiva s,\She^{\shiftp_r\w}(s-r,\aabullet\viiva 0,f^{\shiftp_r\w})\bigr)}\\
    &
    =
    \frac{\She^\w(t,x\viiva s,f^{\shiftp_s\w})}{\She^\w(0,0\viiva s,f^{\shiftp_s\w})}  
    \end{aligned}\ee
    hold simultaneously for all $(t,x)$, 
    on a full $\P$-probability event that depends on the pair $r<s$ in $(-\infty, t)$.  We apply \eqref{u999} only for rational $r<s$ so that the null events   do not accumulate.    
Then we can conclude that  $\Glob(t,x)$ is well-defined for all $(t,x)\in\R^2$ on a single event of full $\P$-probability and its  definition \eqref{f->u} is  independent of the choice of the rational  $s<\min(t,0)$.  

The definition \eqref{f->u} and the cocycle property \eqref{CK8} of $\She$ imply that, $\P$-almost surely, for all pairs $t> s$ in $\R^2$ and all $y\in\R$, 
$\She\bigl(t,y\viiva s,\Glob(s,\aabullet)\bigr)=\Glob(t,y)$.

Next we check the  shift-covariance claim that  for each $r\in\R$ there exists an event $\Omega_r$ such that $\P(\Omega_r)=1$ and   $\eqcl{\Glob^{\shiftp_r\w}(t,\aabullet)}=\eqcl{\Glob^{\w}(t+r,\aabullet)}$ for all $\w\in\Omega_r$ and $t\in\R$.
  No proof is needed for $r=0$ so let $r\ne 0$. For each $t\in\R$ pick and fix a rational $s<\min(t,0)$. Then 
\begin{align*}
\Glob^{\shiftp_r\w}(t,x)
= \frac{\She^{\shiftp_r\w}(t,x\viiva s,f^{\shiftp_{s+r}\w})}{\She^{\shiftp_r\w}(0,0\viiva s,f^{\shiftp_{s+r}\w})}
= \frac{\She^\w(t+r,x\viiva s+r,f^{\shiftp_{s+r}\w})}{\She^\w(r,0\viiva s+r,f^{\shiftp_{s+r}\w})}= \frac{\Glob^\w(t+r,x)}{\Glob^\w(r,0)}\,. 
\end{align*}
The first equality above is the definition \eqref{f->u}  of $\Glob$. The second equality is the shift-covariance \eqref{eq:cov.shift} that holds on an $r$-dependent full-probability event, simultaneously for all real $t$ and rational $s<\min(t,0)$.  The third equality is a combination of  \eqref{f->u} and \eqref{u999}.   \eqref{u999} is used here, at the expense of another $r$-dependent null event, in case $s+r$ is not rational. The properties of $\Glob$ claimed in Lemma  \ref{lm:u-f} have been checked. 

Lastly, we verify that the equivalence class of $\Glob^{\w}(0,\aabullet)$ recovers $\clf^\w$ $\P$-almost surely. Fix a rational $s<0$.
\begin{align*}
\eqcl{\Glob^{\w}(0,\aabullet)}  
=\eqcl{\She^\w(0,\aabullet\viiva s,f^{\shiftp_s\w})}
\overset{\eqref{eq:cov.shift}}=\eqcl{\She^{\shiftp_s\w}(-s,\aabullet\viiva 0,f^{\shiftp_s\w})}
=\varphi(-s,\shiftp_s\w,\clf^{\shiftp_s\w})= \clf^{\w},
\end{align*}
$\P$-almost surely.
The last step is from \eqref{fi-inv} and holds almost surely for a given $s$. 
 
 \medskip 

Conversely, let $\Glob$ be a $\sC(\R, \CICM)$-valued random variable on $(\Omega,\sF,\P)$ such that $\P$-almost surely $\She^\w\bigl(t,y\viiva s,\Glob^\w(s,\aabullet)\bigr)=\Glob^\w(t,y)$ for all $t>s$ and $y\in\R$, and for each given $r>0$, $\eqcl{\Glob^{\shiftp_r\w}(0,\aabullet)}=\eqcl{\Glob^{\w}(r,\aabullet)}$ $\P$-almost surely. Then  $\clf^\w=\eqcl{\Glob^\w(0,\aabullet)}$ satisfies 
\[   \varphi(t,\w,\clf^\w)=\eqcl{\She^\w(t,\aabullet\viiva 0,\Glob^\w(0,\aabullet))}
=\eqcl{\Glob^\w(t,\aabullet)}\overset{(*)}=\eqcl{\Glob^{\shiftp_t\w}(0,\aabullet)}
= \clf^{\shiftp_t\w}.\]
By the assumption on $\Glob$, equality $(*)$ above holds on a $t$-dependent full probability event. 
Thus $\clf^\w$ is   a $\varphi$-invariant random variable $\clf:\Omega\to\CICP$. 
\end{proof}

\begin{proof}[Proof of Theorem \ref{thm:1f1s}]
The measurability claim follows from Theorem \ref{main:Bus}\eqref{Bproc.indep}.
%
By Theorem  \ref{bus:exp}  there exists an event $\Omega_0$ such that $\P(\Omega_0)=1$ and for any 
$\w\in\Omega_0$, $\lambda\in\R$, and $\sigg\in\{-,+\}$, $(r,x)\mapsto\Bus^{\lambda\sig}(r,0,r,x)$ is in $\initF_\lambda$. 
By Theorem \ref{thm:bcov}\eqref{thm:bcov.i} there exists an event $\Omega_\ttset^1$ such that $\Bus^{\lambda\sig}(r,0,r,x)=\Bus^{\lambda\sig}(0,0,0,x)\circ\shiftp_r$, for all $\w\in\Omega_\ttset^1$, $r\in-\ttset$, 
$x,\lambda\in\R$, and $\sigg\in\{-,+\}$.
Consequently,
$\clf_{\lambda\sig}^\w\in\HH_{\lambda,\ttset}(\w)$, for $\w\in\Omega_0\cap\Omega_\ttset^1$.
Part \eqref{thm:1f1s.finH} is proved. Part \eqref{thm:1f1s.f>0} is immediate from the definition \eqref{f-la6} and Theorem \ref{main:Bus}\eqref{Bproc.cont}.
Next, note that, on a full $\P$-probability event, we have for $s\in\R$ and $t\ge s$
	\[\eqcl{\She\bigl(t,\aabullet\tsp\viiva s,e^{\Bus^{\lambda\sig}(s,0,s,\aabullet)}\bigr)}=\eqcl{e^{\Bus^{\lambda\sig}(s,0,t,\aabullet)}}=\eqcl{e^{\Bus^{\lambda\sig}(t,0,t,\aabullet)}}\] 
and the left-hand and right-hand sides equal, respectively, 
$\varphi(t-s,\shiftp_s\w,\clf^{\shiftp_s\w}_{\lambda\sig})$ and $\clf^{\shiftp_t\w}_{\lambda\sig}$
if $s$ and $t$ are in $\ttset$ and $\w\in\Omega_\ttset$.
This proves part \eqref{thm:1f1s.phi-inv}.


If $\clg\in\HH_{\lambda,\ttset}$, then there exists $g:\Omega\to\CICM$ such $\clg=\eqcl{g}$ and 
taking $f(r,x)=g^{\shiftp_r\w}(x)$ gives $f\in\initF_{\lambda,\ttset}$. 
The proof of Theorem \ref{l2p-buslim} works word for word if we restrict $r$ to $-\ttset$.
Then if $\lambda\not\in\Bruno$, 
applying \eqref{eq:buslim2} (with $r$ restricted to $-\ttset$) gives that $\frac{\She(0,x\viiva r,f(r,\aabullet))}{\She(0,0\viiva r,f(r,\aabullet))}$ converges,
 locally uniformly in $x$, to $e^{\Bus^\lambda(0,0,0,x)}$, as $r\to-\infty$ in $-\ttset$. 
 This 
 implies \eqref{1f1s-cv} and part \eqref{thm:1f1s.att} is proved. 
 Part \eqref{thm:1f1s.different} is in Theorem \ref{L-char}. 
 \end{proof}

 \begin{proof}[Proof of Theorem \ref{cor:1f1s}]
 Let $\clf:\Omega\to\CICP$ be a $\varphi$-invariant random variable such that for any $f\in\clf$ the distribution of $\{\log f^\w(x)-\log f^\w(0):x\in\R\}$ under $\P$ is that of a Brownian motion with drift $\lambda x$. This implies that the distribution of $\clf^\w$ under $\P$ is the same as that of $\clf^\w_\lambda$ under $\P$. We want to show that in fact the two are equal, $\P$-almost surely.
 
 For $s,x,t,y\in\R$ take any rational $r<\min(s,t)$ and any $f\in\clf$ and let
    \begin{align}\label{auxBusdef}
    \Bus(s,x,t,y)=\log\She(t,y\viiva r,f^{\shiftp_r\w}) -\log\She(s,x\viiva r,f^{\shiftp_r\w}).
    \end{align}
This definition does not depend on the choice of $f\in\clf$. It also does not depend on the choice of the rational $r$ because, similarly to \eqref{u999}, for any rational $r'<r<\min(s,t)$ we have, $\P$-almost surely, for any $s,x,t,y$
    \[\frac{\She(t,y\viiva r',f^{\shiftp_{r'}\w})}{\She(s,x\viiva r',f^{\shiftp_{r'}\w})}
    =\frac{\She(t,y\viiva r,f^{\shiftp_r\w})}{\She(s,x\viiva r,f^{\shiftp_r\w})}\,.\]
In the above computation we used \eqref{CK8} and the $\varphi$-invariance of $\clf$. 

Repeating the above construction with $\clf^\w_\lambda$ instead of $\clf$ gives $\Bus^\lambda$. 
Since $\clf^\w$ and $\clf^\w_\lambda$ have the same distribution, we get that the
 distribution of $\{\Bus(s,x,t,y):s,x,t,y\in\R\}$ under $\P$ is the same as that of $\Bus^\lambda(s,x,t,y):s,x,t,y\in\R\}$. As such, $\Bus$ satisfies the same shape theorem as $\Bus^\lambda$, namely \eqref{b-shapehat}. Then, \eqref{bus:exp1} and \eqref{bus:exp2} in Theorem \ref{bus:exp} imply that, on a full $\P$-probability event, $(r,x)\mapsto e^{\Bus(r,0,r,x)}$ is in $\initF_\lambda$. Applying \eqref{eq:buslim2} and \eqref{auxBusdef} we get that $\Bus\equiv\Bus^\lambda$, $\P$-almost surely. This implies that $\clf^\w=\clf^\w_\lambda$, for $\P$-almost every $\w$. The one force--one solution principle is thus proved.
 
 The synchronization claim follows from Theorem \ref{th:La}\eqref{th:La:b}, Lemma \ref{specialW}, and Theorem \ref{thm:1f1s}\eqref{thm:1f1s.att}.
\end{proof}

{\it Acknowledgments.} The authors thank Le Chen and Davar Khoshnevisan for valuable discussions.

\appendix 

\section{Probability space for white noise}
\label{sec:WNsetting}
The underlying assumption is that the complete Polish probability space $(\Omega,\sF,\bbP)$ supports a space-time white noise $W$ and a collection of measure-preserving automorphisms as described in Section \ref{sec:setting}.   
This section describes a standard example of a separable Hilbert space (a negative index Hermite-Sobolev space) satisfying the required  hypotheses. This is essentially Example 2 in \cite[Section 4]{Wal-86}. In this setting, verifying our hypotheses is relatively simple using the spectral theory of the quantum harmonic oscillator.  We follow \cite[Section 6.4]{Sim-15} and its notation to supply some of the details missing  from  \cite{Wal-86}.

Define the Hermite functions for $x \in \bbR$ by 
\begin{align*}
e_0(x) = \frac{1}{\pi^{1/4}}e^{-\frac{x^2}{2}}\quad \text{ and, for }n \in \bbN,\text{ by } \quad e_n(x) =  \frac{1}{2^{n/2} \sqrt{n!}}\bigg(x-\frac{d}{dx}\bigg)^n e_0(x).
\end{align*}
\cite[Theorem 6.4.3]{Sim-15} shows that the family $\{e_n \}_{n=0}^\infty$ forms an orthonormal basis for $L^2(\bbR)$. By  \cite[Theorem 4.8.11]{Sim-15}, the family $\{e_{m,n}(x,y) = e_m(x)e_n(y) : m,n \}_{m,n=0}^\infty$ is an orthonormal basis for $L^2(\bbR^2)$. Let $H$ be the two dimensional quantum harmonic oscillator Hamiltonian,
\[
H=-\Delta + |\bfx|^2 = -\frac{\partial^2}{\partial y^2} -\frac{\partial^2}{\partial x^2} + x^2 + y^2.
\]
It follows from (6.4.38) in \cite{Sim-15}, that for $f \in \sS(\R^2)$, 
\begin{align}
\langle Hf, e_{n,m}\rangle_{L^2(\R^2)}  = 2(1+n+m)\langle f, e_{n,m}\rangle_{L^2(\R^2)}.\label{eq:HermiteEV}
\end{align}
\cite[Theorem 6.4.7]{Sim-15} shows that if $f \in \sS(\bbR^2)$ is a Schwartz function, then for all $\ell \in \bbN$,
\begin{align*}
\sum_{m,n=0}^\infty (1+m+n)^{2\ell} \langle e_{m,n}\tsp,f \rangle_{L^2(\bbR^2)}^2<\infty.
\end{align*}
Note that there is a typo in the definition of the semi-norm in equation (6.4.45) of \cite{Sim-15}, which does not depend on $\ell$ as written. See equation (6.4.33) 
for the correct form of the definition.

Define a new Hilbert space $\sH$ by taking the closure of $\sS(\bbR^2)$ in the norm (with the inner product defined by polarization)
\begin{align}\label{eq:norm1}
\|f\|_{\sH}^2 &=  \sum_{m=0}^\infty \sum_{n=0}^\infty 16(1+m+n)^4 \langle f, e_{m,n} \rangle_{L^2(\bbR^2)}^2.
\end{align}
Parseval's identity and \eqref{eq:HermiteEV} combined with \eqref{eq:norm1} imply that for $f \in \sS(\R^2)$,
\begin{align}\label{eq:norm2}
\|f\|_{\sH}^2 = \int_{\bbR^2} \bigg|\bigg(-\frac{\partial^2}{\partial y^2} -\frac{\partial^2}{\partial x^2} + x^2 + y^2\bigg)^2 f(x,y)\bigg|^2 dxdy. 
\end{align}

\begin{lemma}
$(\sH,\|\cdot\|_{\sH})$   is separable. 
$\sH$ is a  dense subset of  $L^2(\bbR^2)$, and the inclusion $\iota : (\sH,\|\cdot\|_{\sH} ) \hookrightarrow (L^2(\bbR^2),\|\cdot\|_{L^2(\bbR^2)})$ is Hilbert-Schmidt.
\end{lemma}
\begin{proof}  $\sH$ has a countable  orthonormal basis $\{f_{m,n} : m,n\in\bbZ_{\geq 0}\}$ defined by   $f_{m,n}=4^{-1}(1+m+n)^{-2} e_{m,n}$.  Separability follows.
For $f \in \sS(\bbR^2)$, we have $\|f\|_{\sH} \geq \|f\|_{L^2(\bbR^2)}$. Therefore, $\sH \subset L^2(\bbR^2)$ and the natural inclusion map $\iota$ is continuous. Density of $\sH$ in  $L^2(\bbR^2)$ follows from the density of $\sS(\bbR^2)$ in   $L^2(\bbR^2)$.  To see that $\iota$ is Hilbert-Schmidt, it suffices to observe that by Parseval's identity and the orthonormality of $\{e_{m,n} : m,n\in\bbZ_{\geq 0}\}$ in $L^2(\bbR^2)$, we have
\[ \|\iota\|_{HS}^2 = \sum_{m,n=0}^\infty \|\iota f_{m,n} \|_{L^2(\bbR^2)}^2 = \sum_{m,n=0}^\infty\frac{1}{16(1+m+n)^4}<\infty.\qedhere\]
\end{proof}
Let $\sH^*$ denote the continuous dual of $\sH$, equipped with its norm topology. Denote by $W$ the canonical (i.e., identity) random variable on $(\sH^*,\sB(\sH^*))$ and by $\varphi$ a generic element of $\sH^*$.
\begin{lemma}
There exists a probability measure $\mu$ on $(\sH^*,\sB(\sH^*))$ under which $W$ is space-time white noise. That is, for $f \in \sH$,
\begin{align*}
\oE^{\mu}[e^{iW(f)}]= \int_{\sH^*} e^{i \varphi(f) }\mu(d\varphi) = e^{-\frac{1}{2}\|f\|_{L^2(\bbR^2)}}.
\end{align*}
\end{lemma}
\begin{proof}
Existence follows from \cite[Theorem 4.1]{Wal-86}. The characteristic function identity follows from the characterization of Hilbert-space valued Gaussian random variables by their characteristic functions; see \cite[Theorem 2.2.4]{Bog-98}. The extension to $f \in L^2(\bbR^2)$ follows from the variance isometry and the fact that $\sH$ is dense in $L^2(\bbR^2)$.
\end{proof}
To verify that $\sH^*$ satisfies our hypotheses, we need to construct the automorphisms in our setting. Before doing this, we note that in the sense of equivalence of norms, we have
\begin{align}
\|f\|_{\sH}^2 \simeq \sum_{|\alpha| \leq 4}\|D^{\alpha}f\|_{L^2(\R^2)}^2 + \int(x^8+y^8)f(x,y)^2dxdy \label{eq:SobWe}
\end{align}
where for $\alpha=(k,\ell)$, $D^{\alpha}=\partial_x^k \partial_y^\ell$ denotes the partial derivative with multi-index $\alpha$ and $|\alpha|=k+\ell.$ See \cite[Claim 9.8.7]{Gar-18-v2} for the details of the $\gtrsim$ bound. The $\lesssim$ bound is easier and follows from the triangle inequality, integration by parts,  and repeated applications of Cauchy-Schwarz.

For $\varphi \in \sH^*$ and $f \in \sH$,  $ \shiftd{s}{y}\varphi(f) = \varphi(\shiftf{-s}{-y}f)$, $\refd_1\varphi  (f) = \varphi(\reff_1 f)$, $ \refd_2\varphi(f) = \varphi(\reff_2 f)$, $\sheard{s}{\nu}\varphi(f) = \varphi(\shearf{s}{-\nu} f)$, $ \did{\alpha}{\lambda}\varphi (f) = \varphi(\dif{\alpha^{-1}}{\lambda^{-1}} f)$, and $\ned \varphi(f) = \varphi(\nef f)$.  We first claim that these are continuous operators from $\sH^*$ to itself. 
\begin{lemma}
For each $\sG \in \{\shiftd{s}{y}, \refd_1,\refd_2, \sheard{s}{\nu}, \did{\alpha}{\lambda}, \ned\}$ defined above, $\sG:\sH^*\to\sH^*$ is a continuous linear operator.
\end{lemma}
\begin{proof}
Linearity follows from the definition, so it suffices to show norm boundedness. For this, it suffices to show norm boundedness of each linear map $\oG:\sH\to\sH$ where $\oG\in\{ \shiftf{-s}{-y}$, $\reff_1 \reff_2, \shearf{s}{-\nu} , \dif{\alpha^{-1}}{\lambda^{-1}},\nef \}$. 
Using \eqref{eq:norm2}, boundedness of $\reff_1,\reff_2,$ and $\nef$ are immediate because they are norm preserving. Boundedness of the remaining maps follow immediately from \eqref{eq:SobWe}.
\end{proof}
Call $\overline{\sB(\sH^*)}$ the completion of $\sB(\sH^*)$ with respect to $\mu$ and let $(\Omega,\sF,\bbP)=(\sH^*,\overline{\sB(\sH^*)},\mu)$. We will make use of the following sufficient condition for mixing. In the statement below, if $\sG:\sH^*\to\sH^*$ is a map, then $\sG$ acts on the identity random variable $W$ by $(\sG\!W)^{\varphi} = \sG\!\varphi$.
\begin{proposition}\label{prop:mixing}
Let $\sG: \sH^* \to \sH^*$ be an invertible, norm bounded, linear transformation and suppose that $(\Omega,\sF,\bbP,\sG)$ is a measure-preserving dynamical system. Then the condition
\[
\lim_{n\to\infty} \bbE[W(f)\,\sG^{n}\!W(g) ] = 0 \text{ for all }f,g\in\sH
\]
implies that $(\Omega,\sF,\bbP,\sG)$ is strongly mixing, i.e., for all $A,B\in\sF$
\[
\lim_{n\to\infty} \bbP(A \cap \sG^{-n} B) = \bbP(A)P(B).
\]
 \end{proposition}
\begin{proof}
We first note that $\sB(\sH^*) = \sigma(W(f) : f \in \sH)$ by the equivalence of the norm and weak Borel $\sigma$-algebras on $\sH^*$ \cite[Theorem 1.1]{Edg-77}. If $f_i, i=1,2,\dots$, is an orthonormal basis for $\sH$, then $\sigma(W(f) : f \in \sH)=\sigma(W(f_i) : i \in \bbN)$. By a standard approximation argument, it therefore suffices to check the mixing condition for sets $A,B \in \sigma(W(f_i) : i\leq k)$ for some $k$.

For each $n\in\bbN$, the vector $X_n = (W(f_1),\dots, W(f_k), \sG^n\!W(f_1),\dots, \sG^n\!W(f_k))$ is jointly Gaussian in $\bbR^{2k}$ with mean zero. The hypotheses and orthogonality of $f_1,\dots,f_k$ show that the covariance matrix converges to a diagonal matrix with entries $(\|f_1\|_{L^2(\bbR)}^2,\dots, \|f_k\|_{L^2(\bbR)}^2, \|f_1\|_{L^2(\bbR)}^2,\dots, \|f_k\|_{L^2(\bbR)}^2)$. Continuity of the determinant of the covariance matrices implies that for all sufficiently large $n$, therefore, $X_n$ has a density function. The hypotheses now imply that these density functions converge pointwise and therefore, by Scheffe's lemma, $X_n$ converges in total variation norm to a vector of independent Gaussians with mean zero and the above variances. The result follows.
\end{proof}
\begin{proposition}
With $\shiftd{s}{y}, \refd_1,\refd_2, \sheard{s}{\nu}, \did{\alpha}{\lambda},$ and $\ned$ as above, $(\Omega,\sF,\bbP)$ satisfies the hypotheses of section \ref{sec:prob-space}.
\end{proposition}
\begin{proof}
Since $\sH$ is a separable Hilbert space, so is $\sH^*$, and therefore $\sH^*$ is Polish. Direct computation checks that each of these transformations preserves the mean and covariance structure of $\mu$ and therefore, because the characteristic function is unchanged, each such map is measure preserving.  

Take $f,g \in \sH$. Then for $n \in \bbN$, denoting the $n$-fold composition of maps with a power superscript, we have
\begin{align*}
\bbE[\shiftd{-s}{-y}^n W(f)W(g)] &= \int_{\bbR^2} f(t+ns, x+ ny) g(t,x)dtdx \text{ and } \\
\bbE[\sheard{s}{-\nu}^n W(f) W(g)] &= \int_{\bbR^2} f(t, x+ n\nu(t-s)) g(t,x)dtdx.
\end{align*}
In each of the expressions on the right-hand side above, so long as  $(s,y) \neq (0,0)$ and $\nu \neq 0$,  respectively, the integrals can be seen to converge to zero as $n \to \infty$ for all $f,g \in L^2(\bbR^2) \supset \sH$ by  approximation by  compactly supported functions. Mixing follows from Proposition \ref{prop:mixing}.
\end{proof}

\section{Shape theorem for shift-covariant cocycles}
Suppose $\P$ is a probability measure on a Polish space $(\Omega,\sF)$. Let $d\in\N$.
Suppose $\{T_x:x\in\Z^d\}$ is a group of measurable bijections on $\Omega$: $T_x T_y=T_{x+y}$ and $T_0$ is the identity map. Let $\evec_1,\dotsc,\evec_d$ denote the canonical basis vectors of $\R^d$ and let $\zevec$ denote the zero vector and let $\onevec=\sum_{i=1}^d\evec_i$. 
Assume that for each $i\in\{1,\dotsc,d\}$,
$\P$ is invariant under the action of $T_{\evec_i}$. Let $\cI_i$ be the $\sigma$-algebra of $T_{\evec_i}$-invariant events. 

A measurable function $F:\Omega\times\R^d\times\R^d\to\R$ is called a \emph{cocycle} if 
there exists an event $\Omega_0$ such that $P(\Omega_0)=1$ and 
    \[F(\w,x,y)+F(\w,y,z)=F(\w,x,z)\quad\text{for all }x,y,z\in\R^d\text{ and all }\w\in\Omega_0.\]
The cocycle is said to be \emph{shift-covariant} if 
there exists an event $\Omega_0$ such that $\P(\Omega_0)=1$ and 
    \[F(\w,x+z,y+z)=F(T_z\w,x,y)\quad\text{for all }x,y\in\R^d,\ z\in\Z^d,\text{ and  }\w\in\Omega_z.\]

Let
    \[m(F)=\sum_{i=1}^d\E[F(0,\evec_i)\,|\,\cI_i]\,\evec_i.\]
For $x,y\in\R^d$, $x\le y$ is interpreted coordinatewise. Let $\abs{\aabullet}$ denote any norm on $\R^d$.

\begin{lemma}\label{lm:Bus54}
Let $F$ be a shift-covariant cocycle. 
Assume that for some $p>d$ and $k\in\Z^d$ we have
	\begin{align}\label{sup b}
	F(\zevec,\evec_i)\in L^p(\P)\ \ \forall i\in\{1,\dotsc,d\}\quad\text{and}\quad
	\sup_{x\tspa\in\tspa\R^d: \tspa  k\le x\le k+\onevec}\abs{F(\zevec,x)}\in L^d(\P).
	\end{align}
Then, with $\P$-probability one, for all $C>0$
	\[\lim_{n\to\infty}n^{-1}\sup_{\abs{x}\vee\abs{y}\tspa\le\tspa  Cn}\bigl|F(x,y)-m(F)\cdot(y-x)\bigr|=0.\]
\end{lemma}

\begin{proof}
We derive the result from the corresponding lattice-indexed shape theorem from the literature.
By the cocycle property, $F(x,y)=F(\zevec,y)-F(\zevec,x)$ and therefore it is enough to prove that $\P$-almost surely
	\[\lim_{n\to\infty}n^{-1}\sup_{\abs{x}\le Cn}\bigl|F(\zevec,x)-m(F)\cdot x\bigr|=0.\]

By Theorem 1 of \cite{Boi-Der-91} and Lemma B.4 of \cite{Jan-Ras-18-arxiv} 
we have
	\[\lim_{n\to\infty}n^{-1}\sup_{\ell\in\Z^d:\abs{\ell}\le Cn}\bigl|F(\zevec,\ell)-m(F)\cdot \ell\bigr|=0.\]
By assumption,  $\sup_{k\le x\le k+\onevec}\abs{F(\zevec,x)}\in L^d(\P)$. 
By the shift-invariance of $\P$ this implies that 
$A=\sup_{\zevec\le x\le \onevec}\abs{F(\zevec,x)}\in L^d(\P)$.
Hence, for any $\e>0$
\[ \sum_{x\in\Z^d} \P\{A\ge\e\abs{x}\}\le C+C\E[(A/\e)^d]<\infty, \]
 and then by
the Borel-Cantelli lemma, we have $\P$-almost surely
	\[\lim_{n\to\infty}n^{-1}\max_{\ell\in\Z^d:\abs{\ell}\le Cn}\,\sup_{\zevec\le x-\ell\le\onevec} \abs{F(\ell,x)}=0.\]
The claim of the lemma follows: 
	\begin{align*}
	&\varlimsup_{n\to\infty}n^{-1}\sup_{\abs{x}\le Cn}\bigl|F(\zevec,x)-m(F)\cdot x\bigr|\\
	&\qquad\le\lim_{n\to\infty}n^{-1}\sup_{\ell\in\Z^d,\abs{\ell}\le Cn}\bigl|F(\zevec,\ell)-m(F)\cdot\ell\bigr|
	+\lim_{n\to\infty}n^{-1}\max_{\ell\in\Z^d:\abs{\ell}\le Cn}\,\sup_{\zevec\le x-\ell\le\onevec}\abs{F(\ell,x)}=0. 
	\qedhere
	\end{align*}
\end{proof}

\section{Auxiliary results}\label{app:misc}

%
 

\subsection{Comparison principle}

\begin{lemma}\label{lm:comp}
There exists an event $\Omega_0$ such that $\P(\Omega_0)=1$ and the following statements hold for all $\w\in\Omega_0$.
 For all real $x<y$,  $s<t$, and $v<w$, 
	\begin{align}\label{crossing}
	\frac{\She(t,y\viiva s,v)}{\She(t,x\viiva s,v)}<\frac{\She(t,y\viiva s,w)}{\She(t,x\viiva s,w)}. 
	\end{align}
For all real $x<y$,  $s<t$, and $z$, and any non-negative function $f$ for which the integrals below are nonzero and finite, 
	\begin{align}\label{comp}
	\frac{\int_{-\infty}^z\She(t,y\viiva s,w)f(w)\,dw}{\int_{-\infty}^z\She(t,x\viiva s,w)f(w)\,dw}
	<\frac{\She(t,y\viiva s,z)}{\She(t,x\viiva s,z)}
	<\frac{\int_z^\infty\She(t,y\viiva s,w)f(w)\,dw}{\int_z^\infty\She(t,x\viiva s,w)f(w)\,dw}.
	\end{align}
Similarly, for all real $v<z$,  $s<t$, and $x$, and any non-negative function $f$  for which the integrals below are nonzero and  finite, 
	\begin{align}\label{comp2}
	\frac{\int_{-\infty}^x\She(t,u\viiva s,z)f(u)\,du}{\int_{-\infty}^x\She(t,u\viiva s,v)f(u)\,du}
	<\frac{\She(t,x\viiva s,z)}{\She(t,x\viiva s,v)}
	<\frac{\int_x^\infty\She(t,u\viiva s,z)f(u)\,du}{\int_x^\infty\She(t,u\viiva s,v)f(u)\,du}.
	\end{align}
\end{lemma}

\begin{proof}
According to Theorem 2.15 in \cite{Alb-etal-22-spde-}, the fundamental solution $\She$ is strictly totally positive: that is, on a single event of full probability,  for all $s<t$, $x_1<\dotsm<x_n$ and $y_1<\dotsm<y_n$,  $\det[\She(t,y_j\viiva s,x_i)]_{i,j=1}^n>0$.  The $2\times2$ case gives \eqref{crossing}.
Next, take a non-negative Borel function $f$,
multiply both sides of \eqref{crossing} by $\She(t,x\viiva s,v)f(v)$, take $w=z$ and integrate over $v<z$.    This gives the first inequality in \eqref{comp}, provided the ratios are well-defined.
The other inequalities come similarly.
\end{proof}

\subsection{Total variation distance} 	
\begin{lemma}\label{lm:var-dist}
Let $P$ and $Q$ be two probability measures on a measurable space $(\Omega,\cF)$. Suppose $Q\ll P$
and let $f=\frac{dQ}{dP}$.  
Then for any event $B\in\cF$, 
	\[  
	\frac{1}{2}\norm{P-Q}_{\rm TV}\leq\sup_{A\in\cF}\bigl[P(A)-Q(A)\bigr]\le E^P[(1-f)^+\one_B]+P(B^c).\]
\end{lemma}

\begin{proof}
First, note that for any event $A\in\cF$ we have
	\[Q(A\cap\{f\le 1\})=E^P[\one_A\one\{f\le 1\}f]\le P(A\cap\{f\le1\})\]
and 
	\[Q(A\cap\{f>1\})=E^P[\one_A\one\{f>1\}f]\ge P(A\cap\{f>1\}).\]
Next, write
\begin{align*}
P(A\cap B)-Q(A\cap B)
&=P(A\cap B\cap\{f\le 1\})-Q(A\cap B\cap\{f\le 1\})\\
&\qquad+P(A\cap B\cap\{f>1\})-Q(A\cap B\cap\{f>1\})\\
&\le P(A\cap B\cap\{f\le 1\})-Q(A\cap B\cap\{f\le 1\})\\
&=P(B\cap\{f\le 1\})-Q(B\cap\{f\le 1\})\\
&\qquad-P(A^c\cap B\cap\{f\le 1\})+Q(A^c\cap B\cap\{f\le 1\})\\
&\le P(B\cap\{f\le 1\})-Q(B\cap\{f\le 1\})=E^P[(1-f)^+\one_B].
\end{align*}
Now
\[ 
P(A)-Q(A)\le P(A\cap B)-Q(A\cap B)+P(B^c)\le E^P[(1-f)^+\one_B]+P(B^c).
\qedhere \] 
\end{proof}


\subsection{Increments of the Busemann process} \label{sub:businc}

\begin{proof}[Proof of Claim \eqref{not-indep} in Remark \ref{rk:Bus-prop}]
We work out the case $\sigg=+$, the other case being similar.
Since the process has Gaussian marginals we know that 
the increments have a finite second moment. 
Since the process has stationary increments we know that if it had independent increments, then for $\lambda<\mu$, the variance of 
$\Bus^{\mu+}(t,x,t,y)-\Bus^{\lambda+}(t,x,t,y)$ would be equal to $\sigma^2(\mu-\lambda)$, where $\sigma^2$ is the variance of 
$\Bus^{1+}(t,x,t,y)-\Bus^{0+}(t,x,t,y)$. The central limit theorem would imply then that the increments are normally distributed and then we would conclude that $\lambda\mapsto\Bus^{\lambda+}(t,x,t,y)$ is a Brownian motion with a linear drift. This would contradict the fact that this process is nondecreasing in $\lambda$. 
\end{proof}

\subsection{Membership in $\initF_\lambda$}

\begin{proof}[Proof of Lemma \ref{specialW}]
Assume $f\in\initF_\lambda$.  Consider first the case $\lambda>0$.
Then taking $x=(\lambda+\delta_0)\abs{r}$ in \eqref{W1} and \eqref{W2} (with $\delta=\delta_0$)
and sending $r\to-\infty$ shows that $x^{-1}\log g(x)\to\lambda$ as $x\to\infty$.
Taking $x=r$ in \eqref{W3} and sending $r\to-\infty$ gives $\varliminf_{x\to-\infty} x^{-1}\log g(x)\ge\mu>-\lambda$.
The case $\lambda<0$ is similar.
 For $\lambda=0$, taking $x=\delta r$ and $x=-\delta r$ in the first condition in \eqref{W''} and sending $r\to-\infty$ gives \eqref{G3}.
 
 For the other direction consider the case $\lambda>0$ and assume \eqref{G1}. 
 Choose $\mu\in\R$ to satisfy 
 $\varlimsup_{x\to-\infty}\abs x^{-1}\log g(x)\le -\mu<\lambda$. 
Let $\e>0$ and $\delta_0\in(0,\lambda)$. Choose $x_0>0$ so that
\begin{align*}
\abs{\log g(x)-\lambda x}\le\e\abs{x} \quad &\text{for } \  x\ge x_0  \\
\text{and}\qquad   \log g(x)-\mu x\le\e\abs{x}    \quad &\text{for } \     x\le -x_0.  
\end{align*}  
Let $r\le-x_0/(\lambda-\delta_0)$. Then
if $\abs{\frac{x}r+\lambda}\le\delta_0$ we have $x\ge x_0$ and $\log g(x)-\lambda x\ge-\e x\ge-\e(\lambda+\delta_0)\abs{r}$. \eqref{W1} follows.
Next, 
let  
\[r\le \min\Bigl(-\tspa \e^{-1}\!\!\sup_{0\le y\le x_0}(\log g(y)-\lambda y),0\Bigr).\] 
Then 
\begin{align*}
  \log g(x)-\lambda x\le\e\abs{r}\le\e(\abs{x}+\abs{r})   \quad &\text{for } \   0\le x\le x_0 \\
\text{and}\qquad  \log g(x)-\lambda x\le\e\abs{x}\le\e(\abs{r}+\abs{x})  \quad &\text{for } \   x\ge x_0.  
\end{align*} 
%
\eqref{W2} follows.
Finally, let  \[r\le \min\Bigl(-\e^{-1}\sup_{-x_0\le y\le 0}(\log g(y)-\mu y),0\Bigr).\] 
Then  \eqref{W3} comes from 
\begin{align*}
\log g(x)-\mu x\le\e\abs{r}\le\e(\abs{r}+\abs{x})    \quad &\text{for } \   -x_0\le x\le 0\\
\text{and}\qquad    \log g(x)-\mu x\le\e\abs{x}\le\e(\abs{r}+\abs{x})  \quad &\text{for } \  x\le-x_0.  
\end{align*}  

The case $\lambda<0$ is similar and the case $\lambda=0$ is an easier version of these arguments. %
\end{proof}

\subsection{Extension of $\bfP_S$} 
\label{app:extend} 

We prove the extension of $\bfP_S$ to an ergodic measure $\bfP$ utilized in the proof of Theorem \ref{thm:unique} in Section \ref{sec:unique}.  The result comes in Proposition \ref{pr:bfP3} below after some preliminaries.   The setting developed in Section \ref{sec:unique} is assumed. 

The generic variables on $\Gamma_S$ and $\Gamma$ are denoted here by $(\zeta, \bbvar)$ to avoid confusion with the processes $\She$ and $\bbus_S$ defined on $\Omega\times\CICP$. 
The time translation mapping $(u,\zeta, \bbvar)\mapsto(\shiftp_u\zeta,\shiftp_u\bbvar)$ of \eqref{th8}  is jointly continuous from $\R\times\Gamma$ into $\Gamma$.  
Hence for any probability measure $Q$ on $\Gamma$ that is invariant under the group $\{\shiftp_u\}_{u\in\R}$, $(\Gamma, \cB_\Gamma, Q, \shiftp_\bbullet)$ is a continuous dynamical system in the sense of \cite{DaP-Zab-96}.     We use $S\!:\!T$ to denote evolutions  restricted to the  time interval $[S,T]$:  $\Sols_{S:T}\clf =  \{\Sols_{S, t}\clf:  t\in[S,T]\}$, 
$\bbus_{S:T}=\{\bbus_S(s, x, t, y): S\le s,t\le T; \, x,y\in\R\}$, and similarly for $\coShe_{S:T}$. 
Abbreviate $\bbus_S(\coShe_S, \clf)=\bbus_S(\aabullet, \aabullet, \aabullet, \aabullet; \coShe_S, \clf)$ when the time-space variables are not explicitly needed.    Throughout, $f\in\clf$ and the choice of representative is immaterial.  


We make the cocycle property of $\bbus_S$ explicit. Let  $s,t\ge T>S$. 
\be\label{b400} \begin{aligned} 
&\exp\{\bbus_S(s, x, t, y; \coShe_S, \clf)\} = \frac{\She(t,y\viiva S,f)}{\She(s,x\viiva S,f)}
 = \frac{\She\bigl(t,y\viiva T , \She(T, \aabullet\viiva S,f)\bigr)}{\She\bigl(s,x\viiva  T , \She(T, \aabullet\viiva S,f)\bigr)}  \\
&=\exp\bigl\{\bbus_{T}\bigl(s, x, t, y; \coShe_T,  \eqcl{\She(T, \aabullet\viiva S,f)}\bigl)\bigr\} 
=\exp\bigl\{\bbus_{T}\bigl(s, x, t, y; \coShe_T,  \Sols_{S,T}\clf\bigl)\bigr\}. 
\end{aligned}\ee 
In words, $\bbus_S$ can be calculated from time $T$ onward by letting $\clf$ evolve from $S$ to $T$ and then running SHE evolution from initial condition $\Sols_{S,T}\clf$.

 We make explicit the effect of time shift on $\bbus$. Let $s,t\ge S$,  $\tau>0$. Recall from \eqref{eq:cov.shift} that $\shiftd{\tau}{0}$ denotes temporal shift on the white noise probability space.  
\be\label{b403-old} \begin{aligned}
&\exp\{(\shiftp_\tau\bbus_S)(s, x, t, y; \coShe^\w_S, \clf)\}  =\exp\{\bbus_S(s+\tau, x, t+\tau, y; \coShe^\w_S, \clf)\} \\[3pt]  
&= \frac{\She^\w\bigl(t+\tau,y\viiva S+\tau , \She^\w(S+\tau, \aabullet\viiva S,f)\bigr)}{\She^\w\bigl(s+\tau,x\viiva S+ \tau , \She^\w(S+\tau, \aabullet\viiva S,f)\bigr)}  
=\exp\bigl\{\bbus_S\bigl(s, x, t, y; \coShe^{\shiftd{\tau}{0}\w}_S,  \eqcl{\She^\w(S+\tau, \aabullet\viiva S,f)}\bigl)\bigr\} \\[3pt]
&=\exp\bigl\{\bbus_S\bigl(s, x, t, y; \coShe^{\shiftd{\tau}{0}\w}_S,  \Sols^\w_{S, S+\tau}\clf \bigl)\bigr\} . 
\end{aligned}\ee 



\begin{lemma}  Let $\PMPinit(d\clf)$ be a probability distribution on $\CICP$.   Let $S\le T< S+\tau$.   Then the conditional distribution of $\shiftp_\tau\bbus_S$ under $\P\otimes\PMPinit$, given the evolution  $(\coShe_{S:T}, \Sols_{S:T}\clf)$ over the time interval $[S,T]$, is given by the following formula for bounded Borel functions $\Phi$ on $\sC(\bbU_S,\R)$:  
\be\label{b419} 
E^{\P\otimes\PMPinit}\bigl[  \Phi(\shiftp_\tau\bbus_S) \big\vert  \coShe^\w_{S:T},  \Sols^\w_{S:T}\clf   \bigr] 
=
\int \P(d\w') \int   \piMP(S+\tau ,d\clg\viiva T, \Sols^\w_{S,T}\clf)  \tspb \Phi\bigl(\bbus_S(\coShe^{\shiftd{\tau}{0}\w'}_{S},  \clg)\bigr). 
\ee
\end{lemma} 

Note that on the right  $\w$ is inherited from the left while $\w'$ is integrated over.   

\begin{proof}  
By \eqref{b403-old},  
$  \shiftp_\tau\bbus_S(\coShe^\w_S, \clf)
= \bbus_S(\coShe^{\shiftd{\tau}{0}\w}_S,  \Sols^\w_{S, S+\tau}\clf) $. 
Note that  $\Sols^\w_{S, S+\tau}\clf$ depends only on the white noise on time interval $(S, S+\tau]$, while  $\coShe^{\shiftd{\tau}{0}\w}_S$ 
is a function of the white noise on time interval $(S+\tau, \infty)$. 
Use independence of white noise over disjoint intervals $(S,T]$, $(T, S+\tau]$ and $(S+\tau, \infty)$, the cocycle property  of $\Sols_{s,t}$ in \eqref{eq:Solsco} 
and the  transition probability in \eqref{SHE-kernel} to write
\begin{align*}
&\iint \P(d\w)  \PMPinit(d\clf) \,  \Psi(\coShe^\w_{S:T}, \Sols^\w_{S:T}\clf) \tspb   \Phi\bigl(\bbus_S(\coShe^{\shiftd{\tau}{0}\w}_S,
\Sols^\w_{S, S+\tau}\clf)\bigr) \\[3pt] 
&= \iint \P(d\w)  \PMPinit(d\clf) \,  \Psi(\coShe^\w_{S:T}, \Sols^\w_{S:T}\clf)
   \int \P(d\w'')   \int \P(d\w')    \,    \Phi\bigl(\bbus_S(\coShe^{\shiftd{\tau}{0}\w'}_{S},  
   \Sols^{\w''}_{T, S+\tau}\Sols^\w_{S,T}\clf)\bigr) \\
   &= \iint \P(d\w)  \PMPinit(d\clf) \,  \Psi(\coShe^\w_{S:T}, \Sols^\w_{S:T}\clf)
   \int   \piMP(S+\tau, d\clg\viiva T, \Sols^\w_{S,T}\clf)   \int \P(d\w')    \,    \Phi\bigl(\bbus_S(\coShe^{\shiftd{\tau}{0}\w'}_{S}, 
   \clg)\bigr)
\qedhere \end{align*}
\end{proof}  

\medskip 


 Recall that  $\bfP_S$ on   $\Gamma_S$  is the distribution of  $(\coShe_S, b_S)$    under $\P(d\w)\otimes\PMPinit(d\clf)$.

\begin{proposition}\label{pr:bfP3}  Let $\PMPinit \in\cM_1(\CICP)$ be an  invariant distribution for the Markov kernel  \eqref{SHE-kernel}.   

{\rm(a)} The  measures $\{\bfP_S: S\in\R\}$ are consistent under projections $\Gamma_S\to\Gamma_T$ for $S<T$.  There is a unique  probability measure $\bfP$ on $\Gamma$ whose projection to $\Gamma_S$ agrees with $\bfP_S$ for each $S\in\R$.  $\bfP$ is invariant   under the time translation group $\{\shiftp_u\}_{u\tsp\in\tsp\R}$.  

{\rm(b)} Assume further that $\PMPinit \in\cM_1(\CICP)$ is ergodic  for the Markov kernel \eqref{SHE-kernel}. Then  $\bfP$ is ergodic under the time translation group $\{\shiftp_u\}_{u\tsp\in\tsp\R}$.  

\end{proposition} 

\begin{proof}
\textit{Step 1. Consistency. }    Let $T>S$ and let $\Phi$ be a function on $\Gamma_{T}$.    
\begin{align*}
\int_{\Gamma_S} \Phi \, d\bfP_S 
&=  \iint \P(d\w)  \PMPinit(d\clf) \, \Phi\bigl(\coShe_S\vert_{\bbU_{T}}, \bbus_S\vert_{\R^4_{T}}(\coShe_S, \clf)\bigr) \\
&\overset{\eqref{b400}}=  \iint \P(d\w)  \PMPinit(d\clf) \, \Phi\bigl(\coShe_{T}, \bbus_{T}(\coShe_{T}, \eqcl{\She(T, \aabullet\viiva S,f)})\bigr) \\
&=  \iint \P(d\w)  \PMPinit(d\clf) \, \Phi\bigl(\coShe_{T}, \bbus_{T}(\coShe_{T}, f)\bigr) 
=\int_{\Gamma_{T}} \Phi \,d\bfP_{T} .  
\end{align*}
The penultimate equality used the independence  
of $\coShe_{T}$ and $\coShe_{S:T}$ 
and the invariance of $\PMPinit$ which gives 
\[  \P\otimes\PMPinit\{ (\w,\clf):  \eqcl{\She^\w(T, \aabullet\viiva S,f)} \in B\} =   \PMPinit(B) \qquad \text{for  Borel }  B\subset\CICP.  \] 
The projection consistency implies  that there is a unique $\bfP\in\cM_1(\Gamma)$ whose restriction to $\Gamma_S$ agrees with $\bfP_S$ for each $S\in\R$ (Corollary 8.22 in \cite{Kal-21}).  
\medskip  

\textit{Step 2. Invariance under time shift. }  It is enough to check the invariance of $\bfP_S$ under positive time shifts.   Let  $\Phi$ be a bounded measurable function on $\Gamma_S$ and $\tau>0$.    Use  the independence 
of $
\coShe^{\shiftd{\tau}{0}\w}_S$ and $
\coShe^\w_{S:S+\tau}$,
the assumed invariance of $\PMPinit$, and the shift-invariance of $\P$ to write
\begin{align*}
\int_{\Gamma_S} \Phi \, d\bfP_S\circ\shiftp_\tau^{-1}  
&=  \iint \P(d\w)  \PMPinit(d\clf) \, \Phi\bigl(
\coShe^{\shiftd{\tau}{0}\w}_S, \shiftp_\tau\bbus_S(\coShe^\w_S, \clf)\bigr) \\ 
&\overset{\eqref{b403-old}}=  \iint \P(d\w)  \PMPinit(d\clf) \, \Phi\bigl(\coShe^{\shiftd{\tau}{0}\w}_S, \bbus_S\bigl(\coShe^{\shiftd{\tau}{0}\w}_S,  \eqcl{\She^\w(S+\tau, \aabullet\viiva S,f)}\bigl)\bigr) \\ 
&=  \iint \P(d\w)  \PMPinit(d\clf) \, \Phi\bigl(\coShe^{\shiftd{\tau}{0}\w}_S, \bbus_S(\coShe^{\shiftd{\tau}{0}\w}_S,  \clf)\bigr) \\ 
&=  \iint \P(d\w)  \PMPinit(d\clf) \, \Phi\bigl(\coShe^\w_S, \bbus_S(\coShe^\w_S,  \clf)\bigr) 
= \int_{\Gamma_S} \Phi \, d\bfP_S.  \\ 
\end{align*}


\textit{Step 3. Ergodicity under time shift.}   We use 
Definition (1.1.4)  of ergodicity from \cite{DaP-Zab-96}:
\be\label{dz-erg100} \lim_{T\to\infty} \frac1T\int_0^T \bfP(A\cap \shiftp_{-t}B) \,dt =  \bfP(A)\tspb\bfP(B) 
\qquad\text{for Borel sets }  A,B\subset\Gamma.  
\ee
By approximation, it suffices to consider sets $A$ and $B$ that depend only on $(\zeta_{S:T}, \bbvar_{S:T})$ for a bounded time interval  $[S,T]\subset\R$.  Then we can replace $\bfP$ with $\bfP_S$.  

  Begin with this auxiliary calculation for  $S\le T\le S+\tau$.  
\begin{align*}
&\int_{\Gamma_S} \Phi(\zeta_{S:T}, \bbvar_{S:T}) \tspb  \Psi(\shiftp_\tau\zeta_{S:T}, \shiftp_\tau \bbvar_{S:T}) \, d\bfP_S    
\overset{\eqref{b419}}=  \iint \P(d\w)  \PMPinit(d\clf)  \, \Phi\bigl(\coShe^\w_{S:T}, \bbus_{S:T}(\coShe^\w_{S:T},  \clf)\bigr)  \\    
&\qquad \qquad   \qquad \qquad  \times 
\int   \piMP(S+\tau, d\clg\viiva T,\Sols_{S,T}^\w\clf)   \int \P(d\w')   \,  \Psi\bigl(
\coShe^{\shiftd{\tau}{0}\w'}_{S:T}, \bbus_{S:T}(
\coShe^{\shiftd{\tau}{0}\w'}_{S:T}, \clg ) \bigr)  \\
&=  \iint \P(d\w)  \PMPinit(d\clf)  \, \Phi\bigl(\coShe^\w_{S:T}, \bbus_{S:T}(\coShe^\w_{S:T},  \clf)\bigr)  \\    
&\qquad \qquad  \qquad \qquad  \times
 \int   \piMP(S+\tau, d\clg\viiva T,\Sols_{S,T}^\w\clf)   \int \P(d\w')   \,  \Psi\bigl(\coShe^{\w'}_{S:T}, \bbus_{S:T}(\coShe^{\w'}_{S:T}, \clg ) \bigr) . 
\end{align*}
In the last step, we used the $\P(d\w')$-almost sure equality $\shiftp_\tau \coShe_{S:T}^{\w'} =\coShe_{S:T}^{ \shiftd{\tau}{0}\w'}$ recorded as \eqref{eq:cov.shift} above and the shift invariance of $\P(d\w')$. 


 We verify the limit in \eqref{dz-erg100}: 
 \begin{align}
\nn &\frac1N\int_0^N  d\tau\, \bfP_S   \{ (\zeta_{S:T}, \bbvar_{S:T}) \in A, \tspb   (\shiftp_\tau\zeta_{S:T}, \shiftp_\tau \bbvar_{S:T}) \in B\}  \\
\nn &=   O(\tfrac{T-S}N) \;  + \; \iint \P(d\w)  \PMPinit(d\clf)  \,  \ind_A(\coShe^\w_{S:T}, \bbus_{S:T})  \\    
\nn &\qquad \qquad  \times  \frac1N\int_{T-S}^N  d\tau\,    \int   \piMP(S+\tau, d\clg\viiva T,\Sols_{S,T}^\w\clf)   \int \P(d\w')   \,  \ind_B\bigl(\coShe^{\w'}_{S:T}, \bbus_{S:T}(\coShe^{\w'}_{S:T}, \clg ) \bigr)   \\
\label{aux68} &=   O(\tfrac{T-S}N) \;  + \; \iint \P(d\w)  \PMPinit(d\clf)  \,  \ind_A(\coShe^\w_{S:T}, \bbus_{S:T})  \\    
\nn &\qquad \qquad    \times  \frac1N\int_{0}^{N-S+T}  d\tau\,    \int   \piMP(\tau, d\clg\viiva 0,\Sols_{S,T}^\w\clf)   \int \P(d\w')   \,  \ind_B\bigl(\coShe^{\w'}_{S:T}, \bbus_{S:T}(\coShe^{\w'}_{S:T}, \clg ) \bigr)\\[3pt] 
\label{aux78}  &\underset{N\to\infty}\longrightarrow \  
\iint \P(d\w)  \PMPinit(d\clf)  \,  \ind_A(\coShe^\w_{S:T}, \bbus_{S:T})  
 \int \PMPinit(d\clg)   \int \P(d\w')   \,  \ind_B\bigl(\coShe^{\w'}_{S:T}, \bbus_{S:T}(\coShe^{\w'}_{S:T}, \clg ) \bigr)\\[4pt] 
\nn &=\bfP_S   \{ (\zeta_{S:T}, \bbvar_{S:T}) \in A \}  \tspb  \bfP_S   \{ (\zeta_{S:T},  \bbvar_{S:T}) \in B\}  .
\end{align} 
In the second equality above, time homogeneity of the  transition probability is used to shift  $T$  to zero and then   $S-T+\tau$   is renamed $\tau$.    
The limit \eqref{aux78}   is  justified as follows.  
By Theorem 3.2.4 on page 25 of \cite{DaP-Zab-96},    an ergodic  invariant distribution $\mu$ of a stochastically continuous Markov semigroup $P_t$ on a Polish space $S$ satisfies 
\be\label{dz-erg104} \lim_{T\to\infty} \frac1T\int_0^T P_t\varphi  \,dt =  \int\varphi\,d\mu  
\quad\text{in } L^2(\mu), 
\ \forall \varphi\in L^2(\mu). 
\ee
We apply  this to the transition probability $\piMP$ of \eqref{SHE-kernel} on the Polish state space $\CICP$.  By Lemma \ref{lm:SHE-MC}\eqref{lm:SHE-MC.iv} the paths are continuous and thereby the stochastic continuity assumption of Theorem 3.2.4 of \cite{DaP-Zab-96} is satisfied.  Thus for $\Psi\in L^2(\PMPinit(d\clf))$, 
\be\label{b560} 
\lim_{N\to\infty}  \frac1N\int_{0}^{N}  d\tau\,    \int   \piMP(\tau, d\clg\viiva 0,\clf) \, \Psi(\clg) 
= \int \Psi(\clf)\,\PMPinit(d\clf)  \quad\text{in } \ L^2(\PMPinit(d\clf)).  
\ee
We take the function $\Psi(\clg) =  \int \P(d\w')   \,  \ind_B(\coShe^{\w'}_{S:T}, \bbus_{S:T}(\coShe^{\w'}_{S:T}, \clg ) )$ above. 

Now the situation at \eqref{aux68}  can be abstracted as follows:   $f_N\to c$ in $L^2(\mu)$,  $(X,Y)$ are jointly defined random variables, $0\le Y\le 1$, and $X$ has distribution $\mu$. Then 
\begin{align*}
\abs{ \,E[ \tspb Y   f_N(X)\tspa] -  EY \cdot c\,} \le 
E\bigl[ Y \tspb \abs{f_N(X)-c}\tspb\bigr]   
\le    
\bigl\{ E[ \tspb \abs{f_N(X)-c}^2\tspb]\bigr\}^{1/2}  \longrightarrow 0. 
\end{align*}  
In the application to \eqref{aux78}, $f_N$ is the average on the second line of \eqref{aux68}, $X=\Sols_{S,T}\clf$ which has distribution $\mu=\PMPinit$ by the invariance of $\PMPinit$, and $Y= \ind_A(\coShe^\w_{S:T}, \bbus_{S:T})$.
This completes the proof of Proposition \ref{pr:bfP3}.  
\end{proof}

\section{Quotient topology on strictly positive measures}\label{sec:quotient}
Recall the space $\sM_+(\R)$ of positive Radon measures on $\R$ and the space 
\[\MP= \{\zeta \in \sM_+(\R) : \supp(\zeta) = \bbR\}\] 
of these measures whose support is the whole real line.
Fix a countable dense subset $\{\varphi_j : j \in \bbN\}$ of $\sC_c(\R,\R_+)$ such that no $\varphi_j$ is identically zero   and   each open interval of $\R$  
contains the support of some $\varphi_j$. Define  the metrics
\begin{align*}
&d_{\sM_+(\R)}(\zeta,\eta) = \sum_{j=1}^\infty 2^{-j} \min\Big\{\;\Big|\int_{\R}\varphi_j \,d\zeta - \int_{\R}\varphi_j \,d\eta \Big|\,,1\Big\} \quad \text{and}\\
&d_{\MP}(\zeta,\eta) = \sum_{j=1}^\infty 2^{-j} \min\Big\{\;\Big|\int_{\R}\varphi_j \,d\zeta - \int_{\R}\varphi_j \,d\eta \Big| + \Big|\frac{1}{\int_{\R}\varphi_j\, d\zeta} - \frac{1}{\int_{\R}\varphi_j\, d\eta} \Big|\,,1\Big\},
\end{align*}
both bounded by one. 
The topology induced by $d_{\sM_+(\R)}$ is the usual vague topology on $\sM_+(\R)$.

\begin{lemma}\label{lem:subspace}
$(\MP,d_{\MP})$ is complete and separable and $d_{\MP}$ generates the subspace topology on $\MP$, viewed as a subset of $\sM_+(\R)$.
\end{lemma}
\begin{proof}
These topologies are metric and therefore sequential. Separability can be seen by adding $\e$ times a Gaussian measure to an approximation by a linear combination, with rational coefficients, of Dirac masses at rational points. 
To see the completeness, note first that a Cauchy sequence $\zeta_n$ in $(\MP,d_{\MP})$ is also Cauchy in $(\sM_+(\R), d_{\sM_+(\R)})$. Call its $d_{\sM_+(\R)}$-limit $\zeta$. Because $(\int \varphi_j\,d\zeta_n)^{-1}$ is Cauchy and  $\int \varphi_j \,d\zeta_n \to \int \varphi_j \,d\zeta$, these limits of integrals must be strictly positive  for all $j\in\N$. 
Therefore $\zeta \in \MP$. Finally,  if $\zeta_n \in \MP$ converges to $\zeta \in \MP$ vaguely, then for all $j$,
\begin{align*}
 \bigg|\int_{\R}\varphi_j d\zeta_n - \int_{\R}\varphi_j d\zeta \bigg| +  \bigg|\frac{1}{\int_{\R}\varphi_j d\zeta_n} - \frac{1}{\int_{\R}\varphi_j d\zeta} \bigg| \to 0
 \end{align*}
 and so   $d_{\MP}(\zeta_n,\zeta)\to0$. Conversely,   $d_{\MP}(\zeta_n,\zeta)\to0$ implies  $d_{\sM_+(\R)}(\zeta_n,\zeta)\to0$. Therefore, the topology induced by $d_{\MP}$ on $\MP$ is the subspace topology.
\end{proof}

Add to $\MP$ a cemetery state $\infimeas$. Call $\radon = \MP \cup\{\infimeas\}$ and define a metric on $\radon$ by setting 
$\dradon(\infimeas,\infimeas)=0$, $\dradon(\zeta,\infimeas) = 1$, and $\dradon(\zeta,\eta) = d_{\MP}(\zeta,\eta)$, for $\eta,\zeta\in\MP$. It is straightforward to see that because $\MP$ is complete and separable under $d_{\MP}$, $\radon$ is complete and separable under $\dradon$.

For $\eta,\zeta\in\MP$, recall the equivalence relation under which $\zeta\sim \eta$ if there exists a finite constant $c>0$ such that $\zeta=c\eta$, and the only element equivalent to $\infimeas$ is $\infimeas$ itself. Denote the equivalence class of $\zeta\in\radon$ by $\eqcl{\zeta}$. 
Denote the quotient space  $\radon/\!\!\sim$ by $\Radon$ and give it the quotient topology. 

\begin{remark}
If one defines the equivalence relation on $\sM_+(\R)$, then the   zero measure  $\zeromeas$  makes the quotient topology uninteresting. This is because $c \zeta \to \zeromeas$ as $c \searrow 0$ 
for all $\zeta \in \sM_+(\R)$. Since all the  measures $c \zeta$ are identified under $\sim$, the continuity of the quotient map implies that every neighborhood of $\eqcl\zeromeas$ contains the entire quotient space. 
\end{remark}

\begin{lemma}\label{lem:quotient}  The quotient topology on $\Radon$ is Polish.  There exists a homeomorphism $f$ from $\Radon$ onto a closed subspace of $\radon$ such that 
\begin{align}\label{dRadon}
	\dRadon(\eqcl{\eta},\eqcl{\zeta})=\dradon(f(\eqcl{\eta}),f(\eqcl{\zeta}))
	\end{align}
	defines a complete separable metric for the 
 quotient topology of $\Radon$.
\end{lemma}

\begin{proof}
Call $\X$ the space $\radon$ equipped with the topology generated by the metric $\dradon$. 
Let 
\begin{align*}
\Y = \big\{\zeta \in \radon\,\setminus\{\infimeas\}: 
\textstyle\int_{\R} \varphi_1\tspb d \zeta =1 \big\} \cup \{\infimeas\}
\end{align*}
inherit   the subspace topology   from $(\radon,\dradon)$.  $\Y$ is Polish, being a closed subset of a Polish space.

Define the surjection $g:\X \to \Y$ by $g(\infimeas) = \infimeas$ and $g(\zeta)= \frac{1}{\int_{\R} \varphi_1 \,d\zeta}\cdot \zeta$ for $\zeta \in \radon\,\setminus\{\infimeas\}$. 
     $g$ is the identity on $\Y$. 
     To check that $g$ is continuous, let $\zeta_n\to\zeta$ in $(\radon,\dradon)$.  If $\zeta = \infimeas$, then $\zeta_n = \infimeas$ for all sufficiently large $n$  because $\infimeas$ is isolated. If $\zeta \neq \infimeas$, then $\zeta_n \neq \infimeas$ for all sufficiently large $n$,  and for $j \in \bbN$,
\begin{align*}
\int_{\R} \varphi_j\, d\bigl(g(\zeta_n)\bigr) =  \frac{\int_{\R} \varphi_j \,d\zeta_n}{\int_{\R} \varphi_1 \,d\zeta_n} \underset{n\to\infty}\longrightarrow \frac{\int_{\R} \varphi_j\,d\zeta}{\int_{\R} \varphi_1 \,d\zeta} = \int_{\R} \varphi_j \,d\bigl(g(\zeta)\bigr).
\end{align*}
Thus in all cases,  $g(\zeta_n) \to g(\zeta)$. 

The quotient  map $p: \radon \to \Radon$ is $p(\zeta)=\eqcl{\zeta}$. Call $\X^*$ the topological space $\Radon$ equipped with the   quotient topology, i.e., the finest topology in which $p$ is continuous. Since $g$ is constant on each $\eqcl{\zeta}$,  \cite[Theorem 22.2]{Mun-00} implies that the map   $f:\X^* \to \Y$ defined by $f(\eqcl{\zeta})=g(\zeta)$  is continuous.  We have the following commuting diagram:
\[
\begin{tikzcd}
\X\arrow[d,"p" left] \arrow[rd,"g"] & \\
\X^*\arrow[r,"f" below] & \Y
\end{tikzcd}
\]
By \cite[Corollary 22.3]{Mun-00}, $f$ is a homeomorphism if and only if $g$ is a quotient map, meaning that $A \subset \Y$ is closed in the topology of $\Y$ if and only if $g^{-1}(A)$ is closed in the topology of $\X$. 

Continuity tells us that if $A$ is closed in the topology of $\Y$, then $g^{-1}(A)$ is closed in the topology of $\X$ and so it suffices to show the converse. Fix $A \subset \Y$ and suppose that $g^{-1}(A)$ is closed in $\X$. Because $\Y$ is sequential, it suffices to show that $A$ is sequentially closed. Take any sequence $\zeta_n \in A$ and suppose that $\zeta_n \to \zeta \in \Y$. Notice that $\zeta_n \in g^{-1}(\zeta_n)$ and so because $g^{-1}(A)$ is closed, we must have $\zeta \in g^{-1}(A)$. But since $\zeta \in \Y$, we have $g(\zeta)=\zeta$ and so $\zeta \in g(g^{-1}(A)) =A$.

Because $f$ is a homeomorphism,   the quotient topology of $\Radon$  is Polish. 
\end{proof}


The space $\CICM$ of \eqref{CICM} is given the metric defined for $f,g \in \CICM$ by
\be\begin{aligned}\label{d_CICM} 
d_{\CICM}(f,g) &= \sum_{m=1}^\infty 2^{-m} \bigg(1\wedge \sup_{-m \leq x \leq m}\bigg[|f(x)-g(x)| + \bigg|\frac{1}{f(x)} -\frac{1}{g(x)}\bigg|\,\bigg]\bigg)   \\
&\qquad +\sum_{m=1}^\infty  2^{-m} \biggl( 1\wedge \biggl\lvert\,\int_{\R} e^{-\tspb\frac1my^2} f(y)dy  - \int_{\R}e^{-\tspb\frac1my^2} g(y)dy \, \biggr\rvert \biggr) . 
\end{aligned}\ee
Lemma D.2 in \cite{Alb-etal-22-spde-} shows that this metric is complete and separable.

\begin{lemma}\label{lm:density} The set
	\be\label{mm8} \bigl\{\eta \in\sM_+(\R):\eta(dx) = f(x)dx \text{ for some strictly positive and continuous $f$}\bigr\}\ee 
is a Borel  subset of $\sM_+(\R)$.
\end{lemma}

\begin{proof}  The set \eqref{mm8} can be represented as follows: 
\begin{align*} 
&\Bigl\{\eta\in\sM_+(\R):   \text{the limit } \, u(q)=\lim_{n\to\infty} \tfrac12n \tspc\eta[q-n^{-1}, q+n^{-1}] \, \text{ exists in $\R_+$ for each  $q\in\Q$,}  \\
&\quad \forall k\in\N \, \text{ the  function $u$  on $[-k,k]\cap\Q$ is uniformly continuous and bounded away from zero,} \\
&\quad \text{and the continuous extension $f$ of $u$  to $\R$  satisfies } \   \int_\R\varphi_n(x)\tspa f(x)\,dx= \int_\R\varphi_n\,d\eta \ \forall n\in\N \Bigr\}. 
\end{align*}
Note that  $\int\varphi_n\tspb f \tspb dx$ can be evaluated as a Riemann integral from the values of $u$ on the rationals.
%
\end{proof}

\section{Random dynamical systems}\label{app:RDS}

In this appendix we give proofs of some well-known results for which we could not find references in the RDS literature proving the precise statements we wanted. These results build up to a proof that with the definitions in Section \ref{sub:1F1S}, $\PMPinit$-synchronization implies that the 1F1S principle holds for $\PMPinit$. 

The setting in this section is that of Section \ref{sec:1F1S}. The independence of $\fil_{-\infty:0}$ and $\fil_{0:\infty}$ implies that the RDS $\varphi$ is of white noise type under Definition 2.11 in \cite{Cra-08}.
Recall the abbreviation $\shiftp_t = \shiftd{t}{0}$.
Define the \emph{skew-product flow} (or semigroup)  $(\Theta_t)_{t\ge0}$ acting on the spaces  $\bigl(\Omega\times\CICP, \sF\otimes\sB(\CICP)\bigr)$ and $\bigl(\Omega\times\CICP, \fil\otimes\sB(\CICP)\bigr)$ by 
$\Theta_t(\w,\clf)=\bigl(\shiftp_t\w,\varphi(t,\w,\clf)\bigr)$. 
A probability measure $\overline\PMPinit$ on $\bigl(\Omega\times\CICP, \sF\otimes\sB(\CICP)\bigr)$ 
is
 \textit{invariant for the RDS $\varphi$} if its marginal on $\Omega$ is equal to $\P$ and $\overline\PMPinit$ is invariant under the action of $\Theta_t$ for all $t>0$. 
 
  For a probability measure $\PMPinit$ on $\CICP$ and $t>0$ call $\varphi(t,\w)\PMPinit\in\sM_1(\CICP)$ the probability measure determined for $\Phi:\CICP\to\R$ bounded and Borel measurable by \[\oE^{\varphi(t,\w)\PMPinit}[\Phi]=\int\Phi(\clg)\, \varphi(t,\w)\PMPinit(d\clg)=\int\Phi(\varphi(t,\w,\clg))\,\PMPinit(d\clf) = \int\Phi(\Sols_{0:t}^\w\clg)\,\PMPinit(d\clf).\]
 
For a probability measure $\overline\PMPinit$ on $\bigl(\Omega\times\CICP, \sF\otimes\sB(\CICP)\bigr)$ with marginal $\P$ on $\Omega$, 
let $\overline\PMPinit_\w$ denote its regular conditional probability, given $\sF$. 
Then $\overline\PMPinit$ is invariant for $\varphi$ if, and only if, for all $t>0$, for $\P$-almost every $\w$,
$\varphi(t,\w)\overline\PMPinit_\w=\overline\PMPinit_{\shiftp_t\w}$; that is, for any measurable set $A\subset\CICP$, $\overline\PMPinit_\w\bigl\{\clf:\varphi(t,\w,\clf)\in A\bigr\}=\overline\PMPinit_{\shiftp_t\w}(A)$. See Section 2 in \cite{Cra-01}. 
Note that the Borel $\sigma$-algebra on $\CICP$ is countably generated and so it suffices to check this equality on a countable generating $\pi$-system.  
%
%
%
If $\w\mapsto\overline\PMPinit_\w$ is $\fil_{-\infty:0}$-measurable, then $\overline\PMPinit$ is said to be {\it Markovian}. 
See Definition 6.14 in \cite{Cra-02}. 


\begin{lemma}\label{1to1}
For each probability measure $\PMPinit$ on $\CICP$ that is invariant for the  Markov process with kernel \eqref{SHE-kernel} there exists a unique Markovian  probability measure on $\bigl(\Omega\times\CICP, \sF\otimes\sB(\CICP)\bigr)$ that is invariant for $\varphi$ and whose $\CICP$-marginal is $\PMPinit$. Conversely, the $\CICP$-marginal of any Markovian invariant measure for $\varphi$ is invariant for the  Markov process with kernel \eqref{SHE-kernel}.
\end{lemma}


\begin{proof}
Suppose that $\PMPinit\in \sM_1(\CICP)$ is invariant for the Markov process with kernel \eqref{SHE-kernel},
i.e.\ $\PMP$ (defined in Section \ref{subsec:ergodic}) is invariant under the shift $\shiftp_t$ for each $t>0$. Proposition 4.2 of \cite{Cra-08} implies that 
there exists a $\fil_{-\infty:0}$-measurable random measure $\w\mapsto\mu_\w$ taking values in $\sM_1(\CICP)$ and such that, for any sequence $t_k\to\infty$, 
\begin{align}\label{def-mu}
\text{for $\P$-almost every $\w$,
$\varphi(t_k,\shiftp_{-t_k}\w)\PMPinit\to\mu_\w$  in the weak topology on $\sM_1(\CICP)$} 
\end{align}
and
\begin{align}\label{Th-inv}
\text{for each $t>0$, for $\P$-almost every $\w$, $\varphi(t,\w)\mu_\w=\mu_{\shift_t\w}$.}
\end{align}

Let $\overline{P}(d\w,d\clf)=\mu_\w(d\clf)\P(d\w)$ be the measure on $\bigl(\Omega\times\CICP, \sF\otimes\sB(\CICP)\bigr)$ which satisfies for bounded measurable $\Psi:\Omega\times \CICP\to\R$,
\begin{align}\label{skew}
\int \Psi(\w,\clg)\,\overline\PMPinit(d\w,d\clg)=\iint\Psi(\w,\clg) \,\mu_\w(d\clg)\P(d\w).
\end{align}
This identifies that $\mu_{\w}=\overline{P}_\w$ is a version of the regular conditional distribution given $\sF$. 
 \eqref{Th-inv} implies that $\overline\PMPinit$ is invariant for the RDS $\varphi$.
It also follows that $\overline\PMPinit$ has $\Omega$-marginal $\P$ and is Markovian. 
By the invariance of $\PMPinit$ for the Markov process, 
$\int\varphi(t,\shift_{-t}\w)\PMPinit\,\P(d\w)=\PMPinit$. Combined with  \eqref{def-mu}, this implies that 
$\overline\PMPinit$ has $\CICP$-marginal $\PMPinit$.  

For the uniqueness claim
consider a probability measure $\overline\PMPinit'$
on $\bigl(\Omega\times\CICP, \sF\otimes\sB(\CICP)\bigr)$ satisfying
\begin{enumerate}[label={\rm(a.\roman*)},ref={\rm(a.\roman*)}] 
\item\label{E1.i}  the $\Omega$-marginal of $\overline\PMPinit'$ is $\P$, 
\item $\overline\PMPinit'$ is  invariant under $\Theta_t$ for each $t>0$, 
\item\label{E1.iii} the $\CICP$-marginal $\overline\PMPinit'_\w$ of its regular conditional probability given $\sF$ is such that $\int\Phi(\clf)\,\overline\PMPinit'_\w(d\clf)$ is $\fil_{-\infty:0}$-measurable for each  bounded measurable $\Phi:\CICP\to\R$, and 
\item\label{E1.iv} the $\CICP$-marginal of $\overline\PMPinit'$ is $\PMPinit$.
\end{enumerate}

Take $\Phi$ bounded  and take a bounded $F:\Omega\to\R$ that is $\fil_{-t:\infty}$-measurable for some $t>0$. Then 
\begin{align*}
&\int F(\w)\Phi(\clg)\,\overline\PMPinit'(d\w,d\clg)
=\iint F(\shiftp_{t}\w)\Phi(\varphi(t,\w,\clg))\,\overline\PMPinit'(d\w,d\clg)\\
&\qquad=\iint F(\shiftp_{t}\w)\Phi(\varphi(t,\w,\clg))\,\overline\PMPinit'_\w(d\clg)\,\P(d\w)
=\iiint F(\shiftp_{t}\w)\Phi(\varphi(t,\w,\clg))\,\overline\PMPinit'_{\w'}(d\clg)\,\P(d\w)\,\P(d\w')\\
&\qquad=\iint F(\shiftp_{t}\w)\Phi(\varphi(t,\w,\clg))\,\P(d\w)\PMPinit(d\clg)
=\iint F(\w)\Phi(\varphi(t,\shiftp_{-t} \w, \clf))\,\P(d\w)\PMPinit(d\clg)\\
&\qquad=\int F(\w)\oE^{\varphi(t,\shift_{-t}\w)\PMPinit}[\Phi]\,\P(d\w).
\end{align*}
The first equality is  $\Theta_{-r}$-invariance. The third equality used the facts that $F(\shiftp_{-r}\w)\Phi(\varphi(t,\w,\clg))$ is $\fil_{0:\infty}$-measurable, $\overline\PMPinit'$ is Markovian, and $\fil_{0:\infty}$ is independent of $\fil_{-\infty:0}$. The fourth equality used the fact that $\PMPinit$ is the $\CICP$-marginal of $\overline\PMPinit'$. 
The fifth equality used the $\shift_r$-invariance of $\P$. 

Applying \eqref{def-mu} gives
\[\int F(\w)\Phi(\clg)\,\overline\PMPinit'(d\w,d\clg)=\int F(\w)\Phi(\clg)\,\overline\PMPinit(d\w,d\clg).\]
By the monotone class theorem,
the above equality holds for any bounded measurable $\Phi:\CICP\to\R$ and any bounded $\fil_{-\infty:\infty}$-measurable  $F:\Omega\to\R$.   Lastly, we extend this to any bounded $\sF$-measurable $F$ by utilizing the Markovian property of $\overline\PMPinit$ and the Markovian assumption on    $\overline\PMPinit'$, which say that  $\int\Phi(\clg)\,\overline\PMPinit_\w(d\clg)$
and
$\int\Phi(\clg)\,\overline\PMPinit'_\w(d\clg)$ are $\fil_{-\infty:0}$-measurable. 
Then 
\begin{align*}
\int F(\w)\Phi(\clg)\,\overline\PMPinit'(d\w,d\clg)
&=\int \E[F\viiva\fil_{-\infty:\infty}]\Phi(\clg)\,\overline\PMPinit'(d\w,d\clg)\\
&=\int \E[F\viiva\fil_{-\infty:\infty}]\Phi(\clg)\,\overline\PMPinit(d\w,d\clg)
=\int F(\w)\Phi(\clg)\,\overline\PMPinit(d\w,d\clg).
\end{align*}
This implies that $\overline\PMPinit'=\overline\PMPinit$ and proves the uniqueness of $\overline\PMPinit$.
%
%
The last claim is in part (ii) of \cite[Proposition 4.2]{Cra-08}.
\end{proof}

The next lemma shows how $\varphi$-invariant random variables are related to ergodic probability measures for the Markov process.

\begin{lemma}\label{lm:attr2}
Let $\clf:(\Omega,\sF)\to(\CICP,\sB(\CICP))$ be a $\varphi$-invariant random variable and let $\overline\PMPinit$ on $\bigl(\Omega\times\CICP, \sF\otimes\sB(\CICP)\bigr)$
be the distribution  of $(\w,\clf^\w)$ under $\P$.  Let  $\PMPinit$ be the distribution of $\w\mapsto\clf^\w$ under $\P$,   equivalently,  the $\CICP$-marginal of $\overline\PMPinit$.  Then the following hold.

\begin{enumerate}[label={\rm(\roman*)},ref={\rm\roman*}] 
\item\label{lm:attr2.i} $\overline\PMPinit$  
is invariant under the action of  $\Theta_t$ for each $t>0$. If $\clf^\w$ is $\fil$-measurable, then $\bigl(\Omega\times\CICP, \fil\otimes\sB(\CICP),\overline\PMPinit\bigr)$
is ergodic under $\Theta_t$ for all $t>0$. If $\clf^\w$ is only $\sF$-measurable but we assume further that $(\Omega,\sF,\P)$ is ergodic under $\shift_t$ for a given $t>0$, then $\bigl(\Omega\times\CICP, \sF\otimes\sB(\CICP),\overline\PMPinit\bigr)$ is ergodic under the action of $\Theta_t$ for that $t$. 

\item\label{lm:attr2.ii} If $\clf_1:\Omega\to\CICP$ is a $\varphi$-invariant random variable such that the distribution of $(\w,\clf_1^\w)$ under $\P$ is $\overline\PMPinit$, then $\clf=\clf_1$ $\P$-almost surely. 

\item\label{lm:attr2.iii}  
Assume $\clf$ is Markovian. Then
$\PMPinit$  is invariant and totally ergodic under the Markov kernel  \eqref{SHE-kernel}. If $\clf_1:\Omega\to\CICP$ is a Markovian $\varphi$-invariant random variable such that the distribution of $\w\mapsto\clf_1^\w$ under $\P$ is $\PMPinit$, then $\clf=\clf_1$ $\P$-almost surely. 
\end{enumerate}
\end{lemma}

\begin{proof} Part \eqref{lm:attr2.i}. 
Consider the event $\{\w:(\w,\clf^\w)\in A\}$ for a measurable $A\in\sF\otimes\sB(\CICP)$.
For each $t>0$ the $\varphi$-invariance of $\clf$
implies that $\P$-almost surely
\[   \shiftp_t^{-1}\{\w:(\w,\clf^\w)\in A\} =  \{\w:(\shiftp_t\w,\clf^{\shiftp_t\w})\in A\} =  \{\w:(\shiftp_t\w, \varphi(t,\w,\clf))\in A\} =\{\w:(\w,\clf^\w)\in \Theta_t^{-1}A\}. \]
Therefore, the invariance of $\P$ under $\shiftp_t$ implies the invariance of 
$\overline P$ under $\Theta_t$. 

For the first ergodicity claim, suppose that $A\in\fil\otimes\sB(\CICP)$ is invariant under $\Theta_t$ and $\clf^\w$ is $\fil$-measurable. Then $\{\w:(\w,\clf^\w)\in A\}\in\fil$  is invariant under $\shiftp_t$ and
the ergodicity of $(\Omega,\fil,\P)$ under $\shiftp_t$ implies that $\overline P(A)=\P\{\w:(\w,\clf^\w)\in A\}\in\{0,1\}$.
The second ergodicity claim follows similarly after replacing $\fil$ by $\sF$.

 Part \eqref{lm:attr2.ii}. The assumptions give  for any bounded measurable $\Phi:\CICP\to\R$ and $F:\Omega\to\R$, 
    \[ \int F(\w)\Phi(\clf_1^\w)\,\P(d\w)
    =\int F(\w)\Phi(\clg)\,\overline\PMPinit(d\w,d\clg)=\int F(\w)\Phi(\clf^\w)\,\P(d\w).\]
This implies that $\Phi(\clf^\w)=\Phi(\clf_1^\w)$ for $\P$-almost every $\w$. Since $\CICP$ is Polish, there exists a countable collection  of bounded measurable functions $\Phi$ that separates points. Thus  $\clf^\w=\clf_1^\w$, for $\P$-almost every $\w$.

Part \eqref{lm:attr2.iii}. 
Since $\overline\PMPinit$ is the distribution of $(\w,\clf^\w)$ under $\P$ we have that $\overline\PMPinit_\w=\delta_{\clf^\w}$, $\P$-almost surely.
If $\clf^\w$ is $\fil_{-\infty:0}$-measurable, then $\overline\PMPinit$ is Markovian and Lemma \ref{1to1} gives  the invariance of  its marginal $\PMPinit$ under  the Markov   kernel \eqref{SHE-kernel}. 

By Part \eqref{lm:attr2.i}, the distribution of $(\w,\clf_1^\w)$ under $\P$ is invariant under $\Theta_t$ for all $t>0$. It has marginals $\P$ and $\PMPinit$ and is Markovian. The uniqueness in Lemma \ref{1to1} implies then that this probability measure is $\overline\PMPinit$ and then Part \eqref{lm:attr2.ii} implies that $\clf_1^\w=\clf^\w$, $\P$-almost surely.

For the ergodicity claim we will use the ergodicity criterion in 
Theorem 3.2.4(iii) of \cite{DaP-Zab-96}.
Take any $t>0$ and 
consider a measurable set $A\subset\CICP$ such that 
$\pi(t,A\viiva 0,\clg)=\one_A(\clg)$, for $\PMPinit$-almost every $g\in\CICP$.  
This is the same as 
$\P(\Sols_{0,t}^\w \clg\in A)=\one_A(\clg)$ for $\PMPinit$-almost every $g\in\CICP$,
which in turn says that for $\P$-almost every $\w$, either  $\clf^\w\not\in A$ or 
$\int\one\{\Sols_{0,t}^{\w'}\clf^\w\in A\}\,\P(d\w')=1$.  
By the measurability condition on $\clf^\w$, this says that for $\P$-almost every $\w$, either $\clf^\w\not\in A$ or 
$\Sols_{0,t}^\w\clf^\w\in A$. But then we have
\[\overline\PMPinit\bigl\{(\Omega\times A)\setminus\Theta_t^{-1}(\Omega\times A)\bigr\}
=\P\bigl\{\clf^\w\in A\text{ and }\Sols_{0,t}^\w\clf^\w\not\in A\bigr\}=0.\]
This says that the set $\Omega\times A$ is $\overline P$-almost surely invariant 
under the action of $\Theta_t$. By Part \eqref{lm:attr2.i}, $\bigl(\Omega\times\CICP, \fil\otimes\sB(\CICP),\overline\PMPinit\bigr)$ is ergodic under this action and so 
$\PMPinit(A)=\overline\PMPinit(\Omega\times A)\in\{0,1\}$. 
\end{proof}


The last result shows that synchronization implies the 1F1S principle.

\begin{proposition}\label{syn=>uniq}
Let $\PMPinit$ be a probability measure on $\CICP$ that is invariant 
for the Markov process with kernel \eqref{SHE-kernel}.
Let $\clf:\Omega\to\CICP$ be a
$\varphi$-invariant random variable. 
%
Fix a countable subset $\ttset\subset[0,\infty)$ with $\sup\ttset=\infty$.
Suppose that there exist events $\wt\cC_0\subset\CICP$ and $\Omega_0\subset\Omega$ such that $\PMPinit(\wt\cC_0)=\P(\Omega_0)=1$, and for any $\clg\in\wt\cC_0$ and $\w\in\Omega_0$, 
\begin{align}\label{sync-aux}
\lim_{\ttset\ni t\to\infty}\dCICP\bigl(\varphi(t,\shiftp_{-t}\w,\clg),\clf^\w\bigr)=0.
\end{align}
Then $\PMPinit$ is the distribution of $\w\mapsto\clf^\w$ under $\P$ and $\clf$ is Markovian.
Furthermore, if $\clf_1:\Omega\to\CICP$ is a $\varphi$-invariant random variable such that the distribution of $\w\mapsto\clf_1^\w$ under $\P$ is $\PMPinit$, then $\clf_1^\w=\clf^\w$  for $\P$-almost every $\w$.
\end{proposition}

\begin{proof}
Let $\overline\PMPinit$ be the unique Markovian probability measure on $\bigl(\Omega\times\CICP, \sF\otimes\sB(\CICP),\overline\PMPinit\bigr)$,
from Lemma \ref{1to1}, with marginals $\P$ and $\PMPinit$. Define $F(\w,\clg)=\dCICP(\clg,\clf^\w)$.
The invariance of $\P$ under $\shiftp_t$, for all $t$, and 
the fact that \eqref{sync-aux} holds $\overline\PMPinit$-almost surely imply that 
$F\bigl(\Theta_t(\w,\clg)\bigr)\to0$ in $\overline\PMPinit$-probability, as $t\to\infty$ in $\ttset$.
Since $\overline\PMPinit$ is invariant under $\Theta_t$ for all $t>0$ we get that $\dCICP(\clg,f^\w)=F(\w,\clg)=0$, $\overline\PMPinit$-almost surely.
This implies that $\overline\PMPinit_\w=\delta_{\clf^\w}$, $\P$-almost surely. As a result, $\clf$ is Markovian and the marginal $\PMPinit$ is the distribution of $\w\mapsto\clf^\w$ under $\P$.

By assumption $\P(\clf_1^{\shiftp_{-t}\w}\in\wt\cC_0)=\PMPinit(\wt\cC_0)=1$ so there exists an event $\Omega_1\subset\Omega$ such that $\P(\Omega_1)=1$ and $\clf_1^{\shiftp_{-t}\w}\in\wt\cC_0$ for all $\w\in\Omega_1$ and $t\in\ttset$. 
Let $\Omega_2\subset\Omega$ be an event such that $\P(\Omega_2)=1$ and $\varphi(t,\shiftp_{-t}\w,\clf_1^{\shiftp_{-t}\w})=\clf_1^\w$, for all $\w\in\Omega_2$ and $t\in\ttset$.
Then for $\w\in\Omega_0\cap\Omega_1\cap\Omega_2$ we have that 
    \[\dCICP(\clf_1^\w,\clf^\w)=\dCICP\bigl(\varphi(t,\shiftp_{-t}\w,\clf_1^{\shiftp_{-t}\w}),\clf^\w\bigr)\to0\] 
as $t\to\infty$ in $\ttset$. The proposition is proved.
\end{proof}

\footnotesize

\bibliographystyle{aop-no-url}

\bibliography{KPZ-1F1S}

\end{document}